\documentclass[12pt,phd,nobabel,twoside]{dms}

\usepackage[utf8x]{inputenc}		
\usepackage[T1]{fontenc}		
\usepackage{lmodern}			
\usepackage{pslatex}

\usepackage{graphicx}			
\usepackage{amsmath, amsfonts, amssymb, amsthm}		
\usepackage[tight,heads=LaTeX]{diagrams}
\usepackage{enumerate}

\numberwithin{equation}{section}
\numberwithin{table}{chapter}
\numberwithin{figure}{chapter}

\usepackage{palatino}
\usepackage{euler}
\usepackage{amsmath}

\newtheorem{prp}{Proposition}[section]
\newtheorem*{prpnn}{Proposition}
\newtheorem{cor}[prp]{Corollary}
\newtheorem{defi}[prp]{Definition}
\newtheorem{exe}[prp]{Example}
\newtheorem{lem}[prp]{Lemma}
\newtheorem*{lemnn}{Lemma}
\newtheorem{rem}[prp]{Remark}
\newtheorem*{remnn}{Remark}
\newtheorem{thm}[prp]{Theorem}
\newtheorem*{thmnn}{Theorem}

\def\C{\mathbb{C}}
\def\N{\mathbb{N}}

\def\R{\mathbb{R}}
\def\Z{\mathbb{Z}}

\def\cO{\mathcal{O}}

\def\gl{\widetilde{Gl_n(\R)}^{\pm}}
\def\ddel{D\bar{\partial}}
\def\klk{K_{\ell,k}}
\def\ulk{U_{\ell,k}}
\def\cllkr{\mathcal{C}\ell_{\ell,k}^\otimes}
\def\ulkr{\mathcal{U}_{\ell,k}^\otimes}
\def\cllks{\mathcal{C}\ell_{\ell,k}^\bullet}
\def\ulks{\mathcal{U}_{\ell,k}^\bullet}

\def\qklk{Q_{\ell,k}}
\def\qcllkr{\mathcal{QC}\ell_{\ell,k}^\otimes}
\def\qulkr{\mathcal{QU}_{\ell,k}^\otimes}

\begin{document}






\pagenumbering{roman}

\thispagestyle{empty}
\let\cleardoublepage\clearpage

\newpage
\null
\vskip 2em
\begin{center}
{\LARGE\bfseries \centering \openup\medskipamount Source Spaces and Perturbations for Cluster Complexes \par}
\vskip 1.5em
{\large \centering \mdseries\lineskip .5em François Charest \par}%
\vskip 1em
{\large \par}
\end{center}\par
\vskip 1.5em

\let\thefootnote\relax\footnotetext{{\em Date:} November 2012 }

\begin{center}
{\bfseries Abstract \vspace{.5em}}
\end{center}

{\listparindent 1.5em \itemindent    \listparindent \rightmargin   \leftmargin
We define objects made of marked complex disks connected by metric line segments and construct nonsymmetric and symmetric moduli spaces of these objects. This allows choices of coherent perturbations over the corresponding versions of the Floer trajectories proposed by Cornea and Lalonde (\cite{CL}). These perturbations are intended to lead to an alternative description of the (obstructed) $A_\infty$-structures studied by Fukaya, Oh, Ohta and Ono (\cite{FOOO2}, \cite{FOOO}).

Given a $Pin_{\pm}$ monotone lagrangian submanifold $L \subset (M,\omega)$ with minimal Maslov number $N_L \geq 2$, we define an $A_\infty$-algebra (resp. differential graded algebra) structure from the critical points of a generic Morse function on $L$. It is written as a cochain (resp. chain) complex extending the pearl complex introduced by Oh (\cite{Oh}) and further explicited by Biran and Cornea (\cite{BC}), equipped with its quantum product. We verify that the construction is homotopy invariant, defining a functor from a homotopy category of $Pin_{\pm}$ monotone lagrangian submanifolds $h\mathcal{L}^{mono, \pm}(M,\omega)$ to the homotopy category of cochain (resp. chain) complexes $hK(\Lambda \text{-mod})$ where $\Lambda$ is a Novikov ring with coefficients in $\mathbb{Z}$.

\vfill
{\bf Acknowledgements: } The author is greatly indebted to his advisors, F. Lalonde and O. Cornea, for their guidance and encouragement. This research was partially supported by an NSERC Canada Graduate Scholarship.
}

\tableofcontents				

%
%
%
%


\NoChapterPageNumber 
\pagenumbering{arabic}


\chapter*{Introduction}

\section*{Background}

One of the basic results of differential topology is the computation of the (singular) cohomology $H^*(L,\Z)$ of a manifold $L$, say compact and without boundary, by looking at a smooth function defined on it. We very briefly recall this construction, called Morse cohomology, using notations that will be used in the main text.

One first sets a Morse-Smale pair $(f,g)$, that is,
\begin{itemize}
\item $f:L \rightarrow \R$ is smooth with nondegenerate hessian over the critical locus (the critical points),
\item $g$ is a smooth metric on $L$ such that the negative gradient flow of $f$ with respect to $g$ has transverse stable and unstable manifolds.
\end{itemize}
This Morse-Smale condition is in fact generically satisfied, meaning that one can pick $f$ and $g$ in a countable intersection of open dense subsets.

Now let
\[
C^*(L,f,g) = crit(f) \otimes \Z
\]
be the free $\Z$-module over the set of critical points of $f$. For $x\in crit(f)$, denote by $\mu(x)=\mu^+(x)$ the number of positive eigenvalues of the hessian matrix of $f$ at $x$, and pick an orientation $\mathcal{O}_x$ of the stable manifold $W^s_x(f)$ of the negative gradient flow of $(f,g)$ to $x$.

One then looks for maps $u:\overline{\R} \rightarrow L$ with $u(-\infty)=y, u(\infty)=x \in crit(f)$ satisfying the negative gradient flow equation $du(-\frac{\partial}{\partial t}(t))= -\nabla_g f \circ u (t)$, that is, $u$ is a negative gradient flow line from $x$ to $y$. Set the index of such a gradient trajectory to be $\mu(u)= \mu(y) - \mu(x)$ ($= |x| - |y|$, where $|x| = n- \mu^+(x)$ is the usual (homological) Morse index).

Then one defines the Morse differential $\delta = \delta(L,f,g): C^*(L,f,g) \rightarrow C^*(L,f,g)$ as
\[
\delta(x) = \underset{ \substack{y \in crit(f) \\ \mu(u)=\mu(y) - \mu(x)=1 } }{\sum} \, <\mathcal{O}_u \# \mathcal{O}_x, \mathcal{O}_y>  y
\]
where $\mathcal{O}_u \# \mathcal{O}_x$ is the orientation on $W^s_y(f)$ induced by $\mathcal{O}_x$ plus the choice of the orientation $+ \frac{\partial}{\partial t}$ over $u$ and $<\_,\_>=1$ (resp. $-1$) if the entries coincide (resp. are opposite).

\begin{thmnn}
For $\delta$ as defined above, we have $\delta \circ \delta = 0$.
\end{thmnn}

The proof is based on the fact that $\delta \circ \delta(x)$ counts exactly the (oriented) boundary components of the $1$-parameter families (i.e. of index $2$) of negative gradient trajectories of $f$ starting from $x$. Moreover,

\begin{thmnn}
For $\delta$ as defined above, we have $H^*(C^*(L,f,g), \delta(L,f,g)) \cong H^*(L,\Z)$.
\end{thmnn}

It has been noticed (see \cite{Fu}) that under this isomorphism, the singular cohomology cup product and its higher order Massey products correspond to counting trajectories $u$ made of trees of gradient flow lines between critical points of a collection of (mutually) Morse-Smale pairs. The right source spaces for $u$ are then planar trees in which every interior edge is given a length. These planar metric trees are well known to form a polytopal moduli identifiable to Stasheff associahedra (\cite{Sta}) (see figure \ref{fig32}).

\begin{figure}[h] 
        \centering 
	\includegraphics[width=110mm]{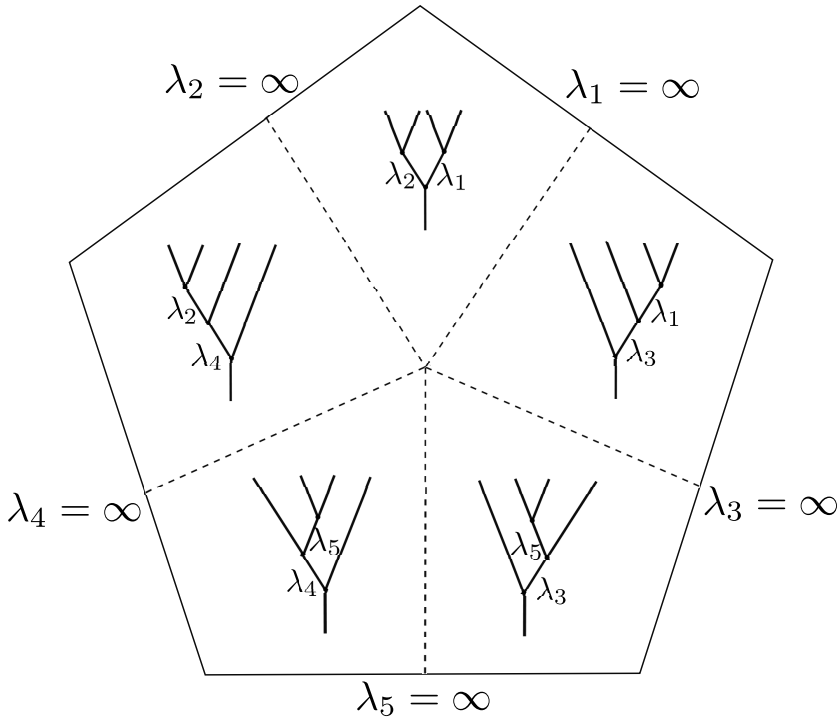}
        \caption{Stasheff associahedron $K_4$ as planar metric trees} \label{fig32}
\end{figure}

These Morse theoretical constructions can be seen as a special case of what is called lagrangian intersection Floer theory: To every Morse function $f$ on the $n$-dimensional manifold $L$ one can associate the graph $L_f$ of its associated $1$-form $df$ in $T^*L$. It is easy to see that $L_f$ is a lagrangian submanifold of $(T^*L, \omega_{st})$ (i.e. it is $n$-dimensional and $\omega_{st}$ vanishes on it). We also identify $L$ with the $0$-section in $T^*L$. Then the points of $L_f \bigcap L$ correspond to critical points of $f$ and Floer (\cite{F}) defined, as in the Morse case, a cochain complex $(CF^*(L_f,L), \delta^{CF}(L_f,L))$ by counting pseudoholomorphic strip trajectories, one side of the strip being mapped to $L_f$, the other to $L$ and the ends converging to some points of $L_f \bigcap L$. By pseudoholomorphic map to a symplectic manifold $(M,\omega)$ we will mean the following: For a $\omega$-tamed almost complex structure $J: TM \rightarrow TM$, that is $J^2 = -Id$ and $\omega(v,Jv) > 0$ whenever $v \neq 0$, a map $u: (\Sigma,j) \rightarrow (M,J)$ from a complex surface $(\Sigma,j)$ is said to be ($J$-)pseudoholomorphic, or simply ($J$-)holomorphic, if $du \circ j = J \circ du$.

Floer showed that the cohomology $HF^*(L_f,L) = H^*(CF^*(L_f,L), \delta^{CF}(L_f,L))$ is again $H^*(L,\Z)$, the moduli of gradient trajectories corresponding to those made of Floer pseudoholomorphic trajectories, then proving (partially) a celebrated conjecture of Arnold.

\begin{thmnn}
$HF^*(L_f,L) \cong H^*(L,\Z)$.
\end{thmnn}

Fukaya and Oh (\cite{FO}) extended this correspondence, gradient tree trajectories then corresponding to pseudoholomorphic polygons with boundary in lagrangian graphs (see figure \ref{fig34}).

\begin{figure}[h] 
        \centering 
	\includegraphics[width=110mm]{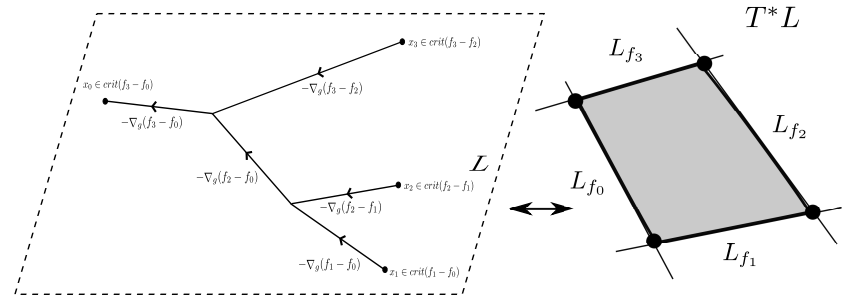}
        \caption{Morse and Floer product trajectories} \label{fig34}
\end{figure}

Thus, in $T^*L$, Morse cohomology of $L$ corresponds to counting pseudoholomorphic strips between intersection points of two lagrangian sections, and taking products corresponds to counting pseudoholomorphic polygons between intersection points of many lagrangian sections. We point out that source spaces for these polygonal Floer trajectories are complex disks with boundary punctures (or markings), again known to form a moduli identifiable to Stasheff associahedra (see figure \ref{fig33} or \cite{MW}).

\begin{figure}[h] 
        \centering 
	\includegraphics[width=110mm]{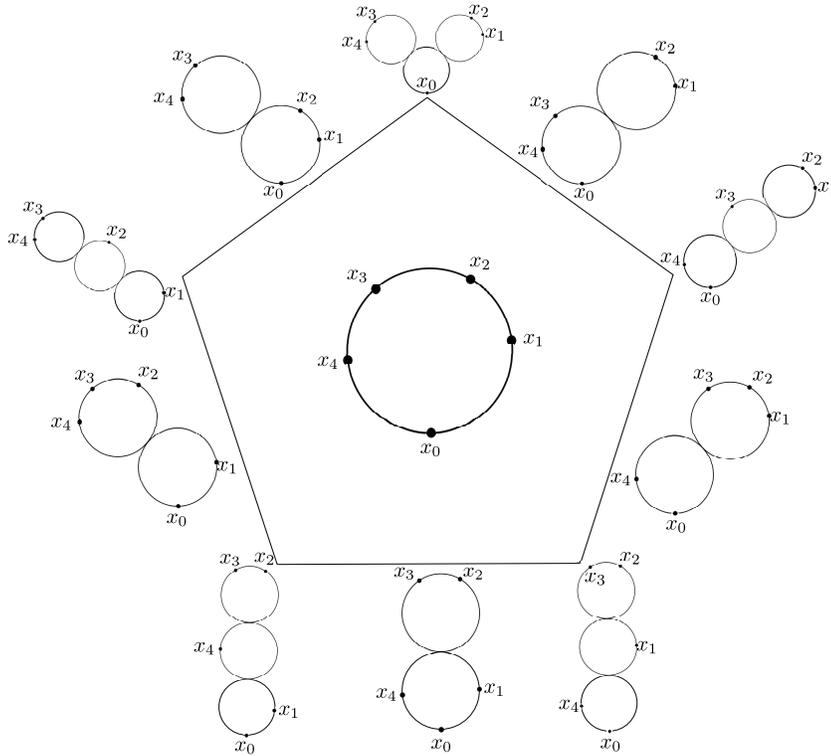}
        \caption{Stasheff associahedron $K_4$ as complex marked disks} \label{fig33}
\end{figure}

Generalizing this construction to a more general symplectic manifold $(M, \omega)$ leads to the definition of what is called a Fukaya category $\mathcal{F}(M, \omega)$ of $(M,\omega)$. Its objects are made of lagrangian submanifolds of a certain class, its morphisms are intersection points between them and composition of morphisms is again given by counting pseudoholomorphic polygons connecting intersection points. Fukaya categories have attracted much attention in the recent years due, in part, to the homological mirror symmetry conjecture made by Kontsevich (\cite{KS}) motivated by high energy physics theories.

A main difference between this more general situation and the cotangent one is the possible presence of new "quantum" trajectories, having no Morse theory counterparts. For a fixed lagrangian submanifold $L$, we can recover this information in the Morse setting by adding pseudoholomorphic disks (with boundary on $L$) to the gradient trajectories (see figure \ref{fig35}), as introduced by Oh (\cite{Oh}). Under rather strong restrictions on $L$ (i.e. $L$ will be said to be monotone), Morse cohomology trajectories are then replaced by the so-called pearl trajectories made of linear gradient trees in which the interior vertices are replaced by pseudoholomorphic disks. These trajectories and the associated binary quantum products on this pearl complex have been studied by Biran and Cornea (\cite{BC}), where two incident gradient lines might now meet on the boundary of a pseudoholomorphic disk (see figure \ref{fig35}).

\begin{figure}[h] 
        \centering 
	\includegraphics[width=110mm]{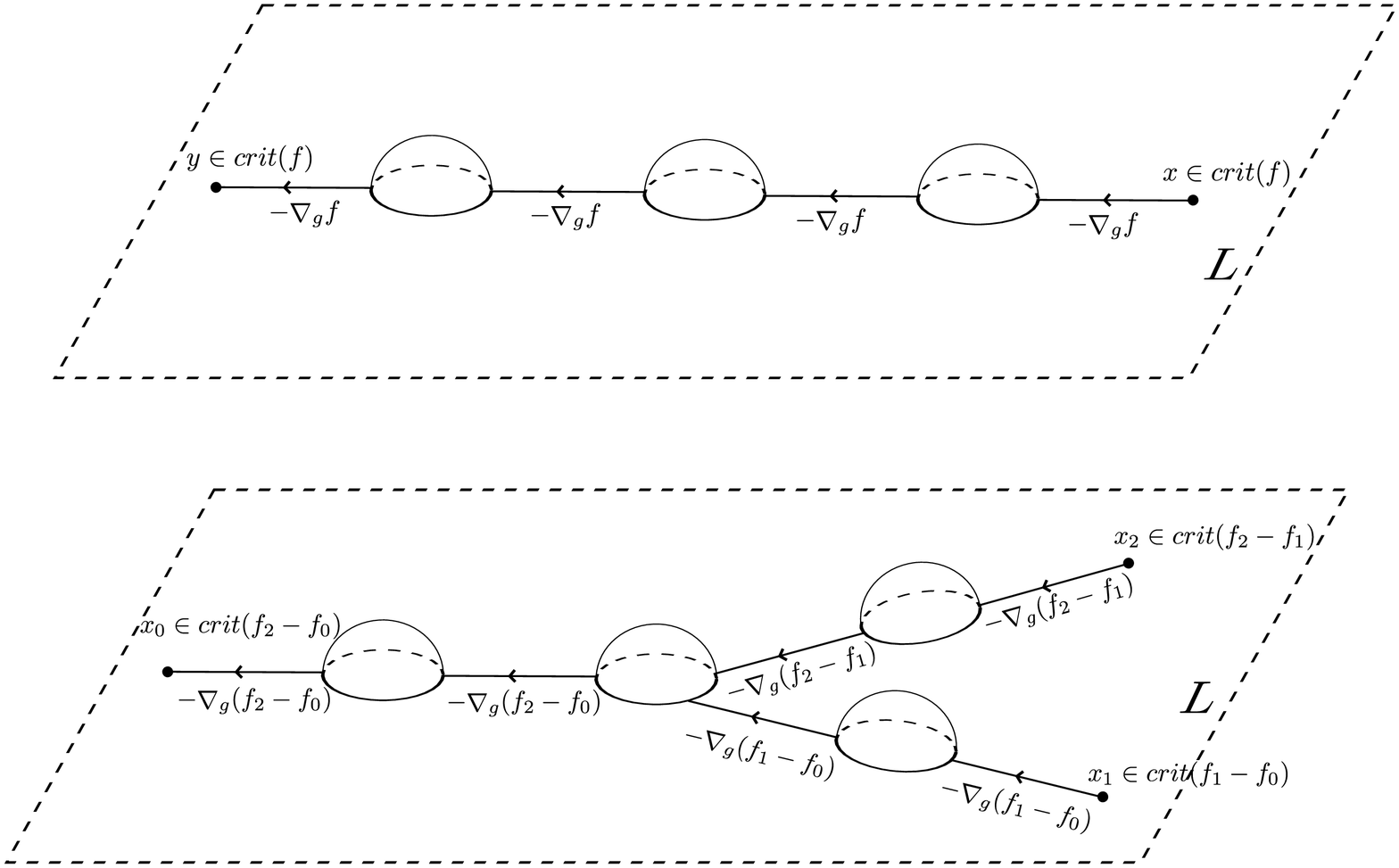}
        \caption{A pearl trajectory and an associated quantum product trajectory} \label{fig35}
\end{figure}

These quantum product configurations tend to be easier to compute than their equivalent pseudoholomorphic triangles bounded by perturbations of $L$ or their hamiltonian variants described by Seidel (\cite{Sei}). Their study have led to a new definition of relative enumerative invariants (see \cite{BC}) and have found other important applications. The extension to higher order products is therefore an interesting problem. This would result in a new geometric description of the part of the Fukaya category associated with $L$, that is, the morphisms from $L$ to itself. Algebraically, this can be viewed either as an $A_\infty$-algebra or, from a dual point of view, as a differential graded algebra (DGA) associated with a monotone $L$.

Further generalizing the above "trees with disks" quantum product configurations should extend the construction to a (almost) general bounding lagrangian (generating an obstructed $A_\infty$-algebra or DGA). Such an approach was proposed by Cornea and Lalonde (\cite{CL}): cluster (co)homology. More generally, it was claimed that the obstruction to the composition of morphisms between lagrangians of the Fukaya category could be encoded by these cluster product configurations on each of the bounding lagrangians. However, analytically, pseudoholomorphic trajectories with general bounding lagrangian are much more delicate to deal with than those in the former monotone case: One needs to use perturbations of the defining equation of these trajectories.

\section*{Strategy}

After the works of Gromov (\cite{G}) and Floer (\cite{F}), pseudoholomorphic maps have become a very widely used tool in the study of symplectic manifolds. When considering spaces of pseudoholomorphic maps with lagrangian boundary conditions, we expect to find codimension one strata corresponding to nodal maps, commonly referred as (relative) bubbling, whose smooth components are of arbitrary index. Furthermore, it is known that in general, the regularity of the spaces of these maps depends upon the ability to manage perturbations of their defining Cauchy-Riemann equation.

Regarding the codimension one relative bubbling, Cornea and Lalonde (\cite{CL}) proposed a geometrical way of encoding pseudoholomorphic disk bubbling of arbitrary index using the flow lines of Morse functions on the bounding lagrangians. This can be seen as extending the pearl complex construction of Oh (\cite{Oh}) later studied by Biran and Cornea (\cite{BC}) and should lead to a geometric (non-hamiltonian) description of the obstructed Fukaya category of $(M,\omega)$. In this setting, adding a flow line of a Morse-Smale pair $(f:L \rightarrow \R, g)$ connecting the boundaries of the disks generates a family of pseudoholomorphic maps complementary to the usual glued family of pseudoholomorphic disks (see figure \ref{fig28}).

\begin{figure}[h] 
        \centering 
	\includegraphics[width=110mm]{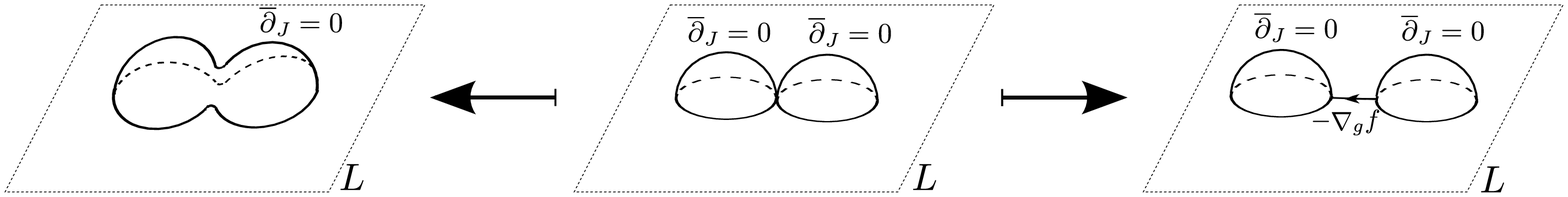}
        \caption{} \label{fig28}
\end{figure}

One way to deal with the associated perturbation problem is given by \\ Cieliebak and Mohnke (\cite{CM}), using Donaldson's (\cite{D}) construction of symplectic hypersurfaces $Z$ that are Poincaré dual to $D[\omega] \in H^2(M,\Z)$, for an integer $D \gg 1$. Since holomorphic curves must have a certain minimal $\omega$-area, or energy, they all must intersect $Z$ several times. Then, they choose coherent perturbation data over moduli spaces of marked source spaces (i.e. marked Riemann surfaces) and then, given a map $u: \Sigma \rightarrow M$, use the perturbed Cauchy-Riemann equation associated with $(\Sigma, u^{-1}(u(\Sigma) \bigcap Z))$. In the more restrictive setting that allows this construction to be performed, it is a geometric alternative to the use of the polyfold theory of Hofer, Wysocki and Zehnder (\cite{HWZ}) being considered in the cluster setting by Li and Wehrheim (\cite{LW}), or to the framework of Kuranishi structures used by Fukaya, Oh, Ohta and Ono (\cite{FOOO}).

The aim of this work is to present the first steps towards an adaptation of the Cieliebak-Mohnke method to manage perturbations over symmetric and nonsymmetric versions of the cluster pseudoholomorphic trajectories of Cornea-Lalonde. As shown by Auroux, Gayet and Mohsen (\cite{AGM}), there are symplectic hypersurfaces in the complement of a given lagrangian submanifold. Therefore, we describe appropriate moduli of marked source objects built from a moduli of marked disks, called marked clusters, over which we will choose coherent perturbation data. In fact, these spaces will be isomorphic to spaces of disks with both boundary and interior markings, making them generalizations of Stasheff associahedra (\cite{Sta}) seen as spaces of disks with only boundary markings (as in \cite{MW}, for example). We note that such an approach has been considered independently by Sheridan (\cite{She1},\cite{She2}) to study homological mirror symmetry for Calabi-Yau hypersurfaces in projective spaces of dimension greater than four.

In the case when a monotone lagrangian submanifold $L \subset (M,\omega)$ with minimal Maslov number $N_L \geq 2$ bounds the cluster trajectories, using a constant almost complex structure $J \in \mathcal{J}(M,\omega)$ in the Cauchy-Riemann equation over the disks is possible so that no symplectic hypersurface is needed. Therefore, a (source) perturbation datum $P$ of the Morse-Smale pair $(f,g)$ will be enough to achieve regularity. This will be shown using structural results for pseudoholomorphic disks of Lazzarini (\cite{L}) and ideas of Biran and Cornea (\cite{BC}). In the nonsymmetric (and symmetric) case, this results in the definition of a cluster cochain (or chain) complex \\ $(\mathcal{C}\ell^\otimes(M,\omega,L,J,P,f,g), \delta^\otimes(M,\omega,L,J,P,f,g))$. The configurations counted by the differential of this cluster complex could be thought of as geometric limits of the relative hamiltonian orbit products described in \cite{Sei}. They again generate an $A_\infty$-algebra that sits in the Fukaya category of $(M,\omega)$ as the compositions of the morphisms from $L$ to itself, or as a DGA encoding the same information.

Furthermore, the functoriality of the construction will be verified in the nonsymmetric case, using moduli spaces of objects called quilted clusters, built from spaces of quilted disks being slight generalizations of those described by Ma'u and Woodward (\cite{MW}) as realizing Stasheff multiplihedra.


\section*{Overview}

Since the main part of the present document is of rather technical nature, we first give an outline of its content, trying to emphasize on its differences and similarities with previous works.

Essentially, a cluster $C$ will be a tree with marked complex disks as vertices and boundary connecting metric segments as edges. A cluster Floer trajectory will be a map $u: (C, \partial C) \rightarrow (M,L)$ being pseudoholomorphic over the disks and satisfying gradient equations over the lines, $\partial C$ being the complement of the interior of the disks. In a family of Floer trajectories, when two incidence points of gradient segments tend to the same incidence point on the boundary of a pseudoholomorphic disk (a phenomenon being of codimension one in the source moduli, and thus in the space of Floer trajectories), one wants to pursue this family to avoid having the latter singular configuration as a boundary point in the space of Floer trajectories.

As proposed by Cornea and Lalonde (\cite{CL}), to make the above singular trajectory an interior point in the space of Floer trajectories, one might consider a new Floer family in which the order of the nearby incidence points (on the boundary of the incidence disk) has been switched (see figure \ref{fig31}). Otherwise, one might add a new gradient segment connecting the incidence disk and the point where the two incoming lines meet (see figure \ref{fig31}), as it is done, for example, for Morse $A_\infty$ products (\cite{Fu}) or to define the quantum product on the pearl complex (\cite{Oh}, \cite{BC}). The latter will be referred to as the nonsymmetric, or $\otimes$, case while the former will be called the symmetric, or $\bullet$, case.

\begin{figure}[h] 
        \centering 
	\includegraphics[width=110mm]{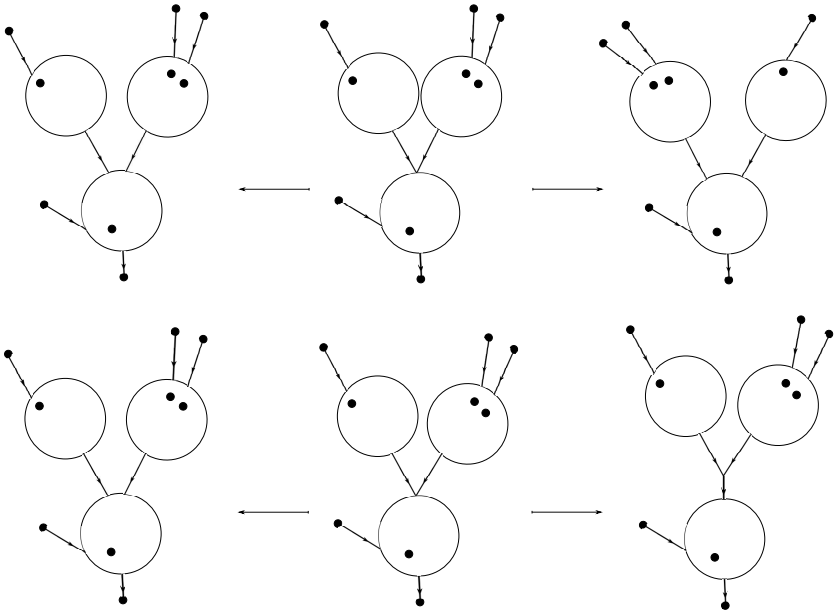}
        \caption{Above: A $\bullet$ -cluster Floer family \hspace{0.2cm} Below: A $\otimes$-cluster Floer family} \label{fig31}
\end{figure}

\subsection*{The $\otimes$ (nonsymmetric) case}

In the $\otimes$ case, in which there will not be switches on the planar structure of the clusters, the space of stable $\otimes$-clusters $\cllkr$ will be built by extending the spaces of stable disks with ordered boundary markings, plus some interior markings.

\textbf{Source spaces.} First, we look at the spaces of stable complex disks with $\ell + 1$ counterclockwise ordered boundary markings and $k$ interior markings. These allow compactifications $K_{\ell,k}$ having the structure of real $\ell -2 +2k$ dimensional orientable manifolds with embedded corners. One could see these as generalizations of the Stasheff associahedra (\cite{Sta}): The spaces $K_{\ell,0}$ is a well-known realization of the Stasheff associahedral polytopes. A $m$-corner of $K_{\ell,k}$ is of the form $K_{\ell^{(1)},k^{(1)}} \times \ldots \times K_{\ell^{(m+1)},k^{(m+1)}} \in C_m(K_{\ell,k})$ and corresponds to a family of nodal disks with $m$ real nodes, the structure on every smooth component varying independently.

Given a pair of disks $(D^{(1)},D^{(2)}) \in K_{\ell^{(1)},k^{(1)}}\times K_{\ell^{(2)},k^{(2)}} \in C_1(K_{\ell,k})$ related by a node, we add a complementary family of objects made of the same disks, but in which the nodal point is replaced with a line of length $\lambda \in \overline{\R_+}$. Letting $(D^{(1)},D^{(2)})$ vary, this procedure corresponds, on the disk moduli, to adding a collar neighborhood on the corresponding $1$-corner $K_{\ell^{(1)},k^{(1)}}\times K_{\ell^{(2)},k^{(2)}} \in C_1(K_{\ell,k})$. The resulting space can be considered as being a manifold with embedded corners isomorphic to $\klk$, so we can iterate this procedure over $C_1(K_{\ell,k})$ (see figure \ref{fig29}). Denote by $\cllkr$ the result, the space of stable $\otimes$-clusters (although in fact, we will define $\cllkr$ combinatorially and show that it can be given the same structure as $\klk$).

\begin{figure}[h] 
        \centering 
	\includegraphics[width=110mm]{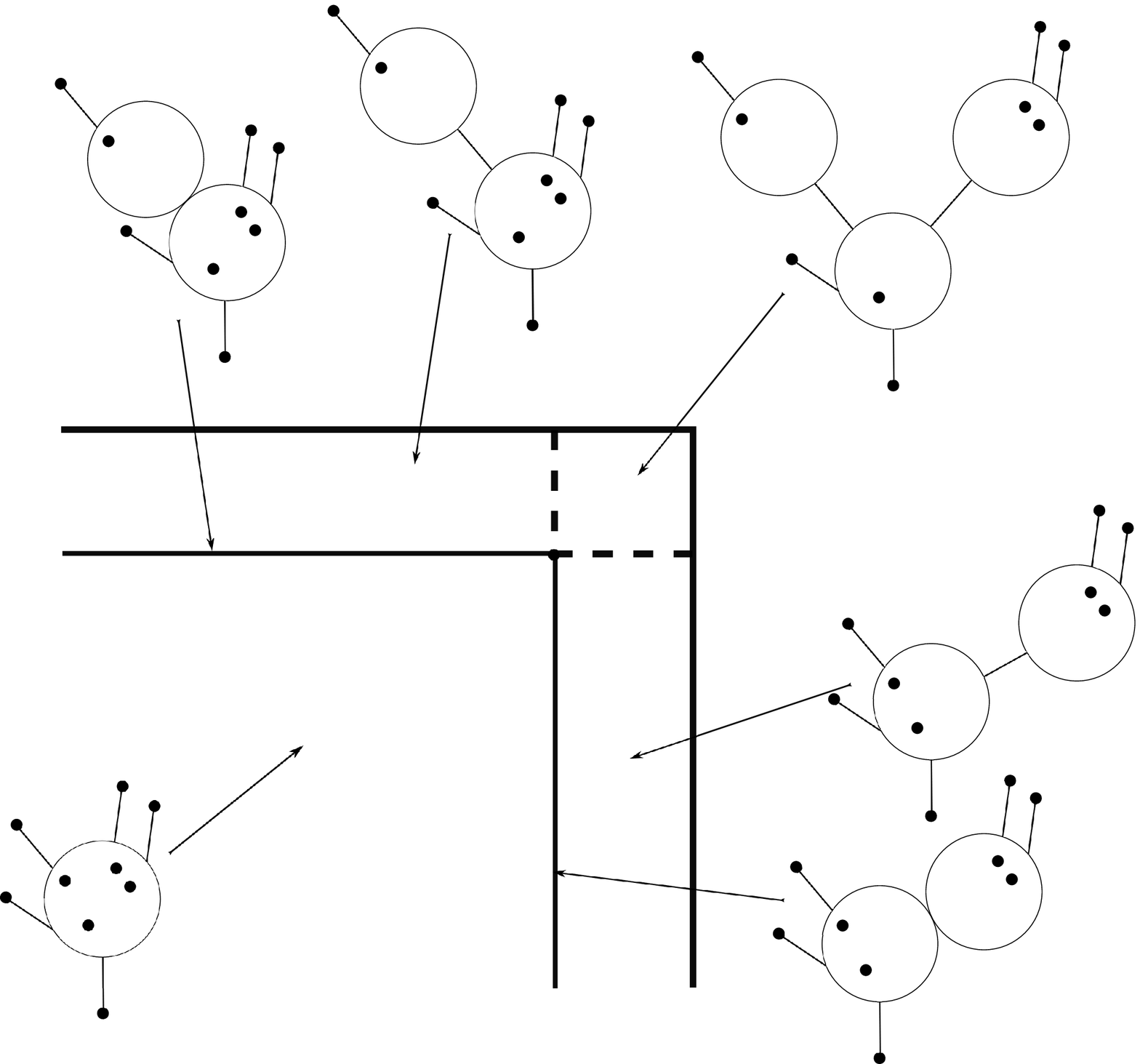}
        \caption{Adding lines between disks near a $2$-corner of $\klk$} \label{fig29}
\end{figure}

\begin{lemnn}
The space of stable $\otimes$-clusters $\cllkr$ is an orientable manifold with embedded corners isomorphic to $\klk$.
\end{lemnn}

We emphasize that families of $\otimes$-clusters do not behave like those in \cite{CL}: The above $\otimes$-cluster families carry an ordering on the boundary markings reminescent of the one over the disks of $\klk$ (see figure \ref{fig31}).

To each point of $C \in \cllkr$, we associate a metric tree $C$ (the interior edges having a length in $\overline{\R_+}$) in which the vertices are replaced with stable marked complex disks using the definition of $\cllkr$. Also, lines of infinite length will be replaced by broken lines (see figure \ref{fig30}). The resulting object will be referred to as a cluster.

\begin{figure}[h] 
        \centering 
	\includegraphics[width=110mm]{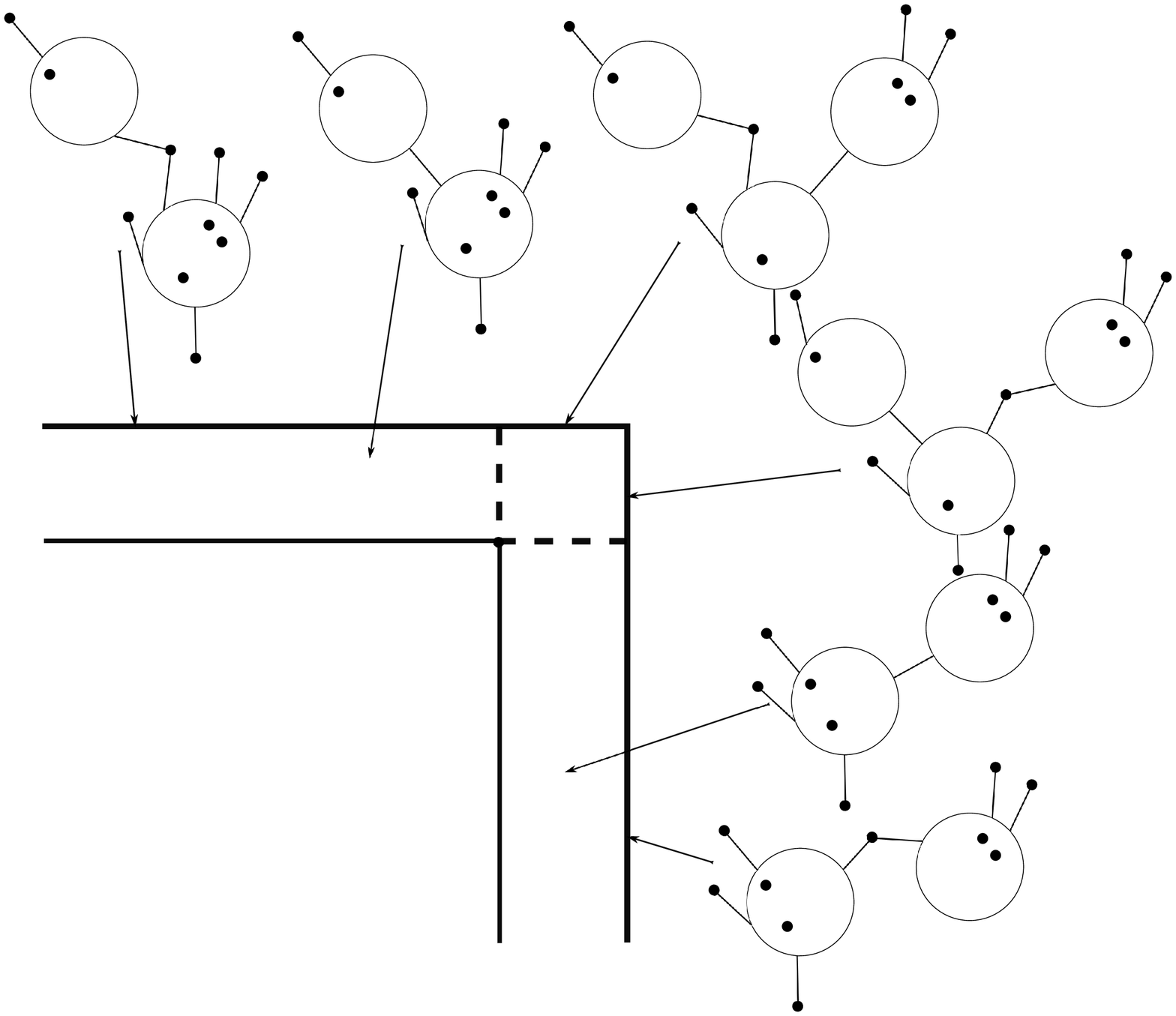}
        \caption{} \label{fig30}
\end{figure}

\begin{remnn}
We note that besides being source spaces for $\otimes$-cluster complexes, the $\otimes$-clusters are also the right source spaces for the more general lagrangian intersection fine Floer homology proposed in \cite{CL}.
\end{remnn}

\textbf{Perturbations.} Since the ghost disks (those having no interior markings) will be mapped to points when considering spaces of Floer trajectories, using a single function $f$ over every line segment may result in degeneracies of the defining equation above these disks (see \cite{CL}). We therefore want to choose coherent systems of perturbation of $f$ over the connecting lines of the clusters and combine them with perturbations of an almost complex structure $J$ over the disks, like those of \cite{CM}. Another solution to the ghost degeneracy problem would be to use different Morse functions over segments that might touch, as it is done in \cite{Fu} or \cite{BC}. However, this strategy fits in the algebraic framework of pre-$A_\infty$ algebras (see \cite{KS}), so we might prefer to use the $\cllkr$ spaces to allow the generation of genuine $A_\infty$ algebras or differential graded algebras (as in \cite{CL}).

To ensure that the perturbed trajectories will break on critical points of $f$, we specify a coherent system of neighborhoods of the endpoints (root, leaves and breakings) where only $f$ will be chosen (a similar strategy has been used in previous works, for example, \cite{W1} and \cite{Sei}).

\textbf{Orientations.} To manage orientations over spaces of trajectories from clusters, we proceed as Welschinger (\cite{W1},\cite{W2}). For maps $u^{(i)}: C^{(i)} \in {\mathcal{C}\ell_{\ell^{(i)},k^{(i)}}} \rightarrow (M,L)$, $i = 1,2$, satisfying pseudoholomorphic and gradient equations, we can associate a linear Fredholm operator $D\overline{\partial}_{u^{(i)}}$. Observe that if $Ind(D\overline{\partial}_{u^{(i)}})= -(\ell^{(i)} -2 +2 k^{(i)})$, then it is possible to show that for generic perturbation data $Ker(D\overline{\partial}_{u^{(i)}}) = \{0\}$ and $Coker(D\overline{\partial}_{u^{(i)}}) = T_{C^{(i)}} \mathcal{C}\ell_{\ell^{(i)},k^{(i)}}^\otimes$. Thus, the orientation problem for the differential $\delta^\otimes$ will amount to a comparison of a product reference orientation $\mathcal{O}^\otimes_{\ell^{(1)},k^{(1)}} \wedge \mathcal{O}^\otimes_{\ell^{(2)},k^{(2)}}$ on a $1$-corner of $\cllkr$ (the source spaces of $\delta^\otimes \circ \delta^\otimes$) and the orientation induced on it by a reference orientation $\mathcal{O}^\otimes_{\ell,k}$ on $\cllkr$ (the source spaces of the glued Floer family).

\begin{lemnn}
        Let $\ell^{(1)} \geq 1$, $\ell^{(2)} \geq 0$ and $k^{(1)}, k^{(2)} \geq 0$ such that $(\ell^{(2)}, k^{(2)}) \neq (0,0)$. Then let $\ell +1 = \ell^{(1)} + \ell^{(2)}$, $k = k^{(1)} + k^{(2)}$, $C^{(1)} \in \mathcal{C}\ell_{\ell^{(1)},k^{(1)}}^\otimes$, $C^{(2)} \in \mathcal{C}\ell_{\ell^{(2)},k^{(2)}}^\otimes$ smooth and define $C =  C^{(1)} \bigsqcup C^{(2)} \big{/}_{v^{(1)}_j \sim v^{(2)}_0}$, that is, $C$ is the concatenation of $C^{(1)}$ and $C^{(2)}$ on the $j^{th}$ leaf of $C^{(1)}$. Then
\[
        \partial_{\ddel_C} \mathcal{O}^\otimes_{\ell,k} = (-1)^{(\ell^{(1)} -j)\ell^{(2)} + (j-1)} \mathcal{O}^\otimes_{\ell^{(1)},k^{(1)}} \wedge
\mathcal{O}^\otimes_{\ell^{(2)},k^{(2)}}.
\]
\end{lemnn}

\begin{remnn}
The resulting signs in the differential $\delta^\otimes$ will agree with those of Getzler and Jones (\cite{GJ}, \cite{KS}), which are widely used.
\end{remnn}

\textbf{Cochain complexes.} We now resrict the ambient $((M,\omega),L)$ setting to the case of a $\gl$ (so it is not necessarily orientable) monotone lagrangian $L$ in the sense of \cite{BC} so we can use uniform $J$ over the disks and forget about the interior markings. The use of the most general perturbation setting is deferred to a different paper.

We observe that regularity is obtained by standard arguments over strata of simple Floer trajectories (\cite{MS}). In the above setting, using decomposition results of Lazzarini (\cite{L}), we achieve regularity in chapter \ref{tran_floe} as in \cite{BC} by performing a reduction procedure on nonsimple trajectories, making them simple but with index dropped by at least $2$ and therefore nongeneric.

Regarding the Morse-Smale functions, we choose $f$, the function that will be perturbed, and (possibly) a generic collection $(f_0, \ldots, f_c)$ that will not be perturbed. The considered cluster trajectories will go from a properly ordered collection of critical points of $f$ and of the $f_{j_2} - f_{j_1}$, $j_1 < j_2$ to a contracted collection following the same contraction pattern as in \cite{Fu}: For every disk, the sum (with $f$ counting as $0$) of the incident line functions' equals the outgoing line function's (see figure \ref{fig10}). Forgetting the perturbed trajectories (those with $f$ over some lines) would result in the quantized Morse products (\cite{Fu},\cite{BC}) of order up to $c$. 
This would result in the pre-$A_\infty$ algebraic setting introduced by Kontsevich and Soibelman (see \cite{KS}). If using only $f$, the considered product trajectories are reminescent of the ones defining the $A_\infty$ algebra associated with $L$ as described by Seidel (\cite{Sei}) using a fixed hamiltonian perturbation over the strip-like ends instead of a fixed function over linear ends. The two are expected to coincide through a PSS (see \cite{PSS}) type comparison morphism (see \cite{CL}). 

Using the source setting above, we define complexes in a form similar to the usual bar complex of an $A_\infty$ algebra (see \cite{GJ}). In fact, it can be seen as a bar complex viewed on the unsuspended algebra (see \cite{K}) over a Novikov ring $\Lambda$. 
The codifferential then has the form
\[
        \delta^\otimes = \delta^\otimes(M,\omega,J,L,P,f,f_0, \dots, f_c,g)= \underset{q \geq 1}{\sum} \underset{1 \leq j \leq q}{\sum} \, \underset{ \ell \geq 1}{\sum} (-1)^{(q-j)\ell + (j-1)} Id^{j-1} \otimes m_\ell \otimes Id^{q-j}
\]
where $m_\ell$ counts rigid (i.e. of index $2-\ell$) Floer trajectories having $\ell$ inputs and $1$ output with orientation $\mathcal{O}_{\ell,k}^\otimes$, as explained above. In this notation, the signs, coming for the combinatorial structure of $\klk$, are independent of the source and target critical points (but when applied to a collection of critical points, additional signs appear according to the usual Koszul sign rule $f \otimes g(x_1 \otimes x_2) = (-1)^{\mu(g) \mu(x_1)} f(x_1) \otimes g(x_2)$). From the above considerations, we get that

\begin{thmnn}
        For $\delta^\otimes$ defined as above, we have $\delta^\otimes \circ \delta^\otimes = 0$ so that $(\mathcal{C}\ell^\otimes,\delta^\otimes)$ is a cochain complex.
\end{thmnn}

Notice that the classical Morse products trajectories ($\ell \geq 2$, $k=0$ so that every disk is ghost) are part of this differential, unlike the differential of \cite{CL}. Having $\delta^\otimes \circ \delta^\otimes = 0$ is equivalent to the fact that the $m_\ell$ form a $A_\infty$ algebra (\cite{K}). It is known that, unlike in the latter symmetric case, the resulting cohomology groups will be trivial. Moreover, to recover the differential graded algebra setting of \cite{CL}, one has to reverse the differential trajectories, take the suspension and complete the generator set. 

\textbf{Source setting for morphisms.} The next step is to see that the construction is homotopy invariant, that is, a hamiltonian isotopy of monotone lagrangians and a homotopy of the intermediate data give rise to a complex (iso)morphism. It is known that quilted disks, that is, complex marked disks together with an inner circle tangent to the boundary at the root marking (the seam) (see figure \ref{fig14}), are appropriate source spaces for morphisms in lagrangian intersection Floer theory (see \cite{MW}, \cite{W1}).

Starting from the moduli spaces of quilted marked disks $\qklk$, we define moduli spaces of quilted clusters $\qcllkr$ as in the non-quilted case, adding collar components. For $k=0$, this amounts to enlarge Stasheff multiplihedra, seen as a space of quilted disks, while for higher $k$, one works with proper real loci of the complexification of the multiplihedra introduced by Ma'u and Woodward (\cite{MW}). However, in every case, these spaces are toric singular manifolds with embedded corners so the collar procedure is more delicate.

\begin{lemnn}
The space of stable quilted $\otimes$-clusters $\qcllkr$ is an orientable toric singular manifold with embedded corners isomorphic to $\qklk$.
\end{lemnn}

Families in $\qcllkr$ behave much like in the non-quilted case, generating isolated ghost disks when incident lines collide. However, some disks will contain the seam of the quilt and thus there will be two types of codimension one boundary components, as in the case of Stasheff multiplihedra (see \cite{MW}), depending on whether the breaking happens over the quilted components, resulting in a $\mathcal{QC}\ell_{\ell^{(1)},k^{(1)}}^\otimes \times \mathcal{C}\ell_{\ell^{(2)},k^{(2)}}^\otimes$ $1$-corner, or breakings happen under the quilted components, resulting in a $\mathcal{C}\ell_{q,k}^\otimes \times \mathcal{QC}\ell_{\ell^{(1)},k^{(1)}}^\otimes \times \ldots \times \mathcal{QC}\ell_{\ell^{(q)},k^{(q)}}^\otimes$ $1$-corner (see figure \ref{fig14}).

The quilted $\otimes$-clusters will play the same role as the quilted disks in Floer theory: Given complexes $(\mathcal{C}\ell^\otimes(M,\omega,L^{(i)},J^{(i)},P^{(i)},f^{(i)},f_0^{(i)}, \dots, f_c^{(i)},g^{(i)}), \\ \delta^\otimes(M,\omega,L^{(i)},J^{(i)},P^{(i)},f^{(i)},f_0^{(i)}, \dots, f_c^{(i)},g^{(i)}) )$, $i = 0,1$ and a homotopy between their construction parameters, we interpolate the $(0)$ and $(1)$ data over the seamed disks of the quilted clusters, use the $(1)$ data above the seamed disks and the $(0)$ data below. This allows to define morphism $H$ as
\[
        H = \underset{ \sum_i \ell^{(i)} \geq 0}{\sum} \, \underset{q \geq 1}{\sum} (-1)^{\sum^q_{i=1} (q -i)(\ell^{(i)} - 1)}  h_{\ell^{(1)}} \otimes \dots \otimes h_{\ell^{(q)}}. 
\]
where $h_{\ell^{(i)}}$ counts quilted Floer trajectories of index $1-\ell^{(i)}$. The sign again comes from the combinatorial structure of the source moduli $\qcllkr \cong \qklk$.

\begin{prpnn}
        $H$ is a cochain map from $((\mathcal{C}\ell^\otimes)^{(1)},(\delta^\otimes)^{(1)})$ to \\
$((\mathcal{C}\ell^\otimes)^{(0)},(\delta^\otimes)^{(0)})$. That is, we have $H \circ (\delta^\otimes)^{(1)} = (\delta^\otimes)^{(0)} \circ H$.
\end{prpnn}

\begin{remnn}
It is worth mentioning that besides being source spaces for morphisms between the $\otimes$-cluster complexes themselves, the quilted $\otimes$-clusters are also expected to be the right source spaces for other comparison morphisms like PSS morphisms (\cite{PSS}) or morphisms of lagrangian intersection fine Floer complexes (\cite{CL}).
\end{remnn}

Then, we show that homotopic interpolations between the $(0)$ and $(1)$ data give rise to homotopic cochain maps. The source moduli $[0,1] \times \qcllkr$ are enough for this need. We define a cochain homotopy $K$ on elements of cardinality $\sum_{i=1}^{q} \ell^{(i)} = q'$ having the form
\[
        K = \underset{ q \geq 1}{\sum} (-1)^q \, \underset{ q' }{\sum} (-1)^{\sum^q_{i=1} (q -i)(\ell^{(i)} - 1)} \, \overset{q}{\underset{ p= 1}{\sum}} (-1)^{\sum^{p-1}_{i=1} (\ell^{(i)} - 1)} \underset{ t\in [0,1]}{\sum} h^{(t)}_{\ell^{(1)}} \otimes \dots \otimes k^{(t)}_{\ell^{(p)}} \otimes \dots \otimes h^{(t)}_{\ell^{(q)}}. 
\]

\begin{prpnn} 
        For $K$, defined as above, is a cochain homotopy between $H^{(1)}$ and $H^{(0)}$. That is, we have $H^{(1)} - H^{(0)} =  K \circ (\delta^\otimes)^{(1)} + (\delta^\otimes)^{(0)} \circ K$.
\end{prpnn}

To complete the proof of the functoriality property of the $\otimes$-cluster construction, we will use quilted clusters with two seams built from the corresponding spaces of disks (\cite{MWW}):

\begin{prpnn}
        For $H^{(i)}$ a cochain map between $\big((\mathcal{C}\ell^\otimes)^{(i+1)},(\delta^\otimes)^{(i+1)}\big)$ and \\
$\big((\mathcal{C}\ell^\otimes)^{(i)},(\delta^\otimes)^{(i)}\big)$, $i=0,1$, defined as above , and $H^{(1) \circ (0)}$ a cochain map between \\ 
$\big((\mathcal{C}\ell^\otimes)^{(2)},(\delta^\otimes)^{(2)}\big)$ and $\big((\mathcal{C}\ell^\otimes)^{(0)},(\delta^\otimes)^{(0)}\big)$ built from the concatenation of their perturbation homotopies. Then $H^{(0)} \circ H^{(1)}$ and $H^{(1) \circ (0)}$ are homotopic.
\end{prpnn}

Now write $\Lambda$-mod for the (abelian) category of $\Lambda$-modules, $K(\Lambda \text{-mod})$ for the category of cochain complexes over $\Lambda$-mod and $hK(\Lambda \text{-mod})$ for the (triangulated) homotopy category of the latter. The above arguments resume to

\begin{thmnn}
\begin{diagram}
h\mathcal{L}^{mono, \pm}(M,\omega) & && \rTo^{\mathcal{C}\ell^\otimes} && & hK(\Lambda \text{-mod}) \\
(L,J,P,f,g) & && \rMapsto  && & (\mathcal{C}\ell^\otimes(M,\omega,L,J,P,f,g),\delta^\otimes(M,\omega,L,J,P,f,g))
\end{diagram}
is a contravariant functor.
\end{thmnn}

Here, $h\mathcal{L}^{mono, \pm}(M,\omega)$ is a homotopy category of monotone $\gl$ (or, equivalently, $Pin_{\pm}$) lagrangian submanifolds of $(M,\omega)$ where the morphisms are built from hamiltonian isotopies and interpolations of the perturbation data.

\subsection*{The $\bullet$ (symmetric) case}

\textbf{Source setting.} Instead of being constructed from disks with $\ell +1$ ordered boundary markings, the symmetric $\bullet$-cluster moduli are constructed from a moduli of disks with varying order on their boundary markings. We look for these disks in the locus of real marked spheres having $\ell+1$ real markings and $k$ pairs of complex conjugate markings partially studied by Ceyhan (\cite{Cey}).

Let $\mathcal{M}_{\ell,k}$ be the Deligne-Mumford-Knudsen space of stable complex genus zero Riemann surfaces with $\ell +1 +2k$ markings denoted by $\{x_j\}_{0\leq j\leq \ell}$ and $\{z_h\}_{1\leq h \leq 2k}$. It has the structure of a compact complex $(\ell -2 +2k)$-manifold and has an antiholomorphic involution $\sigma_{\ell,k}$ defined as the composition of the natural complex conjugation with the transpositions $(z_h z_{h+k})$, $1\leq h \leq k$. The real locus $\R\mathcal{M}_{\ell.k} \equiv fix(\sigma_{\ell,k})$ is then a smooth real $(\ell -2 +2k)$-manifold, corresponding to curves with real $x_j$'s and complex conjugate pairs $(z_h, z_{h+k})$ (see \cite{Cey}).

We next restrict to $Im(\iota) \subset \R \mathcal{M}_{\ell,k}$, the real spheres where $\{z_1, \ldots, z_k\}$ lie in the same hemisphere that will be considered as a disk with $\ell+1$ real markings and $k$ interior markings. It will be seen as $l!$ copies of $\klk$ attached together on strata having ghost components, where the ordering of the real markings switch. The main point is that, in general, $Im(\iota)$ is singular precisely over the strata having at least one internal ghost sphere. For example, a neighborhood of the singular locus in the case $\ell=0$, $k=2$ is displayed in figure \ref{bul_fig01nn}. Next we give a local description of $Im(\iota) \subset \R \mathcal{M}_{\ell,k}$ near any point $S \in Im(\iota)$ lying in a codimension $m$ open stratum $\mathfrak{S} \subset \R \mathcal{M}_{\ell,k}$.

\begin{figure}[h] 
        \centering 
	\includegraphics[width=100mm]{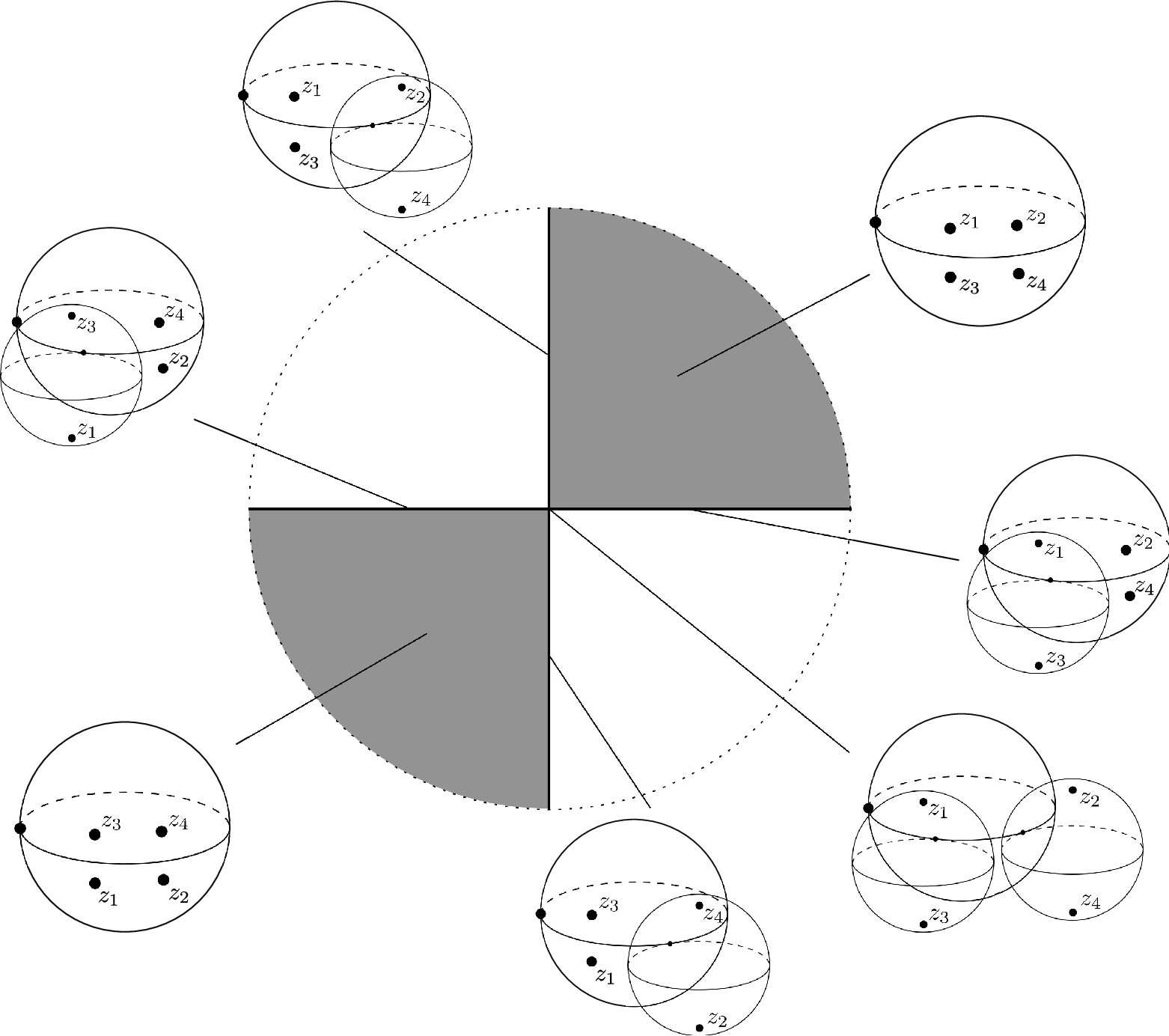}
        \caption{Singularity in $Im(\iota) \subset \R \mathcal{M}_{0,2}$} \label{bul_fig01nn}
\end{figure}

We can choose normal coordinates $(n_1, \ldots, n_m)$ to $\mathfrak{S}$ at $S$ corresponding to $m$ real gluing parameters, one for each real node of $S$ (see \cite{MW}, \cite{Liu}, \cite{MS}). Then one can orient the $n_i$ coordinates so that

\begin{lemnn}
	In the $(n_1, \ldots, n_m)$ normal coordinates to $\mathfrak{S}$ at $S$, $Im(\iota) = G \cdot \R_+^m$, where $G \equiv \underset{d \text{ ghost}}{\prod} \Z/2\Z$ and the $d$ factor generator acts by changing the signs of the coordinates corresponding to the nodes on $d$.
\end{lemnn}

These singularities would not be problematic if we were to consider only disks, but will certainly be when adding the collar components corresponding to connecting metric segments as in the $\otimes$ case (see figure \ref{bul_fig01}). However, the above lemma ensures that we can perform blowups over the singular strata and get a manifold with embedded corners (MWEC). We proceed in three steps to construct an appropriate space of disks:

\begin{enumerate}
	\item Attach the $\klk$ tiles along the $1$-corners corresponding to transpositions of two real markings. The resulting space, $\klk^I$, is an orientable MWEC.
	\item Iteratively identify and blow up $\klk^I$ along its 2-corners corresponding to nodal spheres having one interior ghost component with one real marking separating two non-ghost components. The resulting space, $\klk^{II}$, is again an orientable MWEC.
	\item Iteratively identify and blow up $\klk^{II}$ along its 3-corners corresponding to three non-ghost disks related by a ghost disk with only three nodes. Call the resulting $\klk^{\bullet}$.
\end{enumerate}

\begin{lemnn}
	$\klk^\bullet$ is an orientable smooth MWEC with every $1$-corner having no ghost disks associated being isomorphic to a $K_{\ell^{(1)},k^{(1)}}^\bullet \times K_{\ell^{(2)},k^{(2)}}^\bullet$ product where $k^{(1)}, k^{(2)} \geq 1$ and $\ell^{(1)}+\ell^{(2)}-1 = \ell$.
\end{lemnn}

The above blowup procedure produces exceptional stata where the structure of the disks will not vary along the fibers. This will not cause degeneracy problems later as we will be allowed to choose nonconstant perturbation data along these fibers.

Now we add up connecting metric lines exactly as in the $\otimes$ case: $\cllks$ is defined by an iterative collar enlargement procedure of $\klk^\bullet$.

\begin{lemnn}
	$\cllks$ is an orientable smooth MWEC isomorphic to $\klk^\bullet$.
\end{lemnn}

The families of clusters of $\cllks$ behave as those in \cite{CL}, containing switches of the planar order over pairs of segments meeting the boundary of a trivalent ghost disk. We do not perform this switch over more than trivalent ghosts disks because they will later appear in codimension at least two by choosing perturbations not depending on the moduli over these disks.

\textbf{Orientations.} Next we will, as in the $\otimes$ case, compare a product reference orientation on a cluster $C$ with one breaking with that induced by the reference orientation of its glued family. 

The resulting formula will be
\[
	\partial_{\ddel_C} \mathcal{O}^\bullet_{\ell,k} = (-1)^{(\ell^{(1)} - p_{min})\ell^{(2)} + (p_{min} -1)} (-1)^{\sigma} \mathcal{O}^\bullet_{\ell^{(1)},k^{(1)}} \wedge
\mathcal{O}^\bullet_{\ell^{(2)},k^{(2)}}.
\]
where $\sigma$ and $p_{min}$ will depend on the boundary order of the leaves of $C$.

The resulting signs in the differential $\delta^\bullet$ will agree with those appearing in the literature (see \cite{Cho}).

\chapter{Moduli of $\otimes$-clusters}

We first describe the source spaces that will be used to define the cochain complex. They are planar trees of complex marked disks connected by metric lines and are chosen to form a moduli where the planar structure cannot vary. These sources can be seen as generalizations of the sources used in \cite{Fu}, \cite{Fu2}, \cite{Oh} and \cite{BC}. 

\section{Moduli of marked disks} \label{disks}

Let $\mathring{\klk}$ be the moduli space of stable complex disks with $k \geq 0$ distinct marked interior points $z_1,
z_2, \dots, z_k$ and $\ell+1 \geq 1$ distinct ordered marked points $x_0 < x_1 < \dots < x_\ell$ on its boundary, which is
chosen to be positively oriented with respect to the complex orientation (i.e. the outward normal times the latter equals the complex orientation). This space allows a Deligne-Mumford-Knudsen type compactification $\klk$ which admits an orientable \\ $(\ell -2 + 2k)$-dimensional manifold with embedded corner structure via cross-ratio coordinates (see figure \ref{fig01}, \cite{MS} and for a more general setting \cite{Liu}). Therefore, $\klk$ is a manifold with corners in the sense of \cite{M}, but we follow the conventions of \cite{J}. 

\begin{figure}[h] 
        \centering 
	\includegraphics[width=110mm]{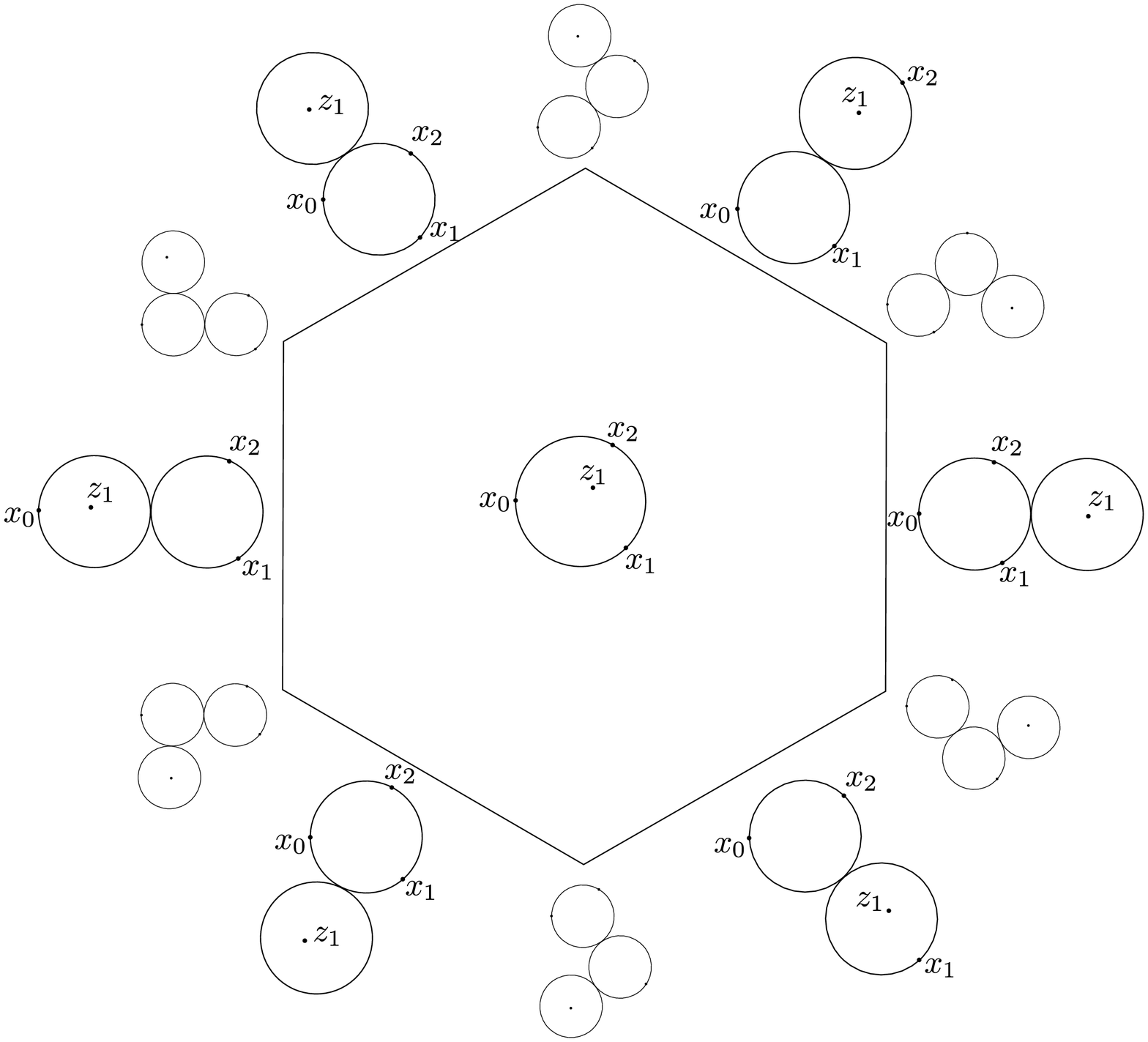}
        \caption{$K_{2,1}$} \label{fig01}
\end{figure}

Denote by $ \ulk \overset{\pi}{\rightarrow} \klk$ the so-called universal curve so that $\pi^{-1}(D)$ is a smooth (resp. nodal) marked disk representing $D$ for every $D \in \mathring{\klk}$ (resp. $D \in \klk \setminus \mathring{\klk}$).

We emphasize that every $1$-corner is the moduli of two smooth marked disks intersecting on a boundary nodal point, so it is canonically isomorphic to a product $K_{\ell^{(1)},k^{(1)}} \times K_{\ell^{(2)},k^{(2)}}$ with $k^{(1)} + k^{(2)} = k$ and $\ell^{(1)} + \ell^{(2)} -1 = \ell$. 

As usual, it is convenient to encode the combinatorial type of a marked nodal disk in a tree $T$ having one vertex for each smooth component and one edge for each special point. The set of such trees has a natural partial order relation given by $T^{(1)} \leq T^{(2)}$ if $T^{(1)}$ is obtained from $T^{(2)}$ by contracting some of its interior edges. Also, it will be useful to see the smooth marked disks as marked complex upper half-planes with the point at infinity corresponding to $x_0$. 

\begin{defi}
Let $X^{\R_+}(T) \equiv \{E^{int}(T) \rightarrow \R_+\}$. The elements of $X^{\R_+}(T)$ are called labelings of $T$ with values in $\R_+$.

If $T^{(1)} \leq T^{(2)}$, a labeling $X^{(2)}$ on $T^{(2)}$ determines one on $T^{(1)}$, denoted by $X^{(2)}|_{T^{(1)}}$, by simply taking $X^{(2)}|_{T^{(1)}}(l) = X^{(2)}(l)$.
\end{defi}


We will use the corner charts given in the following form (see \cite{MW},\cite{Liu}, \cite{MS}): 

\begin{prp} \label{neighk}
        Let $K_{\ell,k, T} \subset \klk$ be the strata of the marked disks having combinatorial type $T$. Then, there is an isomorphism $\psi_T$ from a neighborhood of $K_{\ell,k, T} \times \{0\} = K_{\ell,k, T} \times X^{0}(T)$ in $K_{\ell,k, T} \times X^{\R_+}(T)$ to a neighborhood $\nu(K_{\ell,k, T})$ of $K_{\ell,k, T}$ in $\klk$.
\end{prp}

Without losing generality, one can assume there is $0< \epsilon < 1$ such that
\begin{diagram}
K_{\ell,k, T} \times X^{[0,\epsilon[}(T) & & \rTo^{\psi_T} & & \nu(K_{\ell,k, T})
\end{diagram}
is an isomorphism.

Note that since $\klk$ is a manifold with embedded corners, these charts could be extended to $K_{\ell,k,\geq T}$, that is, to the closure of $K_{\ell,k, T}$. For $T^{max}$ being maximal, the map $\psi_{T^{max}}$ can be made explicit in terms of simple ratio (or cross ratio) coordinates and the above restriction of labelings allows to construct $\psi_T$ for smaller trees (see proposition 6.2 and corollary 6.5 of \cite{MW}).

\begin{figure}[h] 
        \centering 
	\includegraphics[width=120mm]{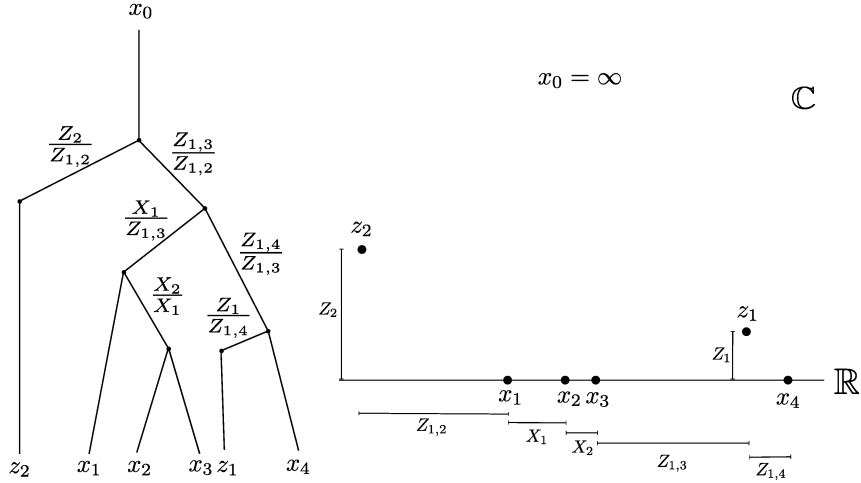}
        \caption{Construction of a marked disk from a labeling $X \in X^{\R_+}(T^{max})$} \label{fig02}
\end{figure}

Following \cite{MW}, we can make the charts $\psi_T$ explicit in terms of simple ratios. Let $T^{max}$ be a maximal planar tree, corresponding to a marked nodal disk $D_{T^{max}}$ with every smooth component being either a disk with $3$ boundary markings and no interior markings or a disk with one boundary marking and one interior marking. The planarity of $T^{max}$ makes the set of markings $\{x_1, \ldots, x_\ell, z_1, \ldots, z_k \}$ into an ordered set $\{y_1, \ldots, y_{\ell+k}\}$. To every trivalent vertex $v$ of $T^{max}$ we associate a marking $y_j$: It can be seen as the uppermost intersection point of a simple path from $y_j$ to the root and a simple path from $y_{j+1}$ to the root. Then take $\Delta_v = y_j - y_{j+1}$. To every bivalent vertex $v$ of $T^{max}$ we associate the marking $z_h$ lying just above it and then set $\Delta_v = Im(z_h)$. Now on an interior edge $l \in E(T^{max})$ with top vertex $v_a$ and bottom vertex $v_b$, we will consider the label $X(l)= \frac{\Delta_{v_a}}{\Delta_{v_b}}$ (see figure \ref{fig02}). 

\begin{diagram}
\nu(K_{\ell,k, T^{max}}) & & \rTo^{\psi_{T^{max}}^{-1}} & &  X^{\C}(T^{max}) \\
D & & \rMapsto & &  X(l) = \frac{\Delta_{v_a}}{\Delta_{v_b}}
\end{diagram}

This will make the nodes correspond to zero labels on the corresponding edge of $T$. When there is no interior markings, $\psi_{T^{max}}$ corresponds to the simple ratio charts of \cite{MW}, taking real positive values. In general, the resulting labelings are complex, but still can be identified with real positive labelings. For example, in a sufficiently small neighborhood $\nu(K_{\ell,k, T^{max}})$ of $K_{\ell,k, T^{max}}$, one can extract the real part of every label as in figure \ref{fig02}. We emphasize that when $k\geq 1$, in any maximally degenerate disk, each interior marking $z_h$ lies in a component with only one node and no other marking, so in the corresponding maximal tree it contributes as a bivalent vertex in the boundary of the edge associated with $z_h$. The label on the edge below that vextex is then seen as the imaginary part $Im(z_h)$ of $z_h$ (see figure \ref{fig02}).

\begin{rem} \label{rema_sphe}
It might be convenient to see the spaces of discs as lying inside the spaces $\R\mathcal{M}_{\ell,k}$ of spheres with $\ell +1$ real and $k$ pairs of complex conjugate markings. We recall the construction of these spaces that appear in \cite{Cey}.

Let $\mathcal{M}_{\ell+1+2k}$ be the space of complex spheres with $\ell +1 +2k$ markings denoted by $\{x_j\}_{0\leq j\leq \ell}$ and $\{z_h\}_{1\leq h \leq 2k}$. It has the structure of a complex $(\ell -2 +2k)$-manifold and has an antiholomorphic involution $\sigma_{\ell,k}$ defined as the composition of the natural complex conjugation $(\Sigma,j) \rightarrow (\Sigma,-j)$ with the transpositions $(z_h z_{h+k})$, $1\leq h \leq k$. The real locus $\R\mathcal{M}_{\ell.k} \equiv fix(\sigma_{\ell,k})$ is then a smooth real $(\ell -2 +2k)$-manifold.

The complex double operation over the complex disks gives a natural map $\klk \overset{\iota}{\rightarrow} \R\mathcal{M}_{\ell,k}$. The image of this map is the closure of the set of spheres represented as $x_0 = \infty, \{x_1 < \ldots < x_\ell \} \subset \R P^1 \subset \C P^1$ and $Im(z_h) > 0$ for $1 \leq h \leq k$.
\end{rem}


\section{Constructing $\cllkr$, the moduli of $\otimes$-clusters} \label{collar}

Now we add up cells to $\klk$ that will be used to encode length variations of metric lines connecting the marked disks. The resulting objects, the $\otimes$-clusters, support planar structures reminescent of the complex structures over the disks that will not vary over $\cllkr$.

\begin{defi} \label{defi_col}
\[
        col(\klk) = \underset{T}{\bigsqcup} \hspace{0.3cm}  K_{\ell,k,\geq T} \times X^{[0,1]}(T) /_{\sim}
\]
with $(D^{(1)},X^{(1)}) \in K_{\ell,k,\geq T^{(1)}} \times X^{[0,1]}(T^{(1)}) \sim (D^{(2)},X^{(2)}) \in K_{\ell,k,\geq T^{(2)}} \times X^{[0,1]}(T^{(2)})$ if
\begin{itemize}
\item $D^{(1)} = D^{(2)}$,
\item $T^{(1)} < T^{(2)}$,
\item $X^{(2)}|_{T^{(1)}}=X^{(1)}$, that is, $X^{(2)}$ is the natural extension of $X^{(1)}$ having $1$ labels on its additional edges.
\end{itemize}
\end{defi}

\begin{figure}[h] 
        \centering 
	\includegraphics[width=95mm]{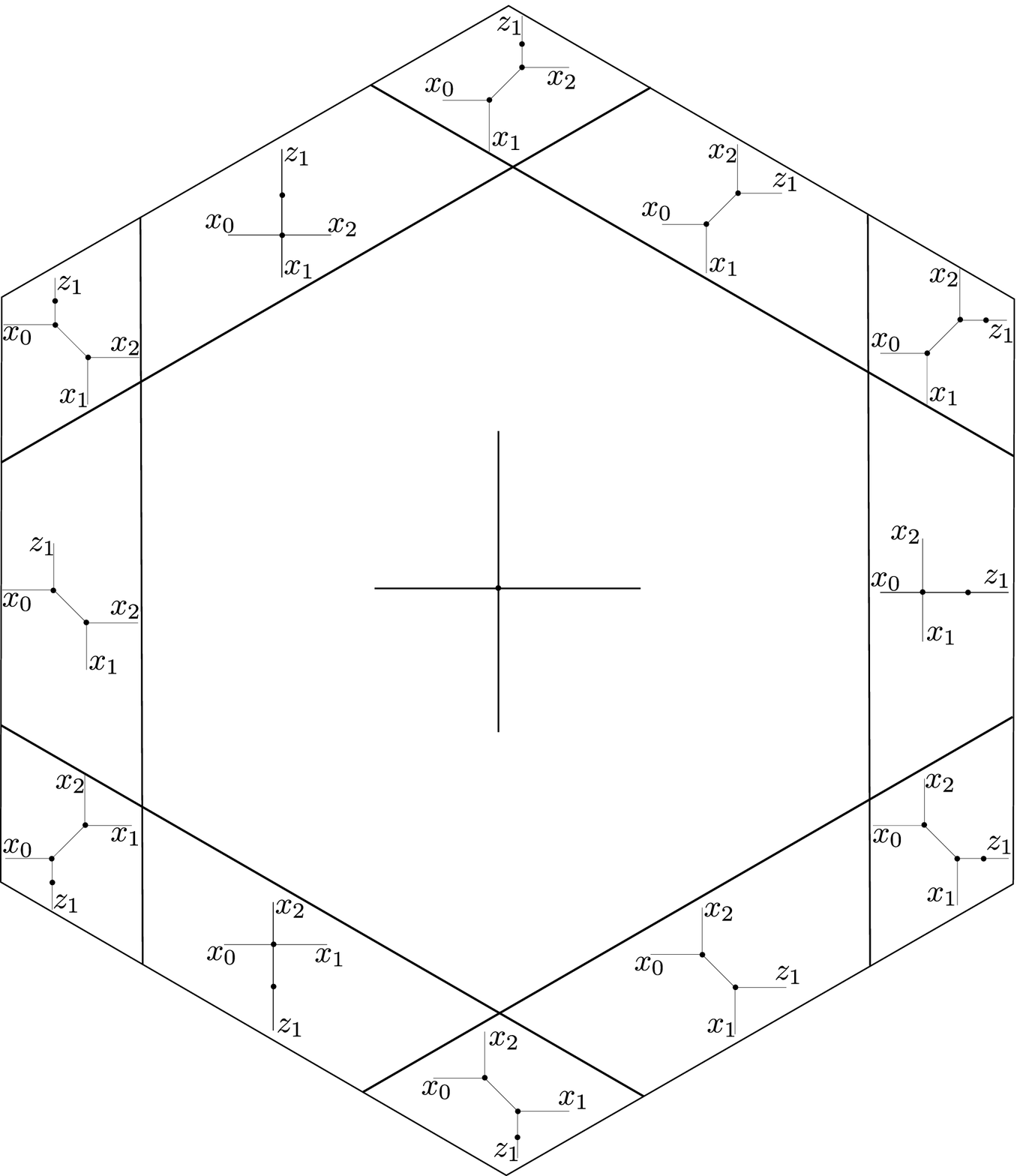}
        \caption{$col(K_{2,1})$} \label{fig03}
\end{figure}



\begin{defi}
        For $1 \leq \ell -1 + 2k$, let $(\cllkr)^{PDIFF} = col(\klk)$ and otherwise take $(\cllkr)^{PDIFF}$ to be a point.
\end{defi}

More geometrically, the above procedure could be seen as a "collar" extension of $\klk$ using, near each maximal corner, a dual cell decomposition reminescent of that of the associahedra seen as a space of metric trees (see \cite{BV}). The cells of this decomposition are given by labelings of maximal trees taking the $1$ value over a subset of interior edges.

Although above the piecewise smooth (manifold with corners) structure would suffice, it will be convenient to smoothen the above charts and see that $\cllkr$, the smoothened version, is isomorphic to $\klk$. As with disks, the combinatorial type of a cluster will be encoded in a tree $T$ having one vertex for each smooth component. Let $\mathcal{C}\ell_{\ell,k,T}^\otimes$ be the clusters having type $T$.

\begin{lem} \label{smoo_lemm}
There exists a piecewise smooth isomorphism $(\cllkr)^{PDIFF} \rightarrow \klk$ sending $\mathcal{C}\ell_{\ell,k,T}^\otimes$ to $K_{\ell,k,T}$ for every combinatorial type $T$.
\end{lem}

The idea will be to decompose a neighborhood of the corners of $\klk$ into pieces isomorphic to $K_{\ell,k,\geq T} \times X^{[0,1]}(T)$, $|T| \geq 1$, see that their complement is isomorphic to $\klk$ and that the pieces attach as in definition \ref{defi_col} (see figure \ref{fig03}). 

\begin{proof}
Since $\klk$ is a smooth MWEC, given a $1$-corner $F \in C_1(\klk)$, one can add a collar neighborhood $F \times [0,1]$ along $F$. The resulting space is again a smooth MWEC isomorphic to $\klk$. The added collar can then be seen as a compact neighborhood of the image of $F$ under this isomorphism.

Considering another $1$-corner $F' \in C_1(\klk)$, one can look at its corresponding $1$-corner in the above enlargement of $\klk$ and repeat the same collar addition manipulation. If $F \bigcap F' = \emptyset$, then we are left with a new copy of $\klk$ with preferred neighborhoods of $F$ and $F'$, and if $F$ and $F'$ have a common $2$-corner $G \in C_2(\klk)$, then we get a new copy of $\klk$ that decomposes as a preferred $G \times [0,1]^2$ neighborhood of $G$, cells $F \times [0,1]$ and $F' \times [0,1]$ and a main cell isomorphic to $\klk$ (see figure \ref{fig29.1}).

\begin{figure}[h] 
        \centering 
	\includegraphics[width=95mm]{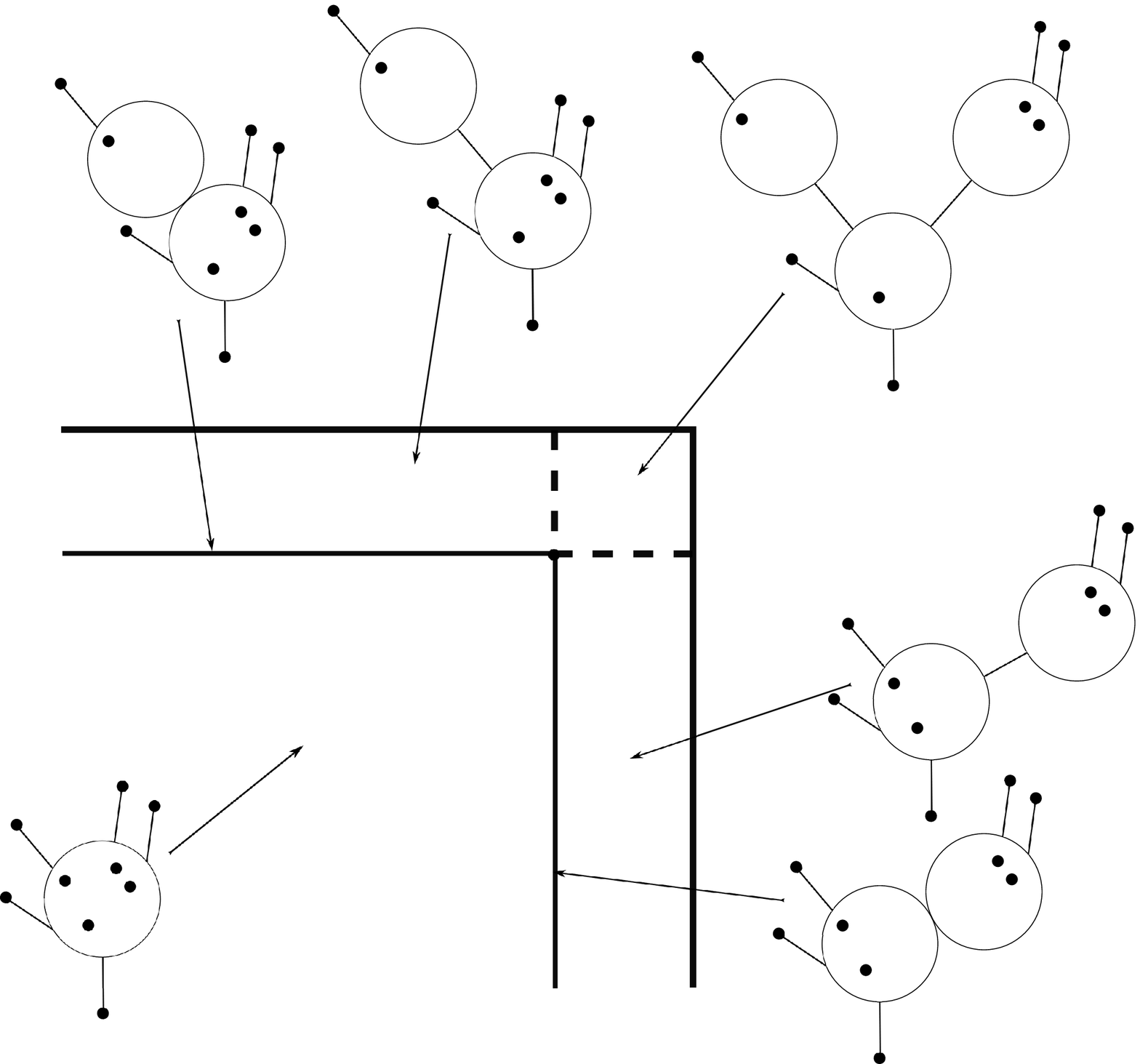}
        \caption{The collar procedure on $\klk$} \label{fig29.1}
\end{figure}

One can iterate over all of $C_1(\klk)$ and see that the resulting space contains for every corner $K_{\ell,k,\geq T}$ a cell isomorphic to $K_{\ell,k,\geq T} \times X^{[0,1]}(T)$ attaching to the $K_{\ell,k,\geq T'} \times X^{[0,1]}(T')$, $T' < T$,  cells as in definition \ref{defi_col}.
\end{proof}

Another way of managing the above construction will be necessary later in the case of quilted $\otimes$-clusters; we include it as appendix \ref{appe_A} to better illustrate the content of appendix \ref{appe_B}.

\begin{defi}
        Let $\cllkr$ be $(\cllkr)^{PDIFF} = col(\klk)$ endowed with the manifold with embedded corner structure induced by the above identification.
\end{defi}

Since every $1$-corner of $\klk$ is naturally isomorphic to a product $K_{\ell^{(1)},k^{(1)}} \times K_{\ell^{(2)},k^{(2)}}$, we get that every $1$-corner of $\cllkr$ is naturally isomorphic to a product $\mathcal{C}l_{\ell^{(1)},k^{(1)}}^\otimes \times \mathcal{C}l_{\ell^{(2)},k^{(2)}}^\otimes$. The same can be said about corners of lower dimension.

\section{Universal curves}

For $\ell -2 +2k \geq 1$, we now describe an extension of the universal curve over $\klk$ to $\cllkr$ such that the fiber over a point of the collar is a set of marked disks connected by metric lines.

Let again $\ulk \overset{\pi}{\rightarrow} \klk$ be that nodal family. We add metric lines between the components of the marked disks lying in the collar part of $\cllkr$ using the tree labelings.

\begin{defi}
        For $\ell -2 +2k \geq 0$, $C = (D,X) \in K_{\ell,k,\geq T} \times X^{[0,1]}(T)$, take the normalization of the marked nodal disk $\pi^{-1}(D)$ modified so that the two markings corresponding to $e \in E^{int}(T)$ are connected by a metric line of length $-log(X(e))$. In addition, for every $1 \leq j \leq \ell$ (resp. $j=0$), identify the origin of a copy of $\overline{\mathbb{R}}_{+}$ (resp. $\overline{\mathbb{R}}_{-}$) with the boundary marking $x_j(D)$ (see figure \ref{fig05}). For $\ell = 1$ and $k=0$, consider a line $\overline{\mathbb{R}}$ and for $\ell = k = 0$, take a half-line $\overline{\mathbb{R}}_{-}$.

        We denote the resulting by $(\pi^\otimes)^{-1}(C)$ and refer to it as the $\otimes$-cluster, with $\ell$ leaves and $k$ interior markings, associated with $C \in \cllkr$. 
\end{defi}

\begin{figure}[h]
        \centering 
	\includegraphics[width=120mm]{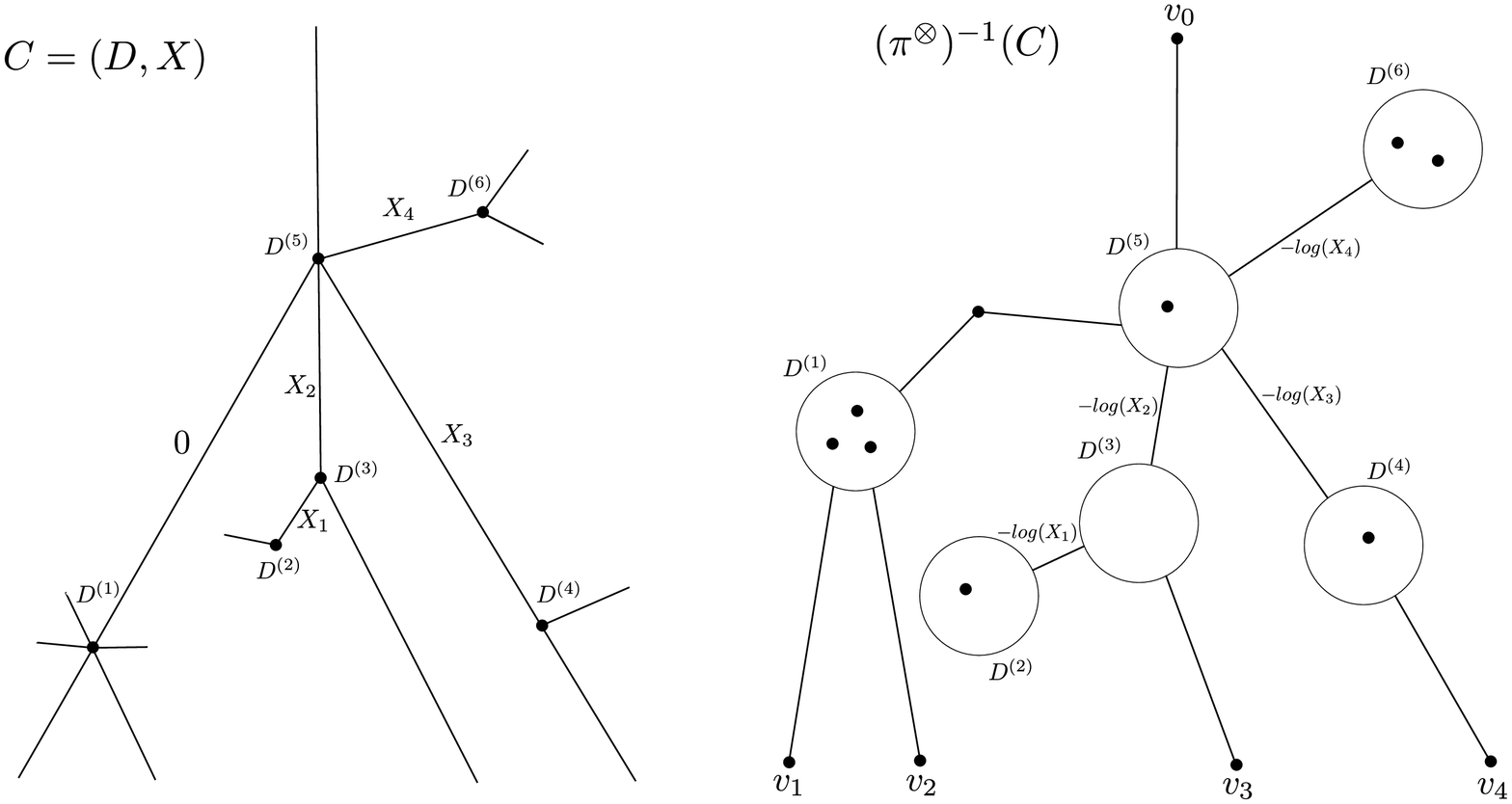}
        \caption{Cluster $(\pi^\otimes)^{-1}(C)$ associated with $C =(D,X)$, $D= D^{(1)} \cup \dots \cup D^{(6)} \in C_5(K_{4,8})$} \label{fig05}
\end{figure}

Like the notation suggests, we will consider the half-lines to be (metrically) semi-infinite, the points at infinity marked as $v_j(C)$, $0 \leq j \leq \ell$, from which we refer to $v_0(C)$ as the root of $C$ and to $v_j(C)$ as the $j^{th}$ leaf of $C$ for $1 \leq j \leq \ell$. 

It will be useful to see a connecting line of infinite length as a broken line $\overline{\mathbb{R}}_{+} \sqcup
\overline{\mathbb{R}}_{-} / _{+\infty \sim -\infty} $ with the $+$ half-line being closer to the root.

\begin{defi}
The points $\{v_j(C)\}_{0 \leq j \leq \ell}$ and the breaking points of $C$ are called the endpoints of $C$. We call a cluster irreducible if it does not contain any broken (i.e. infinite length) line.

$\partial C$ will denote the complement of the interior of the disks in $C$.
\end{defi}

Notice that the breakings determine a decomposition of any $\otimes$-cluster $C$ into a finite composition of irreducible $\otimes$-clusters $\underset{i}{\bigcup}C^{(i)}$.

Taking $\ulkr = \underset{C \in \cllkr}{\bigcup} (\pi^\otimes)^{-1}(C)$, we get a map $\pi^\otimes: \ulkr \rightarrow \cllkr$ and a commutative diagram
\begin{diagram}
        \ulk  && \rInto^{} && \ulkr \\
        \dTo^{\pi} && && \dTo^{\pi^\otimes} \\
        \klk && \rInto^{i} && \cllkr \\
\end{diagram}

Slightly extending the notion of deformation of stable marked bordered Riemann surface (see \cite{Liu}, chap. 3) so that it takes the normalization at a node as the opposite deformation to the (real) gluing at that node, one can say that $\ulkr$ is the union of a smooth family of marked bordered Riemann surfaces and a collection of smooth manifolds with corners made up from the added lines. $\pi^\otimes$ will be thought as a piecewise smooth universal family of $\otimes$-clusters.

\begin{figure}[h]
        \centering 
	\includegraphics[width=110mm]{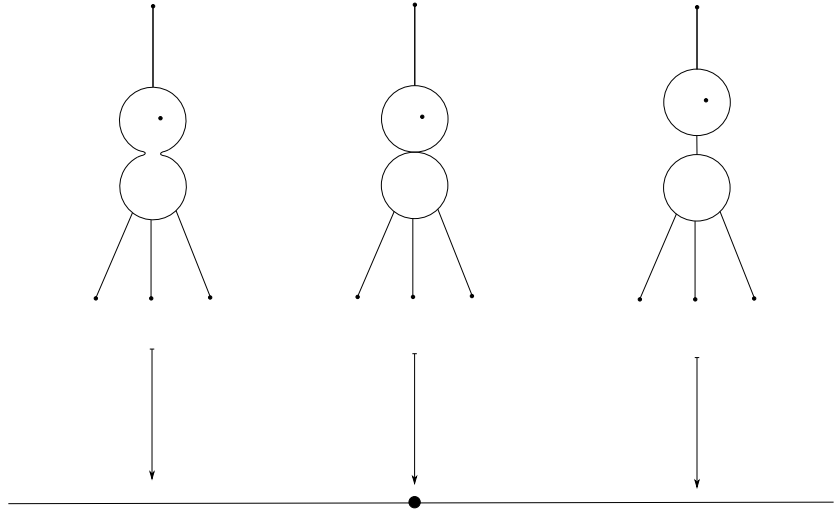}
        \caption{$1$-parameter deformation of a nonsmooth $C \in \mathcal{C}\ell_{3,1}^\otimes$} \label{fig05.1}
\end{figure}

\begin{defi}
        For every $\ell$  and $k$, we set $v_j: \cllkr \rightarrow \ulkr$, $1 \leq j \leq \ell$ (resp. $j=0$) as being the smooth sections that map to the leaves (resp. root) of the $\otimes$-clusters and $z_h: \cllkr \rightarrow \ulkr$, $1 \leq h \leq k$ the smooth sections which maps to the interior markings. Also, for every $1$-corner $F \in C_1(\cllkr)$, set $v_F: F \rightarrow (\pi^\otimes)^{-1}(F) \subset \ulkr$ as the smooth section which maps to the breaking point associated with $F$.
\end{defi}

\section{Coherent systems of ends} \label{sect_ends}

We next build, for every $\ell, k \geq 0$, half-line neighborhoods of the endpoints over $\cllkr$ that are coherent with respect to the product structure of its corners. These ends will later support predetermined Morse functions whereas perturbations of these functions will be allowed on their complement.

\begin{defi} \label{cohe_ends}
        A coherent system of ends is a collection of closed subsets $\mathcal{V} = \{\mathcal{V}_{\ell,k} \subset \ulkr \}_{\ell,k \geq 0}$ coming from the closure of a corresponding collection of open subsets such that
\begin{enumerate}
        \item $\mathcal{V}_{0,0} = \mathcal{U}^\otimes_{0,0} \equiv \overline{\R}_-$, $\mathcal{V}_{1,0} = \mathcal{U}^\otimes_{1,0} \equiv \overline{\R}$. Otherwise if $\ell -2 +2k \geq 0$, then for every irreducible component $C^{(i)}$ of $C$, $\mathcal{V}_{\ell,k} \bigcap C^{(i)}$ is a disjoint union of closed neighborhoods of its endpoints, each being homeomorphic to $[-\infty, -R] \subsetneq \overline{\mathbb{R}}_{-}$, and closed intervals on some interior lines, each being homeomorphic to $[-\frac{\lambda'}{2}, \frac{\lambda'}{2}] \subsetneq [-\frac{\lambda}{2}, \frac{\lambda}{2}]$, \label{c1}

        \item for every irreducible component $C^{(i)} \in \mathcal{C}\ell^\otimes_{\ell^{(i)},k^{(i)}}$ of $C$, $\mathcal{V}_{\ell,k} \bigcap C^{(i)} = \mathcal{V}_{\ell^{(i)},k^{(i)}} \bigcap C^{(i)}$. Therefore, the neighborhoods over $C^{(i)}$ are invariant over deformations of $C \setminus C^{(i)}$, over permutations of the $z_h$'s which leave the $z^{(i)}_h$'s invariant, it is independent of $\ell$, $k$ and of the relative position of $C^{(i)}$ in $C$. \label{c2}
\end{enumerate}

\end{defi}

Although this is not strictly necessary, one can construct coherent systems of ends explicitely.

For $k=0$, condition $\ref{c1}$ of definition $\ref{cohe_ends}$ sets $\mathcal{V}_{0,0} = \mathcal{U}^\otimes_{0,0} \equiv \overline{\R}_-$, $\mathcal{V}_{1,0} = \mathcal{U}^\otimes_{1,0} \equiv \overline{\R}$.

Let $D_0(r) \subset [0,1]^2 \subset \mathbb{R}^2$ be the closed disk of radius $0<r<1$ centered at the origin with respect to the standard metric. Now endow the negative gradient flow lines of $h: (x,y)\in [0,1]^2 \mapsto y^2 - x^2$ with the metric induced from $\R$ by their usual time parametrization.

For $k \geq 1$, identify any line $l$ of $C \in \cllkr$ of finite length with the flow line of the corresponding length and set $\mathcal{V}_{\ell,k} \bigcap l = D_0(r) \bigcap l$. Otherwise if $l$ touches precisely one endpoint (breaking, leave or root) of $C$, identify it with $\{0\} \times [0,1]$ (resp. $[0,1] \times \{0\}$) if $l$ is above (resp. below) that endpoint, the endpoint being identified with $(0,0)$. Then say $\mathcal{V}_{\ell,k} \bigcap l = D_0(r) \bigcap l$. 

Condition $\ref{c1}$ follows directly from the above construction. Note that only the lines of sufficient length will intersect $\mathcal{V}_{\ell,k}$, and that for any disk it is possible to find a neighborhood of it that is not intersecting $\mathcal{V}_{\ell,k}$. Condition $\ref{c2}$ is satisfied simply because the choice of the ends is made independently over each line, the length being the only considered parameter.



\begin{figure}[h]
        \centering 
	\includegraphics[width=110mm]{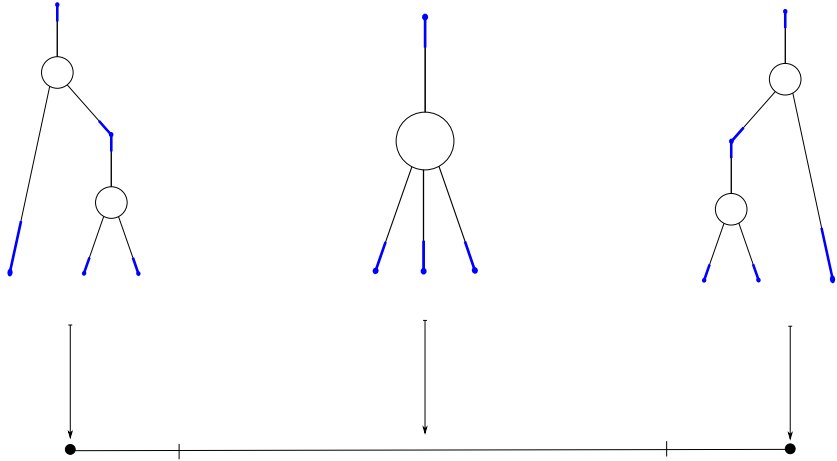}
        \caption{Choice of ends $\mathcal{V}_{3,0}$ over $\mathcal{C}\ell_{3,0}^\otimes$} \label{fig06}
\end{figure}

Therefore,

\begin{lem} \label{ends}
        There exist coherent systems of ends.
\end{lem}

\begin{rem}
        The requirement that $\mathcal{V}$ comes from the closure of open subsets is chosen to ensure that for a sufficiently large length parameter, an interior line will necessarily intersect $\mathcal{V}$ since the nearby breaking point lies in $\mathcal{V}$ (see figure \ref{fig06}).
\end{rem}

Note that one can define a coherent system of strip-like ends $\mathcal{S} = \{\mathcal{S}_{\ell,k} \subset \mathcal{U}_{\ell,k} \}$ over the disks of $\ulkr$ by pulling back a coherent system of strip-like ends $S = \{S_{\ell,k} \subset U_{\ell,k} \}$ (defined in a similar fashion, see \cite{Sei}, \cite{W1}, \cite{CM}) by the natural projection $\ulkr \rightarrow U_{\ell,k}$ that collapses the connecting lines.

\begin{rem}
An appropriate coherent choice of strip-like ends $S = \{S_{\ell,k} \subset U_{\ell,k} \}$ allows an explicit construction of the corner charts $\psi_T$ of proposition \ref{neighk}: one can set $\psi_T(D,X=(X_1, \ldots, X_{|T|}))$ to be the linear gluing of $D \in K_{\ell,k,\geq T}$ using the parameter $\frac{1}{X_i}$ on the corresponding node, $1\leq i \leq |T|$. 

The same operation can also be seen to generate corner charts for $\{\cllkr\}_{\ell,k \geq 0}$.
\end{rem}

\section{Coherent systems of perturbations}

First, we assign a pair of integers to each of the leaves of a $\otimes$-cluster so that these labels respect in a specific way its ribbon structure. This will later encode an appropriate choice of Morse functions over the ends. One could then either use a collection of  Morse functions as in \cite{Fu} and \cite{BC} or a single one as in \cite{CL}.

First, fix an integer $c \geq 0$.

\begin{defi} \label{lab}
        We call $\mathfrak{l}: \{0,1,2, \dots, \ell\} \rightarrow \{0, 1, ...,c\}^2$ an end labeling of length $\ell$ if
$ 0\leq \pi_1(\mathfrak{l}(0)) = \pi_1(\mathfrak{l}(1)) \leq \pi_2(\mathfrak{l}(1)) = \pi_1(\mathfrak{l}(2)) \leq \pi_2(\mathfrak{l}(2))= \pi_1(\mathfrak{l}(3)) \leq \dots  \leq \pi_2(\mathfrak{l}(\ell-1))= \pi_1(\mathfrak{l}(\ell)) \leq \pi_2(\mathfrak{l}(\ell)) = \pi_2(\mathfrak{l}(0)) \leq  c$, where $\pi_1$ and $\pi_2$ are the projection on the first and second factor of the product, respectively. Moreover, we consider every trivial labelings of length $\ell$ to be equivalent.

        We take $\mathfrak{L}$ to be the set of all the end labelings, $\mathcal{C}\ell^\otimes_{\ell, k, \mathfrak{l}} \equiv (\cllkr,\mathfrak{l})$ and $\mathcal{U}^\otimes_{\ell, k, \mathfrak{l}} \equiv (\ulkr,\mathfrak{l})$ for $\mathfrak{l} \in \mathfrak{L}$ of length $\ell$.

\end{defi}

Note that a labeling $\mathfrak{l}$ of the leaves of a $\otimes$-cluster $C \in \cllkr$ canonically determines an end labeling $\mathfrak{l}^{(i)}$ over each of its irreducible component $C^{(i)}$: Let $v^{(i)}_a$ be the $a^{th}$ leaf of $C^{(i)}$ and $E$ the numbers of the leaves that lie above $v^{(i)}_a$. Now say $\{0,1,2, \dots, \ell\} \setminus E = \{0, 1, 2, \dots, b-1, b\} \bigcup \{a, a +1, \dots, \ell -1, \ell, 0\}$, then take $\mathfrak{l}^{(i)}(a) = (\pi_2(b),\pi_1(a))$. In fact, one can use $\mathfrak{l}$ to associate in the same fashion a pair of integers $\mathfrak{l}(l)$ with each line $l$ of $C$ (see figure \ref{fig07}).

\begin{figure}[h]
        \centering 
	\includegraphics[width=85mm]{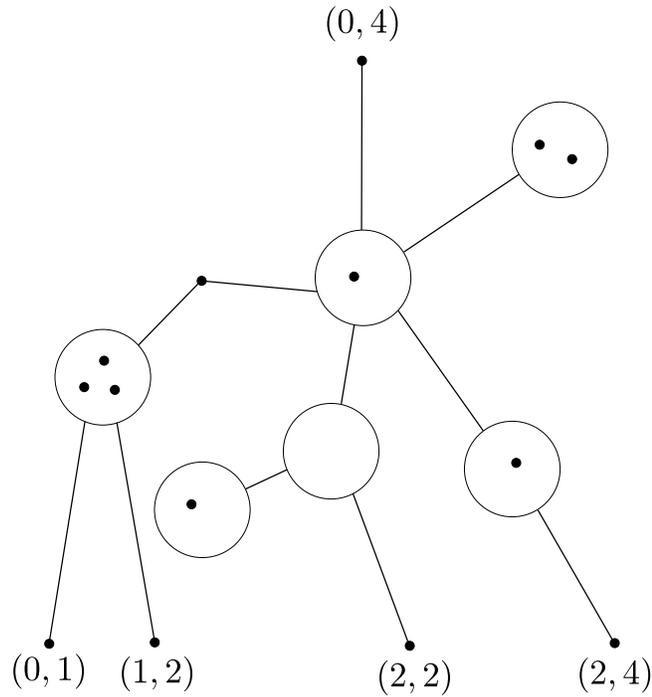}
        \caption{Labeling of $C \in \mathcal{C}\ell^\otimes_{4,8}$ for $c=4$} \label{fig07}
\end{figure}

Next, we assign to every point of $\mathcal{U}^\otimes_{\ell, k, \mathfrak{l}}$ an element in a product of Banach manifolds so that this choice is coherent with respect to the product structure of the corners of $\mathcal{C}\ell^\otimes_{\ell, k, \mathfrak{l}}$. This will later encode the choice of a Morse function plus a $\omega$-tamed almost complex structure on the target space at every point of a $\otimes$-cluster.

Let $\mathcal{M} \times \mathcal{J}$ be a product of Banach spaces with $\pi_\mathcal{M}$ and $\pi_\mathcal{J}$ the
projections on the first and second factor, respectively. 

\begin{defi} \label{pert}
Let $\mathcal{V}$ be a coherent system of ends and $\mathcal{S}$ be a coherent system of strip-like ends. A coherent system of perturbations vanishing on $\mathcal{V}$ and $\mathcal{S}$ is a collection of maps $P = \{ \mathcal{U}^\otimes_{\ell, k, \mathfrak{l}} \overset{p_{\ell,k,\mathfrak{l}}}{\rightarrow} \mathcal{M} \times \mathcal{J} \}_{\ell,k \geq 0, \mathfrak{l} \in \mathfrak{L}}$ such that
\begin{enumerate}
        \item $p_{\ell,k,\mathfrak{l}}$ is piecewise smooth,

        \item $\pi_\mathcal{J} \circ p_{\ell,k,\mathfrak{l}} \equiv 0$ on $\mathcal{S}_{\ell,k,\mathfrak{l}}$ and on the boundaries of the $\otimes$-clusters, 

        \item $\pi_\mathcal{M} \circ p_{\ell,k,\mathfrak{l}} \equiv 0$ on $\mathcal{V}_{\ell,k}$ and $\pi_\mathcal{M} \circ p_{\ell,k,\mathfrak{l}} \equiv 0$ over every line $l$ such that $\pi_1(\mathfrak{l}(l)) \neq \pi_2(\mathfrak{l}(l))$,

        \item for every disk $D$ of $C$ with no interior markings, let    
$\{ x^{D_0}_1, \dots, x^{D_0}_{|D_0|} \}$ (resp. $\{x_0^{D_0} \}$) be the ordered boundary markings of $D$  away from (resp. towards) the root such that $\pi_1(\mathfrak{l}(x^{D_0}_j)) = \pi_2(\mathfrak{l}(x^{D_0}_j))$. Then $\overset{|D_0|}{\underset{j=1}{\sum}} \pi_\mathcal{M} \circ p_{\ell,k,\mathfrak{l}}(x_j^{D_0}) - \pi_\mathcal{M} \circ p_{\ell,k,\mathfrak{l}}(x_0^{D_0}) = 0 $, where the last summand is understood to be $0$ if $\{x_0^{D_0} \} = \emptyset$, \label{line2} 

        \item for every irreducible component $(C^{(i)},\mathfrak{l}^{(i)})$ of $(C,\mathfrak{l})$, $p_{\ell,k,\mathfrak{l}}|_{C^{(i)}} = p_{\ell^{(i)},k^{(i)},\mathfrak{l^{(i)}}}|_{C^{(i)}}$. Therefore, $p_{\ell,k,\mathfrak{l}}$ invariant over changes of $(C,\mathfrak{l})$ which preserve $(C^{(i)},\mathfrak{l}^{(i)})$, that is, independent of changes of  $\ell$, of
$k$, of the relative position of $C^{(i)}$ in $C$, of permutations of the $z_h$'s that preserve the $z_h^{(i)}$'s and of $\mathfrak{l}$ that preserve $\mathfrak{l}^{(i)}$.

\end{enumerate}

We set $\mathfrak{P}(\mathcal{V},\mathcal{S})$ to be the set of all the coherent systems of perturbations vanishing on $\mathcal{V}$ and $\mathcal{S}$.

\end{defi}

\begin{figure}[h]
        \centering 
	\includegraphics[width=110mm]{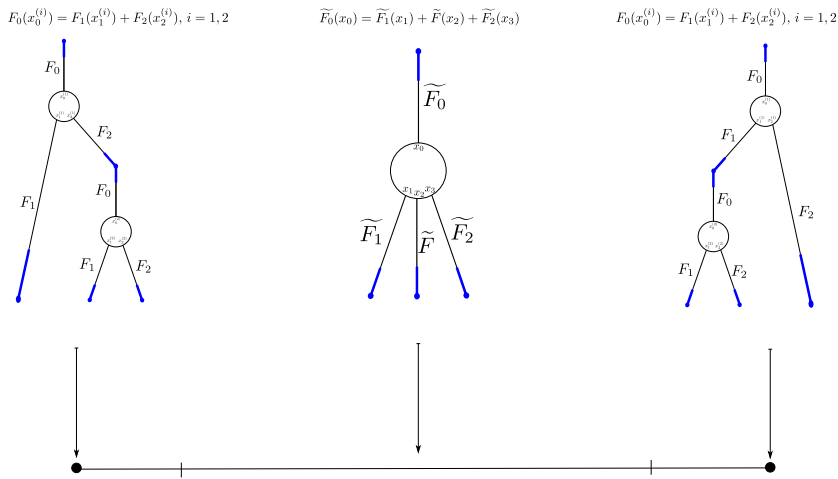}
        \caption{Choice of perturbation $P_{3,0}$ constant over $\mathcal{V}_{3,0}$ for $c=0$} \label{fig08}
\end{figure}

\begin{rem} \label{sub}
        The lines $l$ with $\pi_1(\mathfrak{l}(l)) = \pi_2(\mathfrak{l}(l))$ will later lie in the part of the cluster where perturbations of a single Morse function $f$ are used. $\pi_\mathcal{M} \circ p_{\ell,k,\mathfrak{l}}$ will be a generic choice of perturbation of $f$ needed to achieve transversality near trajectories where two such lines coincide in the target (see \cite{CL}). Condition \ref{line2} is only used to ensure that a hessian is defined naturally on these contact points, the disks with no interior markings corresponding to the so-called ghost disks.

\end{rem}


First notice that when considering disks as real spheres as in remark \ref{rema_sphe}, the choice of coherent perturbation datum over the disks (on $\ulk \rightarrow \klk$) is essentially performed in \cite{CM}. The pullback of such data to clusters under the line collapsing map could be used as $\pi_\mathcal{J} \circ p_{\ell,k,\mathfrak{l}}$.

For simplicity, it will be convenient to generate coherent perturbations that are compatible with the linear gluing of clusters for sufficiently large gluing parameters even though in practice we do not choose exactly these perturbations. We show how to construct such perturbations by induction over the dimension of the corners. More precisely, we extend smoothly a coherent choice over the corners $C(\cllkr)$ of $\cllkr$ to a neighborhood of $C(\cllkr)$. 

Definition \ref{pert} forces the use of trivial perturbation maps for $\ell -2 + 2k < 0$. The case $\ell -2 + 2k = 0$ splits into cases $\ell = 0, k=1$ and $\ell = 2, k=0$ with every possible labeling, over which we can choose arbitrary perturbation maps which satisfy the desired properties.

If we assume a coherent system of perturbations for $\ell -2 +2k < p$, then for $\ell -2 +2k = p$, notice that every $(C,\mathfrak{l}) \in C(\mathcal{C}\ell^\otimes_{\ell, k, \mathfrak{l}})$ splits into irreducible components $(C^{(i)},\mathfrak{l}^{(i)}) \in (\mathcal{C}l_{\ell^{(i)},k^{(i)}}^\otimes)_{\mathfrak{l}^{(i)}}$ such that $\ell^{(i)} -2 + 2k^{(i)} < p$ so we take $p_{\ell,k,\mathfrak{l}}|_{C^{(i)}} = p_{\ell^{(i)},k^{(i)},\mathfrak{l}^{(i)}}|_{C^{(i)}}$. This defines an appropriate choice of perturbation over $C(\mathcal{C}\ell^\otimes_{\ell, k, \mathfrak{l}})$ so we are left with an extension problem over $\mathcal{C}\ell^\otimes_{\ell, k, \mathfrak{l}}$.

First set a system of neighborhoods of the corners: Take $K_{\ell,k,\geq T'} \times X^{[0,\epsilon[}(T') \subset K_{\ell,k,\geq T'} \times X^{[0,1]}(T') \subset \cllkr$, where $0 < \epsilon < 1$. From definition \ref{defi_col}, it is not hard to see that $\nu^{[0,\epsilon[}(\mathcal{C}\ell_{\ell,k,\geq T}) \equiv \underset{T' \geq T}{\bigcup} K_{\ell,k,\geq T'} \times X^{[0,\epsilon[}(T')$ is a neighborhood of $\mathcal{C}\ell_{\ell,k,\geq T}^\otimes$. This neighborhood can be thought as linear gluings on the breakings of $\mathcal{C}\ell_{\ell,k,\geq T}^\otimes$ with sufficiently large gluing parameters.

Then, choose a set of smooth cutoff functions $\beta = \{ \beta_{\ell,k,\geq T} \}_T$ with $\beta_{\ell,k,\geq T}$ being constant to $1$ over $ \nu^{[0,\epsilon[}(\mathcal{C}\ell_{\ell,k,\geq T}^\otimes)$ and having support in $ \nu^{[0,\epsilon+1/N[}(\mathcal{C}\ell_{\ell,k,\geq T}^\otimes)$.

Now if a coherent system of perturbations has been chosen for $\ell -2 +2k < m$, then for $\ell -2 +2k = m$, take $\underset{T}{\prod} \beta_{\ell,k,\geq T}(C)$ times the perturbation over $C\in \cllkr$ induced by seeing $C$ as the result of a linear gluing. Therefore, the extended perturbation can be chosen so that its support lies in an arbitrarily small neighborhood of $C(\cllkr)$.

Define $\mathfrak{P}^{\ell -2 + 2k < m}(\mathcal{V}, \mathcal{S})$ (resp. $\mathfrak{P}^{\ell -2 + 2k \leq m}(\mathcal{V}, \mathcal{S})$) to be the choices of perturbation data on the strata of dimension $\ell -2 + 2k < m$ (resp. $\ell -2 + 2k \leq m$). Let then
\[ 
\mathfrak{P}^{\ell -2 + 2k \leq m}(\mathcal{V}, \mathcal{S}) \overset{R_m}{\longrightarrow} \mathfrak{P}^{\ell -2 + 2k < m}(\mathcal{V}, \mathcal{S})
\]
be the natural restriction map and
\[ 
\mathfrak{P}^{\ell -2 + 2k < m}(\mathcal{V}, \mathcal{S}) \overset{E^\beta_m}{\longrightarrow} \mathfrak{P}^{\ell -2 + 2k \leq m}(\mathcal{V}, \mathcal{S})
\]
the extension map using $\beta$. Thus (see corollary $3.7$ of \cite{CM}),

\begin{lem} \label{rest_bair}
        For $m\geq 1$, $E^\beta_m$ is a continuous linear right inverse for $R_m$, and therefore the preimage of a Baire subset of $\mathfrak{P}^{\ell -2 + 2k < m}(\mathcal{V}, \mathcal{S})$ under $R_m$ is a Baire subset of $\mathfrak{P}^{\ell -2 + 2k \leq m}(\mathcal{V}, \mathcal{S})$.
\end{lem}


\section{Operators over $\otimes$-clusters}

Now we describe a general operator over a cluster defined as a real Cauchy-Riemann operator over its disks and as a hessian connection over its lines. We compute the index of such an operator. Also, for every smooth cluster $C$, we define an operator $\ddel_C$ on its tangent bundle that has a cokernel canonically isomorphic to $T_C\cllkr$ if $\ell-2 +2k \geq 0$.

Take $C \in \cllkr$ and equip any line $l$ of $C$ with a Riemannian metric using the following model: Endow the standard
negative gradient flow lines of the function $h: (x,y) \in [0,1]^2 \mapsto y^2 - x^2$ with the metric induced from $\R$ by
their usual parametrization.

If $l$ is of finite length, take $l$ to be isometric to the flow line of this length. Otherwise, $l$ touches precisely one
endpoint of $C$. If $l$ is above (resp. under) this point, take $l$ to be
isometric to $[0,1] \times \{0\} \subset [0,1]^2$ (resp. $ \{0\} \times [0,1] \subset [0,1]^2$), with the endpoint corresponding to $(0,0)$. For the special case $C = \overline{\R}$, we equip it with the standard metric of $\R$. 

Set $real(\C^n,i) = \{ V \underset{subspace}{\subset} \C^n | iV \oplus V = \C^n\}$. A map $F:\partial
C \rightarrow real(\C^n,i)$, smooth on $\partial d$ for every disk $d$ and such that $F|_l \equiv \R^n
\subset \C^n$ for every line $l$
is called a real boundary condition on $\C^n
\times C \rightarrow C$ and by slight abuse of notation, denote by $F \rightarrow \partial C$ the
associated bundle. For such a real boundary condition, $L^{m,p}(C, \C^n, F) = \{ \xi \in L^{m,p}(C, \C^n) |
\xi(x) \in F_x \forall x \in \partial C\}$ and $L_\pi^{m-1,p}(C, \C^n) = L_\pi^{m-1,p}(C, \C^n, F) = {\underset{l \in lines(C)}\prod}
L^{m-1,p}(l,F|_l) \times \underset{d \in discs(C)}{\prod} L^{m-1,p}(d, \C^n)$.

Let an endpoint condition be a tuplet $\vec{A} = (A_0, \ldots, A_\ell, (A_b)_{b \in breakings(C)})$ of invertible and
diagonalizable $n \times n$  real matrices, then a map \\
$A \in L^{m-1,p} (\partial C, Mat_n(\R))$ such that $A(v_j)=A_j$
for $0 \leq j\leq \ell$ and $A(b)=A_b$ for $b \in breakings(C)$ is said to be compatible with $\vec{A}$.

\begin{defi}
        A real linear Cauchy-Riemann operator of class $(m-1,p)$ on $\C^n \times C \rightarrow C$ with real
boundary condition $F$ and endpoint condition $\vec{A}$ is a real linear operator
\begin{diagram}
        L^{m,p}(C,\C^n,F)&& \rTo^{\ddel}&& L_\pi^{m-1,p}(C,\C^n)\\
\end{diagram}
such that $\forall f \in L^{m,p}(C,\C,\R)$

\begin{diagram}
        \ddel(f\cdot \xi)|_d = f \cdot \ddel(\xi)|_d + \frac{\partial}{\partial \bar{z}} f \cdot \xi |_d, 
\forall d \in discs(C)\\
        \ddel(f\cdot \xi)|_l = f \cdot (\frac{\partial}{\partial t} - A)\xi|_l + \frac{\partial}{\partial t}
f \cdot \xi |_l,  \forall l \in lines(C)\\
\end{diagram}
for some $A \in L^{m-1,p}(\partial C,Mat_n(\R))$ compatible with $\vec{A}$.

\end{defi}

We now compute the index of such an arbitrary Cauchy-Riemann operator. Given $F\rightarrow \partial C$ a real
boundary condition, then since $real(\C^n,i) = Gl_n(\C) / Gl_n(\R)$, which has fundamental group canonically isomorphic
to $\Z$, define $\mu(F)$ to be the index of its associated cycle. Also, given an endpoints condition $\vec{A}$, define
$\mu^+(A_j)$ (resp. $\mu^-(A_j)$) to be the number of positive (resp. negative) eigenvalues of $A_j$ for $0 \leq j \leq
\ell$.

\begin{prp} \label{inde_comp}
        For any $\ell, k \geq 0$ and smooth $C \in \cllkr$, that is, with no lines of zero length nor broken lines, a real linear Cauchy-Riemann operator $\ddel$ on $C$ with real boundary condition $F \rightarrow \partial C$ and endpoint condition $\vec{A}$ is Fredholm of index

\[
        Ind(\ddel) = \mu^+(A_0) - \sum_{j=1}^\ell \mu^+(A_j) + \mu(F)
\]

\end{prp}

\begin{proof}
        Let $n$ be the rank of $F \rightarrow \partial C$. First, we notice that the index of the operator is
well-known over every smooth component of $C$.

Indeed, for any $l \in intlines(C)$, $\ddel |_l$ is Fredholm with $Ind(\ddel|_l) = n$, if $l$ contains $v_0$ then $\ddel
|_{l}$ is Fredholm with $Ind(\ddel|_{l}) = \mu^+(A_0)$ and for $1 \leq j \leq \ell$, if $l$ contains $v_j$ then $\ddel
|_{l}$ is Fredholm with $Ind(\ddel|_{l}) = \mu^-(A_j)$ (see \cite{S2}, \cite{W2}).

Moreover, for any $d \in discs(C)$, a Riemann-Roch-type theorem (see \cite{MS}) ensures that $\ddel|_d$ is
Fredholm with $Ind(\ddel|_d) = n+\mu(F|_d)$.

Then looking at the short exact sequence of operators
\begin{diagram} \label{seq}
          L^{m,p}(C,\C^n) && \rTo &&& \prod_{l \in lines(C)} L^{m,p}(l,F|_l) \times \prod_{d \in
discs(C)} L^{m,p}(d, \C^n,F|_d) &&&& \rTo^{\Pi ev_d} &&& \prod_{d \in discs(C)} F_{x^d_0}
\times \ldots \times F_{(x^d_{|d| -1})}\\
          \dTo^{\ddel} &&&&& \dTo^{\prod_l \ddel|_l \times \prod_d \ddel|_d} &&&&&&& \dTo  \\
          L^{m-1,p}_\pi(C,\C^n) && = &&& L_\pi^{m-1,p}(C,\C^n) &&& \rTo &&&& 0 &  \\
\end{diagram}
where $ev_d$ takes the differences at the contact points between $d$ and the lines touching it.

We have that the middle operator has index
\begin{align*}
        \sum_{l \in lines(C)} Ind(\ddel|_l) +& \sum_{d \in discs(C)} Ind(\ddel|_d) \\
&= \mu^+(A_0) +
\sum_{j=1}^\ell \mu^-(A_j) + n|intlines(C)| + n|discs(C)| + \mu(F) \\
        & = \mu^+(A_0) + \sum_{j=1}^\ell \mu^-(A_j) + \mu(F) +2n |intlines(C)| +n \\
\end{align*}
while the operator on the right has index
\[
        n \sum_{d \in discs(C)} |d| = 2n |intlines(C)| + n(\ell+1)
\]
so that
\begin{align*}
 Ind(\ddel) &= \mu^+(A_0) + \sum_{j=1}^\ell (\mu^-(A_j) -n) + \mu(F)\\
            &= \mu^+(A_0) - \sum_{j=1}^\ell \mu^+(A_j) + \mu(F)\\
\end{align*}
\end{proof}

We would also like to have an analogue of the Cauchy-Riemann operator over the tangent bundle for a cluster. However, this bundle is not well defined at nonsmooth points but if we consider only tangent sections $L^{m,p}_0 (C, TC, T \partial C)$ that vanish at these points besides vanishing at the interior markings, we get an operator

\begin{diagram}
        L^{m,p}_0 (C, TC, T \partial C)&& \rTo^{\ddel_C} && L^{m-1,p}_\pi (C,TC)\\
\end{diagram}
defined as the canonical Cauchy-Riemann operator over the disks and the Levi-Civita connection over the edges.

\begin{lem} \label{coker}
        For smooth $C \in \cllkr$ with $\ell -2 + 2k \geq 0$, $\ddel_C$ is injective with $coker(\ddel_C) = T_C \cllkr \cong \R^{\ell -2 +2k}$ whereas for arbitrary $C \in \cllkr$, $\ddel_C$ is injective with 
\[
coker(\ddel_C) \cong \R^{\ell -2 +2k -|\text{breakings}(C)| -|\text{real nodes}(C)| - 2|\text{complex nodes}(C)|}
\].
For $\ell = 1, k=0$, $\ddel_C$ is surjective with $ker(\ddel_c) \cong \R$, and for $\ell = k = 0$, $\ddel_C$ is an isomorphism.
\end{lem}

\begin{proof}
        The second part of the statement refers to the operator over $\overline{\R}$ (resp. $\overline{\R}_-$) equipped with the standard metric: The only kernel element is the translation over $\R$ (resp. the zero section) and the above index computation returns $1$ (resp. $0$), so the cokernel must be empty.
        
        When $\ell -2 + 2k \geq 0$, obviously, the kernel of $\ddel_C$ is trivial for $C$ being stable. The rest of the statement depends on standard Kodaira-Spencer deformation theory arguments.

        The Dolbeault theorem insures that the cokernel over a smooth disk with $(\ell+1)$ boundary markings and $k$
interior markings is canonically isomorphic to the $\ell - 2 + 2k$ independent first order deformations of its complex
structure, and that the cokernel over any edge of nonzero finite length is canonically isomorphic to the first order
variation of its length. Therefore, we immediately get the result for $C \in \klk \subset \cllkr$ and it extends canonically to the interior of any $K_{\ell,k,\geq T} \times X^{[0,1]}(T)$ cell with $|V(T)| \geq 2$.

        Furthermore, we notice that the creation of every breaking or real boundary node decreases the dimension of the first order deformations, and therefore the dimension of the cokernel, by one whereas a complex node decreases this dimension by two.
\end{proof}

\section{Orientations on $\{\cllkr\}_{\ell,k \geq 0}$}

As the subtitle suggests, a reader who is not interested in orientation considerations should skip this subsection. Orientation considerations on $\{\cllkr\}_{\ell,k \geq 0}$ will later be used to choose correct signs in the definition of $\delta^\otimes$.

Lemma \ref{coker} shows that for smooth $C$ with $\ell -2 +2k \geq 0$, $coker(\ddel_C)$ is canonically isomorphic to $T_C \cllkr$ so for $C = C^{(1)} \bigcup C^{(2)} \in F \in C_1(\cllkr)$ with $C^{(1)}$ and $C^{(2)}$ smooth, $coker(\ddel_{C^{(1)} \bigcup C^{(2)}})$ is canonically isomorphic to both $T_C F$ and $T_{(C^{(1)},C^{(2)})} \mathcal{C}l^\otimes_{\ell^{(1)},k^{(1)}} \times \mathcal{C}l^\otimes_{\ell^{(2)},k^{(2)}}$. These two identifications will be compared, up to orientation, via the operation of (linear) gluing of $\ddel_C$ at the breaking point. This procedure also applies to the cases where $\ell^{(1)} = 1$ and/or $\ell^{(2)} \leq 1$.

First, we fix a reference orientation of $\ddel_C$ for every smooth $C \in \cllkr$ with $(\ell,k) \neq (0,0)$.

\begin{defi}  \label{ori}

Let $ \mathcal{O}_{z_h} = \frac{\partial}{\partial Re(z_h)} \wedge \frac{\partial}{\partial Im(z_h)} $, $0\leq h \leq k$.

For $k \geq 1$, let $\mathcal{O}_\ell = \frac{\partial}{\partial x_1} \wedge \dots \wedge
\frac{\partial}{\partial x_\ell}$, take the orientation $\mathcal{O}_{\ell,k} = \mathcal{O}_\ell \wedge \mathcal{O}_{z_2} \wedge \dots \wedge \mathcal{O}_{z_k}$ over $\klk$ and extend it to an orientation $\mathcal{O}^\otimes_{\ell,k}$ on $\cllkr$.

For $k = 0$ and $\ell = 1$, we take $\mathcal{O}_{\ell,k}^\otimes = \frac{\partial}{\partial t}$ to be the field that generates the positive (away from the root) translation over $\overline{\mathbb{R}}$.

\end{defi}

Now we find the combinatorial formula for the difference between the orientation $\partial_{F} \mathcal{O}^\otimes_{\ell,k}$
induced on a component $F = K_{\ell^{(1)},k^{(1)}} \times K_{\ell^{(2)},k^{(2)}} \subset C_1(\cllkr)$ by
$\mathcal{O}^\otimes_{\ell,k}$ by using the usual outward pointing direction first convention, and $\mathcal{O}^\otimes_{\ell^{(1)},k^{(1)}} \wedge
\mathcal{O}^\otimes_{\ell^{(2)},k^{(2)}}$.

Note that this orientation induction procedure is also defined in the cases where $C^{(1)}$ or $C^{(2)}$ is $\overline{\R}$ so that their concatenation $C$ does not belong to $C_1(\cllkr)$. One can still generate a one-dimensional family of $\otimes$-clusters $\pi: \#C \rightarrow ]0,\infty]$, $\#_{\rho}C \mapsto \rho$ by the linear gluing maps around the concatenation point (see \cite{S}). Indeed, we have the exact sequence of operators

\begin{diagram} \label{seq3}
        L^{m,p}_0 (\#_{\rho}C, T\#_{\rho}C, T \partial \#_{\rho}C) & & \rTo &&& L^{m,p}(\#_{\rho}C,\iota_{\rho}^*T\#C,\iota_{\rho}^*T\partial\#C) &&&& \rTo & \frac{L^{m,p}_0(\#_{\rho}C,\iota_{\rho}^*T\#C,\iota_{\rho}^*T\partial\#C)}{L^{m,p}_0 (\#_{\rho}C, T\#_{\rho}C, T \partial \#_{\rho}C)}\\
         \dTo^{\ddel_{\#_{\rho}C}} &&&& & \dTo^{\ddel_{\iota_\rho}} &&&& &\dTo^{\ddel_{\iota_\rho} / \ddel_{\#_{\rho}C}}\\
         L^{m-1,p}_\pi (\#_{\rho}C,T\#_{\rho}C) && & \rTo && L^{m-1,p}(\#_{\rho}C,\iota_{\rho}^*T\#C) && \rTo && & \frac{L^{m-1,p}(\#_{\rho}C,\iota_{\rho}^*T\#C)}{L^{m-1,p}_\pi (\#_{\rho}C,T\#_{\rho}C)} \\
\end{diagram}

\noindent where $\iota_\rho: \pi^{-1}(\rho) \hookrightarrow \#C$ is the canonical inclusion. For $\rho= \infty$, the right operator is an isomorphism, the other two having the same kernel and cokernel, and therefore we can choose $\mathcal{O}^\otimes_{\ell^{(1)},k^{(1)}} \wedge \mathcal{O}^\otimes_{\ell^{(2)},k^{(2)}}$ to orient the center operator. For $\rho < \infty$, the right operator has a one-dimensional kernel and the left one has index one less than at $\rho = \infty$. Then for large $\rho$, one can consider the orientation induced from $\rho = \infty$ on the center operator and choose orientation $\mathcal{O}^\otimes_{\ell,k}$ on the left one, so the right operator inherits an orientation. We say it is outward normal if it projects onto the $\frac{\partial}{\partial \rho}$ orientation of $]0,\infty]$ and then assert that $\mathcal{O}^\otimes_{\ell^{(1)},k^{(1)}} \wedge \mathcal{O}^\otimes_{\ell^{(2)},k^{(2)}}$ on the composition induces the outward normal orientation on its glued 1-parameter family (relative to $\mathcal{O}^\otimes_{\ell,k}$), or simply $\partial_{\ddel_C} \mathcal{O}^\otimes_{\ell,k} = \mathcal{O}^\otimes_{\ell^{(1)},k^{(1)}} \wedge \mathcal{O}^\otimes_{\ell^{(2)},k^{(2)}}$.

\begin{lem} \label{out}
        Let $\ell^{(1)} \geq 1$, $\ell^{(2)} \geq 0$ and $k^{(1)}, k^{(2)} \geq 0$ such that $(\ell^{(2)}, k^{(2)}) \neq (0,0)$. Then let $\ell +1 = \ell^{(1)} + \ell^{(2)}$, $k = k^{(1)} + k^{(2)}$, $C^{(1)} \in \mathcal{C}\ell_{\ell^{(1)},k^{(1)}}^\otimes$, $C^{(2)} \in \mathcal{C}\ell_{\ell^{(2)},k^{(2)}}^\otimes$ smooth and define $C =  C^{(1)} \bigsqcup C^{(2)} \big{/}_{v^{(1)}_j \sim v^{(2)}_0}$, that is, $C$ is the concatenation of $C^{(1)}$ and $C^{(2)}$ on the $j^{th}$ leaf of $C^{(1)}$ (see figure \ref{fig09}). Then

\[
        \partial_{\ddel_C} \mathcal{O}^\otimes_{\ell,k} = (-1)^{(\ell^{(1)} -j)\ell^{(2)} + (j-1)} \mathcal{O}^\otimes_{\ell^{(1)},k^{(1)}} \wedge
\mathcal{O}^\otimes_{\ell^{(2)},k^{(2)}}.
\]

\end{lem}

\begin{figure}[h]
        \centering 
	\includegraphics[width=110mm]{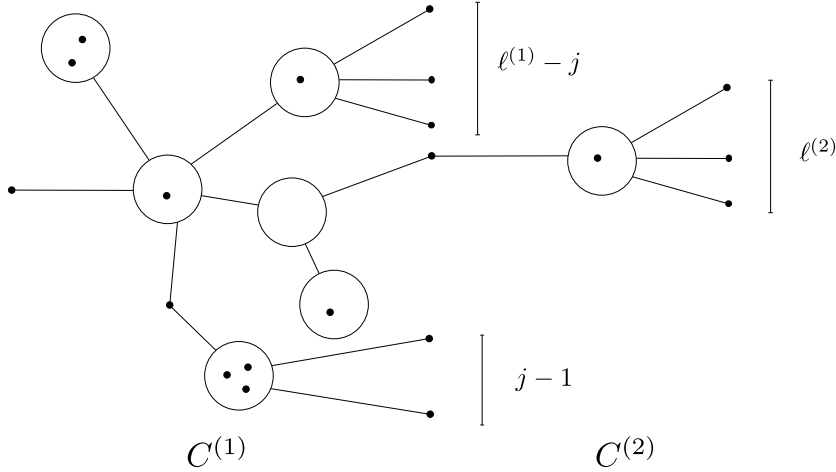}
        \caption{Cluster $C \in C_1(\mathcal{C}\ell_{8,9}^\otimes)$} \label{fig09}
\end{figure}

\begin{proof}
         In the cases where $k^{(1)} \geq 1$ and $k^{(2)} \geq 1$, the composition of $\otimes$-clusters is a stable $\otimes$-cluster with one breaking and therefore lies in a codimension one face $F$ of $\cllkr$. Then as in lemma 2.9 of \cite{W1}, we compute using coordinates in a neighborhood of $F$, the outward normal coordinate being denoted by $\frac{\partial}{\partial n_F}$:

\begin{align*}
& \frac{\partial}{\partial n_F} \wedge (-1)^{ (\ell^{(1)} -j) \ell^{(2)}+ (j-1)} \mathcal{O}^\otimes_{\ell^{(1)},k^{(1)}} \wedge \mathcal{O}^\otimes_{\ell^{(2)},k^{(2)}} \\
&= \bigwedge_{h=2}^{k^{(1)}} \mathcal{O}_{z^{(1)}_h} \wedge \bigwedge_{h=2}^{k^{(2)}} \mathcal{O}_{z^{(2)}_h} \wedge (-1)^{ (\ell^{(1)} - j) \ell^{(2)}+ (j-1)} \frac{\partial}{\partial n_F} \wedge \cO_{\ell^{(1)}} \wedge \cO_{\ell^{(2)}}\\
&= \bigwedge_{h=2}^{k^{(1)}} \mathcal{O}_{z^{(1)}_h} \wedge \bigwedge_{h=2}^{k^{(2)}} \mathcal{O}_{z^{(2)}_h} \wedge (-1)^{ (\ell^{(1)} - j) \ell^{(2)}+ (j-1)} \frac{\partial}{\partial n_F} \wedge \frac{\partial}{\partial x_1^{(1)}} \wedge \ldots \wedge \frac{\partial}{\partial x_{j}^{(1)}} \wedge \ldots \wedge \frac{\partial}{\partial x_{\ell^{(1)}}^{(1)}} \wedge \\ & \frac{\partial}{\partial x_1^{(2)}} \wedge \ldots \wedge
\frac{\partial}{\partial x_{\ell^{(2)}}^{(2)}}\\
&= \bigwedge_{h=2}^{k^{(1)}} \mathcal{O}_{z^{(1)}_h} \wedge \bigwedge_{h=2}^{k^{(2)}} \mathcal{O}_{z^{(2)}_h} \wedge (-1)^{ (\ell^{(1)} - j) \ell^{(2)}+ (j-1)} \\
&- \frac{\partial}{\partial Im(z_1^{(2)})} \wedge \frac{\partial}{\partial x_{1}} \wedge \ldots \wedge \frac{\partial}{\partial x_{j-1}} \wedge \frac{\partial}{\partial Re(z_1^{(2)})} \wedge \frac{\partial}{\partial x_{j+\ell^{(2)}}} \wedge \ldots \wedge \frac{\partial}{\partial x_{\ell^{(2)}+\ell^{(1)}-1}} \wedge \frac{\partial}{\partial x_{j}} \wedge \ldots \wedge \frac{\partial}{\partial x_{j+\ell^{(2)}-1}}\\
&= \bigwedge_{h=2}^{k^{(1)}} \mathcal{O}_{z^{(1)}_h} \wedge \bigwedge_{h=1}^{k^{(2)}} \mathcal{O}_{z^{(2)}_h} \wedge \frac{\partial}{\partial x_{1}} \wedge \ldots \wedge \frac{\partial}{\partial x_{j-1}} \wedge \frac{\partial}{\partial x_{j}} \wedge \ldots \wedge
\frac{\partial}{\partial x_{j+\ell^{(2)}-1}} \wedge \frac{\partial}{\partial x_{j+\ell^{(2)}}} \wedge \ldots \wedge
\frac{\partial}{\partial x_{\ell^{(2)}+\ell^{(1)}-1}} \\
&= \bigwedge_{h=1}^{k} \mathcal{O}_{z_h} \wedge \cO_{\ell} \\
&= \cO_{\ell,k}^\otimes. \\
\end{align*}



If $\ell^{(1)} = 1$, $k^{(1)} = 0$, so that $C^{(1)}$ is a line, and $k^{(2)} \geq 1$, then $\frac{\partial}{\partial \rho}$ goes to $0$ under the connecting map so it lifts to a generator of the kernel of the center operator $\ddel_{\iota_\rho}$. It follows from the definition of the linear gluing that this element gets approximated by $\frac{\partial}{\partial t}$ as $\rho$ goes to $\infty$. Therefore, we compute, up to large $\rho$ approximation,

\begin{align*}
\mathcal{O}^\otimes_{\ell^{(1)},k^{(1)}} \wedge \mathcal{O}^\otimes_{\ell^{(2)},k^{(2)}}
&= \frac{\partial}{\partial t} \wedge \mathcal{O}^\otimes_{\ell^{(2)},k^{(2)}} \\
&= \frac{\partial}{\partial \rho} \wedge \mathcal{O}^\otimes_{\ell,k} \\
&= \partial_{\ddel_C} \mathcal{O}^\otimes_{\ell,k}. \\
\end{align*}

If $\ell^{(2)} = 1$, $k^{(2)} = 0$, so that $C^{(2)}$ is a line, and $k^{(1)} \geq 1$, then $\frac{\partial}{\partial \rho}$ goes to $0$ under the connecting map so it lifts to a generator of the kernel of $\ddel_{\iota_\rho}$. It follows from the definition of the linear gluing that this element gets approximated by $-\frac{\partial}{\partial t}$ as $\rho$ goes to $\infty$. Therefore,

\begin{align*}
(-1)^{\ell^{(1)} -1 } \mathcal{O}^\otimes_{\ell^{(1)},k^{(1)}} \wedge \mathcal{O}^\otimes_{\ell^{(2)},k^{(2)}}
&= (-1)^{\ell^{(1)} } \mathcal{O}^\otimes_{\ell^{(1)},k^{(1)}} \wedge -\frac{\partial}{\partial t} \\
&= \frac{\partial}{\partial \rho} \wedge \mathcal{O}^\otimes_{\ell,k} \\
&= \partial_{\ddel_C} \mathcal{O}^\otimes_{\ell,k}. \\
\end{align*}

If $\ell^{(1)} = \ell^{(2)} = 1$ and $k^{(1)} = k^{(2)} = 0$, then it follows from the definition of the linear gluing that $\frac{\partial}{\partial \rho}$, as $\rho$ goes to $\infty$, is approximated by $\frac{\partial}{\partial t^{(1)}} - \frac{\partial}{\partial t^{(2)}} =  \mathcal{O}^\otimes_{\ell^{(1)},k^{(1)}} - \mathcal{O}^\otimes_{\ell^{(2)},k^{(2)}}$. Moreover, $\frac{\partial}{\partial t} = \mathcal{O}^\otimes_{\ell,k}$ goes, in the limit, to $\frac{\partial}{\partial t^{(1)}} + \frac{\partial}{\partial t^{(2)}} =  \mathcal{O}^\otimes_{\ell^{(1)},k^{(1)}} + \mathcal{O}^\otimes_{\ell^{(2)},k^{(2)}}$. Therefore,

\begin{align*}
2 \mathcal{O}^\otimes_{\ell^{(1)},k^{(1)}} \wedge \mathcal{O}^\otimes_{\ell^{(2)},k^{(2)}}
&= (\mathcal{O}^\otimes_{\ell^{(1)},k^{(1)}} - \mathcal{O}^\otimes_{\ell^{(2)},k^{(2)}}) \wedge (\mathcal{O}^\otimes_{\ell^{(1)},k^{(1)}} + \mathcal{O}^\otimes_{\ell^{(2)},k^{(2)}}) \\
&= \frac{\partial}{\partial \rho} \wedge \mathcal{O}^\otimes_{\ell,k} \\
&= \partial_{\ddel_C} \mathcal{O}^\otimes_{\ell,k}. \\
\end{align*}

\end{proof}

\chapter{The $\otimes$-cluster complex}

\section{$\widetilde{Gl_n(\R)}^{\pm}$-structures}

Let $n\geq 2$. Denote by $\widetilde{Gl_n(\R)}^+$ and $\widetilde{Gl_n(\R)}^-$ the two Lie group structures on the
everywhere nontrivial twofold cover of $Gl_n(\R)$ for which the covering map is a group morphism. Say
$\widetilde{Gl_n(\R)}^+$ (resp. $\widetilde{Gl_n(\R)}^-$) is the group in which a lift of a reflection is of order two
(resp. four). From now on, we choose $\gl$ to denote a particular choice of these group structures, which must be kept
constant in all of the upcoming statements.

\begin{defi}
        A $\gl$-structure on a real vector bundle of rank $n$ $F \rightarrow B$ is the choice of a lift of the
$Gl_n(\R)$-principal structure of $Fr(F) \rightarrow B$, the frame bundle of $F$, to a $\gl$-principal structure, i.e.
the choice of a $\gl$-principal bundle $\widetilde{Fr(F)}^{\pm} \rightarrow B$ that factors through $Fr(F) \rightarrow
B$ by the covering map on every fiber.
\end{defi}

The obstruction to the existence of such a structure is well-known (see \cite{KT}) to be $w_2(F)$ (resp. $w_2(F) + w_1(F)^2$)
for the $+$ (resp. $-$) choice. Note that by the Wu formula  (see \cite{MilS}), for $n \leq 3$, $F$ always has a
$\widetilde{Gl_n(\R)}^-$-structure.

From the homotopy exact sequence associated with \\ $real(\C^n,i) \cong Gl_n(\C) / Gl_n(\R)$, we see that it has a second
homotopy group isomorphic to $\Z / 2\Z$, generated by the action of a retraction in $Gl_n(\C)$ of a nontrivial loop of
$Gl_n(\R)$.

Thus, for any $\mu \in \Z$ and $\otimes$-cluster $C \in \cllkr$, the space $\mathcal{R}(C, \mu)$ of real boundary conditions
of index $\mu$ on $\C^n \times C \rightarrow C$ has fundamental group isomorphic to $(\Z /
2\Z)^{|discs(C)|}$, generated by the action of the above retraction over each disk.

Now if $\widetilde{\mathcal{R}}^{\pm}(C,\mu)$ denotes the space of real boundary conditions of Maslov index $\mu$ on
$\C^n \times C \rightarrow C$ together with a choice of $\gl$-structure on them, it is not hard to see that

\begin{lem} \label{def_gl}
        $\pi_1(\mathcal{R}(C,\mu))$ acts transitively on the fibers of the cover
\[
        \widetilde{\mathcal{R}}^{\pm}(C,\mu) \rightarrow \mathcal{R}(C,\mu)
\]
\end{lem}

In other words: Up to homotopy of the boundary condition, every choice of $\gl$-structure on it is the same.

Given any endpoint condition $\vec{A}$ on $C$, the same phenomena happens with the cover $ori(\ddel,\vec{A})
\rightarrow \mathcal{R}(C,\mu)$ with fiber over $F \rightarrow \partial C$ being the two orientations of the
Cauchy-Riemann operators subject to $F$ and $\vec{A}$. This is a consequence of the short exact sequence \ref{seq} and
the following lemma:

\begin{lem} \cite{W1} \cite{FOOO} \label{def_op}
        For any $\mu \in \Z$, take $\mathcal{R}(\C^n,i,\mu)$ to be the space if real boundary contitions of index $\mu$
on $\C^n \times D_\C \rightarrow D_\C$ and $ori(\ddel) \rightarrow \mathcal{R}(\C^n,i,\mu)$ be the cover with fiber
over $F \rightarrow \partial D_\C$ being the two orientations of the Cauchy-Riemann operators on $\C^n \times D_\C
\rightarrow D_\C$ subject to $F$, then
\[
        ori(\ddel) \rightarrow \mathcal{R}(\C^n,i,\mu)
\]
is the nontrivial twofold cover.
\end{lem}

Therefore, up to homotopy of the boundary conditions, every orientation of an operator over a cluster is the same.

Thus, again in the case of clusters, the choice of $\gl$-structures on the boundary conditions allows one to manage orientations over families of operators.

\section{Algebraic and geometric settings} \label{defi_gene}

Let $(M,\omega)$ be a $2n$-dimensional compact symplectic manifold with $n \geq 2$ and $L$ a closed connected monotone lagrangian submanifold with minimal Maslov number $N_L \geq 2$ that admits a $\gl$-structure. That is, there exists $\tau > 0$ such that $\omega([u]) = \tau \mu([u])$ $\forall [u] \in \pi_2(M,L)$, where $\mu$ is the Maslov index and $\omega$ stands for the integration of $\omega$, and $N_L = \text{min}\{\mu([u])>0 \mid [u] \in \pi_2(M,L) \} \geq 2$. 

Now we fix  $c \in \N$, a $\omega$-tamed almost complex structure $J$ on $(M,\omega)$, a metric $g$ on $L$, a Morse-Smale (with respect to $g$) function $f: L \rightarrow \R$ with a single local maximum, and if $c\geq 1$, some additional functions $\{ f_r: L \rightarrow \R\}_{0 \leq r \leq c}$ such that $\{f_{r_1} - f_{r_2}\}_{0 \leq r_1 < r_2 \leq c}$ is a set of Morse-Smale functions, each having a single local maximum.

Let $\Lambda = \Z[t,t^{-1}]$ and $\Lambda^+ = \Z[t]$ be polynomial rings, with the grading induced by setting the index of $t$ to be $\mu(t)= N_L$. These should be seen as simplified versions of the usual full Novikov rings $\widetilde{\Lambda} = \Z((t^{\pi_2(M,L)})) = \{ \sum a_i e^{[u_i]} \mid a_i \in \Z, [u_i] \in \pi_2(M,L), \forall R \in \R, | \{ i | \omega([u_i]) < R \} | < \infty \}$ and $\widetilde{\Lambda^+} = \Z((t^{\pi_2(M,L)^+}))$, respectively, where $\pi_2(M,L)^+ = \{A \in \pi_2(M,L) | \omega(A) \geq 0\}$. For simplicity, we use $\Lambda$ throughout the following, but it can be replaced by either of the above rings.

Set the $\Lambda$-modules $V = crit(f)\otimes \Lambda$ and $V_{j_1,j_2} = crit(f_{j_1} - f_{j_2})\otimes \Lambda$, $0 \leq j_1 < j_2 \leq c$. Let then $T^m V= V^{\otimes m} = V \otimes \ldots \otimes V$ be the rank $m$ tensor product of $V$.

\begin{defi}
Let 
\begin{align*}
        \mathcal{C}\ell^\otimes =& \mathcal{C}\ell^\otimes(L,f,f_0, \ldots, f_c) \\
=&  \, \underset{0 \leq c' \leq c}\bigoplus \, \underset{\substack{0\leq j_0 \leq \ldots \leq j_{c'} \leq c\\ m_0, \ldots, m_{c'} \geq 0}}{\bigoplus} T^{m_0}V \otimes V_{j_0,j_1} \otimes T^{m_1}V \otimes \ldots \otimes T^{m_{c'-1}}V \otimes V_{j_{c'-1},j_{c'}} \otimes T^{m_{c'}}V
\end{align*}
where the $c'=0$ inner union is understood to be $T(V) = \underset{m \geq 1}{\bigoplus} T^{m}V$, the tensor algebra of $V$, and the summation symbol allows finite sums only. 

We write an arbitrary monomial element $x = x_1 \otimes \dots \otimes x_q \otimes t^d \in \mathcal{C}\ell^\otimes$ as $ x_1 \dots x_q t^d $ and define its index to be $\mu(x)= \overset{q}{\underset{j=1}{\sum}} \mu^+(x_j) + \mu(t^d) = \overset{q}{\underset{j=1}{\sum}} (n-|x_j|) + d N_L$ and its cardinality as $q(x)= q \in \N$. 
\end{defi}

We now define the usual trajectories between elements of $\mathcal{C}\ell^\otimes$ that satisfy gradient and pseudoholomorphic equations. For the above monotone setting, a result of Lazzarini (\cite{L}), used as in \cite{BC}, allows one to use a fixed almost-complex structure. It will then be convenient to use perturbations not depending on the interior markings and to consider maps from (possibly unstable) $\otimes$-clusters with no interior markings.

\begin{defi}
For a coherent system of ends $\mathcal{V}$ independent of the positions of the interior markings, let $P = \{ p_{\ell,k,\mathfrak{l}} \}_{\ell,k \geq 0, \mathfrak{l} \in \mathfrak{L}}$ be a coherent system of perturbations constant on $\mathcal{V}$ with target $\mathcal{M} \times \mathcal{J}$ such that
\begin{itemize}
\item $\pi_\mathcal{J} \circ p_{\ell,k,\mathfrak{l}} \equiv 0$, $\forall \ell,k \geq 0$ and $\mathfrak{l} \in \mathfrak{L}$,
\item for any $k \geq 1$, $p_{\ell,k,\mathfrak{l}}$ is the pullback of $p_{\ell,0,\mathfrak{l}}$ via the forgetful map $\mathcal{C}l_{\ell,k,\mathfrak{l}}^\otimes \rightarrow \mathcal{C}l_{\ell,0,\mathfrak{l}}^\otimes$.
\end{itemize}
We will refer to $P$ as being a monotone coherent system of perturbations constant on $\mathcal{V}$.
\end{defi}

Here, $\mathcal{M} = \mathcal{M}_f \equiv exp_f( B_f  \subset C^\epsilon(L,\R))$ where the latter supports a Banach chart of the space of smooth Morse-Smale functions on $L$ centered at $f$ (see \cite{F2}). Also, $\mathcal{J} = \mathcal{J}_J \equiv exp_{J}( B_{J} \subset C^\epsilon(M,T_{J}\mathbf{J}(TM,\omega)))$ supports a Banach chart of smooth $\omega$-tamed almost-complex structures on $(M,\omega)$ centered at $J$ (see \cite{F2}, \cite{MS}).

For every $\ell,k \geq 0$, $\mathfrak{l} \in \mathfrak{L}$ and $\otimes$-cluster $C \in \mathcal{C}l_{\ell,k,\mathfrak{l}}^\otimes$, we consider the configuration resulting from forgetting its interior marked points, but without stabilizing it, and shall refer to it by $C' \in \mathcal{C}l_{\ell,\mathfrak{l}}^\otimes$, or later on by $C \in \mathcal{C}l_{\ell,\mathfrak{l}}^\otimes$. First, note that $C' \in \mathcal{C}l_{\ell,\mathfrak{l}}^\otimes$ might contain disks with only one or two boundary markings, and thus be unstable. Secondly, a monotone coherent system of perturbations $P$ constant on $\mathcal{V}$ defines a perturbation over any $C' \in \mathcal{C}l_{\ell,\mathfrak{l}}^\otimes$.

The tangent operator $\ddel_{C'}$ is still defined and has index $\mu( \ddel_{C'} )= -(\ell -2)$ since there is no contribution from the interior markings anymore, and the unstable disks with one (resp. two) marking(s) contribute by $+2$ (resp. $+1$) to the dimension of both the kernel and the cokernel. Moreover, the orientations of $\ddel_{C'}$ naturally correspond to those of $\ddel_{C}$ via the following exact sequence:

\begin{diagram}
        0 & \rTo & L^{m,p}_0 (C, TC, T \partial C) & \rTo &&& L^{m,p}_0 (C', TC', T \partial C') &&& \rTo^{\Pi ev} & \prod_{h=1}^{k} T_{z_h}C \cong \C^k & \rTo & 0\\
        & & \dTo^{\ddel_C} &&&& \dTo^{\ddel_{C'}} &&&& \dTo & &\\
        0 & \rTo & L^{m-1,p}_\pi (C,TC) && \rTo && L^{m-1,p}_\pi (C',TC') && \rTo && 0 & \rTo & 0 \\
\end{diagram}

\begin{defi}
        Let $x^- = x_0^{(1)} \dots x_0^{(q)}$ and $x^+ = x_1^{(1)}  \ldots x_{\ell^{(1)}}^{(1)}  \dots x_1^{(q)} \ldots x_{\ell^{(q)}}^{(q)}$ be generators of $\mathcal{C}\ell^\otimes$ with $\ell^{(r)} \geq 1$, $1 \leq r \leq q$. Consider a pair $(u, C)$ such that

\begin{enumerate}
        \item $C = C^{(1)} \times \dots \times C^{(q)} \in \mathcal{C}l_{\ell^{(1)}}^\otimes \times \dots \times \mathcal{C}l_{\ell^{(q)}}^\otimes$,

        \item $u: C = \overset{q}{\underset{r=1}{\bigsqcup}} C^{(r)} \rightarrow M$ continuous such that $\forall 1 \leq r \leq q$,
        \begin{enumerate}
                \item $u(v_j(C^{(r)})) = x_j^{(r)}$ for $0 \leq j \leq \ell^{(r)}$, this naturally defines a labeling $\mathfrak{l}^{(r)}$ on the endpoints of $C^{(r)}$ generated by taking $\mathfrak{l}^{(r)}(j) = (j_1,j_2)$ if $x_j^{(r)} \in crit(f_{j_1}-f_{j_2})$,

                \item $u(\partial C^{(r)}) \subset L$,



                \item over every line $l$ of $C^{(r)}$ such that $\mathfrak{l}^{(r)}(j) = (j_1,j_2)$ with $j_1 < j_2$, $u$ satisfies the gradient equation $du(-\frac{\partial}{\partial t}(p)) = -\nabla_g (f_{j_1} - f_{j_2}) \circ u(p)$ $\forall p \in l$, where $\frac{\partial}{\partial t}$ is the unit-length vector field on $l$ that points away from the root,

                \item over every line $l$ of $C^{(r)}$ such that $\mathfrak{l}^{(r)}(j) = (j,j)$, $u$ satisfies the gradient equation $du(-\frac{\partial}{\partial t}(p)) = -\nabla_g(\pi_\mathcal{M} \circ p_{k^{(r)},\ell^{(r)},\mathfrak{l}^{(r)}}(p)) \circ u(p)$ $\forall p \in l$, where $\frac{\partial}{\partial t}$ is the unit-length vector field on $l$ that points away from the root,

                \item over every disc $d$ of $C^{(r)}$, $u$ satisfies the pseudoholomorphic equation $du \circ j(x) = J(u(x)) \circ du(x)$ $\forall x \in d$, where $j$ is the underlying complex structure of $d$.
        \end{enumerate}

\end{enumerate}
One can consider the linearized Cauchy-Riemann operator $\ddel_u$ on $u^*TM \rightarrow C$ with boundary conditions $u|_{\partial C}^*TL \rightarrow \partial C$ and path of matrices $A$ given by the hessians of the functions determined by $P$. Now take $(C,u,o_u)$ where $(C,u)$ is as above and $o_u$ is an orientation of $\ddel_u$ and denote $[(C,u,o_u)]$ the homotopy class of these data.

We call $[(C,u,o_u)]$ a Floer trajectory from $x^+$ to $x^-$ (see figure \ref{fig10}). We define its index $\mu([(C,u,o_u)]) = \mu(x^-) - \mu(x^+) + \mu(F_u)$, where $\mu(F_u)$ is the Maslov index of its boundary condition specified by $L$, and its area $\omega([(C,u,o_u)]) = \underset{C}{\int} u^*\omega$. Define $F(x^+,x^-)$ as the set of all the Floer trajectories from $x^+$ to $x^-$.
\end{defi}

\begin{figure}[h]
        \centering 
	\includegraphics[width=110mm]{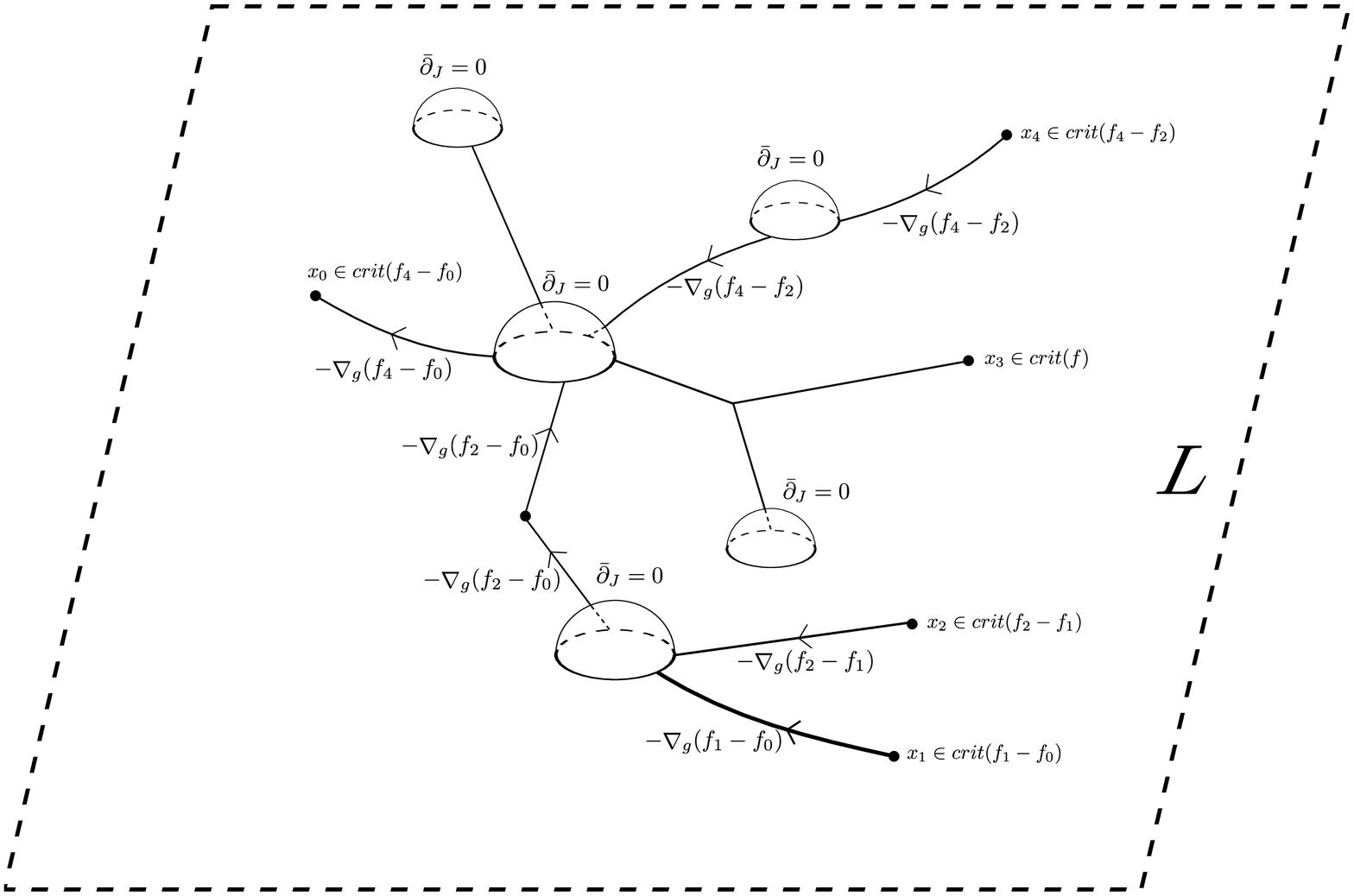}
        \caption{Configuration $(C,u)$ for a Floer trajectory, $C \in C_1(\mathcal{C}\ell^\otimes_{4,8})$} \label{fig10}
\end{figure}

\begin{rem}
By lemma \ref{def_gl}, lemma \ref{def_op}, plus the fact that $L$ admits a $\gl$-structure, we get that the family of operators over a Floer trajectory is orientable. Also note that because the gluing operation canonically transfers orientation (see \cite{FH}, \cite{S2}, \cite{BC}), the composition of trajectories will be well defined and the above homotopy should be considered to allow breakings into compositions of nonbroken trajectories.
\end{rem}

\section{Orientation settings} \label{subs_orie}

This subsection intends to make sense of the $\Lambda$-module elements $+ x_1 \ldots x_q t^d$ and $- x_1 \ldots x_q t^d$ geometrically. Our strategy is reminescent of the orientation procedure in Morse cohomology (see \cite{W2}, and \cite{W1} for the lagrangian Floer setting): Considering $+x$ as an orientation $\mathcal{O}_x$ on $W^s(x)$ (whereas $-x$ stands for $-\mathcal{O}_x$ on $W^s(x)$), and then using the flow over the $\delta$ trajectories to relate the generators. The reader who is not interested in orientation considerations should skip this subsection, replacing $\Z$ with $\Z / 2\Z$ from now on.

To conciliate the preceding algebraic and geometric settings, it will be convenient to consider a Floer trajectory as a particular case of what is called a trajectory $[(C,u,\lambda_{\partial u},o_u)]$ from a generator $x^+$ to a generator $x^-$ (see \cite{W1}, \cite{W2} for other examples) where 

\begin{enumerate}
        \item $C = C^{(1)} \times \dots \times C^{(q)} \in \mathcal{C}l_{\ell^{(1)}}^\otimes \times \dots \times \mathcal{C}l_{\ell^{(q)}}^\otimes$,

        \item $u: C = \overset{q}{\underset{r=1}{\bigsqcup}} C^{(r)} \rightarrow M$ continuous such that $\forall 1 \leq r \leq q$,
        \begin{enumerate}
                \item $u(v_j(C^{(r)})) = x_j^{(r)}$ for $0 \leq j \leq \ell^{(r)}$,


        \end{enumerate}

        \item $\lambda_{\partial u}: \partial C \rightarrow \mathcal{L}ag(u|_{\partial C}^*TM)$ is a lagrangian boundary condition such that $W^s(x_j^{(r)}) \subset (\lambda_{\partial u}(v_j(C^{(r)}))$, $1 \leq r \leq q$, plus a $\gl$-structure on it,

        \item $o_u$ is an orientation of an operator over $u^*TM \rightarrow C$ with boundary conditions $\lambda_{\partial u}$ and endpoint condition $Hess_{x_j^{(r)}}(f_j^{(r)})$ at $v_j(C^{(r)})$, $1 \leq r \leq q$, where $x_j^{(r)} \in crit(f_j^{(r)})$,
\end{enumerate}
and the brackets stand for the homotopy class with fixed area over each disk.

Thus, a trajectory is simply a homotopy class (with fixed area) of maps from a cluster to $(M,\omega)$ with boundary and endpoint conditions together with an orientation of the associated Cauchy-Riemann operators. Note that up to orientation, every Floer trajectory defines a trajectory by using $L$ together with its $\gl$-structure.

\begin{figure}[h]
        \centering 
	\includegraphics[width=110mm]{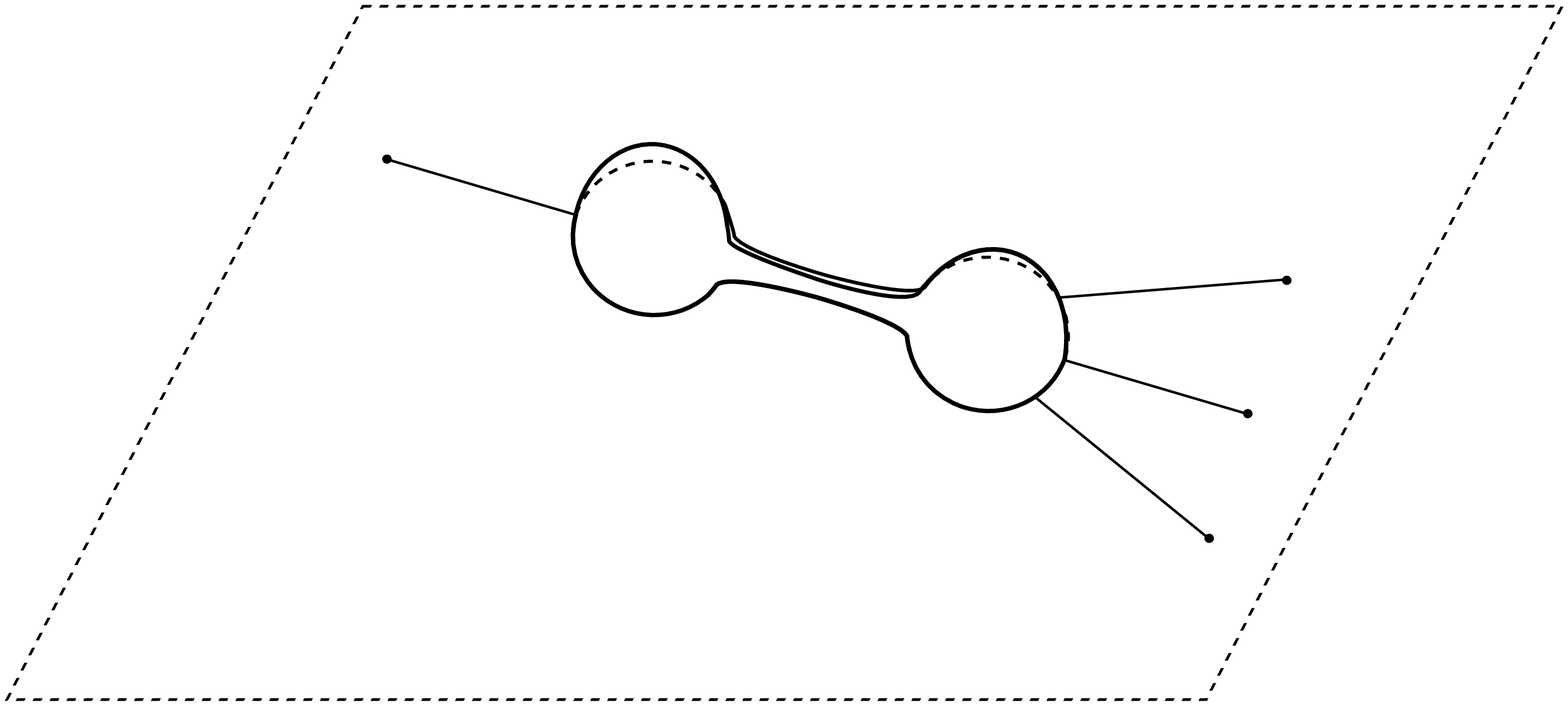}
        \caption{Configuration $(C,u)$ for a trajectory} \label{fig11}
\end{figure}

\begin{exe}
If $x^+ = {\it M} = max(f)$, $x^-=x \in crit(f)$ and area $a=0$ (a trajectory to $x^-$ with $0$ area), then one can think of a trajectory as a semi-infinite flow line of $f$ reaching $x^-$ together with an orientation of $W^s(x^-)$. That is, $u: (\overline{\R_-},-\infty) \rightarrow (M,x^-)$ with endpoint condition $Hess_{x}(f)$ so one can simply think of the trajectory as an orientation of $W^s(x)$.
\end{exe}

\begin{exe}
If $x^+={\it M}$, $x^-= xt^d \in crit(f) \otimes \Lambda$ (and therefore $a = d \tau N_L$), then one can think of a trajectory as a semi-infinite flow line of $f$ connecting $x^-$ with a disk of index $d N_L$ together with an orientation of a Cauchy-Riemann operator over it. 
\end{exe}


We now proceed to choices that will allow a trajectory interpretation for $\pm x_1 \ldots x_q t^d$:

\begin{itemize}
\item For every $x \in crit(f)$ (resp. $crit(f_{j_1}-f_{j_2})$), choose a reference orientation $\mathcal{O}^\otimes_x$ of $W^s(x)$. We will interpret $x = +xt^0 \in \mathcal{C}\ell^\otimes$ as a trajectory from ${\it M}$ to $x$ oriented as $\mathcal{O}^\otimes_x$.

\item Choose a reference orientation $\mathcal{O}^\otimes_{t^1}$ on a trajectory from $x^+={\it M}$ to $x^-= {\it M}$ with index $N_L$. This amounts to choosing an orientation of an operator over disks of index $N_L$ passing through ${\it M}$.

\item Now for every $d \in \N$, set the reference orientation $\mathcal{O}^\otimes_{t^d} = \#^d \mathcal{O}^\otimes_{t^1} = \mathcal{O}^\otimes_{t^1} \# \ldots \# \mathcal{O}^\otimes_{t^1}$, that is, we use the glued orientation over $d$ disks of index $N_L$ passing through ${\it M}$.

\item For every $x \in crit(f)$ (resp. $crit(f_{j_1}-f_{j_2})$), set $\mathcal{O}^\otimes_{xt^d} = \mathcal{O}^\otimes_x \# \mathcal{O}^\otimes_{t^d}$, seeing a trajectory from $M$ to $x$ with index $d N_L$ as the composition of a trajectory from $M$ to $M$ having index $d N_L$ and one from $M$ to $x$ of index $0$.

\item For $ x_1 \ldots x_q t^d \in \mathcal{C}\ell^\otimes$, set $\mathcal{O}^\otimes_{x_1 \ldots x_q t^d} = \mathcal{O}^\otimes_{x_1} \# \ldots \# \mathcal{O}^\otimes_{x_q} \# \mathcal{O}^\otimes_{t^d}$.
\end{itemize}

In what follows, we will consider $\pm x_1 \ldots x_q t^d \in \mathcal{C}\ell^\otimes$ as a trajectory from $1$ to $ x_1 \otimes \ldots \otimes x_{q-1} \otimes x_q$ of index $d N_L$ with orientation $ \pm \mathcal{O}^\otimes_{x_1 \ldots x_q t^d}$.

\begin{rem}
        Also, one can consider an intermediate type of trajectories associated with $\widetilde{\Lambda}$ where $\partial C$ is mapped to $L$ and $\lambda_{\partial u}$ is taken as the pullback of $T^*L$. Moreover, it is possible to restrict to the trajectories with everywhere positive area, resulting in trajectories associated with $\Lambda^+$ and $\widetilde{\Lambda^+}$.
\end{rem}

\section{Definition of $\delta^\otimes$} \label{subs_comp}

We then define a $\Z$-linear differential map that counts rigid Floer trajectories between elements of $\mathcal{C}\ell^\otimes$. We start off with defining intermediate operators that have cardinality $\ell$ inputs and cardinality $1$ outputs:

if $x$ is a generator of $\mathcal{C}\ell^\otimes$ with $q(x) = \ell$, we define
\[
        m_\ell(x) = \underset{\substack{x^- \in \mathcal{C}\ell^\otimes \\ q(x^-)=1}}{\sum} \: \underset{\substack{[(C,u,\mathcal{O}_{\ell}^\otimes)] \in F(x,x^-) \\ \mu(u) = -(\ell -2) }}{\sum} < \mathcal{O}_{\ell}^\otimes \# \mathcal{O}_{x}^\otimes, \mathcal{O}_{x^- t^{\frac{\omega(u)}{\tau N_L}}}^\otimes>  x^- t^{\frac{\omega(u)}{\tau N_L}}.
\]

where $\mathcal{O}_{\ell}^\otimes \# \mathcal{O}_{x}^\otimes$ is the composition at $x$, $<\_, \_> = +1$ (resp. $-1$) if the entries coincide (resp. are opposite), and the orientation $\mathcal{O}_{\ell}^\otimes$ over every $(C,u)$ is chosen via the following exact sequence of operators:

\begin{diagram} \label{seq2}
        0 & \rTo & L^{m,p}_0 (C, TC, T \partial C) & \rTo &&& L^{m,p}(C,u^*TM,u^*TL) &&& \rTo & \frac{L^{m,p}(C,u^*TM,u^*TL)}{L^{m,p}_0 (C, TC, T \partial C)} & \rTo & 0\\
        & & \dTo^{\ddel_C} &&&& \dTo^{\ddel_u} &&&& \dTo^{\ddel_u / \ddel_C} & &\\
        0 & \rTo & L^{m-1,p}_\pi (C,TC) && \rTo && L^{m-1,p}(C,u^*TM) && \rTo && \frac{L^{m-1,p}(C,u^*TM)}{L^{m-1,p}_\pi (C,TC)} & \rTo & 0 \\
\end{diagram}

The right-hand side operator can be made surjective, and hence an isomorphism, by choosing generically the perturbation system and the Morse functions $\{ f_{r_1} - f_{r_2}\}_{0 \leq r_1 \leq r_2 \leq c} $. Thus, 
$\ddel_u$ inherits an orientation from the reference orientation $\mathcal{O}_{\ell}^\otimes$ on $\ddel_C$.

\begin{rem}
Since $\mu(m_\ell)=\mu(x^- t^{\frac{\omega(u)}{\tau N_L}})- \mu(x)=-(\ell-2)$ and $q(m_\ell)=q(x^- t^{\frac{\omega(u)}{\tau N_L}})-q(x)=\ell-1$, then $(\mu+q)(m_\ell)=\mu(m_\ell)+q(m_\ell)=1$, that is, $m_\ell$ is of degree $1$ with respect to the $\mu + q$ grading on $\mathcal{C}\ell^\otimes$.
\end{rem}

We extend $m_\ell$ to a $\Lambda$-module morphism.

\begin{rem}
The expression $m_\ell(x \cdot t^d) = m_\ell(x) \cdot t^d$, $d \in \Z$, is consistent with the interpretation of section \ref{subs_orie}. An output $x^- t^{d + \frac{ \omega(u)}{\tau N_L}}$ from the left hand side is interpreted as a composition of a trajectory from {\it M} to $x$ of index $d N_L$ with one from $x$ to $x^-$ with orientation $< \mathcal{O}_{\ell}^\otimes \# \mathcal{O}_{x t^d}^\otimes, \mathcal{O}_{x^- t^{d + \frac{\omega(u)}{\tau N_L}}}^\otimes>$. The other hand side could be interpreted as a composition of a trajectory from {\it M} to {\it M} of index $d N_L$ with one from {\it M} to $x$ plus one from $x$ to $x^-$ with orientation $< \mathcal{O}_{\ell}^\otimes \# \mathcal{O}_{x}^\otimes, \mathcal{O}_{x^- t^{\frac{\omega(u)}{\tau N_L}}}^\otimes> = < \mathcal{O}_{\ell}^\otimes \# \mathcal{O}_{x}^\otimes \# \mathcal{O}_{t^d}^\otimes, \mathcal{O}_{x^- t^{\frac{\omega(u)}{\tau N_L}}}^\otimes \# \mathcal{O}_{t^d}^\otimes > = < \mathcal{O}_{\ell}^\otimes \# \mathcal{O}_{x t^d}^\otimes, \mathcal{O}_{x^- t^{d + \frac{\omega(u)}{\tau N_L}}}^\otimes>$.
\end{rem}

We define the codifferential $\delta^\otimes: \mathcal{C}\ell^\otimes(L,f,f_0, \ldots, f_c) \rightarrow \mathcal{C}\ell^\otimes(L,f,f_0, \ldots, f_c)$ as
\[
        \delta^\otimes = \delta^\otimes(M,\omega,J,L,P,f,f_0, \dots, f_c,g)= \underset{q \geq 1}{\sum} \underset{1 \leq j \leq q}{\sum} \, \underset{ \ell \geq 1}{\sum} (-1)^{(q-j)\ell + (j-1)} Id^{j-1} \otimes m_\ell \otimes Id^{q-j}. 
\]

Note that it is well defined by compactness of both $M$ and $L$, fairly standard Gromov-type compactness results for pseudoholomorphic discs with lagrangian boundary and the usual compactness results from Morse theory.


\begin{thm} \label{coch_comp}
        For $\delta^\otimes$ defined as above, we have $\delta^\otimes \circ \delta^\otimes = 0$ so that $(\mathcal{C}\ell^\otimes,\delta^\otimes)$ is a cochain complex.
\end{thm}

\begin{proof}

By the transversality results of section \ref{tran_floe}, the gluing theorem of \cite{BC} and some standard compactness results, the broken trajectories counted by $\delta^\otimes \circ \delta^\otimes$ are in 1:1 correspondence with the boundary of a compact 1-dimensional piecewise smooth manifold.



\begin{rem}
For the compactness part of this statement, note that our use of perturbations over $\otimes$-clusters carries well over the degenerations (breaking of flow lines) of trajectories. For instance, a flow line must break on critical points of the prescribed function over this line simply because the other functions are used over pieces of lines that are of finite length, so they cannot approach a critical point.
\end{rem}

It remains to see that the pairs of trajectories counted by $\delta^\otimes \circ \delta^\otimes$ are counted with opposite orientation.

In the special case where the pair of composite trajectories does not express as broken trajectories (see figure \ref{fig12}), this is purely algebraic (see \cite{W1}). More precisely, the associated trajectory is counted by both

\begin{figure}[h]
        \centering 
	\includegraphics[width=110mm]{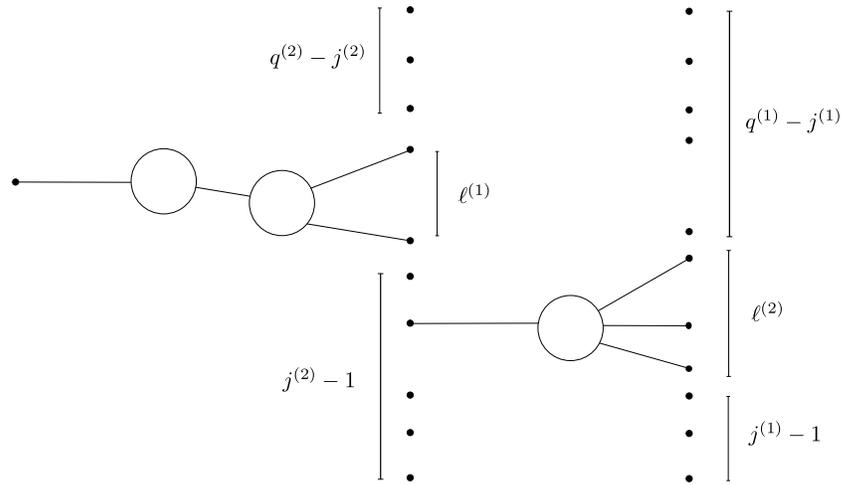}
        \caption{Pair of nonbroken Floer trajectories counted by $\delta^\otimes \circ \delta^\otimes$} \label{fig12}
\end{figure}

\begin{align*}
&(-1)^{(q^{(2)}-j^{(2)})\ell^{(2)} + (j^{(2)}-1)} Id^{(j^{(2)}-1)}\otimes m_{\ell^{(2)}} \otimes Id^{(q^{(2)}-j^{(2)})} \circ \\
&(-1)^{(q^{(1)}-j^{(1)})\ell^{(1)} + (j^{(1)}-1)} Id^{(j^{(1)}-1)}\otimes m_{\ell^{(1)}} \otimes Id^{(q^{(1)}-j^{(1)})}
\end{align*}
,where $j^{(1)} < j^{(2)}$, with orientation
\[
(-1)^{(j^{(2)}-j^{(1)}) + ((q^{(1)}-j^{(1)}) - (q^{(2)}-j^{(2)}))\ell^{(1)} + (q^{(2)}-j^{(2)})(\ell^{(2)} +\ell^{(2)})}  \mathcal{O}_{\ell^{(2)}}^\otimes  \wedge \mathcal{O}_{\ell^{(1)}}^\otimes
\]
and

\begin{align*}
&(-1)^{((q^{(1)}-j^{(1)})-(\ell^{(2)}-1))\ell^{(1)} + (j^{(1)}-1)} Id^{(j^{(1)}-1)}\otimes m_{\ell^{(1)}} \otimes Id^{(q^{(1)}-j^{(1)})-(\ell^{(2)}-1)} \circ \\
&(-1)^{(q^{(2)}-j^{(2)})\ell^{(2)} + (j^{(2)}-1)+(\ell^{(1)}-1)} Id^{(j^{(2)}-1)+(\ell^{(1)}-1)}\otimes m_{\ell^{(2)}} \otimes Id^{(q^{(2)}-j^{(2)})}
\end{align*}
with orientation
\begin{align*}
&(-1)^{(j^{(2)}-j^{(1)}) + ((q^{(1)}-j^{(1)}) - (q^{(2)}-j^{(2)}))\ell^{(1)} + (q^{(2)}-j^{(2)})(\ell^{(2)} +\ell^{(2)})} (-1)^{\ell^{(1)} \ell^{(2)} +1} \mathcal{O}_{\ell^{(1)}}^\otimes  \wedge \mathcal{O}_{\ell^{(2)}}^\otimes = \\
& - (-1)^{(j^{(2)}-j^{(1)}) + ((q^{(1)}-j^{(1)}) - (q^{(2)}-j^{(2)}))\ell^{(1)} + (q^{(2)}-j^{(2)})(\ell^{(2)} +\ell^{(2)})}  \mathcal{O}_{\ell^{(2)}}^\otimes  \wedge \mathcal{O}_{\ell^{(1)}}^\otimes
\end{align*}

For the broken case (see figure \ref{fig13}), relying on lemma \ref{out}, a broken trajectory $(C,u)$ is counted with orientation
\begin{align*}
&(-1)^{ (q^{(2)}-j^{(2)})(\ell^{(1)} +\ell^{(2)})} (-1)^{((q^{(1)}-j^{(1)}) - (q^{(2)}-j^{(2)}))\ell^{(1)} + (j^{(2)}-j^{(1)})}  \mathcal{O}_{\ell^{(2)}}^\otimes  \wedge \mathcal{O}_{\ell^{(1)}}^\otimes = \\
&(-1)^{ (q^{(2)}-j^{(2)})(\ell^{(1)} +\ell^{(2)})}  \partial_{\ddel_C} \mathcal{O}_{\ell^{(1)} +\ell^{(2)}-1}^\otimes
\end{align*}
so it induces $(-1)^{ (q^{(2)}-j^{(2)})(\ell^{(1)} +\ell^{(2)})}$ times the outward normal orientation on its 1-dimensional glued family. This is due to the fact that the orientation transfer on linear gluings considered in \ref{out}, for a sufficiently large gluing parameter, canonically corresponds to orientation transfer over gluings of trajectories (see \cite{S}, \cite{FH}). Since $(q^{(2)}-j^{(2)})(\ell^{(1)} +\ell^{(2)})$ is constant along the homotopy class, we get the result.

\begin{figure}[h]
        \centering 
	\includegraphics[width=110mm]{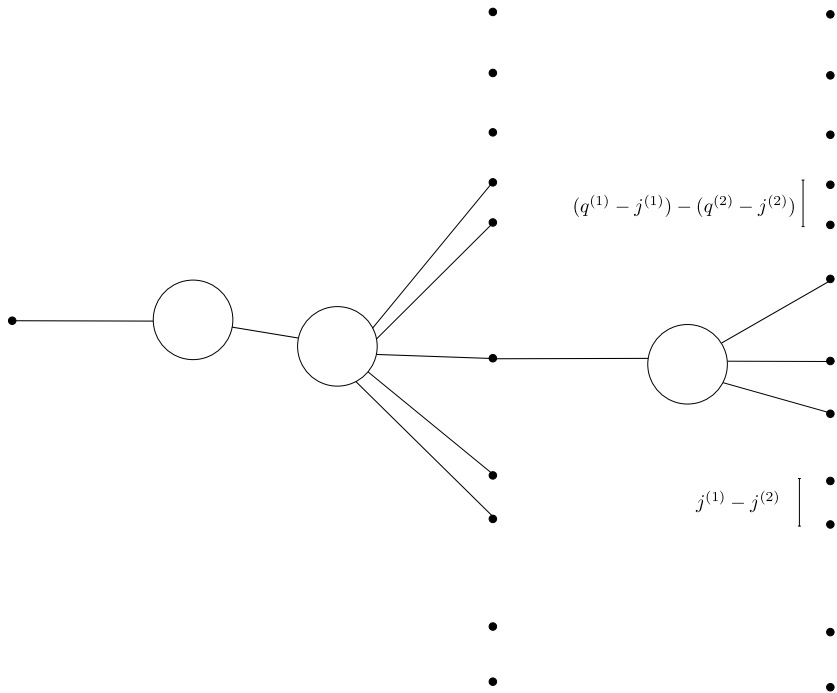}
        \caption{Broken Floer trajectory counted by $\delta^\otimes \circ \delta^\otimes$} \label{fig13}
\end{figure}

\end{proof}

Note that $\delta^\otimes \circ \delta^\otimes =0$ is equivalent, using the sign convention of \cite{GJ}, to the $A_\infty$ associativity relations on the products $\{ m_\ell \}_{\ell \geq 1}$. In other words, $(\mathcal{C}\ell^\otimes,\delta^\otimes)$ could be seen as the so-called (unsuspended) bar complex of this underlying $A_\infty$ algebra. In fact, using the usual Koszul sign rule $f \otimes g(x_1 \otimes x_2) = (-1)^{\mu(g) \mu(x_1)} f(x_1) \otimes g(x_2)$, we compute that the $m_\ell$ satisfy
\[
0 = \underset{\substack{\ell^{(1)}+\ell^{(2)}-1 \\ =q}}{\sum} \hspace{0.2cm} \underset{1 \leq j \leq \ell^{(1)}}{\sum} (-1)^{\epsilon(j,\ell^{(2)})} m_{\ell^{(1)}}(x_1, \ldots, x_{j-1}, m_{\ell^{(2)}}(x_j, \ldots,x_{j+\ell^{(2)}-1} ), x_{j+\ell^{(2)}}, \ldots, x_{\ell^{(1)}+\ell^{(2)}-1})
\]
where $\epsilon(j,\ell^{(2)}) = \ell^{(2)} (\mu(x_1) + \ldots +\mu(x_{j-1})) + (j-1)(\ell^{(2)}-1)+ (\ell^{(1)} -1)\ell^{(2)})$ so they define an $A_\infty$ algebra (in the conventions of \cite{KS} or \cite{PS}).

\begin{rem} \label{rema_susp}
In terms of the suspended tensor algebra $sA$ of a $\mu$-graded module $A$, that is, $(sA)_i = A_{i+1}$,
\begin{diagram}
A^\ell & \rTo^{m_\ell} & A \\
\dTo^{s^\ell} & & \dTo^s \\
(sA)^\ell & \rTo^{b_\ell} & sA \\
\end{diagram}
where $m_\ell$ is of degree $2-\ell$, $b_\ell$ is of degree $1$, and $s$ of degree $-1$. We can compute that
\begin{align*}
0 &= \underset{\ell^{(1)}+\ell^{(2)}-1=q}{\sum} \hspace{0.2cm} \underset{1 \leq j \leq \ell^{(1)}}{\sum} (-1)^{(q-j)\ell^{(2)}+ (j-1)} m_{\ell^{(1)}} \circ Id^{j-1} \otimes m_{\ell^{(2)}} \otimes Id^{q-j} \\
&= \underset{\ell^{(1)}+\ell^{(2)}-1=q}{\sum} \hspace{0.2cm} \underset{1 \leq j \leq \ell^{(1)}}{\sum} (-1)^{(q-j)\ell^{(2)}+ (j-1)} (s^{-1}\circ b_{\ell^{(1)}} \circ s^{\ell^{(1)}}) \circ Id^{j-1} \otimes (s^{-1}\circ b_{\ell^{(2)}}\circ s^{\ell^{(2)}}) \otimes Id^{q-j} \\
&= \underset{\ell^{(1)}+\ell^{(2)}-1=q}{\sum} \hspace{0.2cm} \underset{1 \leq j \leq \ell^{(1)}}{\sum} (-1)^{(j-1)} (s^{-1}\circ b_{\ell^{(1)}} ) \circ s^{j-1} \otimes ( b_{\ell^{(2)}}\circ s^{\ell^{(2)}}) \otimes s^{q-j} \\
&= \underset{\ell^{(1)}+\ell^{(2)}-1=q}{\sum} \hspace{0.2cm} \underset{1 \leq j \leq \ell^{(1)}}{\sum} s^{-1}\circ (b_{\ell^{(1)}}  \circ Id^{j-1} \otimes b_{\ell^{(2)}} \otimes Id^{q-j}) \circ s^q  \\
&= s^{-1}\circ (\underset{\ell^{(1)}+\ell^{(2)}-1=q}{\sum} \hspace{0.2cm} \underset{1 \leq j \leq \ell^{(1)}}{\sum} b_{\ell^{(1)}}  \circ Id^{j-1} \otimes b_{\ell^{(2)}} \otimes Id^{q-j}) \circ s^q  \\
\end{align*}
so that no signs appear in the complex built from the $b_\ell$, that is,
\[
0 = \underset{\ell^{(1)}+\ell^{(2)}-1=q}{\sum} \hspace{0.2cm} \underset{1 \leq j \leq \ell^{(1)}}{\sum} b_{\ell^{(1)}}  \circ Id^{j-1} \otimes b_{\ell^{(2)}} \otimes Id^{q-j}
\]
and
\[
0 = \underset{\ell^{(1)}+\ell^{(2)}-1=q}{\sum} \hspace{0.2cm} \underset{1 \leq j \leq \ell^{(1)}}{\sum} (-1)^{\overline{\epsilon}(j,\ell^{(2)})} b_{\ell^{(1)}}(x_1, \ldots, x_{j-1}, b_{\ell^{(2)}}(x_j, \ldots,x_{j+\ell^{(2)}-1} ), x_{j+\ell^{(2)}}, \ldots, x_{\ell^{(1)}+\ell^{(2)}-1})
\]
where $\overline{\epsilon}(j,\ell^{(2)}) = \mu(x_1) + \ldots +\mu(x_{j-1}) + (j-1)$ and $\mu(x_i)$ is the degree of $x_i$ in $A$. This coincides with the algebraic formalism of \cite{FOOO} and \cite{Sei}.
\end{rem}

Note that it is known that in the $\otimes$ (nonsymmetric) setting no information is stored in the associated cohomology groups. Indeed, a Morse function $f$ with a single local maximum {\it M} has been chosen. Notice that $m_2 \circ ({\it M} \otimes \_) = Id $: If $x \in crit(f_{j_1} -f_{j_2})$, then $m_2$ on ${\it M} \otimes x$ counts the only flow line of $f$ from {\it M} to $x$ and thus returns $x$. Otherwise if $x \in crit(f)$, then, for small enough perturbations, $m_2$ on ${\it M} \otimes x$ counts the only perturbed flow line of $f$ from {\it M} intersecting the perturbed constant flow line from $x$ to $x$ so it again returns $x$. Also, a Morse theory argument tells that $m_1({\it M})=0$ and for index reasons $m_\ell({\it M} \otimes \_)=0$ for $\ell >2$. Therefore, in the language of $A_\infty$ algebras,

\begin{lem}
${\it M} = max(f)$ is a unit of the $(\mathcal{C}\ell^\otimes (L,f), \{ m_\ell\}_{\ell \geq 1}) $ $A_\infty$ algebra.
\end{lem}

Then left multiplication by {\it M} is seen, using the definition of $\delta^\otimes$, to be a contracting homotopy between the identity and the trivial cochain maps:
\[
\delta^\otimes \circ ({\it M} \otimes) + ({\it M} \otimes) \circ \delta^\otimes = (m_2 \circ ({\it M} \otimes \_)) \otimes Id - ({\it M} \otimes) \circ \delta^\otimes + ({\it M} \otimes) \circ \delta^\otimes = Id - 0 
\]
where the first term corresponds to applying $m_2$ on {\it M} and the first factor argument.

\section{Chain complex formulation}

It is also possible to encode the $\otimes$-cluster trajectories information using the homological formalism, as originally proposed in \cite{CL}. 

Let $\Lambda = \Z[t,t^{-1}]$, $\Lambda^+ = \Z[t]$, $\widetilde{\Lambda}$ and $\widetilde{\Lambda^+}$ be the Novikov rings defined as in section \ref{defi_gene}. In what follows, we use $\Lambda$, but the same should apply to the latter versions.

Set the $\Lambda$-modules $V = crit(f)\otimes \Lambda$ and $V_{j_1,j_2} = crit(f_{j_1} - f_{j_2})\otimes \Lambda$, $0 \leq j_1 < j_2 \leq c$. Let then $T^m V= V^{\otimes m} = V \otimes \ldots \otimes V$ be the rank $m$ tensor product of $V$.

\begin{defi}
Let 
\begin{align*}
        \widehat{\mathcal{C}\ell^\otimes} =& \widehat{\mathcal{C}\ell^\otimes}(L,f,f_0, \ldots, f_c) \\
=&  \, \underset{0 \leq c' \leq c}\bigoplus \, \underset{\substack{0\leq j_0 \leq \ldots \leq j_{c'} \leq c\\ m_0, \ldots, m_{c'} \geq 0}}{\widehat{\bigoplus}} T^{m_0}V \otimes V_{j_0,j_1} \otimes T^{m_1}V \otimes \ldots \otimes T^{m_{c'-1}}V \otimes V_{j_{c'-1},j_{c'}} \otimes T^{m_{c'}}V
\end{align*}
where the $c'=0$ inner union is understood to be $T(V) = \underset{m \geq 1}{\widehat{\bigoplus}} T^{m}V$, the tensor algebra of $V$, and now $\widehat{\bigoplus}$ allow formal series. 

\end{defi}

An element of $\widehat{\mathcal{C}\ell^\otimes}$ might now be an infinite sum. Thus the cardinality $q(x)$ of its terms can go to infinity, due to the contribution of products of critical points of $f$. In particular, any infinite sum having bounded energy must have the cardinality of its terms going to infinity (otherwise this would contradict the Novikov condition on the coefficients).

The differential map will now count rigid Floer trajectories between elements of $\widehat{\mathcal{C}\ell^\otimes}$. We start off by defining intermediate operators that have cardinality $1$ inputs and cardinality $\ell$ outputs. If $x$ is a generator of $\mathcal{C}\ell^\otimes$ with $q(x) = 1$, we define $m^{op}_\ell$ as being the opposite of $m_\ell$, that is, we transpose the inputs and outputs of $m_\ell$ so that
\[
        m^{op}_\ell(x) = \underset{\substack{x^+ \in \mathcal{C}\ell^\otimes \\ q(x^+)=\ell}}{\sum} \: \underset{\substack{[(C,u,\mathcal{O}_{\ell}^\otimes)] \in F(x^+,x) \\ \mu(u) = \ell -2 }}{\sum} < \mathcal{O}_{\ell}^\otimes \# \mathcal{O}_{x^+}^\otimes, \mathcal{O}_{x t^{\frac{\omega(u)}{\tau N_L}}}^\otimes>  x^+ t^{\frac{\omega(u)}{\tau N_L}}
\]

Then, $m^{op}_\ell$ is of degree $-1$ with respect to the $\mu+q$ grading on $\widehat{\mathcal{C}\ell^\otimes}$, where now $\mu(t^d)= -dN_L$. Reversing the sign of the chosen Morse functions, one might see the trajectories counted by $m^{op}_\ell$ as flowing from the single input to the cardinality $\ell$ output. With this in mind, we get $\mu(x)=|x|$, where $|\_|$ will now denote the Morse index with respect to the reversed Morse functions, so that now $\mu(x_1 \otimes \ldots \otimes x_q t^d)= \sum_{j=1}^{q}|x_j| - dN_L$.

We define the differential $m^{op}: \widehat{\mathcal{C}\ell^\otimes}(L,f,f_0, \ldots, f_c) \rightarrow \widehat{\mathcal{C}\ell^\otimes}(L,f,f_0, \ldots, f_c)$ as
\[
        m^{op} = m^{op}(M,\omega,J,L,P,f,f_0, \dots, f_c,g)= \underset{q \geq 1}{\sum} \underset{1 \leq j \leq q}{\sum} \, \underset{ \ell \geq 1}{\sum} (-1)^{(q-j)\ell + (j-1)} Id^{j-1} \otimes m^{op}_\ell \otimes Id^{q-j}. 
\]

Note that it is well defined by fairly standard compactness results, but now $m^{op}(x)$ is in general a formal series with increasing cardinalities. From theorem \ref{coch_comp}, we get that


\begin{cor}
        For $m^{op}$ defined as above, we have $m^{op} \circ m^{op} = 0$ so that $(\widehat{\mathcal{C}\ell^\otimes},m^{op})$ is a chain complex.
\end{cor}

Following the suspension procedure of remark \ref{rema_susp}, we get

\begin{diagram}
A & \rTo^{m^{op}_\ell} & A^\ell \\
\dTo^{s} & & \dTo^{s^\ell} \\
sA & \rTo^{b^{op}_\ell=d^\otimes_\ell} & (sA)^\ell \\
\end{diagram}

where $d^\otimes_\ell$ is of degree $\mu(d^\otimes_\ell)= -1$ with respect to the grading $\mu$ on $sA$. Then setting

\[
d^\otimes = \underset{q \geq 1}{\sum} \underset{1 \leq j \leq q}{\sum} \, \underset{ \ell \geq 1}{\sum} Id^{j-1} \otimes d^\otimes_\ell \otimes Id^{q-j},
\]

one obtains that $d^\otimes \circ d^\otimes = 0$.

Moreover, using the usual Koszul sign rule $f \otimes g(x_1 \otimes x_2) = (-1)^{\mu(g) \mu(x_1)} f(x_1) \otimes g(x_2)$, we compute that the $d^\otimes(x_1 \otimes x_2)= d^\otimes \otimes Id(x_1 \otimes x_2) + Id \otimes d^\otimes(x_1 \otimes x_2) = d^\otimes(x_1) \otimes x_2 + (-1)^{\mu(x_1)}x_1 \otimes d^\otimes(x_2)$, where now $\mu(x_1)=|x_1|-1$, so that it satisfies the usual Leibniz rule.

Therefore we get back to the algebraic formulation of \cite{CL}:

\begin{cor}
        For $d^\otimes$ defined as above, we have $d^\otimes \circ d^\otimes = 0$ and $d^\otimes(x_1 \otimes x_2) = d^\otimes(x_1) \otimes x_2 + (-1)^{\mu(x_1)}x_1 \otimes d^\otimes(x_2)$ so that $((s\widehat{\mathcal{C}\ell^\otimes}, \mu),d^\otimes)$ is a differential graded algebra.
\end{cor}

\chapter{Moduli of quilted $\otimes$-clusters and complex morphisms}

Using a construction of \cite{MW}, we now give a description of the source spaces used to define morphisms of complexes.

\section{Moduli of marked quilted disks}

\begin{defi}
	Let $D_\C$ be a complex disk and $x_0 \in \partial D_\C$. A complex disk $C \subsetneq D_\C$ tangent to $\partial D_\C$ at $x_0$ is called a seam of $D_\C$ at $x_0$. Then a pair $(D,C)$, where $D \in \mathring{\klk}$ and $C$ is a seam of $D$ at $x_0$, will be referred to as a quilted marked disk and $\mathring{\qklk}$ will denote the space of such disks.
\end{defi}

Canonically, $\mathring{\qklk} \cong \mathring{\klk} \times ]0,1[$ with the second factor corresponding to the radius of the quilting disk, and notice that under the biholomorphic map that identifies $D_\C$ with the upper complex half-plane sending $x_0$ to $\infty$, $\partial C$ is sent to a horizontal line.

This space allows a Deligne-Mumford-Knudsen type compactification $\qklk$ which admits an orientable $(\ell -1 + 2k)$-dimensional singular manifold with corner structure (see figure \ref{fig14}). The case $k=0$ gives the Stasheff multiplihedra and otherwise we rely on the complexification of the multiplihedra defined in \cite{MW}. We first recall its construction as a moduli of scaled marked genus zero Riemann surfaces.

\begin{figure}[h]
        \centering 
	\includegraphics[width=110mm]{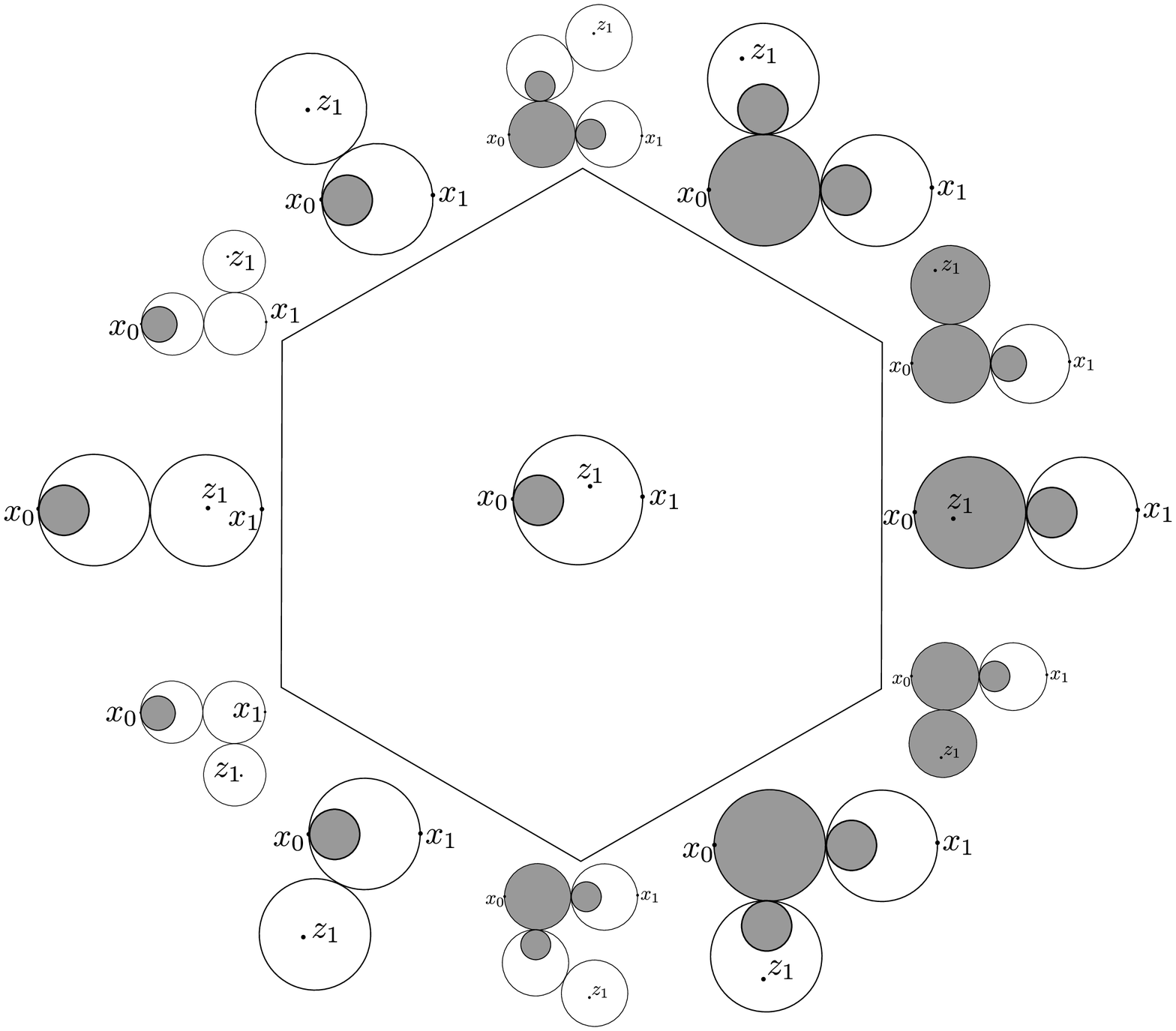}
        \caption{$Q_{1,1}$} \label{fig14}
\end{figure}

\begin{defi}
        Let $(\Sigma, x_0, \dots, x_\ell, z_1, \dots, z_k, z_{k+1}, \dots, z_{2k})$ be a marked genus zero (possibly nodal and unstable) Riemann surface. Let $\Sigma^{(i)}$ be a smooth component of $\Sigma$ and $x^{(i)}_0$ be its closest marking to $x_0$. Then a (possibly zero or infinite) translation invariant $(1,1)$-form $\phi^{(i)}$ on $\Sigma^{(i)} \backslash x^{(i)}_0$ is called a scaling on $\Sigma^{(i)}$. Then the triple $(\Sigma, x_0, \dots, x_\ell, z_1, \dots, z_k, z_{k+1}, \dots, z_{2k},\phi)$, where $\phi$ is a choice of scaling for each component of $\Sigma$, is called a genus zero scaled marked Riemann surface.

        An automorphism of scaled marked Riemann surface is an automorphism of marked Riemann surface preserving the volume forms. A scaled marked Riemann surface is said to be stable if its automorphism group is finite. Let $\mathcal{Q}_q(\C)$ be the set of stable genus zero scaled Riemann surfaces with $q=\ell +1 +2k$ markings.
\end{defi}

In other words, the scaling preserving condition means that we now consider marked Riemann surfaces up to translation in the complement of $x_0$ instead of up to translation and scaling in the complement of $x_0$ as it was the case in section \ref{disks}.

$\mathcal{Q}_q(\C)$ is naturally endowed with natural cross-ratio coordinates giving it the structure of a complex projective variety with toric singularities (see corollary 10.6 of \cite{MW}). We again choose to express the local charts in terms of labelings on combinatorial trees.

Recall that a colored tree $T$ is a tree with a special subset of vertices, called the colored vertices, such that a simple path from a leaf to the root meets exactly one colored vertex. Let $X$ be a labeling on $T$, $v^-$ be a vertex lying below its colored vertices and $v$ be any colored vertex above $v^-$. Then $X$ is called balanced if for every $v^-$ the product of the labels over the simple path from $v$ to $v^-$ is independent of $v$. For any colored tree $T$, let $X^{\C}(T)$ be the set of balanced complex labelings of the interior edges of $T$.

As with (nonbalanced) labelings, if $T^{(1)} \leq T^{(2)}$, a balanced labeling $X^{(2)}$ on $T^{(2)}$ defines a balanced labeling $X^{(2)}|_{T^{(1)}}$  by taking
\[
X^{(2)}|_{T^{(1)}}(l) = X^{(2)}(l)\cdot \prod_{l'} X^{(2)}(l') \prod_{l''} X^{(2)}(l''),
\]
where the products are over contracted edges $l'$ that are below $l$ and connected to $l$ by contracted edges only and contracted edges $l''$ above $l$ that are part of a sequence of contracted edges connecting $l$ with a colored vertex.

\begin{prp}\cite{MW} \label{neigh}
        Let $\mathcal{Q}_{q,T}(\C) \subset \mathcal{Q}_q(\C)$ be the strata of the scaled marked curves having combinatorial type $T$. Then, there is an isomorphism $\psi_T$ from a Zariski neighborhood of $\mathcal{Q}_{q,T}(\C) \times \{0\}$ in $\mathcal{Q}_{q,T}(\C) \times X^{\C}(T)$ to a Zariski neighborhood $\nu(\mathcal{Q}_{q,T}(\C))$ of $\mathcal{Q}_{q,T}(\C)$ in $\mathcal{Q}_q(\C)$.
\end{prp}

Note that here, again, $\mathcal{Q}_{q,T}(\C)$ is not a compact subset and that, in general, the above proposition cannot be extended to model a neighborhood of its closure $\mathcal{Q}_{q,\geq T}(\C)$, unlike in the unquilted case. More precisely, the balancing condition on the labelings implies that $\mathcal{Q}_{q,T}(\C)$ is generally a toric singular stratum (see figure \ref{fig15}).

\begin{figure}[h]
        \centering 
	\includegraphics[width=110mm]{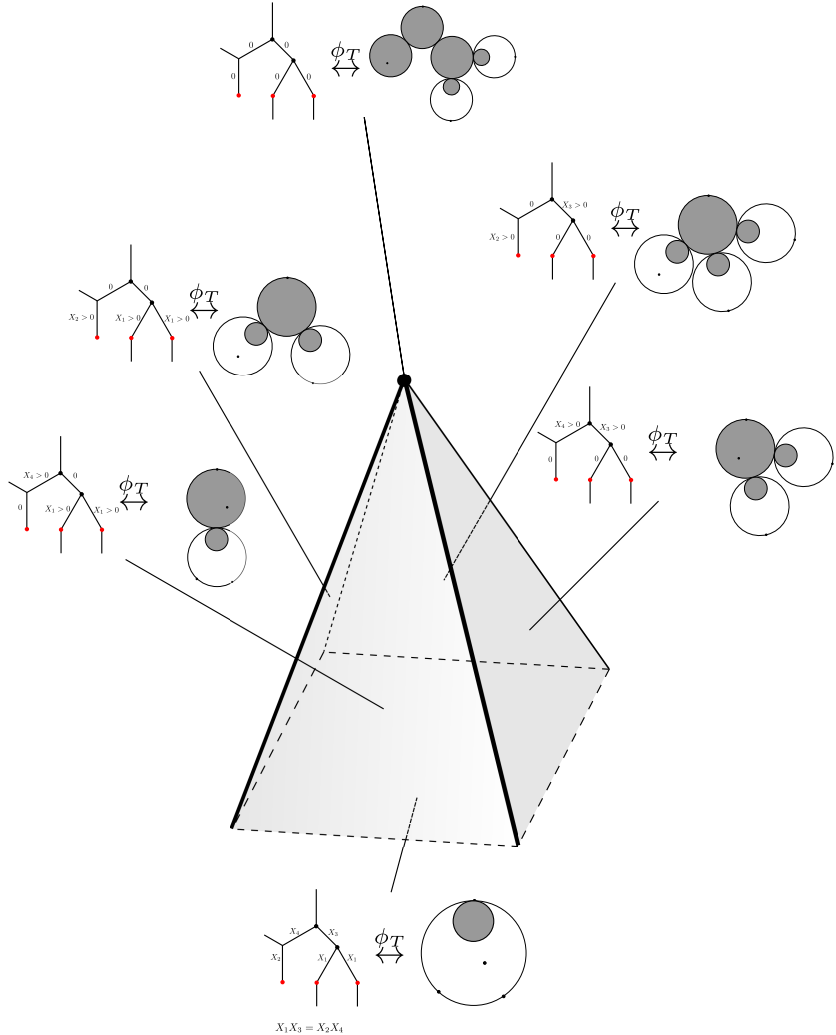}
        \caption{Singularity in $Q_{1,2}$} \label{fig15}
\end{figure}

\begin{defi}
        Set $\sigma_{\ell,k}$ as the anti-holomorphic involution on $\mathcal{Q}_q(\C)$ defined by \\
$\sigma_{\ell,k}(\Sigma, x_0, \dots, x_\ell, z_1, \dots, z_k, z_{k+1}, \dots, z_{2k},\phi) = (\overline{\Sigma}, x_0, \dots, x_\ell, z_k, \dots, z_{2k}, z_{1}, \dots, z_{k},\overline{\phi})$. Define $\qklk(\R) \equiv \text{fix}(\sigma_{\ell,k})$ and call its elements real scaled marked surfaces. Then let $\qklk \subset \qklk(\R)$ be the closure of the set of the smooth real scaled marked spheres $(\C P^1, \infty, x_1, \dots, x_\ell, z_1, \dots, z_{2k}, \phi)$ such that
\begin{itemize}
        \item $x_j \in \R P^1$, $0 \leq j \leq \ell$, with $x_0 < x_1 < \dots < x_\ell$,
        \item $z_h \in \mathbb{H}^+ \subset \C \cong \C P^1 \backslash \{x_0\}$, $1 \leq h \leq k$, 
        \item $\phi = y dz \wedge d\bar{z}$, with $y \in \R_+\backslash \{0\}$. 
\end{itemize}
\end{defi}

Remark that $Q_{\ell,0}$ is the geometric realization of Stasheff multiplihedra (\cite{Sta}) as a space of quilted disks (\cite{MW}). Even though $\qklk$ is not a manifold with corners, it still has embedded boundary components $B(\qklk) = \underset{m}{\bigcup} B_m(\qklk)$, each component having a combinatorial type encoded in a colored tree with at least one interior edge.

Also, the natural projection $\mathring{\qklk} \cong \mathring{\klk} \times ]0,1[ \rightarrow \mathring{\klk}$ extends to a natural projection $\qklk \rightarrow \klk$ defined by forgetting the scaling and stabilizing.

As in \cite{MW}, we can make the charts $\psi_T$ restricted to $\qklk$ explicit in terms of simple ratios. Let $T^{max}$ be a maximal colored planar tree, corresponding to a marked nodal quilted disk $D_{T^{max}}$ with every smooth component being either a disk with $3$ boundary markings and no interior markings, a disk with one boundary marking and one interior marking or a quilted disk with one boundary marking. The planarity of $T^{max}$ makes the set of markings $\{x_1, \ldots, x_\ell, z_1, \ldots, z_k \}$ into an ordered set $\{y_1, \ldots, y_{\ell+k}\}$. To every trivalent vertex $v$ of $T^{max}$ we associate a marking $y_j$: $v$ being the uppermost intersection point of a simple path from $y_j$ to the root and a simple path from $y_{j+1}$ to the root. Then take $\Delta_v = y_j - y_{j+1}$. To every uncolored bivalent vertex $v$ of $T^{max}$ we consider the interior marking $z_h$ lying just above it and then set $\Delta_v = Im(z_h)$. Finally, to every colored bivalent vertex $v$, we set $\Delta_v = Im(S)$ to be the height of the seam $S$ seen in the upper half-plane. Now on an interior edge $l \in E(T^{max})$ with top vertex $v_a$ and bottom vertex $v_b$, we will consider the label $X(l)= \frac{\Delta_{v_a}}{\Delta_{v_b}}$. 

\begin{diagram}
\nu(K_{\ell,k, T^{max}}) & & \rTo^{\psi_{T^{max}}^{-1}} & &  X^{\C}(T^{max}) \\
D & & \rMapsto & &  X(l) = \frac{\Delta_{v_a}}{\Delta_{v_b}}
\end{diagram}

This will make the nodes correspond to zero labels on the corresponding edge of $T$. When there is no interior markings, $\psi_{T^{max}}$ corresponds to the simple ratio charts of \cite{MW}, taking real positive values. In general, the resulting labelings are complex, but still can be identified with real positive labelings. For example, in a sufficiently small neighborhood $\nu(K_{\ell,k, T^{max}})$ of $K_{\ell,k, T^{max}}$, one can again extract the real part over the labelings. 


\section{Constructing $\qcllkr$, the moduli of quilted $\otimes$-clusters} \label{sect_quil}

As in section \ref{collar}, we now add a collar neighborhood to $\qklk$ that has again a decomposition that allows us to encode the lengths of connecting lines.

Let $Q_{\ell,k,T} \subset \qklk$ be the boundary stratum of $\qklk$ corresponding to the colored tree $T$. Then by proposition \ref{neigh}, it has a neighborhood isomorphic to $Q_{\ell,k,T} \times X^{\R_+}(T)$ where $X^{\R_+}(T)$ stands for the set of balanced labelings on the interior edges of $T$ with values in $\R_+$. Indeed, one can take $T$ a maximal tree and associate to any labeling $X^{\R_+}(T)$ a marked quilted disk in the same fashion as in the unquilted case except that now the height of the seam is computed by taking the product of the labels below the colored vertices (see \cite{MW}). For $T$ nonmaximal, we can proceed as in proposition 6.2 and corollary 6.5 of \cite{MW}, using the balanced labelings obtained by collapsing balanced labelings on maximal trees. 

Now we define a collar extension of $\qklk$ as

\begin{defi}
\[
        col(\qklk) = \underset{T}{\bigsqcup} \hspace{0.3cm}  \mathcal{Q}_{\ell,k,\geq T} \times X^{[0,1]}(T) /_{\sim}
\]
with $(Q^{(1)},X^{(1)}) \in \mathcal{Q}_{\ell,k,\geq T^{(1)}} \times X^{[0,1]}(T^{(1)}) \sim (Q^{(2)},X^{(2)}) \in \mathcal{Q}_{\ell,k,\geq T^{(2)}} \times X^{[0,1]}(T^{(2)})$ if
\begin{itemize}
\item $Q^{(1)} = Q^{(2)}$,
\item $T^{(2)} > T^{(1)}$, that is, $T^{(1)}$ is obtained from $T^{(2)}$ by contracting some interior edges,
\item $X^{(2)}$ is the extension of $X^{(1)}$ having $1$ labels on the extra edges.
\end{itemize}
\end{defi}

\begin{defi}
        For $1 \leq \ell -1 + 2k$, let $(\qcllkr)^{PDIFF} = col(\qklk)$ and otherwise take $(\qcllkr)^{PDIFF}$ to be a point.
\end{defi}

Geometrically, one could see the above construction as an extension of $\mathcal{Q}_{\ell,k}$ using a local dual cell decomposition reminescent of that on the multiplihedra seen as a space of painted metric trees (see \cite{For}). These cells are again given as balanced labelings of maximal colored trees having $1$ labels over a subset of edges (see figure \ref{fig16})).

\begin{figure}[h]
        \centering 
	\includegraphics[width=110mm]{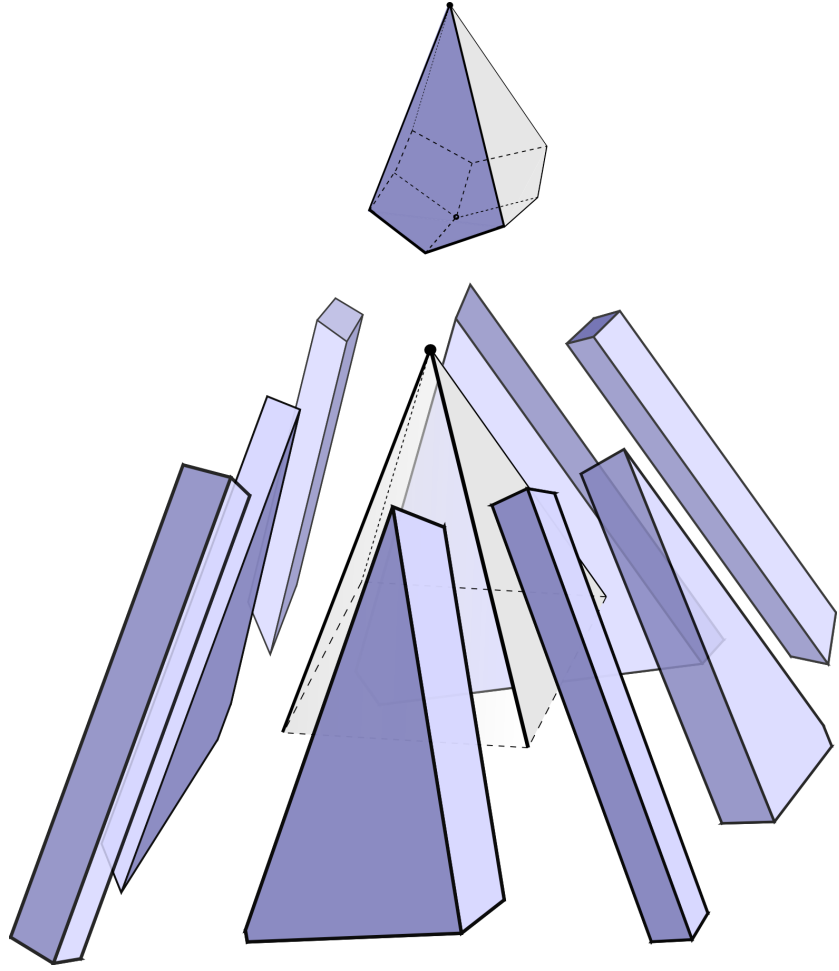}
        \caption{ $col(Q_{1,2})$ near its singularity} \label{fig16}
\end{figure}

Although above the piecewise smooth (singular manifold with corners) structure would suffice, it will be convenient to smoothen the above charts and see that $\qcllkr$, the smoothened version, is isomorphic to $\qklk$. As with disks, the combinatorial type of a cluster will be encoded in a painted tree $T$ having one vertex for each smooth component. Let $\mathcal{QC}\ell_{\ell,k,T}^\otimes$ be the quilted clusters having type $T$.

\begin{lem} \label{qsmoo_lemm}
There exists a piecewise smooth isomorphism $(\qcllkr)^{PDIFF} \rightarrow \qklk$ sending $\mathcal{QC}\ell_{\ell,k,T}^\otimes$ to $\mathcal{Q}_{\ell,k,T}$ for every combinatorial type $T$.
\end{lem}

This is performed as in the unquilted case in appendix \ref{appe_B}, decomposing single ratio charts of $\qklk$.

\begin{defi}
        Let $\qcllkr$ be $(\qcllkr)^{PDIFF} = col(\qklk)$ endowed with the singular manifold with embedded corner structure induced by the above identification.
\end{defi}


Recall that every 1-boundary of $\qklk$ is naturally isomorphic either to a product $Q_{\ell^{(1)},k^{(1)}} \times K_{\ell^{(2)},k^{(2)}}$, where the second factor stands for an unquilted disk not containing the root, or to a product $K_{\ell^{(1)},k^{(1)}} \times Q_{\ell^{(2)},k^{(2)}} \times \dots \times Q_{\ell^{(q)},k^{(q)}}$, where the first factor stands for an unquilted disk containing the root and the $q-1$ nodes. We immediately get that every 1-boundary stratum of $\qcllkr$ is naturally isomorphic either to a product $\mathcal{QC}l_{\ell^{(1)},k^{(1)}}^\otimes \times \mathcal{C}l_{\ell^{(2)},k^{(2)}}^\otimes$ or a product $\mathcal{C}l_{\ell^{(1)},k^{(1)}}^\otimes \times \mathcal{QC}l_{\ell^{(2)},k^{(2)}}^\otimes \times \dots \times \mathcal{QC}l_{\ell^{(q)},k^{(q)}}^\otimes$.

Also, the natural projection $\qklk \rightarrow \klk$ defined by forgetting the quilting extends to a projection $\qcllkr \rightarrow \cllkr$ by summing the lengths of the lines touching a bivalent quilted disk.

\section{Coherent system of perturbation homotopies}

In order to define morphisms between complexes built from two different choices of perturbation data, one needs to choose coherent perturbations over the quilted clusters that interpolate the two data sets.

First, we set the family of nodal disks over which we will choose the interpolation data. Let $\qulkr \overset{\pi}{\rightarrow} \qklk$ be the nodal family of quilted marked disks. Again, we add metric lines between the components of the marked disks lying in the collar part of $\qklk$.

\begin{defi}
        For $\ell -1 +2k \geq 0$, $G = (Q,X) \in Q_{\ell,k,\geq T} \times X^{[0,1]}(T)$, take the marked nodal quilted disk $\pi^{-1}(Q)$ and take its normalization modified so that the two markings corresponding to $e \in E^{int}(T)$ are connected by a metric line of length $-log(X(e))$. In addition, for every $1 \leq j \leq \ell$ (resp. $j=0$), identify the origin of a copy of $\overline{\mathbb{R}}_{+}$ (resp. $\overline{\mathbb{R}}_{-}$) with the boundary marking $x_j(Q)$ (see figure \ref{fig17}). For $\ell = 1$ and $k=0$, consider the pair $(\overline{\mathbb{R}},[0,1])$, and for $\ell = k = 0$, associate to $s \in \overline{\R}_-$ the pair $(\overline{\mathbb{R}}_{-},[s,min\{s+1,0\}])$.

        We denote the resulting by $(\pi^\mathcal{Q})^{-1}(G)$ and refer to it as the quilted $\otimes$-cluster, with $\ell$ leaves and $k$ interior markings, associated with $G \in \qcllkr$. 
\end{defi}

\begin{figure}[h]
        \centering 
	\includegraphics[width=110mm]{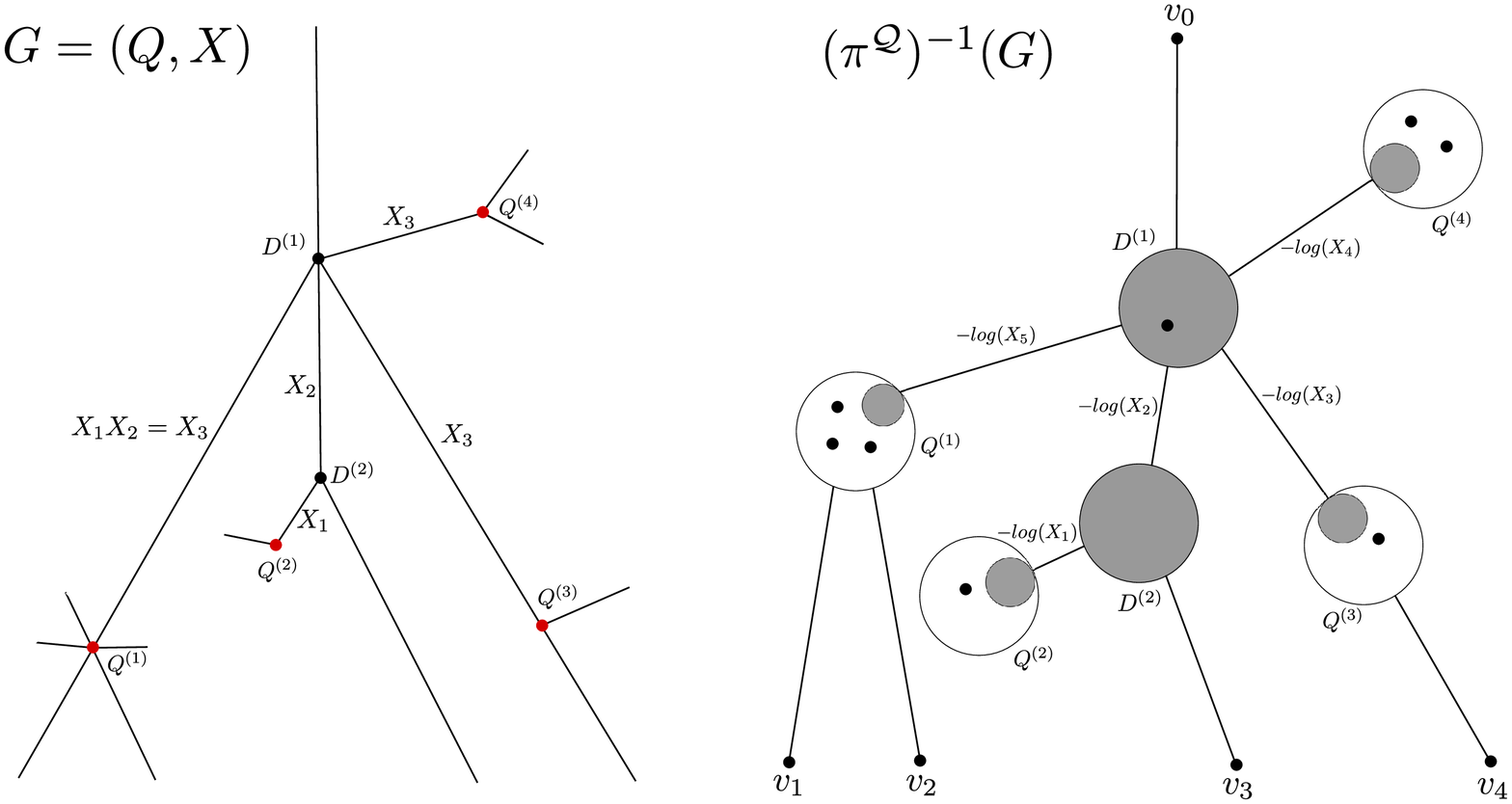}
        \caption{Quilted cluster $(\pi^\mathcal{Q})^{-1}(G)$ associated with $G=(Q,X)$, $Q= D^{(1)} \cup D^{(2)} \cup Q^{(1)} \cup\dots \cup Q^{(4)} \in C_2(Q_{4,8})$} \label{fig17}
\end{figure}

Again in the case of quilted clusters, the half-lines are considered as semi-infinite ending to either a leaf or the root and a connecting line of infinite length is seen as a broken line.

\begin{defi}
The points $\{v_j(G)\}_{0 \leq j \leq \ell}$ and the breaking points of $G$ are called the endpoints of $G$. We call a quilted cluster irreducible if it does not contain any broken line.

$\partial G$ will denote the complement of the interior of the disks in $G$.
\end{defi}


Taking $\qulkr = \underset{G \in \qcllkr}{\bigcup} (\pi^\otimes)^{-1}(G)$, we get a map $\pi^\otimes: \qulkr \rightarrow \qcllkr$ and again, $\pi^\otimes$ will be thought as a universal family of quilted $\otimes$-clusters.

\begin{defi}
        For every $\ell$  and $k$, we set $v_j: \qcllkr \rightarrow \qulkr$, $1 \leq j \leq \ell$ (resp. $j=0$) as being the smooth sections that map to the leaves (resp. root) of the $\otimes$-clusters and $z_h: \qcllkr \rightarrow \qulkr$, $1 \leq h \leq k$ the smooth sections which map to the interior markings. Also, for every $1$-boundary $F \in C_1(\qcllkr)$, set $v_F: F \rightarrow (\pi^\otimes)^{-1}(F) \subset \qulkr$ as the smooth section which maps to the breaking point associated with $F$.
\end{defi}

Now, given two coherent systems of ends $(\mathcal{V}^{(0)},\mathcal{S}^{(0)})$ and $(\mathcal{V}^{(1)},\mathcal{S}^{(1)})$ over $\{ \cllkr \}_{\ell,k \geq 0}$, one can use their preimage under the forgetful map $\qulkr \rightarrow \ulkr$. Note that the endpoints of any quilted cluster split as the endpoints lying below the seam and the ones above it. Let then $(\mathcal{V}^H,\mathcal{S}^H)$ be the coherent system of ends defined by taking
\begin{itemize}
\item $(\mathcal{V}^{(1)},\mathcal{S}^{(1)})$ above the quilted disks,
\item $(\mathcal{V}^{(0)},\mathcal{S}^{(0)})$ under them,
\item over a quilted disk with $\ell=1$ and $q=0$ (i.e. $D\backslash \{x_0,x_1\} \cong \R \times [0,1]$, $x_0$ seen as $-\infty$ and $x_1$ seen as $+\infty$), we use $\mathcal{S}^{(0)}$ over $]-\infty,-1] \times [0,1]$ and $\mathcal{S}^{(1)}$ over $[1,\infty[ \times [0,1]$. We then choose iteratively extensions $\mathcal{S}^H$ over $\{ \qklk \}_{\ell,k}$ (see \cite{W1}).
\end{itemize}
It is again coherent in the sense that the endpoint neighborhoods over any component do not depend upon the deformations of its complement.

\begin{defi} \label{qpht}
Next, take two coherent systems of perturbations $P^{(0)}= \{ \mathcal{U}^\otimes_{\ell, k, \mathfrak{l}} \overset{p^{(0)}_{\ell,k,\mathfrak{l}}}{\rightarrow} \mathcal{M}^{(0)} \times \mathcal{J}^{(0)} \}_{\ell,k \geq 0, \mathfrak{l} \in \mathfrak{L}}$ and $P^{(1)} = \{ \mathcal{U}^\otimes_{\ell, k, \mathfrak{l}} \overset{p^{(1)}_{\ell,k,\mathfrak{l}}}{\rightarrow} \mathcal{M}^{(1)} \times \mathcal{J}^{(1)} \}_{\ell,k \geq 0, \mathfrak{l} \in \mathfrak{L}}$ constant over $(\mathcal{V}^{(0)},\mathcal{S}^{(0)})$ and $(\mathcal{V}^{(1)},\mathcal{S}^{(1)})$, respectively.

Let $ \mathcal{M}^H \times \mathcal{J}^H \overset{s}{\rightarrow} [0,1]$ be a smooth Banach bundle such that $s^{-1}(i) \cong \mathcal{M}^{(i)} \times \mathcal{J}^{(i)}$, $i=0,1$.


Let $\mathcal{QC}\ell_{\ell,k,\mathfrak{l}}^\otimes$ and $\mathcal{QU}_{\ell,k,\mathfrak{l}}^\otimes$ be defined as in the unquilted case. A set of smooth maps $\mathcal{Q}P^H = \{ \mathcal{QU}^\otimes_{\ell, k, \mathfrak{l}} \overset{p^H_{\ell,k,\mathfrak{l}}}{\rightarrow} \mathcal{M}^H \times \mathcal{J}^H \}_{\ell,k \geq 0, \mathfrak{l} \in \mathfrak{L}}$ such that

\begin{enumerate}
        \item $s \circ p^H_{\ell,k,\mathfrak{l}} \equiv 1$ and $p^H_{\ell,k,\mathfrak{l}}= p^{(1)}_{\ell,k,\mathfrak{l}}$ above the quilted components and on the $\mathcal{S}^{(1)}$ strip-like ends,

	\item $s \circ p^H_{\ell,k,\mathfrak{l}} \equiv 0$ and $p^H_{\ell,k,\mathfrak{l}} = p^{(0)}_{\ell,k,\mathfrak{l}}$ below the quilted components and on the $\mathcal{S}^{(0)}$ strip-like ends, 


        \item $\pi_\mathcal{J} \circ p^H_{\ell,k,\mathfrak{l}} \subset \{0\} \times [0,1]$ on the boundaries of the quilted components, 

        \item it is coherent with respect to the product structure of the boundary components of $\mathcal{\mathcal{QC}\ell}^\otimes_{\ell, k, \mathfrak{l}}$


\end{enumerate}

is called a coherent choice of homotopy between $P^{(0)}$ and $P^{(1)}$.

\end{defi}

The procedure of lemma \ref{rest_bair} provides the existence of extension operators so that generic choices over low dimensional strata can be extended to generic choices over strata of greater dimension.

\section{Orientations on $\{ \qcllkr \}_{\ell,k \geq 0} $ }

We first notice that we can associate to every smooth quilted cluster $G$ an operator $\ddel_G \equiv 0 \oplus \ddel_C: L^{m,p}_0 (C, TC, T \partial C) \rightarrow  \R \oplus L^{m-1,p}_\pi (C,TC)$, where $C$ is the underlying cluster of $G$. Therefore, for $\ell -1 +2k \geq 1$ we proceed to the identification $coker(\ddel_G) = \R \times T_C \cllkr = T_G \qcllkr$ where, as before, $\R_+$ corresponds to the increase of the quilting circle radius.

\begin{defi}  \label{qori}

For $G \in \qcllkr $ smooth, the reference orientation on $\ddel_G$ is taken to be $\mathcal{QO}^\otimes_{\ell,k} \equiv \frac{\partial}{\partial s} \wedge \mathcal{O}^\otimes_{\ell,k}$ where $\frac{\partial}{\partial s}$ is a positive generator of the $\R$ factor of the target space.

\end{defi}

Now we find the combinatorial formula for the difference between the orientation $\partial_{F} \mathcal{QO}^\otimes_{\ell,k}$
induced on a component $F \in C_1(\qcllkr)$ by $\mathcal{QO}^\otimes_{\ell,k}$ by using the usual convention, and its reference orientation given by a product of $\mathcal{O}^\otimes$'s and $\mathcal{QO}^\otimes$'s.

There are two different cases to consider, that is, $F = \mathcal{QC}\ell^\otimes_{\ell^{(1)},k^{(1)}} \times \mathcal{C}\ell^\otimes_{\ell^{(2)},k^{(2)}}$ and $F = \mathcal{C}\ell^\otimes_{\ell^{(1)},k^{(1)}} \times \mathcal{QC}\ell^\otimes_{\ell^{(2)},k^{(2)}} \times \dots \times \mathcal{QC}\ell^\otimes_{\ell^{(q)},k^{(q)}}$ (see figure \ref{fig18}).

\begin{figure}[h]
        \centering 
	\includegraphics[width=110mm]{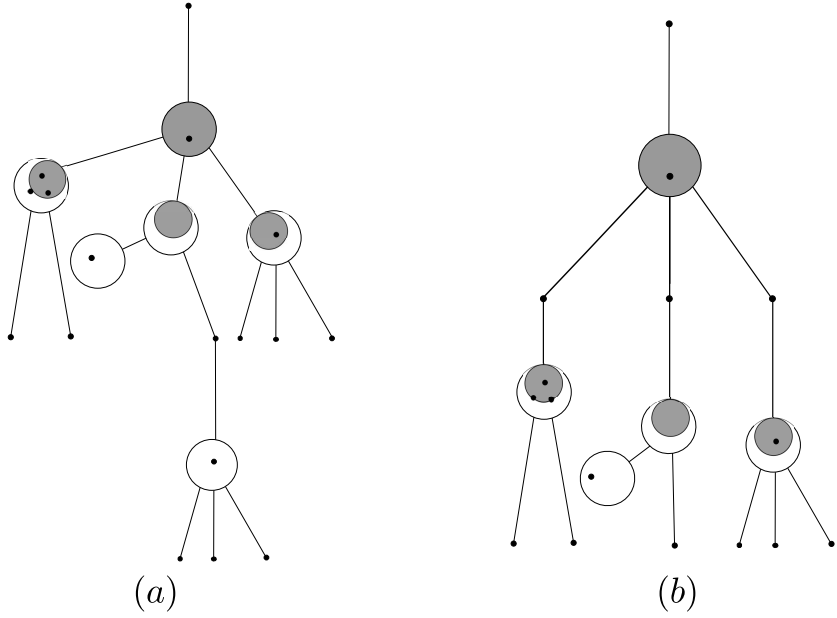}
        \caption{The two types of boundaries of $\qklk$: (a) lower facet (b) upper facet} \label{fig18}
\end{figure}

In the first case, it follows from the corresponding comparison formula for $\qklk$. That is the object of lemma \ref{out} (also see lemma 3.4 of \cite{W1}) and its proof only relies on the combinatorics of $\klk$.

\begin{lem} \label{qout1}
        Let $\ell^{(1)} \geq 1$, $\ell^{(2)} \geq 0$ and $k^{(1)}, k^{(2)} \geq 0$ such that $(\ell^{(2)}, k^{(2)}) \neq (0,0)$. Then let $\ell +1 = \ell^{(1)} + \ell^{(2)}$, $k = k^{(1)} + k^{(2)}$, and $F = \mathcal{QC}\ell^\otimes_{\ell^{(1)},k^{(1)}} \times \mathcal{C}\ell^\otimes_{\ell^{(2)},k^{(2)}} \in C_1(\qcllkr)$ be a lower facet corresponding to an unquilted cluster attached on the $j^{th}$ leaf of a quilted one. Then

\[
        \partial_{F} \mathcal{QO}^\otimes_{\ell,k} = (-1)^{(\ell^{(1)} -j)\ell^{(2)} + j} \mathcal{QO}^\otimes_{\ell^{(1)},k^{(1)}} \wedge
\mathcal{O}^\otimes_{\ell^{(2)},k^{(2)}}.
\]
\end{lem}

\begin{proof}
        For $\ell^{(1)}=1$ and $k^{(1)} =0$, $F$ is the bottom face of $\qcllkr$ so $(-1)^j \frac{\partial}{\partial s}= -\frac{\partial}{\partial s}$ is the outward normal. Therefore $-\frac{\partial}{\partial s} \wedge -\mathcal{O}^\otimes_{\ell,k} = \mathcal{QO}^\otimes_{\ell,k}$ and $\partial_{F} \mathcal{QO}^\otimes_{\ell,k} = -\mathcal{O}^\otimes_{\ell^{(2)},k^{(2)}}$.

        Otherwise, we can take the outward normal $\frac{\partial}{\partial n_F}$ to lie in a level set of the quilting radius parameter $s$ and therefore see it as the outward normal to $\cllkr$. Then, by using definition \ref{ori} and proposition \ref{out}, we get
\begin{align*}
        & \frac{\partial}{\partial n_F} \wedge (-1)^{(\ell^{(1)} -j)\ell^{(2)} + j} \mathcal{QO}^\otimes_{\ell^{(1)},k^{(1)}} \wedge\mathcal{O}^\otimes_{\ell^{(2)},k^{(2)}}\\
        &= \frac{\partial}{\partial n_F} \wedge (-1)^{(\ell^{(1)} -j)\ell^{(2)} + j} \frac{\partial}{\partial s} \wedge \mathcal{O}^\otimes_{\ell^{(1)},k^{(1)}} \wedge \mathcal{O}^\otimes_{\ell^{(2)},k^{(2)}} \\          
        &= \frac{\partial}{\partial s} \wedge (-1)^{(\ell^{(1)} -j)\ell^{(2)} + (j-1)} \frac{\partial}{\partial n_F} \wedge \mathcal{O}^\otimes_{\ell^{(1)},k^{(1)}} \wedge \mathcal{O}^\otimes_{\ell^{(2)},k^{(2)}} \\
        &= \frac{\partial}{\partial s} \wedge \mathcal{O}^\otimes_{\ell,k} = \mathcal{QO}^\otimes_{\ell,k}.\\
\end{align*}

\end{proof}

In the second case, 

\begin{lem} \label{qout2}
        Let $q \geq 1$, $\ell^{(i)} -1 +2k^{(i)} \geq 0$, $1 \leq i \leq q$. Then let $\ell = \sum^{q}_{i=1} \ell^{(i)}$, $k = \sum^{q+1}_{h=1} k^{(h)}$, and $F = \mathcal{C}\ell^\otimes_{q,k^{(q+1)}} \times \mathcal{QC}\ell^\otimes_{\ell^{(1)},k^{(1)}} \times \dots \times \mathcal{QC}\ell^\otimes_{\ell^{(q)},k^{(q)}} \in C_1(\qcllkr)$. Then

\[
        \partial_{F} \mathcal{QO}^\otimes_{\ell,k} = (-1)^{\sum^q_{i=1} (q -i)(\ell^{(i)} - 1)} \mathcal{O}^\otimes_{q,k^{(q+1)}} \wedge
\mathcal{QO}^\otimes_{\ell^{(1)},k^{(1)}} \wedge \dots \wedge \mathcal{QO}^\otimes_{\ell^{(q)},k^{(q)}}.
\]

\end{lem}

\begin{proof}
	Again, the result is direct if $F$ is the top face of $\qcllkr$, that is, when $\ell^{(j)} = 1$ and $k^{(j)} = 0$, $ 1 \leq j \leq q$: $\frac{\partial}{\partial s}$ is the outward normal and $\frac{\partial}{\partial s} \wedge \mathcal{O}^\otimes_{\ell,k} = \mathcal{QO}^\otimes_{\ell,k}$.

Otherwise, if $k^{(j)} \geq 1$, $1 \leq j \leq q$, we compute using appropriate coordinates in a neighborhood of $F$: 

\begin{align*}
& \frac{\partial}{\partial n_F} \wedge (-1)^{\sum^q_{i=1} (q -i)(\ell^{(i)} - 1)} \mathcal{O}^\otimes_{q,k^{(q+1)}} \wedge
\mathcal{QO}^\otimes_{\ell^{(1)},k^{(1)}} \wedge \dots \wedge \mathcal{QO}^\otimes_{\ell^{(q)},k^{(q)}} \\
&= \bigwedge_{h=2}^{k^{(q+1)}} \mathcal{O}_{z^{(q+1)}_h} \wedge \bigwedge_{j=1}^{q} \bigwedge_{h=2}^{k^{(j)}} \mathcal{O}_{z^{(j)}_h} \wedge  (-1)^{\sum^q_{i=1} (q -i)(\ell^{(i)} - 1)} \frac{\partial}{\partial n_F} \wedge \frac{\partial}{\partial x_1^{(q+1)}} \wedge \ldots \wedge \frac{\partial}{\partial x_{q}^{(q+1)}} \wedge \\
& \frac{\partial}{\partial s^{(1)}} \wedge \cO_{\ell^{(1)}} \wedge \ldots \wedge \frac{\partial}{\partial s^{(q)}} \wedge \cO_{\ell^{(q)}}\\
&= \bigwedge_{h=2}^{k^{(q+1)}} \mathcal{O}_{z^{(q+1)}_h} \wedge \bigwedge_{j=1}^{q} \bigwedge_{h=2}^{k^{(j)}} \mathcal{O}_{z^{(j)}_h} \wedge  (-1)^{\sum^q_{i=1} (q -i)(\ell^{(i)} - 1)} \frac{\partial}{\partial s} \wedge \frac{\partial}{\partial Re(z_1^{(1)})} \wedge \ldots \wedge \frac{\partial}{\partial Re(z_{1}^{(q)})} \wedge \\
& \frac{\partial}{\partial Im(z_1^{(1)})} \wedge \cO_{\ell^{(1)}} \wedge \ldots \wedge \frac{\partial}{\partial Im(z_1^{(q)})} \wedge \cO_{\ell^{(q)}}\\
&= \bigwedge_{h=2}^{k^{(q+1)}} \mathcal{O}_{z^{(q+1)}_h} \wedge \bigwedge_{j=1}^{q} \bigwedge_{h=2}^{k^{(j)}} \mathcal{O}_{z^{(j)}_h} \wedge  (-1)^{\sum^q_{i=1} (q -i)} \frac{\partial}{\partial s} \wedge \frac{\partial}{\partial Re(z_1^{(1)})} \wedge \ldots \wedge \frac{\partial}{\partial Re(z_{1}^{(q)})} \wedge \\
& \frac{\partial}{\partial Im(z_1^{(1)})} \wedge \ldots \wedge \frac{\partial}{\partial Im(z_1^{(q)})} \wedge \cO_{\ell^{(1)}} \wedge \ldots \wedge \cO_{\ell^{(q)}}\\
&= \bigwedge_{h=2}^{k^{(q+1)}} \mathcal{O}_{z^{(q+1)}_h} \wedge \bigwedge_{j=1}^{q} \bigwedge_{h=1}^{k^{(j)}} \mathcal{O}_{z^{(j)}_h} \wedge \frac{\partial}{\partial s} \wedge \cO_{\ell}\\
&= \frac{\partial}{\partial s} \wedge \bigwedge_{h=2}^{k} \mathcal{O}_{z_h} \wedge \cO_{\ell} = \frac{\partial}{\partial s} \wedge \mathcal{O}^\otimes_{\ell,k} = \mathcal{QO}^\otimes_{\ell,k}. \\
\end{align*}

\end{proof}

\section{Complex morphisms}

The goal of this section is to compare, via trajectories having quilted clusters as their sources, \\ $((\mathcal{C}\ell^\otimes)^{(0)},(\delta^\otimes)^{(0)})$ and $((\mathcal{C}\ell^\otimes)^{(1)},(\delta^\otimes)^{(1)})$, two complexes built from different data sets by using an appropriate homotopy of these data.

As before fix $c \geq 0$, and let $(\mathcal{C}\ell^\otimes)^{(i)} = \mathcal{C}\ell^\otimes(L^{(i)},f^{(i)},f^{(i)}_0, \dots, f^{(i)}_c)$, $i = 0,1$, be the appropriately ordered products of critical points.

We now define trajectories between $x^- \in (\mathcal{C}\ell^\otimes)^{(0)}$ and $x^+ \in (\mathcal{C}\ell^\otimes)^{(1)}$ which satisfy the gradient and pseudoholomorphic equations associated with a coherent system of perturbation homotopies $\mathcal{QP}^H$ between $P^{(0)}$ and $P^{(1)}$.

Let $(\phi^H_s)_{s\in[0,1]}$ be a hamiltonian flow such that $\phi^H_0 = Id$ and $\phi^H_1(L^{(0)}) = L^{(1)}$. Let then $f^H_s: \phi^H_s(L^{(0)}) \rightarrow \R$ be a Morse-Smale family of functions. Then, the bundle $\mathcal{M}^H \rightarrow [0,1]$ will be chosen so that the fiber over $s$ is a $C^{\epsilon}$ Banach space tangent to $f^H_s$ in the space of smooth Morse-Smale functions on $L$. Also, let $(J^H_s)_{s\in[0,1]}$ be a smooth homotopy between $J^{(0)}$ and $J^{(1)}$ in $\mathbf{J}(TM,\omega)$. Then, $\mathcal{J}^H \rightarrow [0,1]$ will be chosen so that the fiber over $s$ is a $C^{\epsilon}$ Banach space tangent to $J^H_s$ in $\mathbf{J}(TM,\omega)$. Finally, for every pair $0 \leq r_1 < r_2 \leq c$, choose a smooth Morse-Smale homotopy $\{f^H_{r_1} - f^H_{r_2}\} \rightarrow [0,1]$ from $f^{(0)}_{r_1} - f^{(0)}_{r_2}$ to $f^{(1)}_{r_1} - f^{(1)}_{r_2}$ and choose a metric homotopy $g^H \rightarrow [0,1]$ from $g^{(0)}$ to $g^{(1)}$.

Now choose $\mathcal{Q}P^H = \{ \mathcal{QU}^\otimes_{\ell, k, \mathfrak{l}} \overset{p^H_{\ell,k,\mathfrak{l}}}{\rightarrow} \mathcal{M}^H \times \mathcal{J}^H \}_{\ell,k \geq 0, \mathfrak{l} \in \mathfrak{L}}$ a coherent system of perturbation homotopies between $P^{(0)}$ and $P^{(1)}$ such that for any $k \geq 1$, $p^H_{\ell,k,\mathfrak{l}}$ is the pullback of $p^H_{\ell,0,\mathfrak{l}}$ via the forgetful map $\mathcal{QC}l_{\ell,k,\mathfrak{l}}^\otimes \rightarrow \mathcal{QC}l_{\ell,0,\mathfrak{l}}^\otimes$ so that the perturbation data does not depend on the position of the interior markings. Note that this still allows a very general choice of almost-complex structures over the quilted disks.

As to define the differential in the monotone setting, we forget the interior markings over quilted clusters. For every $\ell,k \geq 0$, $\mathfrak{l} \in \mathfrak{L}$ and quilted $\otimes$-cluster $G \in \mathcal{QC}l_{\ell,k,\mathfrak{l}}^\otimes$, we consider the configuration resulting from forgetting its interior marked points, but without stabilizing it, and shall refer to it by $G' \in \mathcal{QC}l_{\ell,\mathfrak{l}}^\otimes$. Again, removing the interior markings on a monovalent or bivalent unquilted disk or on a monovalent quilted disk will result in instability. However, as for non-quilted clusters, the orientations of $\ddel_{G'}$ naturally correspond to those of $\ddel_{G}$. Moreover, a monotone coherent system of perturbation homotopies $\mathcal{Q}P^H$ defined as above defines a perturbation homotopy over any $G' \in \mathcal{QC}l_{\ell,\mathfrak{l}}^\otimes$.

\begin{defi} \label{qtraj}
        Let $x^- = x_0^{(1)} \ldots x_0^{(q)} \in (\mathcal{C}\ell^\otimes)^{(0)}$ and $x^+ = x_1^{(1)} \ldots x_{\ell^{(1)}}^{(1)} \ldots x_1^{(q)} \ldots x_{\ell^{(q)}}^{(q)} \in (\mathcal{C}\ell^\otimes)^{(1)}$ two generators with $\ell^{(r)} \geq 1$, $1 \leq r \leq q$. Consider a pair $(u, G)$ such that

\begin{enumerate}
        \item $G = G^{(1)} \times \dots \times G^{(q)} \in \mathcal{QC}l_{\ell^{(1)}}^\otimes \times \dots \times \mathcal{QC}l_{\ell^{(q)}}^\otimes$,

        \item $u: G = \overset{q}{\underset{r=1}{\bigsqcup}} G^{(r)} \rightarrow M$ continuous such that $\forall 1 \leq r \leq q$,
        \begin{enumerate}

                \item $u(v_j(G^{(r)})) = x_j^{(r)}$ for $0 \leq j \leq \ell^{(r)}$, this naturally defines a labeling on the endpoints of $G^{(r)}$,

                \item $u(p) \subset \phi^H_{s\circ p^H(p)}(L)$, for every $p \in \partial G^{(r)}$,



                \item over every line $l$ of $G^{(r)}$, $u$ satisfies the gradient equation $du(-\frac{\partial}{\partial t}(p)) = -\nabla_{g^H_{s\circ p^H}}(\pi_\mathcal{M} \circ p^H(p)) \circ u(p)$ $\forall p \in l$,

                \item over every disc $d$ of $G^{(r)}$, $u$ satisfies the pseudoholomorphic equation $du \circ j(p) = (\pi_\mathcal{J} \circ p^H(p))(u(p)) \circ du(p)$ $\forall p \in d$, where $j$ is the underlying complex structure of $d$.
        \end{enumerate}

\end{enumerate}

As in the non-quilted case, over such a trajectory, one can consider the linearized Cauchy-Riemann operator $\ddel_u$ on $u^*TM \rightarrow G$ with boundary conditions $u|_{\partial G}^*TL \rightarrow \partial G$ and path of matrices $A$ given by the hessians of the functions determined by $\mathcal{Q}P^H$. As above, take $(G,u.o_u)$ where $(G,u)$ is a trajectory and $o_u$ is an orientation of $\ddel_u$ and denote $[(G,u.o_u)]$ the homotopy class of these data modulo permutations of the interior markings over every quilted cluster.

We call $[(G,u.o_u)]$  a quilted Floer trajectory from $x^+$ to $x^-$. We define its index $\mu([(G,u.o_u)]) = \mu(x^-) - \mu(x^+) + \mu(F_u)$ and its area $\omega([(G,u.o_u)])$ by the area of the trajectory with boundary in $L^{(0)}$ defined by applying $\phi^H$ over the partially quilted part of $(G,u.o_u)$. Define $QF(x^+, x^-)$ as the set of all the quilted Floer trajectories from $x^+$ to $x^-$.
\end{defi}



We then define a map that counts rigid quilted trajectories between elements of $(\mathcal{C}\ell^\otimes)^{(0)}$ and $(\mathcal{C}\ell^\otimes)^{(1)}$. We start off with defining intermediate operators that have cardinality $\ell$ inputs and cardinality $1$ outputs:

if $x$ is a generator of $(\mathcal{C}\ell^\otimes)^{(1)}$ with $q(x) = \ell$, we define
\[
        h_\ell(x) = \underset{\substack{x^- \in (\mathcal{C}\ell^\otimes)^{(0)} \\ q(x^-)=1}}{\sum} \: \underset{\substack{[(G,u, \mathcal{QO}^\otimes_{\ell})] \in QF(x,x^-) \\ \mu([(G,u,\mathcal{QO}^\otimes_{\ell})]) = -(\ell -1) }}{\sum} < \mathcal{QO}_{\ell}^\otimes \# \mathcal{O}_{x}^\otimes ,\mathcal{O}_{x^- t^{\frac{\omega(u)}{\tau N_L}}}^\otimes > x^- t^{\frac{\omega(u)}{\tau N_L}}.
\]

where the orientation over every $(G,u)$ is chosen via the following exact sequence of operators:

\begin{diagram} 
        0 & \rTo & L^{m,p}_0 (G, TG, T \partial G) & \rTo &&& L^{m,p}(G,u^*TM,u^*TL) &&& \rTo & \frac{L^{m,p}(G,u^*TM,u^*TL)}{L^{m,p}_0 (G, TG, T \partial G)} & \rTo & 0\\
        & & \dTo^{\ddel_G} &&&& \dTo^{\ddel_u} &&&& \dTo^{\ddel_u / \ddel_G} & &\\
        0 & \rTo & \R \times L^{m-1,p}_\pi (G,TG) && \rTo && L^{m-1,p}(G,u^*TM) && \rTo && \frac{L^{m-1,p}(G,u^*TM)}{\R \times L^{m-1,p}_\pi (G,TG)} & \rTo & 0 \\
\end{diagram}

The right-hand side operator can be made surjective, and hence an isomorphism, by choosing generically $\mathcal{Q}P^H$. 
So, if we orient $\ddel_G$ as $\mathcal{QO}_{\ell}^\otimes$, its reference orientation obtained from $\mathcal{QO}_{\ell,k}^\otimes$, $\ddel_u$ inherits an orientation denoted by the same symbol.


We extend $h_\ell$ to a $\Lambda$-module map and define the morphism $H$ on arbitrary elements as
\[
        H = \underset{ \sum_i \ell^{(i)} \geq 0}{\sum} \, \underset{q \geq 1}{\sum} (-1)^{\sum^q_{i=1} (q -i)(\ell^{(i)} - 1)}  h_{\ell^{(1)}} \otimes \dots \otimes h_{\ell^{(q)}}. 
\]

Note that it is well defined by compactness of both $M$ and $L$, fairly standard Gromov-type compactness results for pseudoholomorphic discs with lagrangian boundary and the usual compactness results from Morse theory.

\begin{prp} \label{coch_map}
        $H$, defined as above, is a cochain map from $((\mathcal{C}\ell^\otimes)^{(1)},(\delta^\otimes)^{(1)})$ to \\
$((\mathcal{C}\ell^\otimes)^{(0)},(\delta^\otimes)^{(0)})$. That is, we have $H \circ (\delta^\otimes)^{(1)} = (\delta^\otimes)^{(0)} \circ H$.
\end{prp}

\begin{proof}

Again, the transversality results of section \ref{tran_floe}, the gluing theorem of \cite{BC} and some standard compactness results, the composite trajectories counted by $H \circ (\delta^\otimes)^{(1)} - (\delta^\otimes)^{(0)} \circ H$ are in 1:1 correspondence with the boundary of a compact 1-dimensional piecewise smooth manifold. Actually, we can now assume simplicity of the $1$-dimensional families of Floer trajectories by further choosing generically the perturbation over the quilted disk components.

It remains to see that the two composite trajectories corresponding to the two ends of a connected component of this moduli space are counted with opposite orientation. Again, the computations are performed in \cite{W1}, and moreover, lemma 1.9 of \cite{W2} allows one to assume that the trajectories have cardinality one targets, but include a full verification for completeness. 

Indeed,  we see that each of these compositions is counted by either

\begin{align*}
&(-1)^{\sum^q_{i=1} (q -i)(\ell^{(i)} - 1)}  h_{\ell^{(1)}} \otimes \dots \otimes h_{\ell^{(q)}} \circ \\
&(-1)^{(\sum^q_{i=1} \ell^{(i)}-j)\ell + (j-1)} Id^{(j-1)}\otimes m_{\ell} \otimes Id^{(q-j)},
\end{align*}
so by lemma \ref{qout1} it is counted with orientation
\begin{align*}
&(-1)^{\sum^q_{i=1} (q -i)(\ell^{(i)} - 1) + (\sum^q_{i=1} \ell^{(i)}-j)\ell + (j-1)}  \mathcal{QO}_{\ell^{(1)}}^\otimes \wedge \dots \wedge \mathcal{QO}_{\ell^{(q)}}^\otimes \wedge \mathcal{O}_{\ell}^\otimes = \\
&(-1)^{\sum^q_{1} (q -i)(\ell^{(i)} - 1) + \sum^q_{r+1} \ell^{(i)} \ell  + \sum^{r-1}_{1} \ell^{(i)} + \sum^q_{r+1} (\ell^{(i)}-1)\ell} \\
&\mathcal{QO}_{\ell^{(1)}}^\otimes \wedge \dots \wedge \Big((-1)^{((\sum^q_{1} \ell^{(i)}-j) - \sum^q_{r+1} \ell^{(i)})\ell + (j-1) - \sum^{r-1}_{1} \ell^{(i)}} \mathcal{QO}_{\ell^{(r)}}^\otimes \wedge \mathcal{O}_{\ell}^\otimes \Big) \\
&\wedge \dots  \wedge \mathcal{QO}_{\ell^{(q)}}^\otimes = \\
&(-1)^\Phi \mathcal{QO}_{\ell^{(1)}}^\otimes \wedge \dots \wedge \partial_{F} \mathcal{QO}_{\ell^{(r)} + \ell -1}^\otimes  \wedge \dots \mathcal{QO}_{\ell^{(q)}}^\otimes
\end{align*}
where $\Phi = \underset{i \neq r}{\sum} (q -i)(\ell^{(i)} - 1)  + \sum^{r-1}_{1} \ell^{(i)} + (q-r)(\ell^{(r)} + \ell -1)$, or by

\begin{align*}
&(-1)^{(q-r)\ell' + (r-1)} Id^{(r-1)}\otimes m_{\ell'} \otimes Id^{(q-r)}  \circ \\
&(-1)^{\sum^{q+(\ell'-1)}_{i=1} (q+(\ell'-1) -i)(\ell'^{(i)} - 1)} h_{\ell'^{(1)}} \otimes \dots \otimes h_{\ell'^{(q + (\ell' -1))}},
\end{align*}
so by lemma \ref{qout2} it is counted with orientation
\begin{align*}
&(-1)^{\sum^{q+(\ell'-1)}_{1} (q+(\ell'-1) -i)(\ell'^{(i)} - 1) + (q-r)\ell' + (r-1)} \mathcal{O}_{\ell'}^\otimes \wedge \mathcal{QO}_{\ell'^{(1)}}^\otimes \wedge \dots \wedge \mathcal{QO}_{\ell'^{(q+(\ell'-1))}}^\otimes  = \\
&(-1)^{\underset{\substack{\scriptscriptstyle i \leq r \\ \scriptscriptstyle i \geq r+(\ell'-1) } }{\sum} (q -i)(\ell^{(i)} - 1) +(q-r)\sum^{\ell'}_{1} (\ell'^{(m+(r-1))}-1) +(\ell'-1)\sum^{r-1}_{1} (\ell'^{(i)}-1)  +(q-r)\ell' + (r-1)+ \sum^{r-1}_{1} \ell'(\ell'^{(i)}-1)} \\
&\mathcal{QO}_{\ell'^{(1)}}^\otimes \wedge \dots \wedge \Big((-1)^{((\sum^{\ell'}_{m=1} (\ell'-m)(\ell'^{(m+(r-1))}-1) } \mathcal{O}_{\ell'}^\otimes \wedge \mathcal{QO}_{\ell'^{(r)}}^\otimes \wedge \dots \wedge \mathcal{QO}_{\ell'^{(r+(\ell'-1))}}^\otimes\Big) \\
&\wedge \dots \mathcal{QO}_{\ell'^{(q+(\ell'-1)))}}^\otimes = \\
&(-1)^\Phi \mathcal{QO}_{\ell^{(1)}}^\otimes \wedge \dots \wedge \partial_{F'} \mathcal{QO}_{\ell^{(r)} + \ell -1}^\otimes  \wedge \dots \mathcal{QO}_{\ell^{(q)}}^\otimes
\end{align*}

The conclusion is that each broken trajectory induces $(-1)^\Phi$ times the outward normal orientation on its 1-dimensional glued family. Since $\Phi$ is constant along the homotopy class, we get the result. 

\end{proof}

\section{Homotopies between complex morphisms}

The goal of this section is to relate complex morphisms $H^{(0)}$ and $H^{(1)}$ built from different coherent systems of perturbation homotopies $\mathcal{Q}P^{H^{(0)}}$ and $\mathcal{Q}P^{H^{(1)}}$. We want to verify that they are homotopic and consequently complete the proof of the functoriality of the $\otimes$-complex construction.

Choose a smooth homotopy $(\mathcal{Q}P^{H^{(t)}})_{t\in[0,1]}$ from $\mathcal{Q}P^{H^{(0)}}$ to $\mathcal{Q}P^{H^{(1)}}$ so that $\mathcal{Q}P^{H^{(t)}}$ satisfies definition \ref{qpht} for all $t \in [0,1]$. Take $[0,1] \times \qcllkr$ and pullback the universal curve $\qulkr \rightarrow \qcllkr$. Then one can choose to use $\mathcal{Q}P^{H^{(t)}}$ as a choice of perturbation data over $\{t\} \times \qcllkr$.

Then, one defines a pair $[(G,u,o_u)]$ a $t$-quilted Floer trajectory from $x^+$ to $x^-$ as being a Floer quilted trajectory with perturbation data given by $\mathcal{Q}P^{H^{(t)}}$ (see definition \ref{qtraj}). 

We then define a differential map that counts rigid $t$-quilted trajectories between elements of $(\mathcal{C}\ell^\otimes)^{(1)}$ and $(\mathcal{C}\ell^\otimes)^{(0)}$ with $t$ varying in $[0,1]$. We start off with defining intermediate operators that have cardinality $\ell$ inputs and cardinality $1$ outputs:

if $x$ is a generator of $(\mathcal{C}\ell^\otimes)^{(1)}$ with $q(x) = \ell$, we define
\[
        k^{(t)}_\ell(x ) = \underset{\substack{x^- \in (\mathcal{C}\ell^\otimes)^{(0)} \\ q(x^-)=1}}{\sum} \: \underset{\substack{[(G,u,\frac{\partial}{\partial t} \wedge \mathcal{QO}^\otimes_{\ell})] \in tQF(x,x^-) \\ \mu(G,u) = -(\ell) }}{\sum} <\frac{\partial}{\partial t} \wedge \mathcal{QO}^\otimes_{\ell} \# \mathcal{O}_x^\otimes , \mathcal{O}_{x^-}^\otimes > x^-.
\]

where the orientation over every $(G,u)$ is again chosen via an exact sequence of operators with this time the source operator being $0 \oplus \ddel_{G}$ and has reference orientation $\frac{\partial}{\partial t} \wedge \mathcal{QO}^\otimes_{\ell}$. Again, the right-hand side operator in the associated exact sequence is made an isomorphism by choosing generically $\mathcal{Q}P^{H^{(t)}}$ and $\ddel_u$ then inherits an orientation also called $\frac{\partial}{\partial t} \wedge \mathcal{QO}^\otimes_{\ell}$.

We again extend $k^{(t)}_\ell$ linearly and define the homotopy $K$ on elements of cardinality $\sum_{i=1}^{q} \ell^{(i)} = q'$ as
\[
        K = \underset{ q \geq 1}{\sum} (-1)^q \, \underset{ q' }{\sum} (-1)^{\sum^q_{i=1} (q -i)(\ell^{(i)} - 1)} \, \overset{q}{\underset{ p= 1}{\sum}} (-1)^{\sum^{p-1}_{i=1} (\ell^{(i)} - 1)} \underset{ t\in [0,1]}{\sum} h^{(t)}_{\ell^{(1)}} \otimes \dots \otimes k^{(t)}_{\ell^{(p)}} \otimes \dots \otimes h^{(t)}_{\ell^{(q)}}. 
\]


\begin{prp} \label{coch_homo}
        For $K$, defined as above, is a cochain homotopy between $H^{(1)}$ and $H^{(0)}$. That is, we have $H^{(1)} - H^{(0)} =  K \circ (\delta^\otimes)^{(1)} + (\delta^\otimes)^{(0)} \circ K$.
\end{prp}

\begin{proof}

Again, by choosing generically the perturbation data and applying the arguments of section \ref{tran_floe}, the composite trajectories counted by $ H^{(1)} - H^{(0)} -  K \circ (\delta^\otimes)^{(1)} - (\delta^\otimes)^{(0)} \circ K$ are in 1:1 correspondence with the boundary of a compact 1-dimensional piecewise smooth manifold.

It remains to see that the broken trajectories corresponding to the two ends of a connected component induce opposite orientations on their homotopy class. Using lemma 1.11 of \cite{W2}, we will assume that the target of the trajectories are of cardinality $1$.

First,  $ -K \circ (\delta^\otimes)^{(1)}$ counts with

\begin{align*}
&  k^{(t)}_{\ell^{(1)}} \circ (-1)^{(q-j)\ell + (j-1)} Id^{(j-1)}\otimes m_{\ell} \otimes Id^{(q-j)},
\end{align*}
so by lemma \ref{qout1} it counts with orientation
\begin{align*}
& - (-1)^{(q-j)\ell + j} \frac{\partial}{\partial t} \wedge \mathcal{QO}_{\ell^{(1)}}^\otimes \wedge \mathcal{O}_{\ell}^\otimes = \\
& \partial_{F} (\frac{\partial}{\partial t} \wedge \mathcal{QO}_{\ell^{(1)} + \ell -1}^\otimes).
\end{align*}

Second, $ - (\delta^\otimes)^{(0)} \circ K$ counts with

\begin{align*}
& -m_{\ell} \circ (-1)^{\ell} \, (-1)^{\sum^{\ell}_{i=1} (\ell -i)(\ell^{(i)} - 1)} \, (-1)^{\sum^{p-1}_{i=1} (\ell^{(i)} - 1)} h^{(t)}_{\ell^{(1)}} \otimes \dots \otimes k^{(t)}_{\ell^{(p)}} \otimes \dots \otimes h^{(t)}_{\ell^{(\ell)}},
\end{align*}
so by lemma \ref{qout2} it counts with orientation
\begin{align*}
& -(-1)^{\sum^{\ell}_{i=1} (\ell -i)(\ell^{(i)} - 1)} \, (-1)^{\ell} \, (-1)^{\sum^{p-1}_{i=1} (\ell^{(i)} - 1)} \mathcal{O}_{\ell}^\otimes \wedge \mathcal{QO}_{\ell^{(1)}}^\otimes \wedge \dots \wedge \frac{\partial}{\partial t} \wedge \mathcal{QO}_{\ell^{(p)}}^\otimes \wedge \dots \wedge \mathcal{QO}_{\ell^{(\ell)}}^\otimes  \\
&= -(-1)^{\sum^{\ell}_{i=1} (\ell -i)(\ell^{(i)} - 1)} \, \frac{\partial}{\partial t} \wedge \mathcal{O}_{\ell}^\otimes \wedge \mathcal{QO}_{\ell^{(1)}}^\otimes \wedge \dots \wedge \mathcal{QO}_{\ell^{(p)}}^\otimes \wedge \dots \wedge \mathcal{QO}_{\ell^{(\ell)}}^\otimes  \\
&= \partial_{F} (\frac{\partial}{\partial t} \wedge \mathcal{QO}_{\sum^{\ell}_{i=1} \ell^{(i)}}^\otimes).
\end{align*}

Since, by construction, $H^{(1)}$ and $- H^{(0)}$ also count with orientations that induce outward normal orientations on their glued families, we get the result.

\end{proof}

\begin{prp} \label{h0h1_homo}
        For $H^{(i)}$ a cochain map between $\big((\mathcal{C}\ell^\otimes)^{(i+1)},(\delta^\otimes)^{(i+1)}\big)$ and \\
$\big((\mathcal{C}\ell^\otimes)^{(i)},(\delta^\otimes)^{(i)}\big)$, $i=0,1$, defined as above , and $H^{(1) \circ (0)}$ a cochain map between \\ 
$\big((\mathcal{C}\ell^\otimes)^{(2)},(\delta^\otimes)^{(2)}\big)$ and $\big((\mathcal{C}\ell^\otimes)^{(0)},(\delta^\otimes)^{(0)}\big)$ built from the concatenation of their perturbation homotopies. Then $H^{(0)} \circ H^{(1)}$ and $H^{(1) \circ (0)}$ are homotopic.
\end{prp}

\begin{proof}

Again, by choosing generically the perturbation data, the composite trajectories counted by $ H^{(0)} \circ H^{(1)} = H^{(1) \circ (0)}$ are seen to be in 1:1 correspondence with the boundary of a compact 1-dimensional piecewise smooth manifold by the same techniques as in section \ref{tran_floe}.

This time the 1-dimensional homotopies are modeled upon an extension of the spaces of biquilted disks $\{ \mathcal{R}_{\ell,k} \}_{\ell,k \geq 0}$, these having the same type of "balanced labelings" local charts as $\qklk$ (see \cite{MWW}). Loosely speaking, the radius of larger seam will encode the position of the perturbation data that defines $H^{(1)}$ while the radius of the smaller one will encode the position of the data corresponding to $H^{(0)}$.

The $1$-boundaries of the space of biquilted disks come in four types (see \cite{MWW}) illustrated in figure \ref{fig20}. Type (a) rigid Floer trajectories correspond to those counted by $H^{(0)} \circ H^{(1)}$, while type (d) rigid Floer trajectories correspond to those counted by $H^{(1) \circ (0)}$. Finally, the trajectories modeled on type (b) (resp. (d)) correspond to those counted by $R \circ (\delta^\otimes)^{(2)}$ (resp. $(\delta^\otimes)^{(0)} \circ R$), where $R$ stands for the homotopy that counts rigid $(t_1,t_2)$-biquilted Floer trajectories.

\begin{figure}[h]
        \centering 
	\includegraphics[width=110mm]{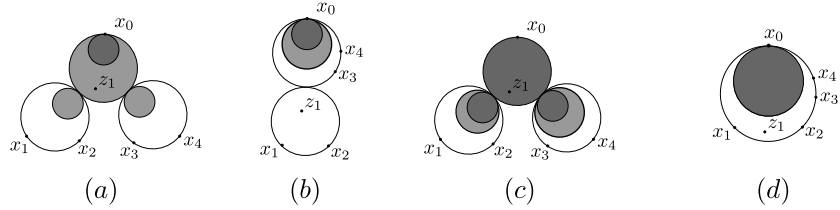}
        \caption{Types of facets of $\mathcal{R}_{\ell,k}$} \label{fig20}
\end{figure}

It remains to see that the broken trajectories corresponding to each term induce opposite orientations on their homotopy class. This sign computation is again very similar to the verification of proposition \ref{coch_homo}, proceeding as in \cite{W2}.

\end{proof}


A consequence of the above propositions is that the complex built using only one Morse function and perturbations, that is when we set $c=0$, is equivalent to the ones built from many. More precisely, let $(\mathcal{C}\ell_{ \{f,f_0, \dots, f_c\} }^\otimes,\delta^\otimes)$ be the complex obtained with $c\geq 1$, $(\mathcal{C}\ell_{\{ f\} }^\otimes,\delta^\otimes)$ be the subcomplex generated by the product of critical points of $f$ only and $(\mathcal{C}\ell_{ \{f_0, \dots, f_c\} }^\otimes,\delta^\otimes)$ the subcomplex generated without the critical points of $f$. In the monotone setting, these complexes support filtrations \\
$ \Big( (\mathcal{C}\ell_{ \{f,f_0, \dots, f_c\} }^\otimes)_{\leq c+d},\delta^\otimes \Big)_{ d \geq 0}$ and $ \Big( (\mathcal{C}\ell_{ \{f\} }^\otimes)_{\leq c+d},\delta^\otimes \Big)_{ c+d \geq 1}$ generated by products of length $c+d$ or less.

\begin{cor}
For $c, d\geq 0$, $\Big( (\mathcal{C}\ell_{ \{f,f_0, \dots, f_c\} }^\otimes)_{\leq c+d},\delta^\otimes \Big)$ and
$\Big( (\mathcal{C}\ell_{ \{f\} }^\otimes)_{\leq c+d},\delta^\otimes \Big)$ are homotopy equivalent.
\end{cor}

Although the underlying morphisms of complex are built just as before using quilted trajectories with the appropriate perturbation datum above and below the quilted components, this requires more explanations.

\begin{proof}
Let $\Big( (\mathcal{C}\ell_{ \{f\} }^\otimes)_{\leq c+d, \mathfrak{L}},\delta^\otimes_{\mathfrak{L}} \Big)$ be the complex built just as $\Big( (\mathcal{C}\ell_{ \{f,f_0, \dots, f_c\} }^\otimes)_{\leq c+d},\delta^\otimes \Big)$, but having its labeled generators coming from the function $f$ only. This means that its differential generally uses different perturbations for different labelings.

A morphism $H:\Big( (\mathcal{C}\ell_{ \{f,f_0, \dots, f_c\} }^\otimes)_{\leq c+d},\delta^\otimes \Big) \rightarrow \Big( (\mathcal{C}\ell_{ \{f\} }^\otimes)_{\leq c+d, \mathfrak{L}},\delta^\otimes_{\mathfrak{L}} \Big)$ is defined by taking the perturbation data associated with $\delta^\otimes$ above the quilted components and the perturbation data of $\delta^\otimes_{\mathfrak{L}}$ under them as in definition \ref{qpht}. Notice that a trajectory from $\mathcal{C}\ell_{ \{f,f_0, \dots, f_c\} }^\otimes$ to $\mathcal{C}\ell_{ \{f\} }^\otimes$ induces a labeling on the root collection of critical points of $f$.

To define a homotopy inverse morphism \\ $H': \Big( (\mathcal{C}\ell_{ \{f\} }^\otimes)_{\leq c+d, \mathfrak{L}},\delta^\otimes_{\mathfrak{L}} \Big) \rightarrow \Big( (\mathcal{C}\ell_{ \{f,f_0, \dots, f_c\} }^\otimes)_{\leq c+d},\delta^\otimes \Big)$ we have to use trajectories with a labeling on the leaves even though they are mapped to critical points of the same function. $H'$ is then defined by taking the perturbation data of $\delta^\otimes_{\mathfrak{L}}$ above the quilted components and the perturbation data of $\delta^\otimes$ below.

Then, $H' \circ H$ and $H \circ H'$ is seen to be a homotopy equivalence. Indeed, proposition \ref{h0h1_homo} ensures that $H' \circ H$ is homotopic to the chain map built from the concatenated perturbation homotopies. We can choose this concatenation to be homotopically trivial, so $H' \circ H$ is homotopic to the chain morphism from $\Big( (\mathcal{C}\ell_{ \{f,f_0, \dots, f_c\} }^\otimes)_{\leq c+d},\delta^\otimes \Big)$ to itself built with the trivial perturbation homotopy. As it is usual in Morse homology, the latter must be the identity, simply because it counts the same trajectories as $\delta^\otimes$, but with index one lower so the only possibility left is the constant trajectories having lines as sources.

It remains to see that $\Big( (\mathcal{C}\ell_{ \{f\} }^\otimes)_{\leq c+d, \mathfrak{L}},\delta^\otimes_{\mathfrak{L}} \Big)$ and $\Big( (\mathcal{C}\ell_{ \{f\} }^\otimes)_{\leq c+d},\delta^\otimes \Big)$ are equivalent. Notice that when defining $\Big( (\mathcal{C}\ell_{ \{f\} }^\otimes)_{\leq c+d, \mathfrak{L}},\delta^\otimes_{\mathfrak{L}} \Big)$, one could have only used the trivial label perturbation data of $\delta^\otimes$. Therefore, $\Big( (\mathcal{C}\ell_{ \{f\} }^\otimes)_{\leq c+d, \mathfrak{L}},\delta^\otimes_{\mathfrak{L}} \Big)$ is equivalent to a complex that is invariant under changes of labelings, but the latter reduces to $\Big( (\mathcal{C}\ell_{ \{f\} }^\otimes)_{\leq c+d},\delta^\otimes \Big)$ when forgetting the labelings.
\end{proof}

\section{Categorical formulation}

The results of propositions \ref{coch_comp}, \ref{coch_map}, \ref{coch_homo} and \ref{h0h1_homo} could be stated in a more compact form. Let $\mathcal{L}^{mono, \pm}(M,\omega)$ be the category whose objects are tuples $(L,J,P,f,g)$ where:
\begin{itemize}
\item $L$ is a closed embedded monotone $\gl$ lagrangian submanifolds of $(M,\omega)$,
\item $J$ is a $\omega$-tamed smooth almost complex structure,
\item $(f,g)$ is a smooth Morse-Smale pair,
\item $P$ is a coherent system of perturbations of $J$ and $f$,
\end{itemize}

with $Hom_{\mathcal{L}^{mono, \pm}(M,\omega)}((L^{(0)},J^{(0)},P^{(0)},f^{(0)},g^{(0)}),(L^{(1)},J^{(1)},P^{(1)},f^{(1)},g^{(1)}))$ being made of pairs $(\phi^t_H,\mathcal{Q}P^H)$ where:
\begin{itemize}
\item $\phi^t_H$ is a hamiltonian isotopy such that $\phi^1_H(L^{(0)})=L^{(1)}$,
\item $\mathcal{Q}P^H$ is a coherent perturbation homotopy from $P^{(0)}$ to $P^{(1)}$ along a smooth homotopy from $(J^{(0)},f^{(0)},g^{(0)})$ to $(J^{(1)},f^{(1)},g^{(1)})$.
\end{itemize}

\begin{defi}
Let $h\mathcal{L}^{mono, \pm}(M,\omega)$ be the quotient category obtained from $\mathcal{L}^{mono, \pm}(M,\omega)$ by identifying homotopic morphisms.
\end{defi}

Now write $\Lambda$-mod for the (abelian) category of $\Lambda$-modules, $K(\Lambda \text{-mod})$ for the category of cochain complexes over $\Lambda$-mod and $hK(\Lambda \text{-mod})$ for the (triangulated) homotopy category of the latter.


\begin{thm}
\begin{diagram}
h\mathcal{L}^{mono, \pm}(M,\omega) & && \rTo^{\mathcal{C}\ell^\otimes} && & hK(\Lambda \text{-mod}) \\
(L,J,P,f,g) & && \rMapsto  && & (\mathcal{C}\ell_{\{ f\} }^\otimes(M,\omega,L,J,P,f,g),\delta^\otimes(M,\omega,L,J,P,f,g))
\end{diagram}
is a contravariant functor.
\end{thm}

\chapter{Regularity for spaces of Floer trajectories} \label{tran_floe}

This section is intended to establish the regularity results needed in theorem \ref{coch_comp} and propositions \ref{coch_map}, \ref{coch_homo} and \ref{h0h1_homo}. This is performed by adapting the reduction to simple trajectories procedures of \cite{BC} making use of the disk decomposition results of \cite{L} and the monotonicity hypothesis.

\section{Decomposition results}

We first set some conventions. Throughout this section, a generic subset of a topological space will mean a second category subset in the sense of Baire, that is, a countable intersection of open dense subsets. Let, as in section \ref{subs_comp}, $\mathcal{J} = \mathcal{J}_J = exp_{J}( B_{J} \subset C^\epsilon(M,T_{J}\mathbf{J}(TM,\omega)))$ be a Banach chart of smooth $\omega$-tamed almost complex structures on $(M,\omega)$ being $C^\epsilon$ relative to $J$ (see \cite{F2}) and $\mathcal{M}_f = exp_f( B_f  \subset C^\epsilon(L,\R))$ be a Banach chart of smooth Morse-Smale functions on $(L,g)$ being $C^\epsilon$ relative to $f$.

\begin{defi} \label{simp}
A Floer trajectory configuration $u:(C,\partial C) \rightarrow (M,L)$ is called simple if the following conditions are satisfied:
\begin{itemize}
\item $u|_D$ is simple for every $D \in \text{Disks}(C)$, that is, for every disk $D$ of $C$ there is an open and dense subset $S\subset D$ such that $du|_D(s) \neq 0$ and $u|_D^{-1}(u|_D(s))=s$ for all $s\in S$,
\item the maps $\{u|_D \}_{D \in \text{Disks}(C)}$ are absolutely distinct, that is, there is no disk $D$ of $C$ such that $u(D)\subset \underset{D' \in \text{Disks}(C) \backslash \{D\} } {\bigcup} u(D')$,
\item the maps $\{u|_L \}_{L \in \text{Lines}(C)}$ are absolutely distinct, that is, there is no line segment $L$ of $C$ such that $u(L)\subset \underset{L' \in \text{Lines}(C) \backslash \{L\} } {\bigcup} u(L')$.
\end{itemize}
\end{defi}

By standard arguments (see \cite{MS}), for generic $p \in \mathcal{J} \times \mathcal{M} = \mathcal{J}_J \times \mathcal{M}_f \times \mathcal{M}_{f_1-f_0} \times \dots \mathcal{M}_{f_c-f_{c-1}}$ and associated coherent perturbation data $P_p$, we get over every such simple $u$ an exact sequence of operators:

\begin{diagram}
        0 & \rTo & L^{m,p}_0 (C, TC, T \partial C) & \rTo &&& L^{m,p}(C,u^*TM,u^*TL) &&& \rTo & \frac{L^{m,p}(C,u^*TM,u^*TL)}{L^{m,p}_0 (C, TC, T \partial C)} & \rTo & 0 \\
        & & \dTo^{\ddel_C} &&&& \dTo^{\ddel_u} &&&& \dTo^{\ddel_u / \ddel_C} & & \\
        0 & \rTo & L^{m-1,p}_\pi (C,TC) && \rTo && L^{m-1,p}(C,u^*TM) && \rTo && \frac{L^{m-1,p}(C,u^*TM)}{L^{m-1,p}_\pi (C,TC)} & \rTo & 0 \\
\end{diagram}

where $\ddel_u / \ddel_C$ is surjective and therefore $ker(\ddel_u / \ddel_C)$ is of dimension $Ind(\ddel_u) - Ind(\ddel_C)$. The latter can be seen as being isomorphic to $T_u \mathfrak{M}$ where $\mathfrak{M}$ is the moduli of (simple) Floer configurations near $u$ with source in the same open stratum as $C$. Combining this with lemma \ref{rest_bair}, the spaces of simple Floer configurations are smooth for generic pairs $(p,P_p)$.

Our goal is now to show that, in the monotone setting, over a generic subset 
\begin{align*}
& \mathcal{P}\text{ert}^{gen} \subset \mathcal{P}\text{ert} \equiv \{ (p,P_p) | p\in \mathcal{J}_J \times \mathcal{M}_f \times \mathcal{M}_{f_1-f_0} \times \dots \mathcal{M}_{f_c-f_{c-1}}, \\
& P_p \hspace{0.1cm} \text{monotone coherent system of perturbations of} \hspace{0.1cm} p \hspace{0.1cm} \text{vanishing on} \hspace{0.1cm} \mathcal{V} \},
\end{align*}
$\mathcal{V}$ being a coherent system of ends, the considered Floer trajectories are simple:

\begin{prp} \label{gene_tran}
Let $L\subset (M,\omega)$ be a monotone lagrangian submanifold with $N_L \geq 2$. There exists a generic subset $\mathcal{P}\text{ert}^{gen} \subset \mathcal{P}\text{ert}$ such that for every $(p,P_p) \in \mathcal{P}\text{ert}^{gen}$, every Floer configuration $u:(C,\partial C) \rightarrow (M,L)$ satisfying $P_p$ and such that $Ind(\ddel_u) \leq -(\ell -2) +1$ is simple.
\end{prp}

Then, we will get the following regularity result:

\begin{cor}
The spaces of Floer trajectory configurations having their source in the same open stratum as $C \in \cllkr$, satisfying $(p,P_p) \in \mathcal{P}\text{ert}^{gen}$ and $Ind(\ddel_u) \leq -(\ell-2) +1$ are smooth manifolds of dimension $Ind(\ddel_u) - Ind(\ddel_C)$.
\end{cor}

The proof of proposition \ref{gene_tran} will be based, as in \cite{BC}, on the following fundamental result of Lazzarini (\cite{L}):

\begin{thm} \label{lazz_deco}
Let $J \in \mathcal{J}(M, \omega)$ be any smooth $\omega$-tamed almost complex structure on $(M,\omega)$ and $v:(D_{\C},\partial D_{\C}) \rightarrow (M,L)$ be a nonconstant $J$-holomorphic disk. Then there exists a graph $\mathcal{G}(v) \subset D_{\C}$ with $\partial D_{\C} \subset \mathcal{G}(v)$ such that for every connected component $\mathcal{D} \subset D_{\C} \backslash \mathcal{G}(u)$, there is a surjective holomorphic map $\pi_{\overline{\mathcal{D}}}: (\overline{\mathcal{D}}, \partial \overline{\mathcal{D}}) \rightarrow (D_{\C},\partial D_{\C})$ and a simple $J$-holomorphic disk $v_{\mathcal{D}}:(D_{\C},\partial D_{\C}) \rightarrow (M,L)$ such that $v\mid_{\overline{\mathcal{D}}} = v_{\mathcal{D}} \circ \pi_{\overline{\mathcal{D}}}$.

If $m_\mathcal{D} \in \N$ stands for the degree of $\pi_{\overline{\mathcal{D}}}$, then, in $H_2(M,L)$, we have
\[
[v] = \underset{\mathcal{D}}{\sum} m_\mathcal{D} [v_{\mathcal{D}}].
\]
\end{thm}

To improve the readability of the general argument, we first prove proposition \ref{gene_tran} for $n=dim(L) \geq 3$. In that case, theorem \ref{lazz_deco} has a stonger statement:

\begin{thm} \label{lazz_high}
If $n\geq 3$, then for a generic choice of $J\in \mathcal{J}(M, \omega)$, any nonconstant $J$-holomorphic disk $v:(D_{\C},\partial D_{\C}) \rightarrow (M,L)$ is multicovered, in the sense that there exists a simple $J$-holomorphic disk $r(v):(D_{\C},\partial D_{\C}) \rightarrow (M,L)$ and a surjective holomorphic map $\pi: (D_{\C},\partial D_{\C}) \rightarrow (D_{\C},\partial D_{\C})$ such that $v= r(v) \circ \pi$ and $\pi^{-1}(\partial D_{\C}) = \partial D_{\C}$. Therefore, in $H_2(M,L)$, we have $[v] = m [r(v)]$ where $m\geq 1$ is the degree of $\pi$.
\end{thm}

Also, in that instance, we have a nonoverlapping result for absolutely distinct simple discs (see lemma 3.2.2 of \cite{BC}):

\begin{lem} \label{non_over}
If $n\geq 3$, then for a generic choice of $J\in \mathcal{J}(M, \omega)$, any two simple $J$-holomorphic disk $v_i:(D_{\C},\partial D_{\C}) \rightarrow (M,L)$, $i\in \{1,2\}$, with $v_1(D_{\C}) \bigcap v_2(D_{\C})$ being an infinite set are such that either $v_1(D_{\C}) \subset v_2(D_{\C})$ and $v_1(\partial D_{\C}) \subset v_2(\partial D_{\C})$ or $v_2(D_{\C}) \subset v_1(D_{\C})$ and $v_2(\partial D_{\C}) \subset v_1(\partial D_{\C})$.
\end{lem}

\section{Case $n\geq 3$}

\begin{proof}[Proof of proposition \ref{gene_tran} for $n\geq 3$]

First, it will be convenient to define combinatorial operations, called reductions, on clusters that are built from elementary cut and paste operations called elementary reductions. These are motivated by the reductions of the Floer trajectories allowed by the above structural results for $J$-holomorphic disks.

\begin{defi} \label{elem_redu}
An elementary reduction of $C \in \cllkr$ is defined as the result of one of the following cut and paste manipulations:

\begin{enumerate}[I)]

\item Let $D \in \text{Disks}(C)$, $k(D)$ be the number of interior markings in $D$ and $d \in \N$ such that $d| k(D)$. For any holomorphic map $\pi:(D, \partial D) \rightarrow (D_{\C}, \partial D_{\C})$ of degree $d$, one can define an object $r_{I}(C)$ by removing $D\backslash \{ x_j(D) \}_{0\leq j \leq \ell(D)}$ from $C$, attaching $x_j(D)$ to $\pi(x_j(D)) \in \partial D_{\C}$ for $0\leq j \leq \ell(D)$ and marking the points $\{ \pi(z_h(D)) \}_{1 \leq h \leq \frac{k(D)}{d}}$ (see figure \ref{fig21}).

\begin{figure}[h] 
        \centering 
	\includegraphics[width=120mm]{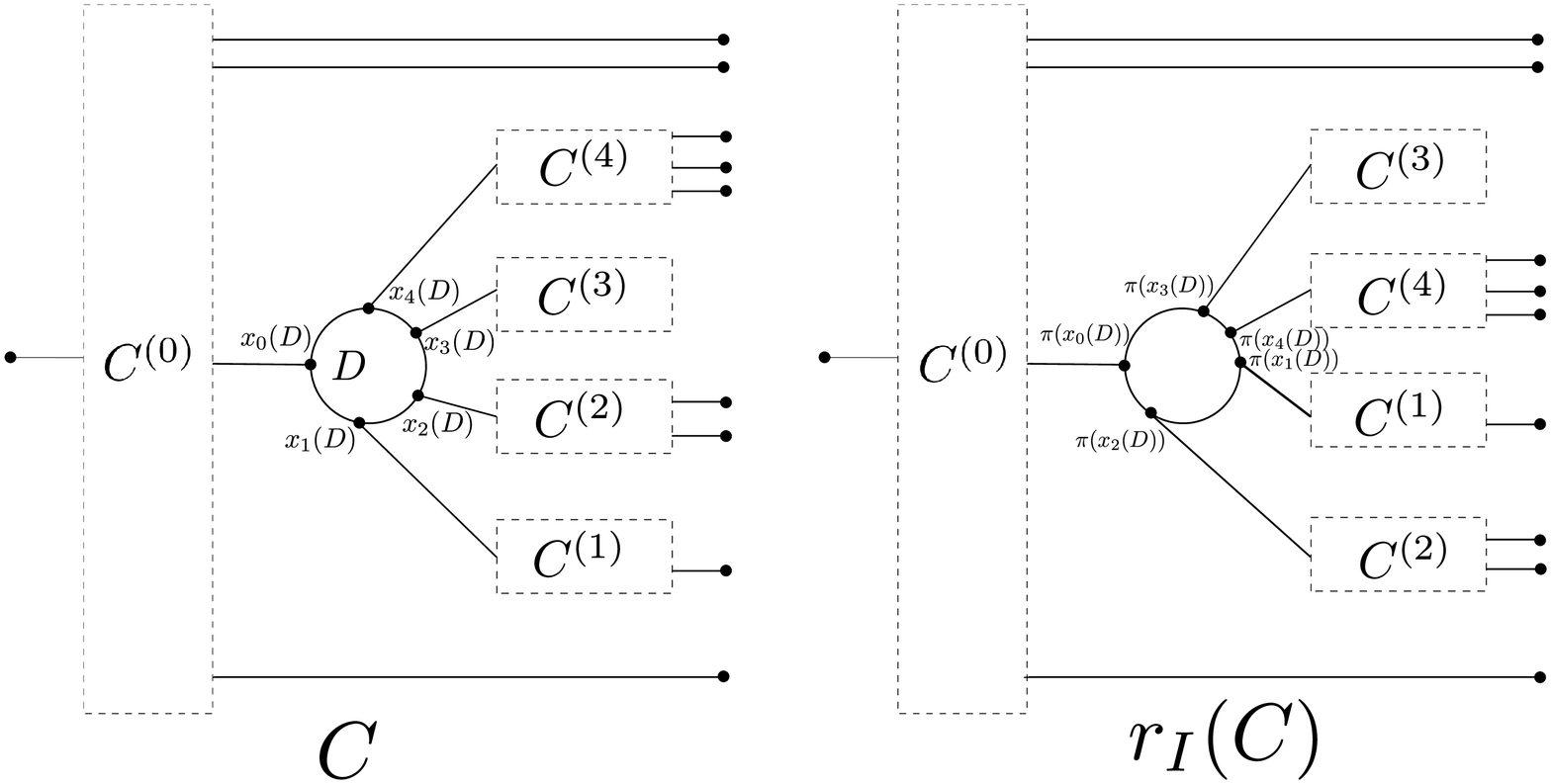}
        \caption{A type $I$ elementary reduction of $C$} \label{fig21}
\end{figure}

\item \begin{enumerate}[a)]
\item Let $D_1, D_2 \in \text{Disks}(C)$ with $D_2$ lying above $D_1$ and $\iota: \{ x_j(D_2) \}_{1\leq j \leq \ell(D_2)} \rightarrow \partial D_1$. Then one can consider the object $r_{IIa}(C)$ defined by removing $D_2 \backslash \{ x_j(D_2) \}_{0\leq j \leq \ell(D_2)}$ and the line touching $x_0(D_2)$ from $C$, and then attaching $x_j(D_2)$ to $\iota(x_j(D_2))$, $1\leq j \leq \ell(D_2)$ (see figure \ref{fig22}).

\begin{figure}[h] 
        \centering 
	\includegraphics[width=120mm]{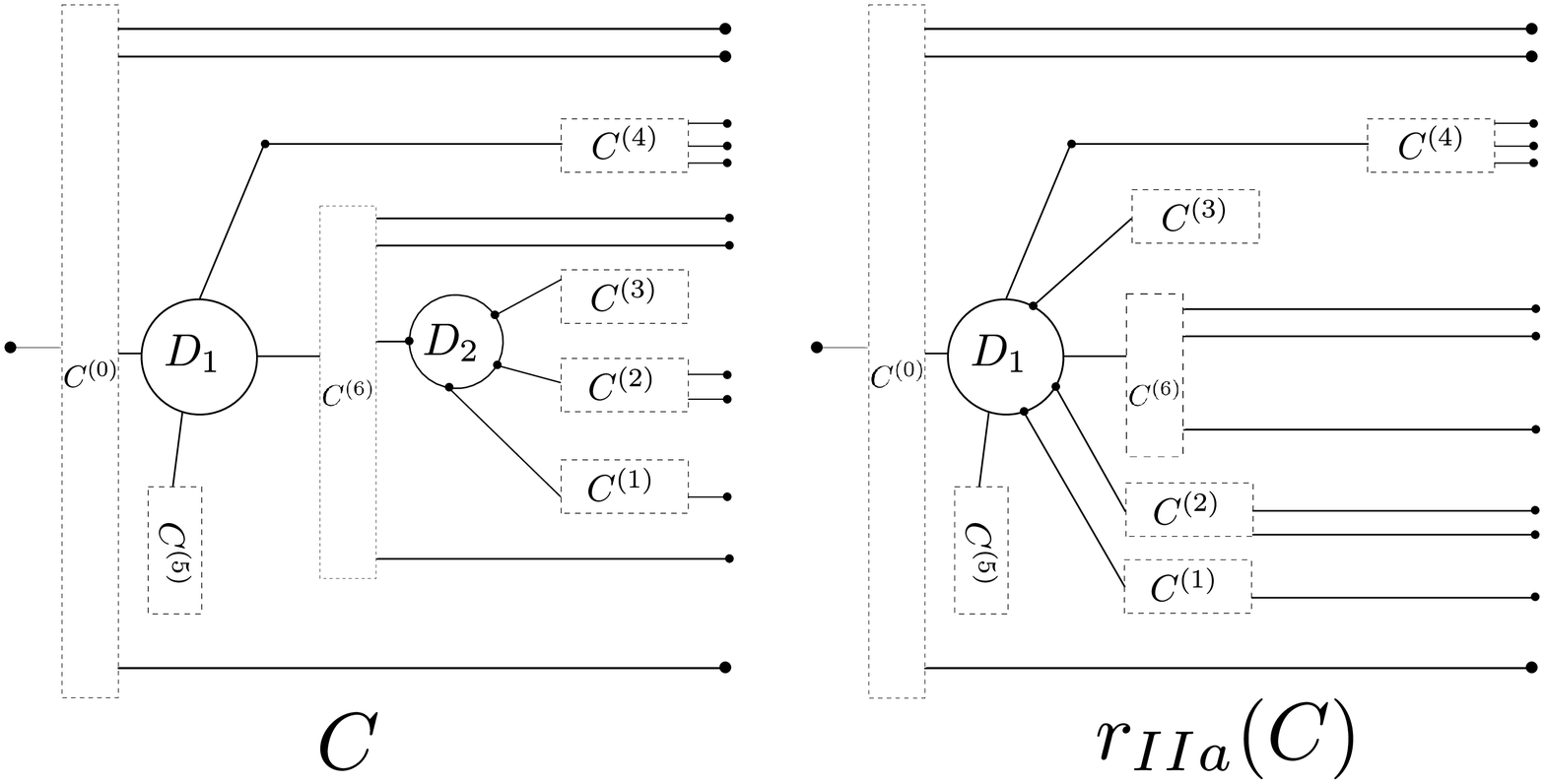}
        \caption{A type $IIa$ elementary reduction of $C$} \label{fig22}
	\centering 
	\includegraphics[width=120mm]{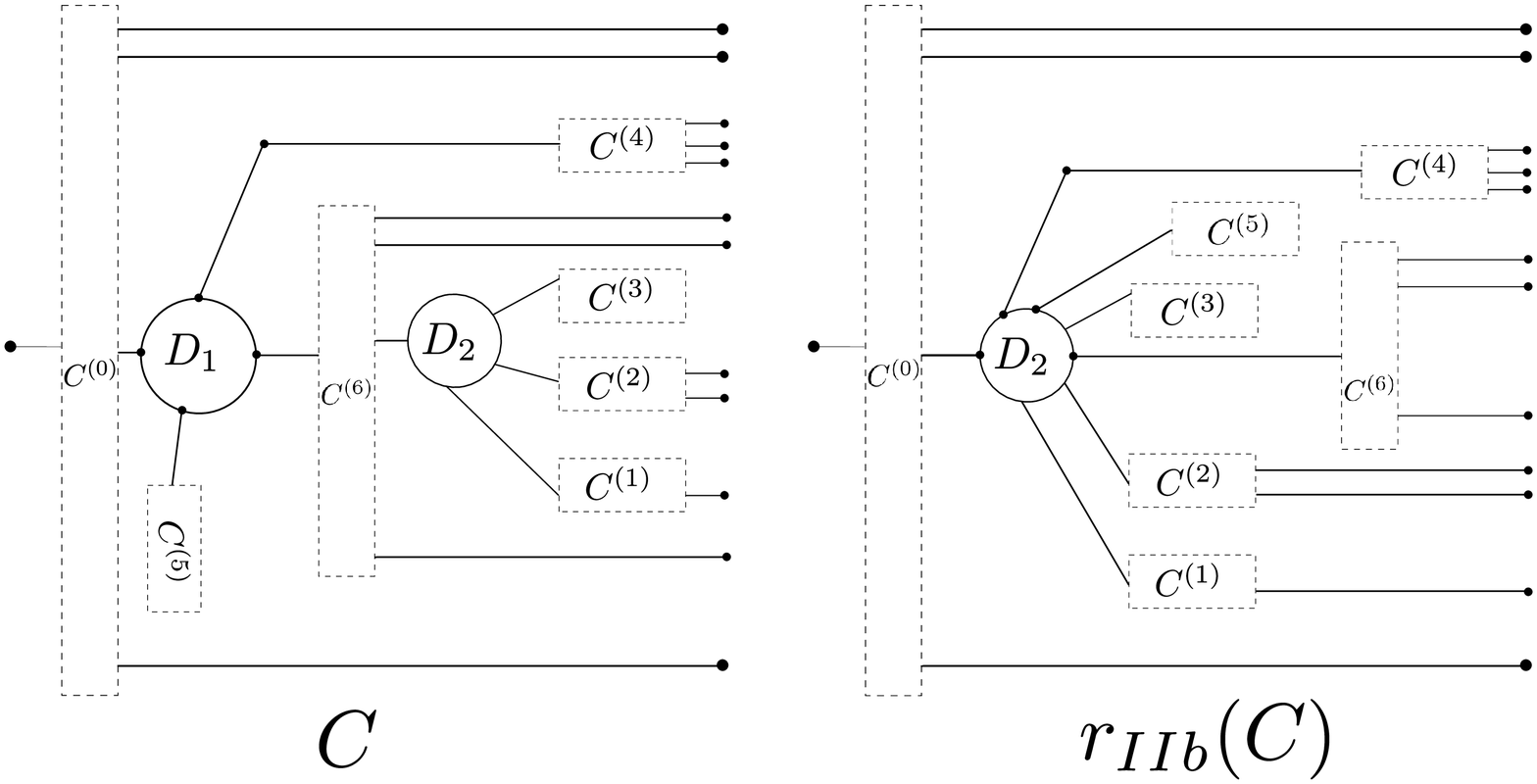}
        \caption{A type $IIb$ elementary reduction of $C$} \label{fig23}
\end{figure}

\item Let $D_1, D_2 \in \text{Disks}(C)$ with $D_1$ lying above $D_2$ and $\iota: \{ x_j(D_2) \}_{0\leq j \leq \ell(D_2)} \rightarrow \partial D_1$. Then one can consider the object $r_{IIb}(C)$ defined by removing $D_2 \backslash \{ x_j(D_2) \}_{0\leq j \leq \ell(D_2)}$ and the line touching $x_0(D_1)$ from $C$, and then attaching $x_j(D_2)$ to $\iota(x_j(D_2))$, $0\leq j \leq \ell(D_2)$ (see figure \ref{fig23}).

\end{enumerate}

\item For every leaf $v_j$ of $C$, $1\leq j \leq \ell$, let $C_j$ be the cluster made of the lines and the marked disks through which every continuous path from $v_0$ to $v_j$ must go. Then one can consider the cluster $r_{III}(C) = \underset{1 \leq j \leq \ell}{\bigcup} C_j$ (see figure \ref{fig24}).

\begin{figure}[h] 
        \centering 
	\includegraphics[width=120mm]{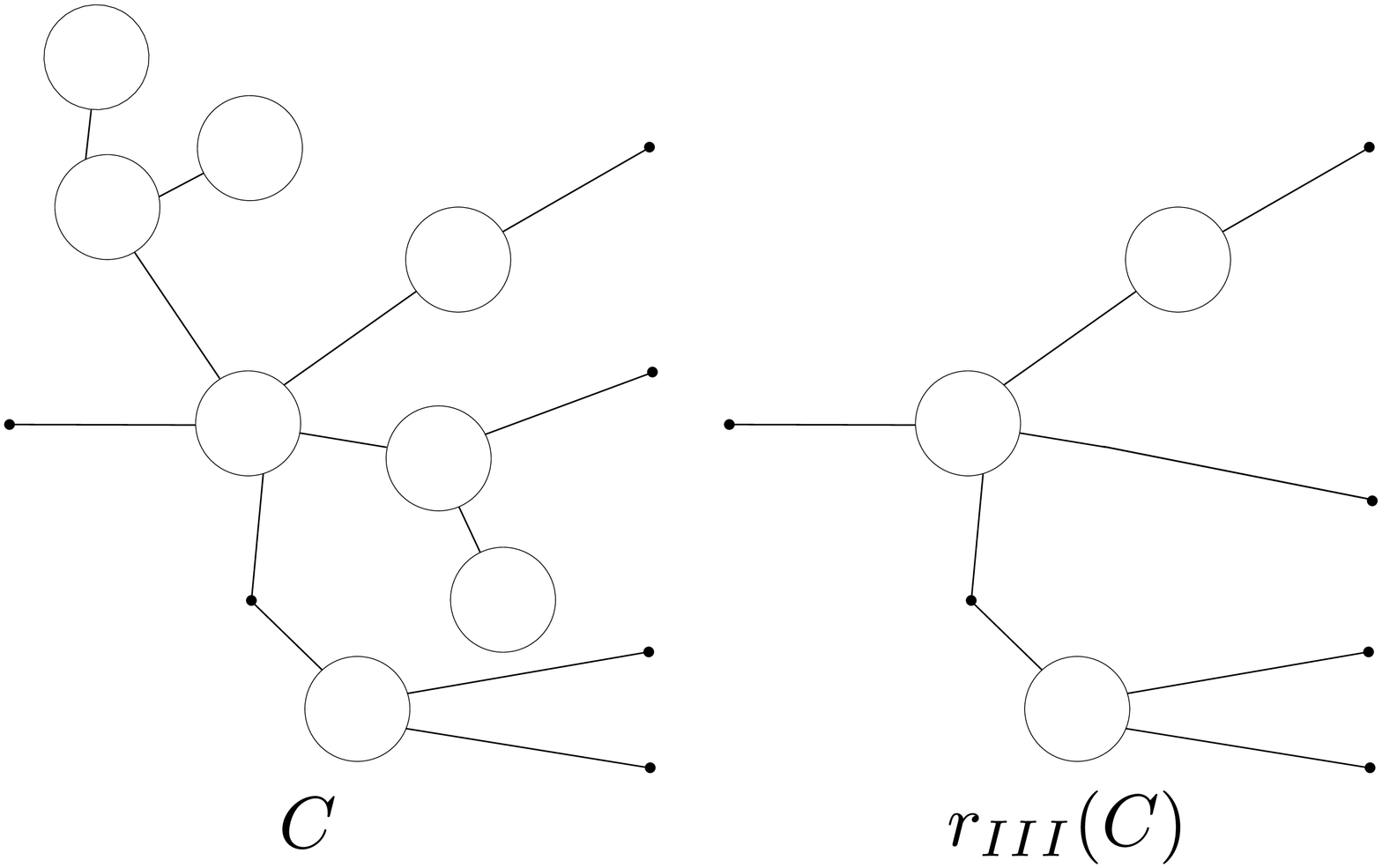}
        \caption{A type $III$ elementary reduction of $C$} \label{fig24}
\end{figure}

\end{enumerate}
A composition of a finite number of elementary reductions is called a reduction.
\end{defi}

Note that the objects resulting from reductions are not exactly $\otimes$-clusters as the incidence points of some lines might coincide, and moreover if that happens, the ribbon structure is not well defined. However, the choice of a perturbation $P|_C$ over $C \in \cllkr$ naturally defines one over any of its reductions $r(C)$ that we will denote by $r_*(P|_C)$, and a tangent operator $D\overline{\partial}_{r(C)}$ is still defined.

We fix a coherent system of ends $\mathcal{V}$ and proceed to the choice of generic elements of $\mathcal{P}\text{ert} \equiv \{ (p,P_p) | p\in \mathcal{J}_J \times \mathcal{M}_f \times \mathcal{M}_{f_1-f_0} \times \dots \mathcal{M}_{f_c-f_{c-1}}, \\ P_p \hspace{0.1cm} \text{monotone coherent system of perturbations of} \hspace{0.1cm} p \hspace{0.1cm} \text{vanishing on} \hspace{0.1cm} \mathcal{V} \}$.

Let us first restrict ourselves to the subset $\mathcal{P}\text{ert}^{red} \subset \mathcal{P}\text{ert}$ defined by restricting the $\mathcal{J}_J$ factor of $p$ to the generic subsets coming from theorem \ref{lazz_high} and lemma \ref{non_over}.

First consider the usual universal moduli space of Floer trajectory configurations $\mathcal{U}\text{niv} = \{ (u,C,P_p) | u:(C,\partial C) \rightarrow (M,L) \hspace{0.1cm} \text{Floer with respect to} \hspace{0.1cm} P_p, \\ Ind(\ddel_u) \leq -(\ell -2) +1\} \rightarrow \mathcal{P}\text{ert}^{red}$. Not every trajectory $(u,C,P_p)$ of $\mathcal{U}\text{niv}$ needs to be simple, but since we restricted to $\mathcal{P}\text{ert}^{red}$, any trajectory can be reduced to a simple trajectory having a reduced cluster as its source. This can be achieved as follows:
\begin{enumerate}[1)]
\item Reduce iteratively every nonsimple disk using theorem \ref{lazz_high} together with the induced type $I$ reductions over the source.

\item Eliminate iteratively every disk $D$ of the resulting trajectory that is sent into the image of the other disks using lemma \ref{non_over} together with the induced type $II$ reductions over the source. After each step, one might additionally perform a type $III$ reduction.

\item Perform a type $III$ reduction over the source.

\item 
On the linear parts of the trajectory where the lines satisfy the gradient flow of a single Morse function, then nondistinction of the lines should also be eliminated. These trajectories are reduced as in \cite{BC}, bypassing the disks lying between nondistinct lines.
\end{enumerate}

Denote the resulting reduced trajectory configuration by $(r(u),r(C),r_*(P|_C))$ and notice that it is simple in the sense of the natural adaptation of definition \ref{simp}. Also note that a nonsimple trajectory configuration of $\mathcal{U}\text{niv}$ might be reducible in many different ways.

We will then consider $\mathcal{U}\text{niv}^{red} = \{ (r(u),r(C),r_*(P|_C)) | r(u):(r(C),\partial r(C)) \rightarrow (M,L) \hspace{0.1cm} \text{reduced Floer} \\ \text{with respect to} \hspace{0.1cm} r_*(P|_C), Ind(\ddel_{r(u)}) \leq -(\ell-2) +1 -2\} \rightarrow \mathcal{P}\text{ert}^{red}$, the space of every possible reduced, and therefore simple, trajectory configurations.

As with $\mathcal{U}\text{niv}$, one can use standard perturbation arguments to show that for a generic $\mathcal{P}\text{ert}^{gen} \subset \mathcal{P}\text{ert}^{red}$, then for every simple $(r(u),r(C),r_*(P|_C)) \in \mathcal{U}\text{niv}^{red}$ above $\mathcal{P}\text{ert}^{gen}$, we get an exact sequence of operators

\begin{diagram} \label{redu_sequ}
         L^{m,p}_0 (r(C), Tr(C), T \partial r(C)) & \rTo &&& L^{m,p}(r(C),r(u)^*TM,r(u)^*TL) &&& \rTo && \frac{L^{m,p}(r(C),r(u)^*TM,r(u)^*TL)}{L^{m,p}_0 (r(C), Tr(C), T \partial r(C))} \\
         \dTo^{\ddel_{r(C)}} &&&& \dTo^{\ddel_{r(u)}} &&&&& \dTo^{\ddel_{r(u)} / \ddel_{r(C)}} \\
         L^{m-1,p}_\pi (r(C),Tr(C)) && \rTo && L^{m-1,p}(r(C),r(u)^*TM) && \rTo &&& \frac{L^{m-1,p}(r(C),r(u)^*TM)}{L^{m-1,p}_\pi (r(C),Tr(C))} \\
\end{diagram}

where $\ddel_{r(u)} / \ddel_{r(C)}$ is surjective and therefore $ker(\ddel_{r(u)} / \ddel_{r(C)})$ is of dimension $Ind(\ddel_{r(u)}) - Ind(\ddel_{r(C)})$.

We argue that every trajectory configuration $(u,C,P_p) \in \mathcal{U}\text{niv}$ above $\mathcal{P}\text{ert}^{gen}$ must be simple so we get the desired result. Indeed, if $(u,C,P_p)$ is not simple, then we can apply a reduction $r$ on it and get the reduced exact sequence \ref{redu_sequ}. Since $Ind(\ddel_{u}) \leq -(\ell -2) +1$, that $Ind(\ddel_{r(C)}) \geq Ind(\ddel_{C})+2(k - k(r(C))) = -(\ell -2 + 2k(r(C)))$ and that the monotonicity hypothesis implies that $Ind(\ddel_{r(u)}) \leq Ind(\ddel_{u}) -2 $, we must have

\begin{align*}
Ind(\ddel_{r(u)} / \ddel_{r(C)}) &= Ind(\ddel_{r(u)}) - Ind(\ddel_{r(C)}) \\
 &\leq -(\ell -2) +1 -2 +(\ell-2 + 2k(r(C))) \\
 &= 2k(r(C)) -1
\end{align*}

Remember that $\ddel_{r(u)} / \ddel_{r(C)}$ must have a kernel of dimension at least $2k(r(C))$ due to the invariance of $J$ over the disks that allows the position of the interior markings over the reduced trajectory to be changed according to $2k(r(C))$ real independent parameters. Therefore, it must have a cokernel of dimension at least one, but that is impossible for a simple configuration lying above $\mathcal{P}\text{ert}^{gen}$.

\end{proof}

\section{Case $n\leq 2$}

\begin{proof}[Proof of proposition \ref{gene_tran} for $n\leq 2$]

We will use the same ideas as in the case $n\geq 3$, except that in the present case we must make use of the more general theorem \ref{lazz_deco}, leading to the use of more general reduced trajectories.

Therefore, we define a generalized reduction procedure over clusters.

\begin{defi}
A generalized elementary reduction of $C \in \cllkr$ is defined as the result of one of the following cut and paste manipulations:

\begin{enumerate}[I)]
\item Given $D \in \text{Disks}(C)$, $G \subset D$ a graph with $\partial D \subset G$ and a surjective holomorphic map $\pi_{\overline{\mathcal{D}}}: (\overline{\mathcal{D}},\partial \overline{\mathcal{D}}) \rightarrow (D_\mathcal{D},\partial D_\mathcal{D}) \cong (D_{\C}, \partial D_{\C})$ for every connected component $\mathcal{D}$ of $D\backslash G$. One can define an object $r^g_{I}(C)$ by substituting $D$ for a nodal disk $D' \cong \underset{\mathcal{D}}{\bigcup} D_{\mathcal{D}}$ where a node $n_{\mathcal{D}_1,\mathcal{D}_2} \in D_{\mathcal{D}_1}\bigcap D_{\mathcal{D}_2}$, $\mathcal{D}_1 \neq \mathcal{D}_2$, must be such that $\pi^{-1}_{\overline{\mathcal{D}_1}}(n_{\mathcal{D}_1,\mathcal{D}_2})= \pi^{-1}_{\overline{\mathcal{D}_2}}(n_{\mathcal{D}_1,\mathcal{D}_2})$ is an interior point of an arc of $\overline{\mathcal{D}_1} \bigcap \overline{\mathcal{D}_2}$ (see figure \ref{fig25}).

\begin{figure}[h] 
        \centering 
	\includegraphics[width=120mm]{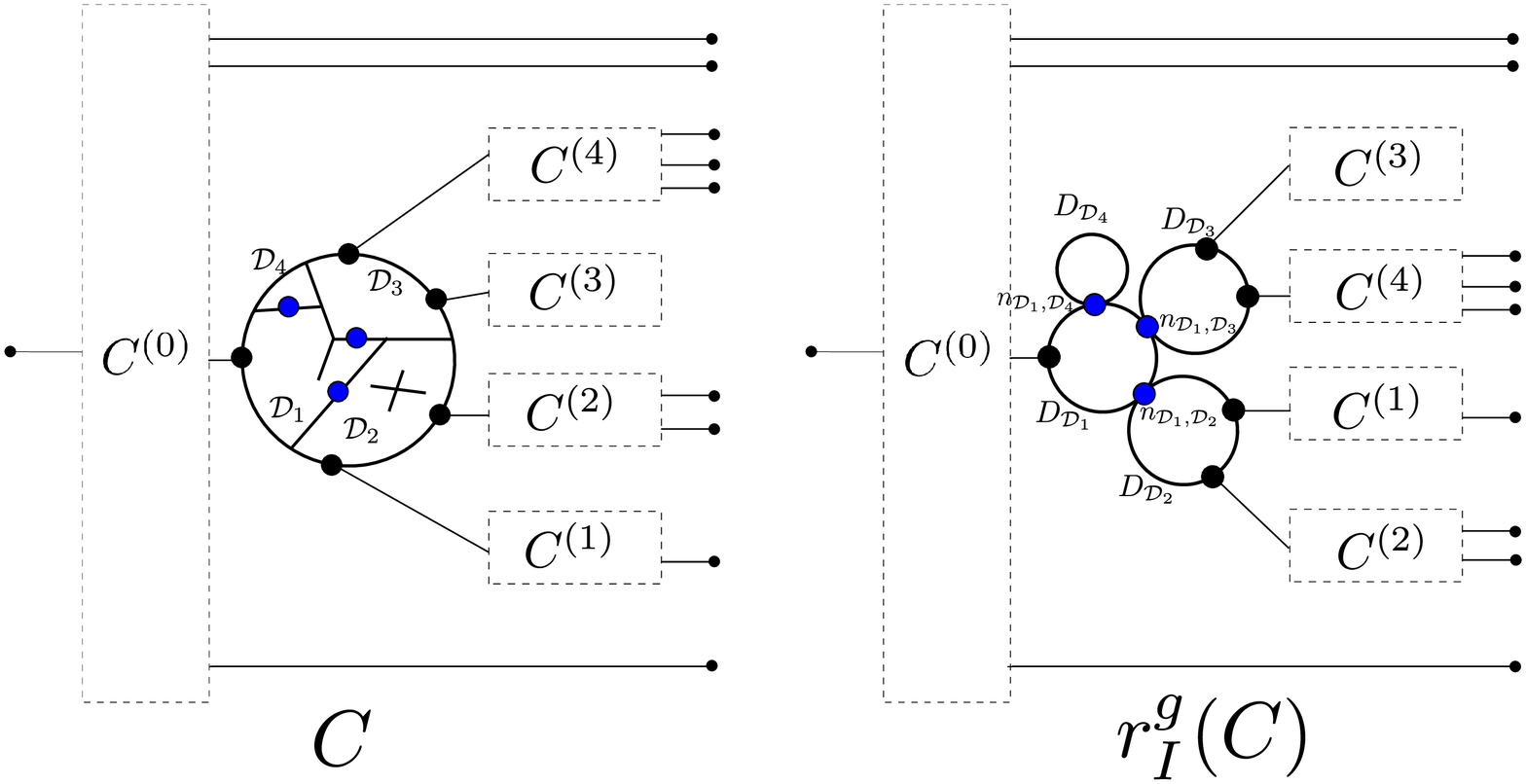}
        \caption{A type $I$ generalized elementary reduction of $C$} \label{fig25}
\end{figure}

\item Given $D \in \text{Disks}(C)$, a map $\iota: \{ x_0(D), \dots , x_{|D|}(D) \} \rightarrow \underset{ D' \in \text{Disks}(C) \backslash D}{\bigcup} D'$ and, for every $0 \leq j \leq |D|$, pairs of points $(\iota_0^{(j)+},\iota_0^{(j)-}), \dots, (\iota_{s_j}^{(j)+},\iota_{s_j}^{(j)-})$ of $\underset{ D' \in \text{Disks}(C) \backslash D}{\bigcup} D'$ such that $\iota_0^{(j)+} = \iota(x_0(D))$, $\iota_{s_j}^{(j)+} = \iota(x_j(D))$ and the points of any pair belong to the same disk, one can perform the following operations and call the resulting $r_{II}^g(C)$:

\begin{enumerate}[i)]
\item for every $j \in \{1, \dots, |D| \}$ such that $\iota(x_j(D))$ is not above $D$, detach the line touching $D$ at $x_j(D)$ and attach it back to $\iota(x_j(D))$ (see figure \ref{fig26}),

\item then, for every $j \in \{1, \dots, |D| \}$ such that $\iota(x_j(D))$ is above $D$, say $\iota(x_j(D)) \in D' \in \text{Disks}(C)$, detach the line touching $D$ at $x_j(D)$, attach it back to $\iota(x_j(D))$ and remove the line connecting to $x_0(D')$. The resulting space has two connected components that we reconnect using the following steps (see figure \ref{fig26}):

\begin{enumerate}[1)]
\item if $s_j = 0$, so that $\iota(x_0(D)) \in D'$, identify $x_0(D)$ and $\iota(x_0(D))$,

\item if $s_j \geq 1$, then \label{iter}

\begin{enumerate}[]
\item if $\iota_{s_j}^{(j)+}$ does not lie in the same connected component of the resulting space as $\iota_{s_j -1}^{(j)-}$, identify them,

\item if $\iota_{s_j}^{(j)+}$ lies in the same connected component of the resulting space as $\iota_{s_j -1}^{(j)-}$, but not in the same disk, say $\iota_{s_j -1}^{(j)-} \in D''$ above $D'$, identify $\iota_{s_j}^{(j)+}$ and $\iota_{s_j -1}^{(j)-}$ and remove the line attaching to $x_0(D'')$ (if any). Repeat \ref{iter} with the index lowered by one.
\end{enumerate}

\end{enumerate}

\end{enumerate}

\begin{figure}[h] 
        \centering 
	\includegraphics[width=120mm]{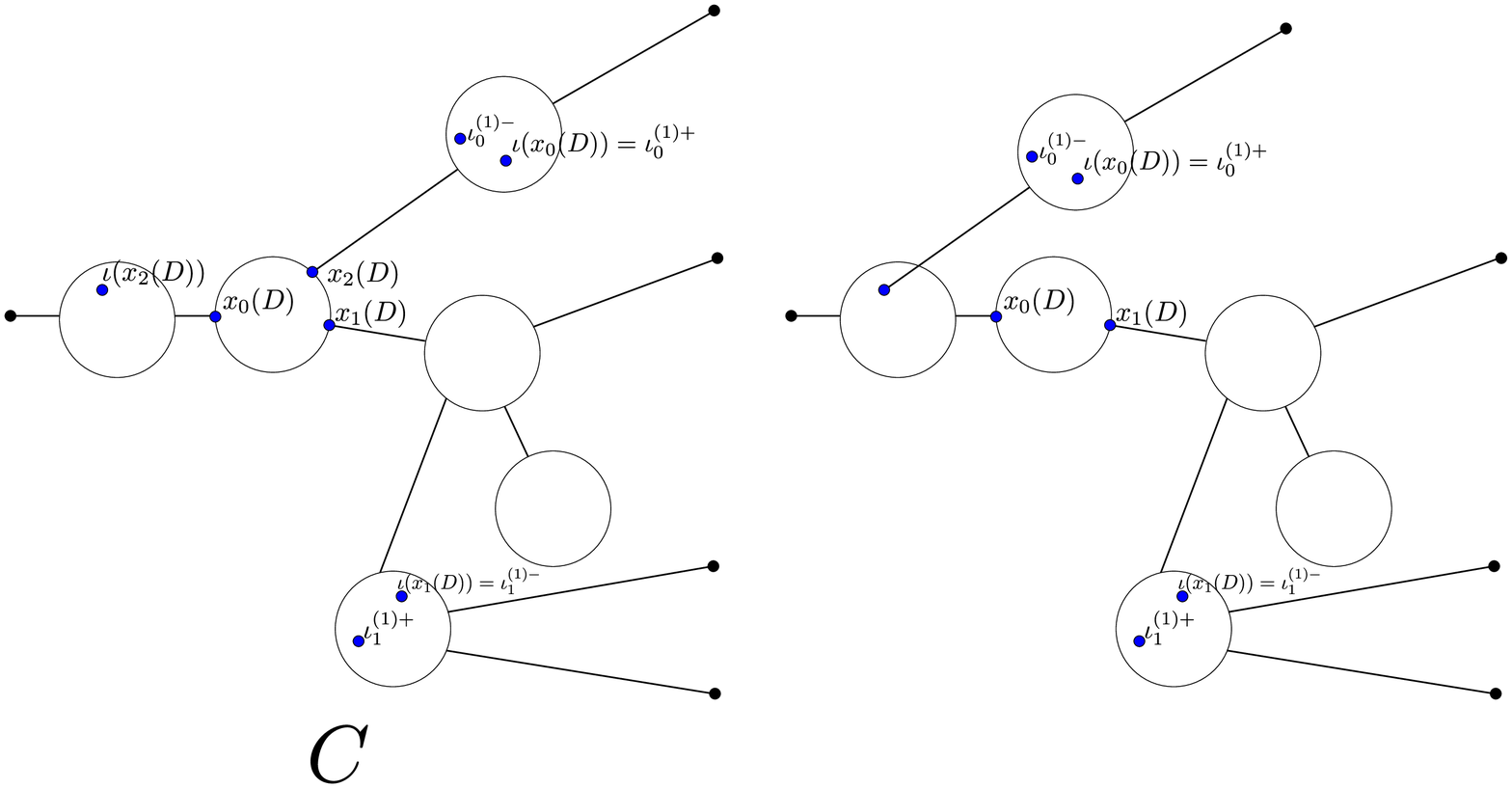}
	\includegraphics[width=120mm]{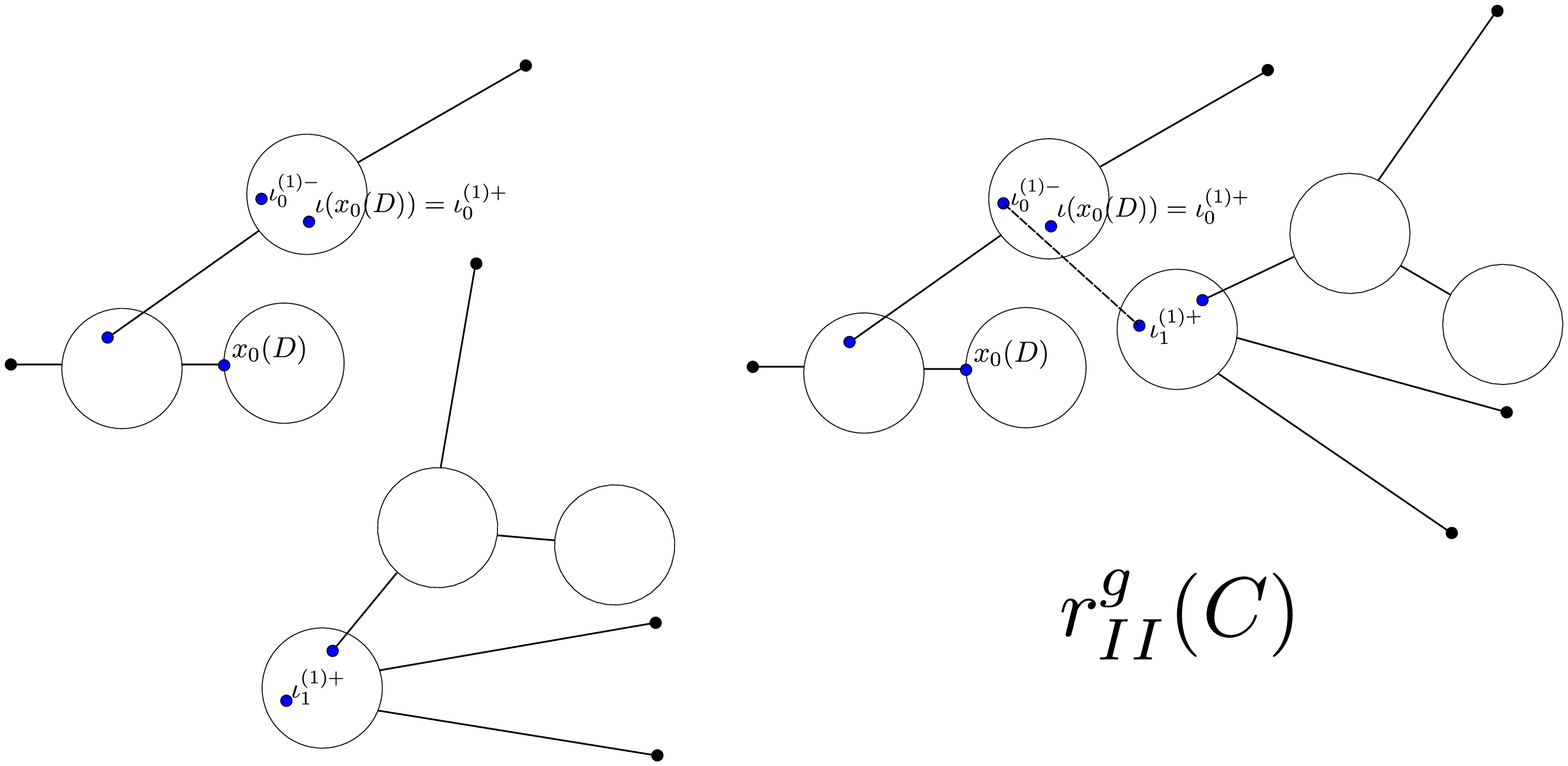}
        \caption{A type $II$ generalized elementary reduction of $C$} \label{fig26}
\end{figure}

\item One can consider the elementary reduction $r^g_{III}(C) = \underset{1 \leq j \leq \ell}{\bigcup} C_j$ as in definition \ref{elem_redu}.


\end{enumerate}
A composition of a finite number of elementary reductions is called a generalized reduction.
\end{defi}

Note that the objects resulting from generalized reductions are not exactly $\otimes$-clusters as the incidence points of some lines might be interior points of disks and many can coincide. However, we again only need a source operator of known index over a general reduction.

We again fix a coherent system of ends $\mathcal{V}$ and proceed to the choice of generic elements of $\mathcal{P}\text{ert} \equiv \{ (p,P_p) | p\in \mathcal{J}_J \times \mathcal{M}_f \times \mathcal{M}_{f_1-f_0} \times \dots \mathcal{M}_{f_c-f_{c-1}}, \\
P_p \hspace{0.1cm} \text{monotone coherent system of perturbations of} \hspace{0.1cm} p \hspace{0.1cm} \text{vanishing on} \hspace{0.1cm} \mathcal{V} \}$.


Consider again the usual universal moduli space of Floer trajectory configurations $\mathcal{U}\text{niv} = \{ (u,C,P_p) | u:(C,\partial C) \rightarrow (M,L) \hspace{0.1cm} \text{Floer with respect to} \hspace{0.1cm} P|_C, \\ Ind(\ddel_u) \leq -(\ell-2) +1\} \rightarrow \mathcal{P}\text{ert}$. Not every trajectory $(u,C,P_p)$ of $\mathcal{U}\text{niv}$ needs to be simple, but from theorem \ref{lazz_deco}, any trajectory can be reduced to a simple trajectory having a generalized reduced cluster as its source. This can be achieved as follows:
\begin{enumerate}[1)]
\item Decompose iteratively every nonsimple disk into a nodal disk with simple components using theorem \ref{lazz_deco} together with the induced type $I$ generalized reductions over the source.

\item Eliminate iteratively every disk $D$ of the resulting trajectory that is sent by $u$ into the image of the other disks using (if necessary) some paths $\gamma_j: [0,1] \rightarrow \underset{ D' \in \text{Disks}(C) \backslash D}{\bigcup} D'$, $1 \leq j \leq |D|$, such that 
\begin{itemize}
\item $u(\gamma_j(0))=u(x_0(D))$,
\item $u(\gamma_j(1))=u(x_j(D))$,
\item $\forall t\in [0,1]$, $\exists z\in \mathring{D}$ such that $u(\gamma_j(t))=u(z)$ and
\item $\gamma_j$ is discontinuous at a finite number of points $0 < t_0^{(j)} < \dots < t_{s_j -1}^{(j)} <1$
\end{itemize}
with the induced type $II$ reductions that use $\iota_{m}^{(j)-}= \underset{t \rightarrow t_m^{(j)-}}{\lim} \gamma_j(t)$, $0 \leq m \leq s_j -2$, and $\iota_{m}^{(j)+}= \underset{t \rightarrow t_{m-1}^{(j)+}}{\lim} \gamma_j(t)$, $1 \leq m \leq s_j -1$.

\item Perform a reduction of type $III$ over the source.
\end{enumerate}

Denote the resulting reduced trajectory configuration by $(r(u),r(C),r_*(P|_C))$ and notice that it is again simple in the sense of the generalization of definition \ref{simp}. Again, a nonsimple trajectory configuration of $\mathcal{U}\text{niv}$ might be generally reduced in countably many different ways, modulo the choice of the new nodal points and incident lines.

We will then consider $\mathcal{U}\text{niv}^{red} = \{ (r(u),r(C),r_*(P|_C)) | r(u):(r(C),\partial r(C)) \rightarrow (M,L) \hspace{0.1cm} \text{generalized reduced Floer} \\ \text{with respect to} \hspace{0.1cm} r_*(P_C), Ind(\ddel_{r(u)}) \leq -(\ell -2) +1 -2\} \rightarrow \mathcal{P}\text{ert}$, the space of every possible generalized reduced trajectory configurations.

As before, one can use standard perturbation arguments to show that for a generic $\mathcal{P}\text{ert}^{gen} \subset \mathcal{P}\text{ert}$, then for every simple $(r(u),r(C),r_*(P|_C)) \in \mathcal{U}\text{niv}^{red}$ above $\mathcal{P}\text{ert}^{gen}$, we get an exact sequence of operators

\begin{diagram} \label{redu_sequ_2}
         L^{m,p}_0 (r(C), Tr(C), T \partial r(C)) & \rTo &&& L^{m,p}(r(C),r(u)^*TM,r(u)^*TL) &&& \rTo && \frac{L^{m,p}(r(C),r(u)^*TM,r(u)^*TL)}{L^{m,p}_0 (r(C), Tr(C), T \partial r(C))} \\
         \dTo^{\ddel_{r(C)}} &&&& \dTo^{\ddel_{r(u)}} &&&&& \dTo^{\ddel_{r(u)} / \ddel_{r(C)}} \\
         L^{m-1,p}_\pi (r(C),Tr(C)) && \rTo && L^{m-1,p}(r(C),r(u)^*TM) && \rTo &&& \frac{L^{m-1,p}(r(C),r(u)^*TM)}{L^{m-1,p}_\pi (r(C),Tr(C))} \\
\end{diagram}

where $\ddel_{r(u)} / \ddel_{r(C)}$ is surjective and therefore $ker(\ddel_{r(u)} / \ddel_{r(C)})$ is of dimension $Ind(\ddel_{r(u)}) - Ind(\ddel_{r(C)})$.

We argue that every trajectory configuration $(u,C,P_p) \in \mathcal{U}\text{niv}$ above $\mathcal{P}\text{ert}^{gen}$ must be simple so we get the desired result. Indeed, if $(u,C,P_p)$ is not simple, then we can apply a reduction $r$ on it and get the reduced exact sequence \ref{redu_sequ_2}.

Now a straightforward index computation shows that \\  $Ind(\ddel_{r(C)}) \geq Ind(\ddel_{C}) +2(k - k(r(C))) - N $, where $N$ is the number of interior incidence points on $r(C)$ plus the number of complex nodes of $r(C)$. Indeed, by reconsidering the computation of lemma \ref{coker}, we see that replacing a boundary incident line segment by an interior incident line segment decreases by one the index over the source space. Also, a complex node can be seen as a connecting line of length zero being incident at interior points of two disks and therefore it decreases the index by one.

By the monotonicity hypothesis, removing $D$ decreases the index $\mu(F)$ of the boundary condition over the trajectory by at least two. Adapting proposition \ref{inde_comp} to the situation where the incidence points might be interior points results in substacting $n$ to the index for every interior incidence point and for every complex node of $r(C)$. Therefore, we get that $Ind(\ddel_{r(u)}) \leq Ind(\ddel_{u}) -2 -nN$.

Since $Ind(\ddel_{u}) \leq -(\ell -2) +1$, we must have

\begin{align*}
Ind(\ddel_{r(u)} / \ddel_{r(C)}) &= Ind(\ddel_{r(u)}) - Ind(\ddel_{r(C)})\\
 &\leq -(\ell -2) +1 -2 -nN +(\ell-2 + 2k(r(C))) +N\\
 &= 2k(r(C)) -1 -(n-1)N
\end{align*}

Note that $\ddel_{r(u)} / \ddel_{r(C)}$ must have a kernel of dimension at least $2k(r(C))$ due to the invariance of $J$ over the disks that allows the position of the interior markings over the reduced trajectory to be changed according to $2k(r(C))$ real independent parameters. Therefore, it must have a cokernel of dimension at least one, which is impossible for a simple configuration lying above $\mathcal{P}\text{ert}^{gen}$.

\end{proof}

\begin{rem}
Note that the above proof implies that the Floer trajectory configurations of index lower than $-(\ell -2)+1$ are generically type $III$ reduced, so we could have restricted to these clusters from the beginning.
\end{rem}

\chapter{Moduli of $\bullet$-clusters} \label{clks_sect}

We first describe the source spaces that will be used to define the cochain complex. They are planar trees of complex marked disks connected by metric lines and are chosen to form a moduli where the planar structure can vary. These sources can be seen as generalizations of the sources used in \cite{Fu}, \cite{Fu2}, \cite{Oh} and \cite{BC}.

\section{Constructing $\cllks$, the moduli of $\bullet$-clusters} \label{sect_klk}

First, we build a new manifold with corners $\klk^\bullet$ that will be interpreted as a moduli of marked disks with varying order on the boundary markings. As in the $\otimes$ case, the space of  $\bullet$-clusters $\cllks$ is then defined as a collar neighborhood of $\klk^\bullet$, the collar part again encoding disks connected by metric line segments. 

First, it will be convenient to see the spaces of discs as lying inside the spaces $\R\mathcal{M}_{\ell,k}$ of spheres with $\ell +1$ real and $k$ pairs of complex conjugate markings. It is useful because in that setting, the ribbon switches as described in \cite{CL} naturally occur. We recall the construction of these spaces that appears in \cite{Cey}.

Let $\mathcal{M}_{\ell+1+2k}$ be the space of complex spheres with $\ell +1 +2k$ markings denoted by $\{x_j\}_{0\leq j\leq \ell}$ and $\{z_h\}_{1\leq h \leq 2k}$. It has the structure of a complex $(\ell -2 +2k)$-manifold and has an antiholomorphic involution $\sigma_{\ell,k}$ defined as the composition of the natural complex conjugation $(\Sigma,j) \rightarrow (\Sigma,-j)$ with the transpositions $(z_h z_{h+k})$, $1\leq h \leq k$. The real locus $\R\mathcal{M}_{\ell.k} \equiv fix(\sigma_{\ell,k})$ is then a smooth real $(\ell -2 +2k)$-manifold. Also remark that permutation of the real markings gives a natural smooth action of $S_\ell$ on $\R\mathcal{M}_{\ell.k}$.

The complex double operation over the complex disks gives a natural map $\klk^{Id} = \klk \overset{\iota}{\rightarrow} \R\mathcal{M}_{\ell,k}$. The image of this map is the closure of the set of spheres represented as $x_0 = \infty, \{x_1 < \ldots < x_\ell \} \subset \R P^1 \subset \C P^1$ and $Im(z_h) > 0$ for $1 \leq h \leq k$. Moreover, for every permutation $p \in S_\ell$ one can denote by $\klk^p$ the space of marked disks with $x_0 < x_{p(1)} < \ldots < x_{p(\ell)}$ and extend $\iota$ to these disks. We refer to $p$ as an ordering on the real markings of the disks.

We will be interested in the image of the extended complex doule map $\underset{p \in S_\ell}{\bigsqcup} \klk^p \overset{\iota}{\longrightarrow} \R\mathcal{M}_{\ell,k}$. It can be seen as a partial manifold with corners tiling of $\R\mathcal{M}_{\ell,k}$ by $\ell !$ tiles, the complete tiling having $\ell ! 2^{k-1}$. $Im(\iota)$ gives a moduli in which the real marked points are no longer required to stay in a given order, but their order rather changes according to the switches proposed in \cite{CL}. The non-injective strata of $\iota$ are exactly those made of configurations having smooth components with no interior markings.

\begin{defi}
	A smooth component of a marked nodal disk (resp. real sphere) with no interior (resp. no complex-conjugate pair of) markings is called a ghost disk (resp. sphere). If  a ghost disk (resp. sphere) separates $z_i$ and $z_j$ for some pair $1 \leq i < j \leq k$, it is referred to as an internal ghost disk (resp. sphere), otherwise, it is referred to as an external ghost disk (resp. sphere).
\end{defi}

The main point is that, in general, $Im(\iota)$ is singular over the strata having at least one internal ghost sphere. For example, a neighborhood of the singular locus in the case $\ell=0$, $k=2$ is displayed in figure \ref{bul_fig01}. Next we give a local description of $Im(\iota) \subset \R \mathcal{M}_{\ell,k}$ near any point $S \in Im(\iota)$ lying in a codimension $m$ open stratum $\mathfrak{S} \subset \R \mathcal{M}_{\ell,k}$.

\begin{figure}[h] 
        \centering 
	\includegraphics[width=100mm]{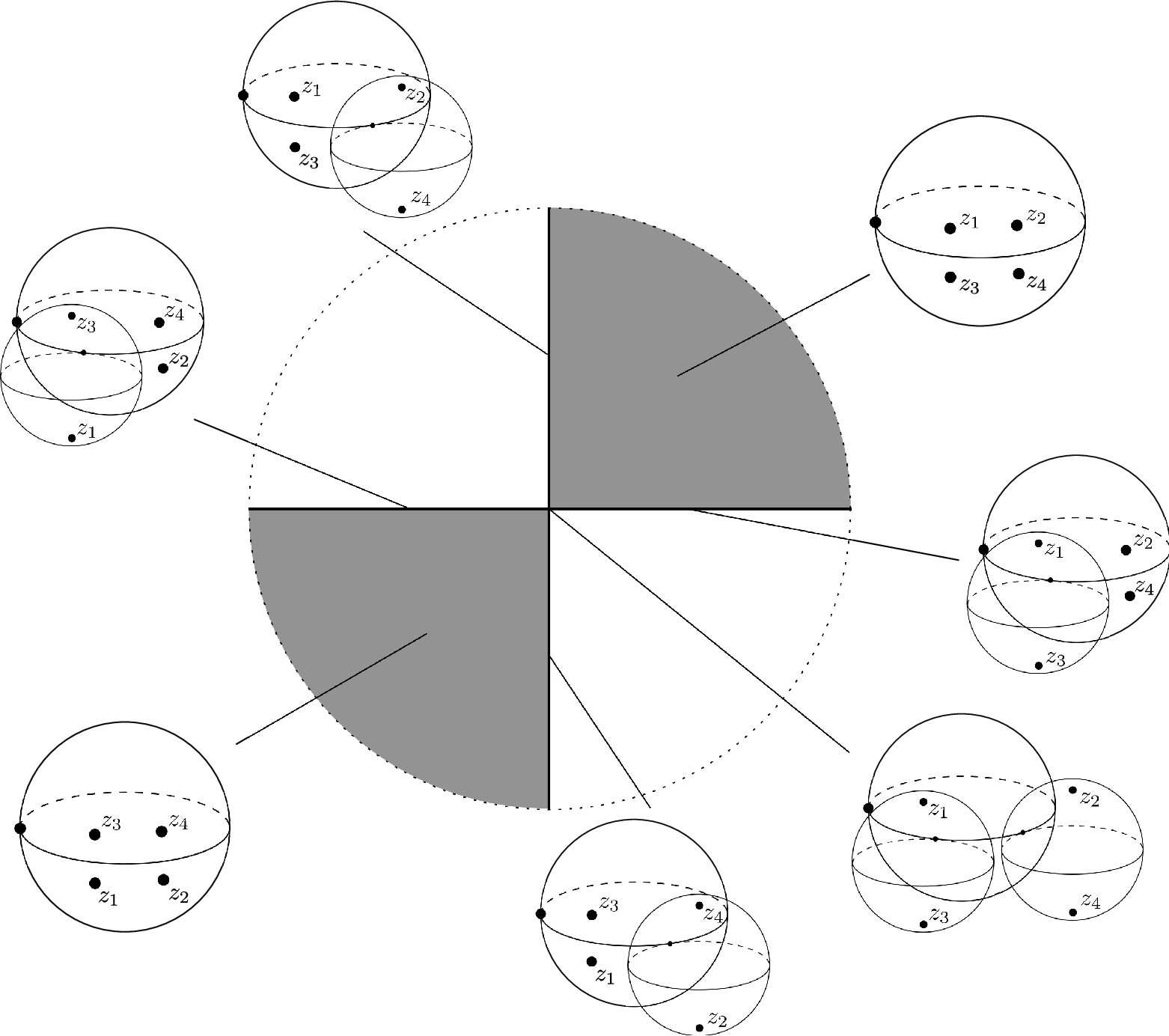}
        \caption{Singularity in $Im(\iota) \subset \R \mathcal{M}_{0,2}$} \label{bul_fig01}
\end{figure}

We can choose normal coordinates $(n_1, \ldots, n_m)$ to $\mathfrak{S}$ at $S$ corresponding to $m$ real gluing parameters, one for each real node of $S$ (see \cite{MW}, \cite{Liu}, \cite{MS}). Choose $D \in \iota^{-1}(S)$, say $D \in \klk^p$, and orient each $n_i$ coordinate so that $ \{n_i \geq 0\}_{1 \leq i \leq m}$ corresponds to $\iota(\klk^p)$. Then one can see that

\begin{lem}\label{action}
	In the $(n_1, \ldots, n_m)$ normal coordinates to $\mathfrak{S}$ at $S$, $Im(\iota) = G \cdot \R_+^m$, where $G \equiv \underset{d \text{ ghost}}{\prod} \Z/2\Z$ and the $d$ factor generator acts by changing the signs of the coordinates corresponding to the nodes on $d$.
\end{lem}

Although it is possible to desingularize the whole $Im(\iota)$ to get an appropriate manifold with embedded corners, we decide not to do so for practical reasons. We rather partially join the tiles $\underset{p \in S_\ell}{\bigsqcup} \klk^p$ on chosen 1-corners, 2-corners and 3-corners.

First, we join the tiles along 1-corners that are transpositions of two real markings.

\begin{defi} \label{transI}
	A 1-codimensional stratum $ T^I \cong \R\mathcal{M}_{\ell-1,k} \times \R\mathcal{M}_{2,0} \subset \R\mathcal{M}_{\ell,k}$ with $k \geq 1$ is called a marking-marking transposition stratum. It is equivalent to being the fixed locus of a transposition for the natural $S_\ell$ action on $\R\mathcal{M}_{\ell,k}$.

	Then, $\klk^I \equiv \underset{p \in S_\ell}{\bigsqcup} \klk^p / \sim $, where $D_1 \sim D_2$ if they differ by a transposition over an external ghost with two real markings (see figure \ref{bul_fig11}). That is, $\iota(D_1) = \iota(D_2) \in T^{I}$, where $T^I$ is a marking-marking transposition stratum.
\end{defi}

\begin{figure}[h] 
        \centering 
	\includegraphics[width=90mm]{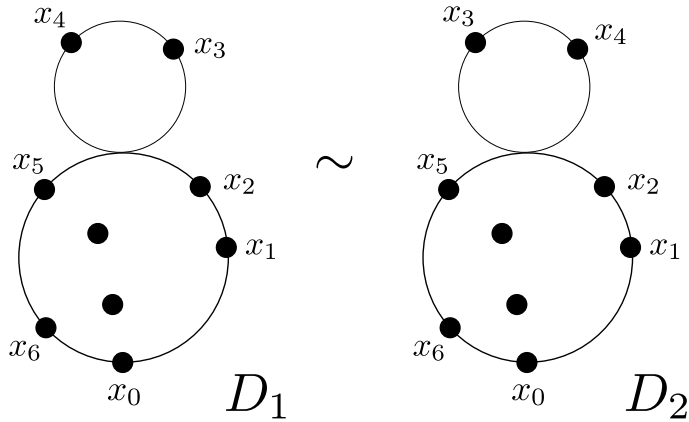}
        \caption{Disks identified in a marking-marking transposition stratum} \label{bul_fig11}
\end{figure}

\begin{lem}\label{klkI}
	$\klk^I$ is an orientable smooth manifold with embedded corners.
\end{lem}

\begin{proof}
	$\underset{p \in S_\ell}{\bigsqcup} \klk^p$ is smooth with embedded corners and orientable so we need to verify smoothness only at $D \in \klk^I$ with $S = \iota(D) \in T^I$. Therefore, $S$ lies in the interior of a codimension $m$ stratum $\mathfrak{S} \cong \R \mathcal{M}_{2,0} \times \ldots \times \R \mathcal{M}_{2,0} \times \R \mathcal{M}_{\ell^{(e+1)},k^{(e+1)}} \times \ldots \times \R \mathcal{M}_{\ell^{(m+1)},k^{(m+1)}}$, where the $e$ first factors correspond to the external ghosts with two real markings and the last factors to the deformations over the other components.

	Therefore, if we choose the $(n_1, \ldots, n_m)$ normal coordinates to $\mathfrak{S}$ at $S$ so that $(n_1, \ldots, n_e)$ correspond to gluing parameters for the external ghosts with two real markings, then $(n_1, \ldots, n_m) \in \R^e \times \R_+^{m-e}$ times a chart of $\mathfrak{S}$ gives a chart of $\klk^I$ near $D$.

	Notice that this local portrait also implies embeddedness of the 1-corners, and therefore embeddedness of every corner. Indeed, every $1$-corner $F \in C_1(\klk^I)$ appears locally as $\{n_i=0\}$, for some $i > e$, so it arises as a smooth gluing of $2^e$ $1$-corners of the $\klk^p$.

	Let $\mathcal{O}$ be an orientation of $\klk^{Id}$. It is not hard to see that by choosing $(-1)^{sg(p)} \mathcal{O}$ on $\klk^p$ and pushing it forward using $\iota$, one defines an orientation of $\klk^I$. Indeed, having $\klk^{p_1}, \klk^{p_2} \subset \klk^I$ identified on one of their 1-corners means that in the normal coordinates to the tranposition strata, $\iota(\klk^{p_1})$ writes as $n_1\geq0$ and $\iota(\klk^{p_2})$ as $n_1\leq0$ while the identified 1-corners is the hypersurface $\{n_1=0\}$. Since $\iota|_{\klk^{p_1}}$ and $\iota|_{\klk^{p_2}}$ only differ by a reflection about $\{n_1=0\}$, we get that, locally,
\[
(\iota|_{\klk^{p_1}})_* (-1)^{sg(p_1)} \mathcal{O} = -(\iota|_{\klk^{p_2}})_* (-1)^{sg(p_1)} \mathcal{O} = (\iota|_{\klk^{p_2}})_* (-1)^{sg(p_2)} \mathcal{O}
\]
\end{proof}

\begin{rem}
	One could also identify the tiles over 1-corners with an arbitrary exterior ghost, but the orientability part of the proof of lemma \ref{klkI} then fails. Indeed, for $\ell=2$ and $k=1$, one then gets the whole $\R \mathcal{M}_{2,1}$ which is $T^2$ with a point blown up, the exceptional divisor representing the spheres with a quadrivalent ghost (see \cite{Cey}).
\end{rem}

Second, we identify $\klk^I$ along 2-corners being internal transpositions of a real node and a real marking. We proceed iteratively, identifying two codimension two strata associated with a given real marking and blowing up along the identification locus to remove the singularity then created.

\begin{defi} \label{transII}
	Let a node-marking transposition of $x_j$, $0\leq j \leq \ell$, be a stratum $ T^{II_j} \cong \R\mathcal{M}_{2,0} \times \R\mathcal{M}_{\ell^{(1)},k^{(1)}} \times \R\mathcal{M}_{\ell^{(2)},k^{(2)}} \subset \R\mathcal{M}_{\ell,k}$ with $k^{(1)} \geq 1$, $k^{(2)} \geq 1$ and the first factor standing for an internal ghost with two real nodes and the real marked point $x_j$. Note that for fixed $j$ two such strata are disjoint.

	Then, $\klk^{II_j} \equiv \klk^I / \sim $, where $D_1 \sim D_2$ if they differ by a node-marking transposition of $x_j$ (see figure \ref{bul_fig12}). That is, $\iota(D_1) = \iota(D_2) \in T^{II_j}$ for some node-marking transposition stratum $T^{II_j}$.

\begin{figure}[h] 
        \centering 
	\includegraphics[width=90mm]{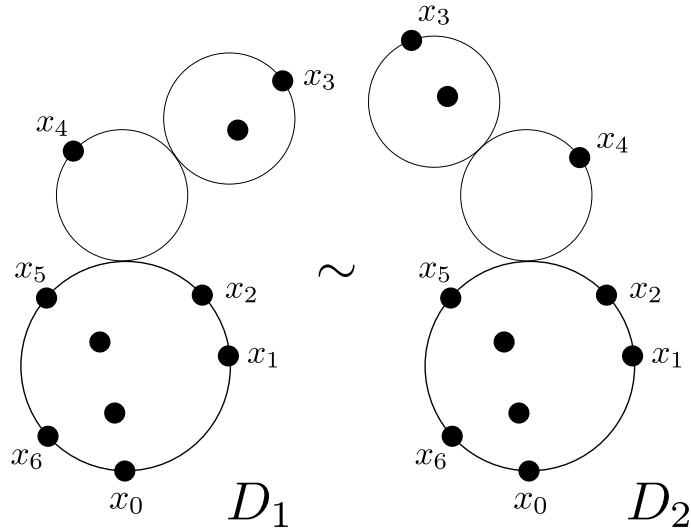}
        \caption{Disks identified in a node-marking transposition stratum} \label{bul_fig12}
\end{figure}

	Let $bu_{II_j}(\R \mathcal{M}_{\ell,k})$ and $bu_{II_j}(\klk^{II_j})$ be the blowups of $\R \mathcal{M}_{\ell,k}$ and $\klk^{II_j}$  along $\bigsqcup T^{II_j}$ and $\iota^{-1}(\bigsqcup T^{II_j})$, respectively, so that we have a commutative diagram

\begin{diagram}
	bu_{II_j}(\klk^{II_j}) && \rTo{\iota} && bu_{II_j}(\R \mathcal{M}_{\ell,k})\\
	\dTo{bd_{II_j}} &&  && \dTo{bd_{II_j}}\\
	\klk^{II_j} && \rTo{\iota} && \R \mathcal{M}_{\ell,k}\\
\end{diagram}
\end{defi}

\begin{lem}\label{klkII}
	$bu_{II_j}(\klk^{II_j})$ is an orientable smooth manifold with embedded corners.
\end{lem}

\begin{proof}
	\textit{Smoothness.} We must only verify smoothness at a point $p \in bd^{-1} \circ \iota^{-1}(T^{II_j}) \subset \klk^{II_j}$ in the exceptional locus of a given internal transposition $T^{II_j}$. Say $S = \iota \circ bd(p)$ lies in the interior of a $m$-codimensional stratum $\mathfrak{S} \subset T^{II_j}$. Then, choose the normal coordinates $(n_1, n_2, n_3, \ldots, n_m)$ to $\mathfrak{S}$ at $S$, so that $n_1$ and $n_2$ correspond to the gluings over the ghost component containing $x_j$. Using these coordinates, we can describe $\klk^{II_j}$, near $p$, as $\{ n_1 n_2 \geq 0 \} \times \R_+^{m'} \times \R^{m -m' -2} \times \R^{dim(\mathfrak{S})} \subset \R^{\ell -2 +2k}$ where $\iota^{-1}(T^{II_j})$ corresponds to $\{ n_1 = n_2 = 0 \} \times \R_+^{m'} \times \R^{m -m' -2} \times \R^{dim(\mathfrak{S})}$ and $3 \leq m' \leq m$.

	Then, $p$ writes, in the blownup local chart, as $([n_1: n_2], 0, \ldots, 0)$ with, say, $n_2 \neq 0$. Thus, the coordinates $(n, \frac{n_1}{n_2}, n_3, \ldots, n_{m})$, where $n$ is normal to the exceptional divisor, times a chart of $\mathfrak{S}$ give a smooth corner chart of $bu_{II_j}(\klk^{II_j})$ near $p$.

\begin{rem}\label{iterateII}
	Furthermore, one verifies that in these new normal coordinates, $Im(\iota)$ lifts to $\R \times ((\underset{d\in ghosts \mid d\neq d_j}{\prod} \Z/2\Z) \cdot \R_+^{m-1})$ where the action is induced by that of $G = \underset{d\in \text{ ghosts}}{\prod} \Z/2\Z$ so we are back to the local model described in lemma \ref{action}.
\end{rem}

	\textit{Embeddedness of the corners. } To show that the corners of the desingularized space are still embedded, it is again sufficient to show that every 1-corner is embedded. Indeed, a 1-corner of $bu_{II_j}(\klk^{II_j})$ arises locally, in local normal coordinates above, as a lift of $\iota^{-1}(G \cdot \{n_i \geq 0, n_j=0 \mid 1 \leq i \leq m, i \neq j\})$ to the blowup, and therefore it is smoothly embedded.

	\textit{Orientability. } We are left to check that an orientation of $\klk^I$ lifts to an orientation of $bu_{II_j}(\klk^{II_j})$. It is enough to look around $p \in \partial_1(bu_{II_j}(\klk^{II_j}))$ such that $\iota \circ bd(p) \in T^{II_j}$ is an interior point, i.e. an endpoint of an exceptional fiber.

	Then, in the normal coordinates $(n_1,n_2)$ to $T^{II_j}$ at $\iota \circ bd(p)$, $Im(\iota)$ corresponds to $\{n_1 n_2 \geq0\}$. Thus, in the blownup local normal model, up to a reordering of the coordinates, $p = (n_1, [n_1:n_2]) = (0,[0: 1])$ so it lies in the lift of the $\{n_1 =0\}$ hyperplane. Remark that $\iota^{-1}(\{n_1 =0\}) \subset \klk^I$ lies in a single 1-corner $F \in C_1(\klk^I)$ that is sent by $\iota$ to a codimension 1 stratum $\R \mathcal{M}_{\ell^{(1)},k^{(1)}} \times \R \mathcal{M}_{\ell^{(2)},k^{(2)}}$ of $\R \mathcal{M}_{\ell,k}$ where, say, $x_j$ contributes to the moduli of the second factor.

	Take $\frac{\partial}{\partial n} \wedge \mathcal{O}_F$ as an orientation of $\klk^I$, where $\frac{\partial}{\partial n}$ is outward normal to $F$ and $\mathcal{O}_F$ is an orientation of $F$.

	First, notice that $\mathcal{O}_F$ pushes to an orientation of the $\{n_1 =0\}$ hyperplane because the attaching of $F$ to itself on $\{n_1 =n_2 =0\}$ is the same as the one used in the construction of $K_{\ell^{(2)},k^{(2)}}^I$. Said differently, $F$ is a cut of $K_{\ell^{(1)},k^{(1)}}^I \times K_{\ell^{(2)},k^{(2)}}^I$ along the transpositions of its real node, seen as a real marking, and one of its real marking, and $\{n_1 =n_2 =0\}$ is part of that cut locus.

	Second, since $\frac{\partial}{\partial n}$ is pushed to $-\frac{\partial}{\partial n_1}$ over $\{n_1,n_2 \geq0\}$ and to $\frac{\partial}{\partial n_1}$ over $\{n_1,n_2 \leq0\}$, it lifts to $-\frac{\partial}{\partial \frac{n_1}{n_2}}$ which is again outward to $\partial_1(bu_{II_j}(\klk^{II_j}))$ near $p$. Therefore, an orientation $\frac{\partial}{\partial n} \wedge \mathcal{O}_F$ of $\klk^I$ lifts to an orientation of $bu_{II_j}(\klk^{II_j})$.

\end{proof}

Remark \ref{iterateII} in the proof of \ref{klkII} emphasizes the fact that it is possible to iterate the identification and blowup process.

\begin{defi}
	Let $\klk^{II}$ be the space $bu_{II_\ell}( \ldots bu_{II_1}(bu_{II_0}(\klk^{II_0})^{II_1}) \ldots ^{II_\ell})$ defined by implementing the construction of definition \ref{transII} iteratively.
\end{defi}

\begin{cor}
	$\klk^{II}$ is an orientable smooth manifold with embedded corners.
\end{cor}

As noticed in the orientability part of the proof of lemma \ref{klkII}, a 1-corner $F \in C_1(\klk^{II})$ with no ghost component is isomorphic to some $K_{\ell^{(1)},k^{(1)}}^I \times K_{\ell^{(2)},k^{(2)}}^I$ with $k^{(1)}, k^{(2)} \geq 1$. This is simply due to the fact that type $II$ identifications restricted to 1-corners with no ghost become type $I$ identifications.

Third, we identify $\klk^{II}$ along 3-corners being transpositions over an internal ghost connecting three nonghost disks. We proceed iteratively, identifying two corresponding codimension 3 strata and blowing up along the identification locus to remove the singularity then created.

\begin{defi} \label{transIII}
	Let a node-node transposition be a stratum $ T^{III} \cong \R\mathcal{M}_{2,0} \times \R\mathcal{M}_{\ell^{(1)},k^{(1)}} \times \R\mathcal{M}_{\ell^{(2)},k^{(2)}} \times \R\mathcal{M}_{\ell^{(3)},k^{(3)}} \subset \R\mathcal{M}_{\ell,k}$, with $k^{(1)},k^{(2},k^{(3)} \geq 1$ and the first factor standing for an internal ghost with three real nodes. Note that these strata are not mutually disjoint in general.

	Then, $\klk^{III} \equiv \klk^{II} / \sim $, where $D_1 \sim D_2$ if they differ by a node-node transposition (see figure \ref{bul_fig13}).

\begin{figure}[h] 
        \centering 
	\includegraphics[width=90mm]{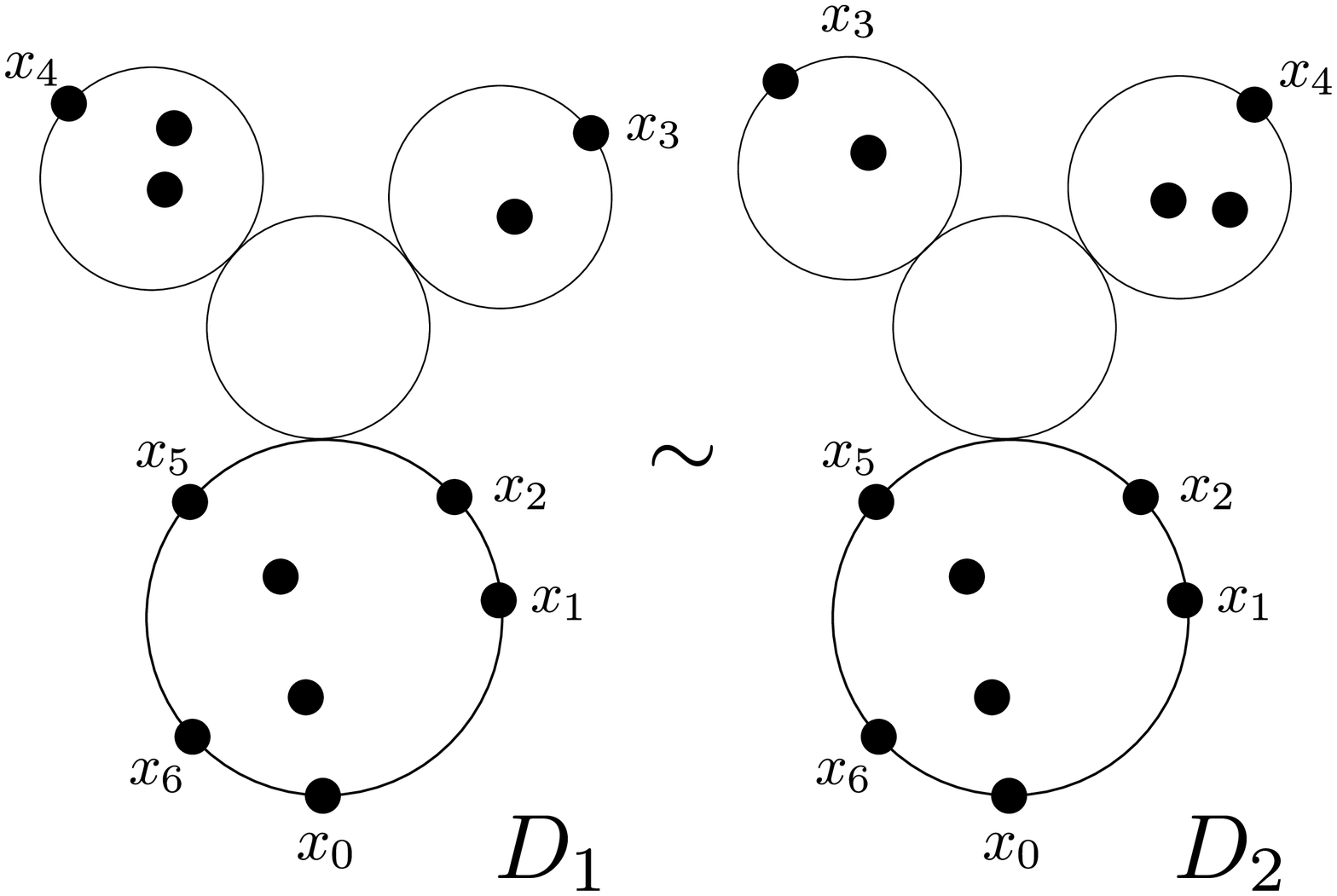}
        \caption{Disks identified in a node-node transposition stratum} \label{bul_fig13}
\end{figure}

	Given an order on the set of node-node transpositions, let $bu_{III}(\R \mathcal{M}_{\ell,k})$ and $\klk^\bullet \equiv bu_{III}(\klk^{III})$ be the blowups of $bu_{II}(\R \mathcal{M}_{\ell,k})$ and $\klk^{III}$  along $\bigcup T^{III}$ and $\iota^{-1}(\bigcup T^{III})$ with respect to this order, respectively, so that we get a commutative diagram

\begin{diagram}
	\klk^\bullet \equiv bu_{III}(\klk^{III}) && \rTo{\iota} && bu_{III}(\R \mathcal{M}_{\ell,k})\\
	\dTo{bd_{III}} &&  && \dTo{bd_{III}}\\
	\klk^{III} && \rTo{\iota} && \R \mathcal{M}_{\ell,k}\\
\end{diagram}
\end{defi}

\begin{lem}\label{klkIII}
	$\klk^\bullet \equiv bu_{III}(\klk^{III})$ is an orientable smooth manifold with embedded corners.
\end{lem}

\begin{proof}
	The proof here uses the same arguments as the proof of lemma \ref{klkII} but we include it for completeness.
  
	\textit{Smoothness.} We must only verify smoothness at a point $p \in bd^{-1} \circ \iota^{-1}(T^{III}) \subset \klk^\bullet$ after the blowup at the transposition $T^{III}$. Say $S = \iota \circ bd(p)$ lies in the interior of a $m$-codimensional stratum $\mathfrak{S} \subset T^{III}$. Then, in the normal coordinates $(n_1, n_2, \ldots, n_m)$ to $\mathfrak{S}$ at $S$, say $n_1$, $n_2$ and $n_3$ are associated with the gluings over the special ghost component so that $T^{III}$ corresponds to $\{n_1=n_2=n_3=0\}$. 

	Then, $p$ writes, in the blownup local model, as $([n_1: n_2: n_3], 0, \ldots, 0)$ with, say, $n_3 \neq 0$. The coordinates $(n, \frac{n_1}{n_3}, \frac{n_2}{n_3}, n_4, \ldots, n_{m})$, where $n$ is normal to the exceptional divisor, times a chart of $\mathfrak{S}$ gives a smooth corner chart of $\klk^\bullet)$ near $p$.

\begin{rem}
	Again, one verifies that in these coordinates, $Im(\iota)$ lifts to $\R \times ((\underset{d\in ghosts \mid d\neq d_j}{\prod} \Z/2\Z) \cdot \R_+^{m-1})$ where the action is induced by that of $G = \underset{d\in \text{ ghosts}}{\prod} \Z/2\Z$ so we are back to the local model described in lemma \ref{action}.
\end{rem}

	\textit{Embeddedness of the corners. } To show that the corners of the desingularized space are still embedded, it is again sufficient to show that every 1-corner is embedded. Indeed, a 1-corner of the blowup at $T^{III}$ arises locally, in the local coordinates above, as a lift of $\iota^{-1}(G \cdot \{n_i \geq 0, n_j=0 \mid 1 \leq i \leq m, i \neq j\})$ to the blowup, and therefore it is smoothly embedded.

	\textit{Orientability. } We are left to check that an orientation of $\klk^{II}$ lifts to an orientation of $\klk^\bullet$. It is enough to look around $p \in \partial_1(\klk^\bullet)$ such that $\iota \circ bd(p) \in T^{III}$ is an interior point.

	Then, in the normal coordinates $(n_1,n_2,n_3)$ to $T^{III}$ at $\iota \circ bd(p)$, $Im(\iota)$ corresponds to $\{n_1,n_2,n_3 \geq0\} \bigcup \{n_1,n_2,n_3 \leq0\}$. Thus, in the blownup local normal model, up to a reordering, $p = (0,[0: p_2: p_3])$ so it lies in the local lift of the $\{n_1 =0\}$ hyperplane. Remark that $\iota^{-1}(\{n_1 =0\}) \subset \klk^{II}$ lies in a single 1-corner $F \in C_1(\klk^{II})$ and that it is sent by $\iota$ to a codimension 1 stratum $\R \mathcal{M}_{\ell^{(1)},k^{(1)}} \times \R \mathcal{M}_{\ell^{(2)}+\ell^{(3)}-1,k^{(2)}+k^{(3)}}$ of $\R \mathcal{M}_{\ell,k}$ where the second factor is attained by the gluings on the nodes corresponding to $n_2$ and $n_3$.

	Take $\frac{\partial}{\partial n} \wedge \mathcal{O}_F$ as an orientation of $\klk^{II}$, where $\frac{\partial}{\partial n}$ is outward normal to $F$ and $\mathcal{O}_F$ is an orientation of $F$.

	First, notice that $\mathcal{O}_F$ pushes to an orientation of the $\{n_1 =0\}$ hyperplane because the attaching of $F$ to itself on $\{n_1 =n_2 =n_3 =0\}$ is the same as the attaching of $K_{\ell^{(2)}+\ell^{(3)}-1,k^{(2)}+k^{(3)}}^{I}$ on an internal node-marking transposition as described in \ref{klkII}. Said differently, $F$ is a cut of $K_{\ell^{(1)},k^{(1)}}^{II} \times K_{\ell^{(2)}+\ell^{(3)}-1,k^{(2)}+k^{(3)}}^{II}$ along the transpositions of a real node and a real marking for which $\{n_1 =n_2 =n_3 =0\}$ is part of that cut locus.

	Second, since $\frac{\partial}{\partial n}$ is pushed to $-\frac{\partial}{\partial n_1}$ over $\{n_1,n_2,n_3 \geq0\}$ and to $\frac{\partial}{\partial n_1}$ over $\{n_1,n_2,n_3 \leq0\}$, it lifts to $-\frac{\partial}{\partial \frac{n_1}{n_3}}$ which is again outward to $\partial_1(\klk^\bullet)$ near $p$. Therefore, an orientation $\frac{\partial}{\partial n} \wedge \mathcal{O}_F$ of $\klk^{II}$ lifts to an orientation of $\klk^\bullet$ near $F$.

\end{proof}

As in the nonsymmetric case, one can then apply a collar neighborhood enlargement procedure to $\klk^\bullet$. However, we choose to restrict the collar procedure $col$ to the corners that do not have ghost components associated, and denote by $col^{\text{n-g}}$ this restricted procedure.

\begin{defi}
	For $k \geq 1$ and $\ell + 2k -2 \geq 1$, let $\mathcal{C}\ell_{\ell,k}^\bullet = col^{\text{n-g}}(\klk^\bullet)$ and otherwise if $\ell = 0$ and $k=1$ or $\ell = 1$ and $k=0$, $\mathcal{C}\ell_{\ell,k}^\bullet$ will be a point. 
\end{defi}

The last property we would like to check is that we get, as in the $\otimes$ case, a product structure on the corners of $\cllks$.

\begin{lem}
	Every 1-corner of $\klk^\bullet$ corresponding to nodal disks with two nonghost components is isomorphic to a product $K_{\ell^{(1)},k^{(1)}}^\bullet \times K_{\ell^{(2)},k^{(2)}}^\bullet$ where $k^{(1)}, k^{(2)} \geq 1$ and $\ell^{(1)}+ \ell^{(2)} -1 = \ell$.
\end{lem}

\begin{proof}
	As noticed in the proof of lemma \ref{klkII}, type $II$ identifications provide the missing type $I$ identifications over the 1-corners, that is, the type $I$ transpositions with the node considered as a real marking. In the same fashion, type $III$ identifications restrict to the missing type $II$ identifications of the 1-corners, more precisely, the type $II$ transpositions with the node considered as a real marking.

	It remains to check that all the type $III$ identifications of the 1-corner are performed after type $III$ identifications. Indeed, the type $III$ identifications over a 1-corner are simply those coming from the global ones because the node, considered as a real marking, is not implied in the associated transpositions. Thus, the considered 1-corner of $\klk^\bullet$ has the same identifications as $K_{\ell^{(1)},k^{(1)}}^\bullet \times K_{\ell^{(2)},k^{(2)}}^\bullet$.

\end{proof}

\begin{cor}
	Every 1-corner of $\cllks$ corresponding to two nonghost components is isomorphic to a product $\mathcal{C}\ell_{\ell^{(1)},k^{(1)}}^\bullet \times \mathcal{C}\ell_{\ell^{(2)},k^{(2)}}^\bullet$ where $k^{(1)}, k^{(2)} \geq 1$ and $\ell^{(1)}+ \ell^{(2)} -1 = \ell$.
\end{cor}

\begin{proof}
	This follows from the fact that $col(M_1 \times M_2) \cong col(M_1) \times col(M_2)$ for any manifolds with embedded corners $M_1$ and $M_2$, and that any $F \in \partial_1(M)$ naturally has a corresponding 1-corner $col(F) \in \partial_1(col(M))$.
\end{proof}

Remark that $\klk^\bullet$ has a second type of 1-corners, namely those corresponding to nodal disks with one ghost component having more than 3 real special points. In fact, ignoring the corners of $\klk^\bullet$ having ghost components associated, we get this more general result:

\begin{lem}
	Every $m$-corner of $\klk^\bullet$ corresponding to nodal disks without ghost components is isomorphic to a product $K_{\ell^{(1)},k^{(1)}}^\bullet \times \ldots \times K_{\ell^{(m+1)},k^{(m+1)}}^\bullet$ \\ where $k^{(1)}, \ldots, k^{(m+1)} \geq 1$ and $\ell^{(1)}+ \ldots+ \ell^{(m+1)} -m = \ell$.
\end{lem}

We note that, unlike in the $\otimes$ case, these isomorphisms do not seem to be canonical. More precisely, we do not know which numbering to choose on the breakings seen as leaves of the smooth component below them. We therefore proceed to some choices in order to define for every $m$-corner without ghost components $F$ of $\cllks$ an isomorphism $\Phi_F: F \rightarrow \mathcal{C}\ell_{\ell^{(1)},k^{(1)}}^\bullet \times \ldots \times \mathcal{C}\ell_{\ell^{(m+1)},k^{(m+1)}}^\bullet$.

Let $D \in \klk^\bullet$ in the lift of $\klk^p$ and $D= D^{(1)} \bigcup \ldots \bigcup D^{(m+1)}$ its natural decomposition into an element of $K_{\ell^{(1)},k^{(1)}} \times \ldots \times K_{\ell^{(m+1)},k^{(m+1)}}$ where $\ell^{(i)}>0$ and $k^{(i)}>0$, $1\leq i \leq m+1$. Consider, for every leaf (boundary special point) $x$ of $D^{(i)}$,
\begin{itemize}
\item $p(x)$ if $x$ is also a leaf of $D$,
\item the minimum of the values of $p$ on the leaves of $D$ lying above $x$. 
\end{itemize}
This defines an ordering $p^{(i)}$ on the leaves of $D^{(i)}$. Also consider the natural ordering on the interior markings of $D^{(i)}$ induced by that on the interior markings of $D$.

Thus $\Phi_F(C) \in \mathcal{C}\ell_{\ell^{(1)},k^{(1)}}^\bullet \times \ldots \times \mathcal{C}\ell_{\ell^{(m+1)},k^{(m+1)}}^\bullet$ can be defined by iterating this procedure over $i$ (see figure \ref{bul_fig02}). More generally, this procedure defines an ordering on the leaves of any subcluster $C'$ of $C$. Moreover, it is coherent in the sense that if $C''$ is a subcluster of $C'$, then the orderings on $C''$ induced by the ordering on $C'$ and by the ordering on $C$ coincide.

\begin{figure}[h] 
        \centering 
	\includegraphics[width=120mm]{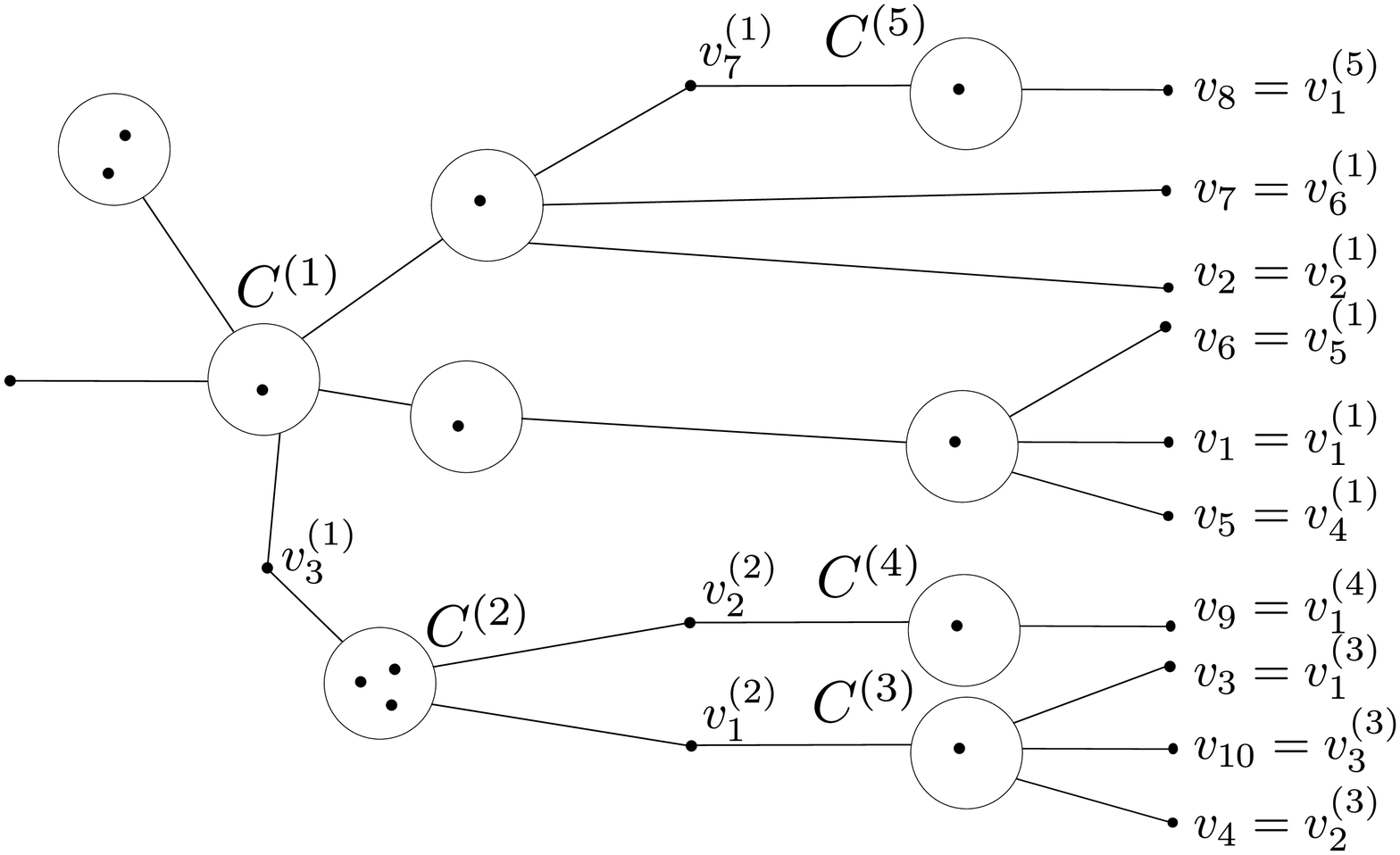}
        \caption{Product decomposition of a $\bullet$-cluster, $\ell^{(i)}>0$, $1\leq i \leq 5$} \label{bul_fig02}
\end{figure}

The above identifications will suffice to choose coherent perturbations over the Floer trajectories in the monotone setting where the perturbations will be chosen to be independent of the position of the interior markings. However, in general, we will have to generalize the above to make the identifications dependent of the position of the interior markings.

In this case, one can modify $\Phi_F$ to cover the general case where $C$ has some components without leaves, that is, $\ell^{(i)}$ can vanish for some $i$'s. Let $\ell \geq 0$, $k \geq 1$ and $\tau$ be a $(\ell,k)$-shuffle, that is, $\tau \in S_{\ell+k}$ such that $\tau(j) < \tau(j+1)$ whenever $j \neq \ell$. Now consider $\mathcal{C}\ell_{\ell,k,\tau}^\bullet = (\cllks,\tau)$ as the moduli of $\bullet$-clusters with $\ell$ leaves $\{v_j\}_{1\leq j \leq \ell}$ and $k$ interior markings $\{k_h\}_{1\leq h \leq k}$ together with the shuffle of the set $\{v_j\}_{1\leq j \leq \ell} \bigcup \{k_h\}_{1\leq h \leq k}$ defined by $\tau$.

Then proceeding exactly as above, one chooses for every $m$-corner without ghost components $F$ of $\mathcal{C}\ell_{\ell,k,\tau}^\bullet$ an isomorphism $\Phi_F: F \rightarrow \mathcal{C}\ell_{\ell^{(1)},k^{(1)},\tau^{(1)}}^\bullet \times \ldots \times \mathcal{C}\ell_{\ell^{(m+1)},k^{(m+1)},\tau^{(m+1)}}^\bullet$, the set of such identifications having the same coherency property (see figure \ref{bul_fig03}). Notice that for $\tau = Id$ and $\ell^{(i)}\geq 1$, $1 \leq i \leq m+1$, then $\tau^{(i)} = Id$, $1 \leq i \leq m+1$, and the above two identifications coincide.

\begin{figure}[h] 
        \centering 
	\includegraphics[width=120mm]{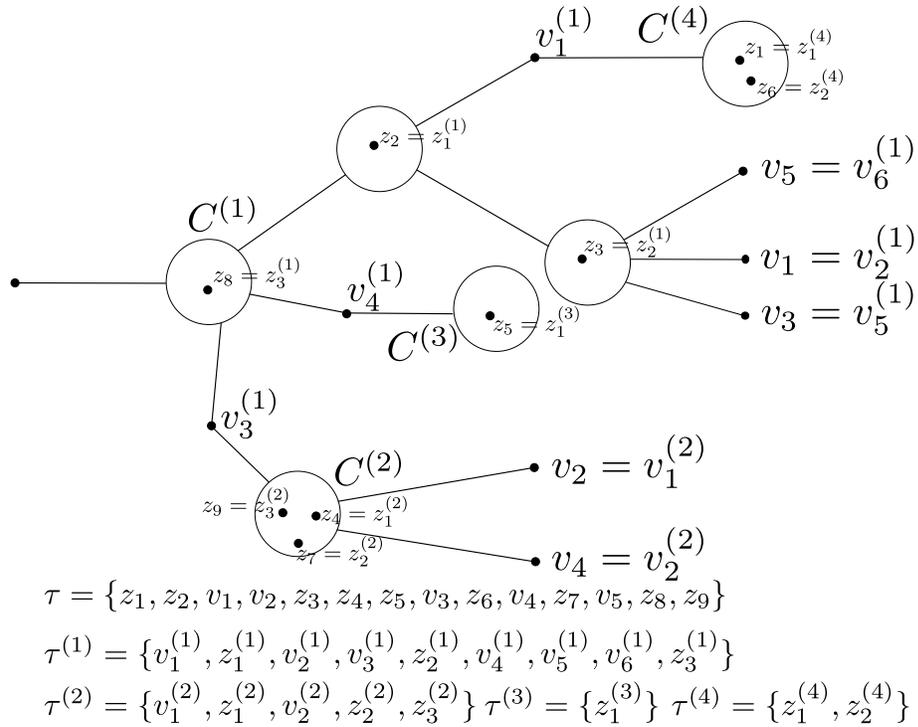}
        \caption{Product decomposition of a general $\bullet$-cluster} \label{bul_fig03}
\end{figure}


\section{Universal curves}

\begin{defi}
Let $\ulk^\bullet \overset{\pi}{\rightarrow} \klk^\bullet$ be the nodal family defined by restricting to $im(\iota)$ the universal curve over $\R \mathcal{M}_{\ell,k}$, forgetting the hemispheres containing any subset of $\{ z_h \}_{k+1 \leq h \leq 2k}$ and then pulling back the resulting under blowdown.
\end{defi}

\begin{rem}
$\ulk^\bullet \overset{\pi}{\rightarrow} \klk^\bullet$ is the pullback of the $\ulk \overset{\pi}{\rightarrow} \klk$ curves over each copy of $\klk$, except that its ghost disks are now complex doubled.
\end{rem}

We add metric lines between the components of the marked disks lying in the collar part of $\cllks$ using the tree labelings over the collar part as in the nonsymmetric case.

\begin{defi}
For $C \in \cllks$, we then define $(\pi^\bullet)^{-1}(C)$ as in the $\otimes$ case. We will refer to it as the $\bullet$-cluster, with $\ell$ leaves and $k$ interior markings, associated with $C \in \cllks$. 
\end{defi}

Taking $\ulks = \underset{C \in \cllks}{\bigcup} (\pi^\bullet)^{-1}(C)$, we get a map $\pi^\bullet: \ulks \rightarrow \cllks$ and a commutative diagram
\begin{diagram}
	\ulk^\bullet  && \rInto^{} && \ulks \\
	\dTo^{\pi} && && \dTo^{\pi^\bullet}\\
	\klk^\bullet && \rInto^{i} && \cllks\\
\end{diagram}

The result of this procedure will, as in the $\otimes$ case, produce families of metric trees with vertices replaced by complex marked disks. However, we emphasize that the $\bullet$ case is different from the former regarding several aspects: 

\begin{itemize}
\item Every ghost disk must be at distance $0$ from a disk having interior markings. This will later imply that classical Morse products are no longer considered in the differential of the $\bullet$-cluster complex.

\item In the context of $\bullet$-clusters, a configuration having two incident lines meeting on a trivalent ghost disk is glued on the one hand using an hemisphere of the ghost disk and on the other hand using the other hemisphere (see figure \ref{bul_fig01}). This will correspond to the flowline switches proposed in \cite{CL}.

\item Let $Sym_n$ stand for the symmetric group on $n$ symbols. There is a natural action of both $Sym_\ell$ and $Sym_k$ on $\cllks$ defined by changing the numbers of the leaves and of the interior markings, respectively.
\end{itemize}

\section{Coherent systems of ends}

We define coherent systems of ends for the $\bullet$-clusters as it was done for the $\otimes$-clusters, except that we require them to be independent of deformations over disk components with no interior markings, the ghost disk components. That is, for every $\ell, k \geq 0$, these half-line neighborhoods of the endpoints over $\cllks$ will again be coherent with respect to the product structure of its corners, and perturbations of Morse functions will later be allowed on their complement.

\begin{defi}
        A coherent system of ends on $\{ \ulks \}_{\ell,k \geq 0}$ is a collection of closed subsets $\mathcal{V} = \{\mathcal{V}_{\ell,k} \subset \ulks \}_{\ell,k \geq 0}$ coming from the closure of a corresponding collection of open subsets such that
\begin{enumerate}
        \item $\mathcal{V}_{0,0} = \mathcal{U}^\bullet_{0,0} \equiv \overline{\R}_-$, $\mathcal{V}_{1,0} = \mathcal{U}^\bullet_{1,0} \equiv \overline{\R}$. Otherwise if $\ell -2 +2k \geq 0$, then for every irreducible component $C^{(i)}$ of $C$, $\mathcal{V}_{\ell,k} \bigcap C^{(i)}$ is a disjoint union of closed neighborhoods of its endpoints, each being homeomorphic to $[-\infty, -R] \subsetneq \overline{\mathbb{R}}_{-}$, and closed intervals on some interior lines, each being homeomorphic to $[-\frac{\lambda'}{2}, \frac{\lambda'}{2}] \subsetneq [-\frac{\lambda}{2}, \frac{\lambda}{2}]$,

        \item for every irreducible component $C^{(i)} \in \mathcal{C}\ell^\bullet_{\ell^{(i)},k^{(i)}}$ of $C$, $\mathcal{V}_{\ell,k} \bigcap C^{(i)} = \mathcal{V}_{\ell^{(i)},k^{(i)}} \bigcap C^{(i)}$. Therefore, the neighborhoods over $C^{(i)}$ are invariant under deformations of $C \setminus C^{(i)}$, under permutations of the $z_h$'s which leave the $z^{(i)}_h$'s invariant, it is independent of $\ell$, $k$ and of the relative position of $C^{(i)}$ in $C$.

	\item for every irreducible component $C^{(i)} \in \mathcal{C}\ell^\bullet_{\ell^{(i)},k^{(i)}}$ of $C$, $\mathcal{V}_{\ell,k} \bigcap C^{(i)}$ is invariant under deformations of the ghost spheres of $C^{(i)}$ having $4$ or more boundary markings.
\end{enumerate}

\begin{rem}
One might further require the coherent system of ends to be invariant under permutation of the numbers of the leaves and of the interior marked points.
\end{rem}

\end{defi}

One can construct coherent systems of ends explicitely using the same procedure as in \ref{sect_ends}, the neighborhoods over a line only depending on its length. Therefore,

\begin{lem}
        There exists coherent systems of ends.
\end{lem}

Note that one can define a coherent system of strip-like ends $\mathcal{S} = \{\mathcal{S}_{\ell,k} \subset \mathcal{U}_{\ell,k} \}$ over the disks (and spheres) of $\ulks$ by pulling back a real coherent system of ends $S = \{S_{\ell,k} \subset U_{\ell,k}^\bullet \}$ (defined in a similar fashion as in \cite{Sei}, \cite{W1}, \cite{CM}) by the natural projection $\ulks \rightarrow U^\bullet_{\ell,k}$ that collapses the connecting lines.

\section{Coherent systems of perturbations}

As in the $\otimes$ case, we assign to every point of $\mathcal{U}^\bullet_{\ell, k, \tau}$ plus an endpoint labeling an element in a product of Banach manifolds so that this choice is coherent with respect to the product structure of the corners of $\mathcal{C}\ell^\bullet_{\ell, k, \tau}$. Again, the labeling on the leaves will later encode a choice of Morse functions over the ends.

First, fix an integer $c \geq 0$.

\begin{defi} \label{bul_lab}
        We call $\mathfrak{l}: \{1,2, \ldots, \ell \} \hookrightarrow \{1, ...,c, c+1\}$ an endpoint labeling of length $\ell$ if whenever $j_1 < j_2$ with $\mathfrak{l}(j_1))<c+1$ and $\mathfrak{l}(j_2))<c+1$, then $\mathfrak{l}(j_1) < \mathfrak{l}(j_2)$. That is, if we exclude the preimages of $c+1$, $\mathfrak{l}$ is increasing.

        We take $\mathfrak{L}$ to be the set of all the endpoint labelings, $\mathcal{C}\ell^\bullet_{\ell, k, \tau, \mathfrak{l}} \equiv (\cllks, \tau, \mathfrak{l})$ and $\mathcal{U}^\bullet_{\ell, k, \tau, \mathfrak{l}} \equiv (\ulks, \tau,\mathfrak{l})$ for $\mathfrak{l} \in \mathfrak{L}$ of length $\ell$.

\end{defi}

First, note that if $c=0$, the only possible labeling is trivial.
 
Note that a labeling $\mathfrak{l}$ of the leaves of a $\bullet$-cluster $C \in \cllks$ canonically determines an end labeling $\mathfrak{l}^{(i)}$ over each of its irreducible component $C^{(i)}$: let $v^{(i)}_a$ be the $a^{th}$ leaf of $C^{(i)}$ and $E$ the numbers of the leaves that lie above $v^{(i)}_a$ (including itself). If $E\neq \emptyset$, take $\mathfrak{l}^{(i)}(a) = min(E)$ and otherwise, take $\mathfrak{l}^{(i)}(a)=c+1$. In fact, one can use $\mathfrak{l}$ to associate in the same fashion an integer $\mathfrak{l}(l)$ to each line $l$ of $C$ (see figure \ref{bul_fig04}).

\begin{figure}[h] 
        \centering 
	\includegraphics[width=120mm]{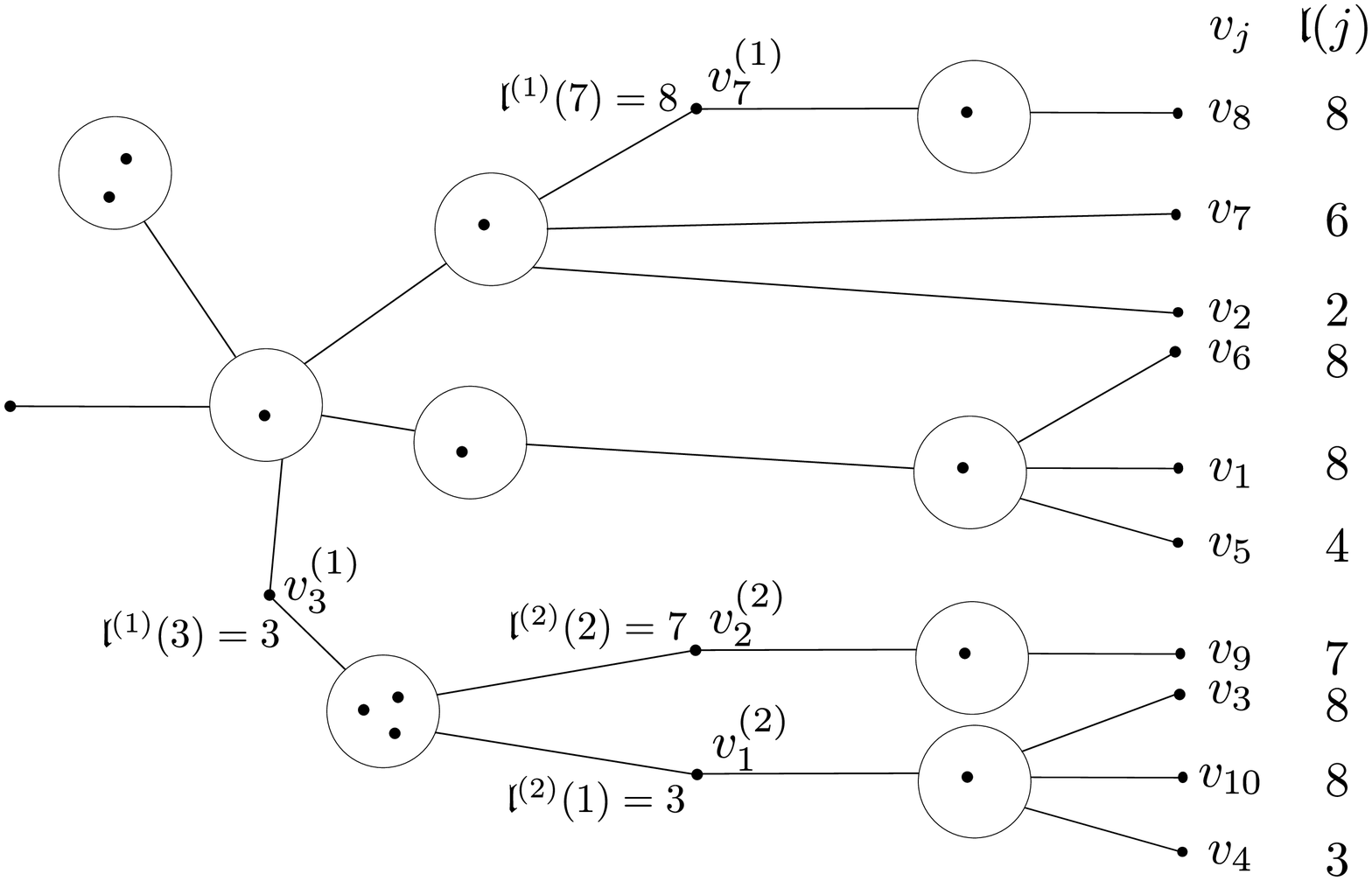}
        \caption{A labeling $\mathfrak{l}$ of the leaves of a $\bullet$-cluster $C \in \cllks$, $c=7$} \label{bul_fig04}
\end{figure}

Next, we assign to every point of $\mathcal{U}^\bullet_{\ell, k, \tau, \mathfrak{l}}$ an element in a product of Banach manifolds so that this choice is coherent with respect to the product structure of the corners of $\mathcal{C}\ell^\bullet_{\ell, k, \tau, \mathfrak{l}}$. This will later encode the choice of a Morse function plus a $\omega$-tamed almost complex structure on the target space at every point of a $\bullet$-cluster.

Let $\mathcal{M} \times \mathcal{J}$ be a product of Banach spaces with $\pi_\mathcal{M}$ and $\pi_\mathcal{J}$ the
projections on the first and second factor, respectively.

\begin{defi} \label{bul_pert}
Let $\mathcal{V}$ be a coherent system of ends and $\mathcal{S}$ be a coherent system of strip-like ends over the disks (see \cite{Sei}, \cite{W1}). A coherent system of perturbations vanishing on $\mathcal{V}$ and $\mathcal{S}$ is a collection of maps $P = \{ \mathcal{U}^\bullet_{\ell, k, \tau,\mathfrak{l}} \overset{p_{\ell,k,\tau,\mathfrak{l}}}{\rightarrow} \mathcal{M} \times \mathcal{J} \}_{\ell,k \geq 0, \tau \hspace{0.15cm}(\ell,k)-\text{shuffle}, \mathfrak{l} \in \mathfrak{L}}$ such that
\begin{enumerate}
        \item $p_{\ell,k,\tau, \mathfrak{l}}$ is piecewise smooth,

        \item $\pi_\mathcal{J} \circ p_{\ell,k,\tau, \mathfrak{l}} \equiv 0$ on $\mathcal{S}_{\ell,k,\tau}$, on the boundaries of the $\bullet$-clusters and on the real locus of ghost disk components,

        \item $\pi_\mathcal{M} \circ p_{\ell,k,\tau, \mathfrak{l}} \equiv 0$ on $\mathcal{V}_{\ell,k}$ and $\pi_\mathcal{M} \circ p_{\ell,k,\tau, \mathfrak{l}} \equiv 0$ over every line $l$ such that $\mathfrak{l}(l) \neq c+1$,

        \item $p_{\ell,k,\tau, \mathfrak{l}}$ is invariant under deformations and real involutions over ghost disk components, 

        \item for every irreducible component $(C^{(i)},\tau^{(i)})$ of $(C,\tau)$, $p_{\ell,k,\tau, \mathfrak{l}}|_{C^{(i)}} = p_{\ell^{(i)},k^{(i)},\tau^{(i)}, \mathfrak{l}^{(i)}}|_{C^{(i)}}$. Therefore, $p_{\ell,k,\tau, \mathfrak{l}}$ invariant over changes of $(C,\tau)$ which preserve $(C^{(i)},\tau^{(i)})$. 

\end{enumerate}

\end{defi}

As in the $\otimes$ case, there are no obtsructions to such a choice.

\begin{lem}
        Given coherent systems of ends $\mathcal{V}$ and $\mathcal{S}$, there exists coherent systems of perturbations vanishing on $\mathcal{V}$ and $\mathcal{S}$.
\end{lem}

Moreover, again, a Baire subset of perturbations over the boundary components extends to a Baire subset of perturbation over the whole moduli.


\section{Orientations on $\{\cllks\}_{\ell,k \geq 0}$}

As before, identify the complex marked disks with upper hemispheres of real marked spheres and take $\frac{\partial}{\partial x_j}$, $1\leq j \leq \ell$, and $\mathcal{O}_{z_h} = \frac{\partial}{\partial Re(z_h)} \wedge \frac{\partial}{\partial Im(z_h)} $, $0\leq h \leq k$, to be defined by the corresponding variations of the positions of the markings.

For $k \geq 1$, let $\mathcal{O}_\ell = \frac{\partial}{\partial x_1} \wedge \ldots \wedge
\frac{\partial}{\partial x_\ell}$. Then the proofs of lemmata \ref{klkI}, \ref{klkII} and \ref{klkIII} show that $\mathcal{O}_{\ell,k} = \mathcal{O}_\ell \wedge \mathcal{O}_{z_2} \wedge \ldots \wedge \mathcal{O}_{z_k}$ on $\klk$ defines an orientation on $\klk^\bullet$ that restricts as $(-1)^p \iota_* \mathcal{O}_{\ell,k}= (-1)^p \frac{\partial}{\partial x_{p(1)}} \wedge \ldots \wedge \frac{\partial}{\partial x_{p(\ell)}} \wedge \bigwedge_{h=2}^{k} \mathcal{O}_{z_h} =  \frac{\partial}{\partial x_{1}} \wedge \ldots \wedge \frac{\partial}{\partial x_{\ell}} \wedge \bigwedge_{h=2}^{k} \mathcal{O}_{z_h}$ on the disks having ordering $p$. We extend it to an orientation $\mathcal{O}^\bullet_{\ell,k}$ on $\cllks$.

For $k = 0$ and $\ell = 1$, we take $\mathcal{O}_{\ell,k}^\bullet = \frac{\partial}{\partial t}$ to be the field that generates the positive (away from the root) translation over $\overline{\R}$.

Next we will compare the product reference orientation on a cluster with one breaking with that induced by the reference orientation of its glued family.

Let $C =  C^{(1)} \bigcup C^{(2)}$ be a concatenation with $C^{(1)} \in \mathcal{C}\ell_{\ell^{(1)},k^{(1)}}^\bullet$, $C^{(2)} \in \mathcal{C}\ell_{\ell^{(2)},k^{(2)}}^\bullet$, $\ell^{(2)}\geq 1$, smooth and $p$ be an ordering of its leaves. Then the ordering $p$ on the leaves of $C$ naturally induces an ordering $p^{(2)}$ on the leaves of $C^{(2)}$. By the convention of section \ref{sect_klk}, we consider on the root of $C^{(2)}$, seen as a leaf of $C^{(1)}$, the minimal value $p_{min}$ of $p$ on the leaves of $C^{(2)}$ so that $C^{(1)}$ also inherits an ordering $p^{(1)}$ of its leaves (see figure \ref{bul_fig05}).

\begin{figure}[h] 
        \centering 
	\includegraphics[width=120mm]{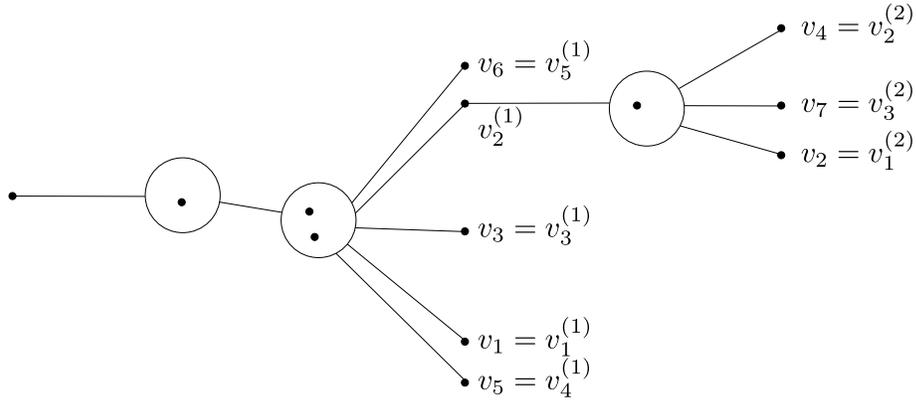}
        \caption{Orderings on a concatenation } \label{bul_fig05}
\end{figure}

For convenience, we write $p= \sigma \circ p^{(1)} \circ p^{(2)}$, where $p^{(2)}$ orders the leaves of $C^{(2)}$, acting trivially on the leaves of $C^{(1)}$, $p^{(1)}$ orders the leaves of $C^{(1)}$, shifting the leaves of $C^{(2)}$ as $x_{p_{min}}^{(1)}$, and $\sigma$ shuffles up some leaves of $C^{(2)}$ through those of $C^{(1)}$ (see figure \ref{bul_fig06}). Note that $\sigma$ is nontrivial exactly when the leaves of $C^{(2)}$ have nonconsecutive numbers and that in any case, $\sigma(j)= j$ for $j \leq p_{min}$.

\begin{figure}[h] 
        \centering 
	\includegraphics[width=120mm]{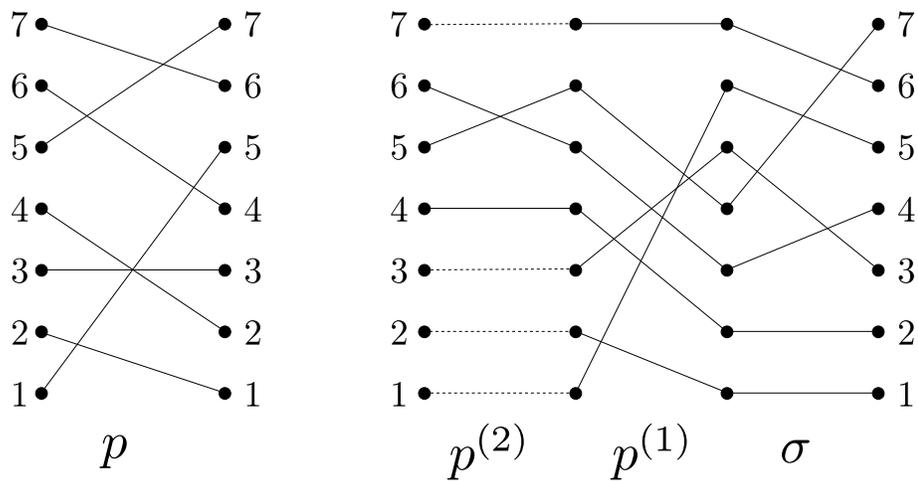}
        \caption{Decomposition of the ordering on a concatenation } \label{bul_fig06}
\end{figure}

We can therefore consider $C$ as an element of the product $\mathcal{C}\ell_{\ell^{(1)},k^{(1)}}^\bullet \times \mathcal{C}\ell_{\ell^{(2)},k^{(2)}}^\bullet$. Then, as in lemma \ref{ori}, we get

\begin{prp} \label{bul_out}
	Let $\ell^{(1)} \geq 1$, $\ell^{(2)} \geq 0$ and $k^{(1)}, k^{(2)} \geq 0$ such that $(\ell^{(2)},k^{(2)})\neq (0,0)$. Then let $\ell = \ell^{(1)} + \ell^{(2)} -1$, $k = k^{(1)} + k^{(2)}$, $C^{(1)} \in \mathcal{C}\ell_{\ell^{(1)},k^{(1)}}^\bullet$, $C^{(2)} \in \mathcal{C}\ell_{\ell^{(2)},k^{(2)}}^\bullet$ smooth, $C =  C^{(1)} \bigsqcup C^{(2)} \big{/}_{v^{(1)}_{p_{min}} \sim v^{(2)}_0}$, that is, $C$ is the concatenation of $C^{(1)}$ and $C^{(2)}$ on the $p_{min}^{th}$ leaf of $C^{(1)}$.  Let $p$ be an ordering of the $\ell$ leaves of $C$ such that $p_{min}$ is the minimal value of $p$ on the leaves of $C^{(2)}$ and take $\sigma$ to be the associated shuffle. Then

\[
	\partial_{\ddel_C} \mathcal{O}^\bullet_{\ell,k} = (-1)^{(\ell^{(1)} - p_{min})\ell^{(2)} + (p_{min} -1)} (-1)^{\sigma} \mathcal{O}^\bullet_{\ell^{(1)},k^{(1)}} \wedge
\mathcal{O}^\bullet_{\ell^{(2)},k^{(2)}}.
\]
\end{prp}

\begin{proof}
Proposition \ref{ori} shows that the formula holds whenever $\sigma$ is trivial, that is, the leaves of $C^{(2)}$ have successive numbers according to $p$. It remains to show that it still holds when $p$, and thus $\sigma$, shuffles nontrivially the leaves of $C^{(2)}$ and the leaves of $C^{(1)}$ that are not breakings. This can happen only when $\ell^{(1)} \geq 2$ and $\ell^{(2)} \geq 2$.

        In this case, $C$ is a stable $\bullet$-cluster with one breaking, and therefore lies in a $1$-corner $F$ of $\cllks$. We compute using coordinates in a neighborhood $F$:

\begin{align*}
& \frac{\partial}{\partial n_F} \wedge (-1)^{ (\ell^{(1)} -p_{min}) \ell^{(2)}+ (p_{min}-1)} (-1)^{\sigma} \mathcal{O}^\bullet_{\ell^{(1)},k^{(1)}} \wedge \mathcal{O}^\bullet_{\ell^{(2)},k^{(2)}} \\
&= \bigwedge_{h=2}^{k^{(1)}} \mathcal{O}_{z^{(1)}_h} \wedge \bigwedge_{h=2}^{k^{(2)}} \mathcal{O}_{z^{(2)}_h} \wedge (-1)^{ (\ell^{(1)} - p_{min}) \ell^{(2)}+ (p_{min}-1)} (-1)^{\sigma} \frac{\partial}{\partial n_F} \wedge \cO_{\ell^{(1)}} \wedge \cO_{\ell^{(2)}}\\
&= \bigwedge_{h=2}^{k^{(1)}} \mathcal{O}_{z^{(1)}_h} \wedge \bigwedge_{h=2}^{k^{(2)}} \mathcal{O}_{z^{(2)}_h} \wedge (-1)^{ (\ell^{(1)} - p_{min}) \ell^{(2)}+ (p_{min}-1)} (-1)^{\sigma}\\
& \frac{\partial}{\partial n_F} \wedge \frac{\partial}{\partial x_1^{(1)}} \wedge \ldots \wedge \frac{\partial}{\partial x_{p_{min}}^{(1)}} \wedge \ldots \wedge \frac{\partial}{\partial x_{\ell^{(1)}}^{(1)}} \wedge \frac{\partial}{\partial x_1^{(2)}} \wedge \ldots \wedge
\frac{\partial}{\partial x_{\ell^{(2)}}^{(2)}}\\
&= \bigwedge_{h=2}^{k^{(1)}} \mathcal{O}_{z^{(1)}_h} \wedge \bigwedge_{h=2}^{k^{(2)}} \mathcal{O}_{z^{(2)}_h} \wedge (-1)^{ (\ell^{(1)} - p_{min}) \ell^{(2)}+ (p_{min}-1)} (-1)^{\sigma}\\
&- \frac{\partial}{\partial Im(z_1^{(2)})} \wedge \frac{\partial}{\partial x_{1}} \wedge \ldots \wedge \frac{\partial}{\partial x_{p_{min}-1}} \wedge \frac{\partial}{\partial Re(z_1^{(2)})} \wedge \frac{\partial}{\partial x_{\sigma(p_{min}+\ell^{(2)})}} \wedge \ldots \wedge
\frac{\partial}{\partial x_{\sigma(\ell^{(2)}+\ell^{(1)}-1)}} \wedge \\
& \frac{\partial}{\partial x_{\sigma(p_{min})}} \wedge \ldots \wedge
\frac{\partial}{\partial x_{\sigma(p_{min}+\ell^{(2)}-1)}}\\
&= \bigwedge_{h=2}^{k^{(1)}} \mathcal{O}_{z^{(1)}_h} \wedge \bigwedge_{h=1}^{k^{(2)}} \mathcal{O}_{z^{(2)}_h} \wedge (-1)^{\sigma} \frac{\partial}{\partial x_{1}} \wedge \ldots \wedge \frac{\partial}{\partial x_{p_{min}-1}} \wedge \frac{\partial}{\partial x_{\sigma(p_{min})}} \wedge \ldots \wedge
\frac{\partial}{\partial x_{\sigma(p_{min}+\ell^{(2)}-1)}} \wedge \\
& \frac{\partial}{\partial x_{\sigma(p_{min}+\ell^{(2)})}} \wedge \ldots \wedge
\frac{\partial}{\partial x_{\sigma(\ell^{(2)}+\ell^{(1)}-1)}} \\
&= \bigwedge_{h=2}^{k^{(1)}} \mathcal{O}_{z^{(1)}_h} \wedge \bigwedge_{h=1}^{k^{(2)}} \mathcal{O}_{z^{(2)}_h} \wedge \frac{\partial}{\partial x_{1}} \wedge \ldots \wedge \frac{\partial}{\partial x_{p_{min}-1}} \wedge \\
& \frac{\partial}{\partial x_{p_{min}}} \wedge \ldots \wedge \frac{\partial}{\partial x_{p_{min}+\ell^{(2)}-1}} \wedge \frac{\partial}{\partial x_{p_{min}+\ell^{(2)}}} \wedge \ldots \wedge
\frac{\partial}{\partial x_{\ell^{(2)}+\ell^{(1)}-1}} \\
&= \bigwedge_{h=1}^{k} \mathcal{O}_{z_h} \wedge \cO_{\ell} \\
&= \cO_{\ell,k}^\bullet. \\
\end{align*}

\end{proof}

\appendix
\chapter{Smoothing $(\cllkr)^{PDIFF}$ charts} \label{appe_A}

The goal of this technical section is to establish lemma \ref{smoo_lemm} without assuming that $\klk$ is a MWEC, but rather decomposing neighborhoods of the corners using the simple ratio charts on $\klk$.

\begin{lemnn}[\ref{smoo_lemm}]
There exists a piecewise smooth isomorphism
\[
(\cllkr)^{PDIFF} = col(\klk) \rightarrow \klk
\]
sending $\mathcal{C}\ell_{\ell,k,T}^\otimes$ to $K_{\ell,k,T}$ for every combinatorial type $T$.
\end{lemnn}

We begin by defining an identification near any maximal corner:

\begin{lem}
For every maximal combinatorial type $T^{max}$, there is a piecewise smooth isomorphism from a neighborhood of $K_{\ell,k,T^{max}} \times X^{[0,1]}(T^{max}) \subset col(\klk)$ to a neighborhood $V$ of $K_{\ell,k,T^{max}} \in \klk$, or equivalently, to a neighborhood $V$ of $X^{0}(T^{max}) \in X^{\R_+}(T^{max})$.
\end{lem}

That is, we will decompose a neighborhood $V$ of $X^{0}(T^{max}) \in X^{\R_+}(T^{max})$ into $V= \underset{T\leq T^{max}}{\bigcup} V_T \bigcap V$ where $V_T \bigcap V$ is isomorphic to a neighborhood of $K_{\ell,k,T^{max}} \times X^{[0,1]}(T)$ in $K_{\ell,k,\geq T} \times X^{[0,1]}(T)$. This will be based on the labeling reduction operation.

\begin{proof}
{\bf 1) Building the $V_{T^{max}}$ cell.} Set $0 < \epsilon \leq 1$ and $V \supset X^{[0,\epsilon]}(T^{max})$ an open neighborhood of $X^{[0,\epsilon]}(T^{max})$ small enough so that it is contained in the $\psi_{T^{max}}$ chart. Note that, canonically, $X^{[0,\epsilon]}(T^{max}) \cong X^{[0,1]}(T^{max})$ and thus $K_{\ell,k,T^{max}} \times X^{[0,1]}(T^{max})$ gets identified with the cell $V_{T^{max}} = X^{[0,\epsilon]}(T^{max}) \subset V$.

{\bf 2) Building a neighborhood of the $\geq T$ corner.} For any $T \leq T^{max}$, recall that a labeling $X$ on $T^{max}$ determines one on $T$, denoted by $X|_{T}$, by simply taking $X|_{T}(l) = X(l)$. For $l \in E(T)$, let $X^{[0,\epsilon]}(T)$ be the labelings on $T$ taking values in $[0,\epsilon]$. Denote by $ X_T^{[0,\epsilon]}(T^{max}) = |_T^{-1}(X^{[0,\epsilon]}(T)) \subset X^{\R_+}(T^{max})$ the labelings on $T^{max}$ that restric to $T \leq T^{max}$ in this range. Near $K_{\ell,k,T^{max}}$, $X_T^{[0,\epsilon]}(T^{max})$ is as a neighborhood of the corner $K_{\ell,k,\geq T}$. Remark that from this definition, $T^{(1)} < T^{(2)} \leq T^{max}$ implies that $X_{T^{(2)}}^{[0,\epsilon]}(T^{max}) \subset X_{T^{(1)}}^{[0,\epsilon]}(T^{max})$ since $(X|_{T^{(2)}})|_{T^{(1)}} = X|_{T^{(1)}}$.

{\bf 3) Building the $V_{T_0}$ cell.} Now set $V_{T_0} = \overline{( \underset{|T| = 1}{\bigcup} X_T^{[0,\epsilon]}(T^{max}) )^c}$. Then $V_{T_0} \bigcap V$ could be thought of as a cell of $V$ that is complementary to the above corner neighborhoods. Note that the labeling $X^\epsilon(T^{max})$ (taking value $\epsilon$ on each edge of $T^{max}$) belongs to $V_{T_0}$ since the sequence $\{X^{\epsilon+\frac{1}{n}}(T^{max})\}_{n\in \N}$ lies in $( \underset{|T| = 1}{\bigcup} X_T^{[0,\epsilon]}(T^{max}) )^c$.

More generally, $X_T^{[0,\epsilon]}(T^{max}) \bigcap V_{T_0} = X_T^{\epsilon}(T^{max}) \bigcap V_{T_0}$, the labelings (of $V_{T_0}$) reducing to $T$ with values $\epsilon$. This is clear when $|T|=1$, and otherwise notice that $V_{T_0} = \overline{( \underset{T}{\bigcup} X_T^{[0,\epsilon]}(T^{max}) )^c}$. $X_T^{\epsilon}(T^{max}) \bigcap V_{T_0} \neq \emptyset$ since it contains $X^\epsilon(T^{max})$, but in fact, is seen to be isomorphic to $K_{\ell,k,\geq T}$ near $K_{\ell,k, T^{max}}$. 

\begin{lem} \label{lemm_retr_1}
There is an isomorphism from a neighborhood of $X^{0}(T^{max})$ in $X^{\R_+}(T^{max})$ to a neighborhood of  $X^\epsilon(T^{max})$ in $V_{T_0}$.
\end{lem}

From proposition \ref{neighk}, we get that

\begin{cor} \label{coro_retr_1}
There is an isomorphism from a neighborhood of $K_{\ell,k,T^{max}}$ in $\klk$ to a neighborhood of $X^\epsilon(T^{max})$ in $V_{T_0}$.
\end{cor}

\begin{proof}
Define a smooth map
\begin{diagram}
X^{\R_+}(T^{max}) && \rTo{\chi} && V_{T_0} \\
X && \rMapsto && \chi(X)(l)= \epsilon + X(l).
\end{diagram}

It is obvious that this map has the desired properties. It is not hard to see that the labelings on $T^{max}$ having $0$ values exactly on $T$ (corresponding to $K_{\ell,k,\geq T}$) correspond under $\chi$ to $X^{[0,\epsilon]}_T(T^{max}) \bigcap V_{T_0} = X^{\epsilon}_T(T^{max}) \bigcap V_{T_0}$: Let $X \in X^{\R_+}(T^{max})$ that is $0$ over $T$ and nonzero over $T^{max} \backslash T$, then an edge $l \in E(T)$ gets label $\epsilon$.

\end{proof}

{\bf 4) Building the $V_{T}$ cell.} Given $T \leq T^{max}$, let $\underset{i}{\bigcup} T^{(i)} =T^{max} \backslash T$ be the decomposition into connected maximal subtrees. Define $V_{T^{(i)}_0} \subset X^{\R_+}(T^{(i)})$ as in step $3$, so it is made of labelings $X^{(i)}$ over every component subtree $T^{(i)}$ of $T^{max}\backslash T$ such that $X^{(i)}$ does not restrict to a $[0,\epsilon]$ labeling on any of the subtrees of $T^{(i)}$. Lemma \ref{lemm_retr_1} shows that near $K_{\ell^{(i)},k^{(i)}, T^{(i)}}$, the $\chi^{(i)}$ maps will retract the $K_{\ell^{(i)},k^{(i)},\geq T^{(i)}_0}$ onto $V_{T^{i}_0}$. Consequently, near $K_{\ell,k,T^{max}}$, $\underset{i}{\prod} \chi^{(i)}$ retracts the $K_{\ell,k,\geq T} = \underset{i}{\prod} K_{\ell^{(i)},k^{(i)},\geq T^{(i)}_0}$ corner onto $\underset{i}{\prod} V_{T^{(i)}_0}$. As in the proof of \ref{lemm_retr_1}, the labelings $X$ with value $0$ on $T$ and restricting to $\epsilon$-labelings over some subtrees of the $X^{(i)}$ constitute a boundary stratum of $\underset{i}{\prod} V_{T^{(i)}_0}$.

Combining the above labels with labels over $T$, $\underset{i}{\prod} V_{T^{(i)}_0} \times X^{\R_+}(T)$ naturally sits in $X^{\R_+}(T^{max})$ and we can therefore define $V_{T} = (\underset{i}{\prod} V_{T^{(i)}_0} \times X^{\R_+}(T)) \bigcap X_T^{[0,\epsilon]}(T^{max})$. Through the restriction map associated with $T$, $V_{T}$ is seen to be isomorphic to $\underset{i}{\prod} V_{T^{(i)}_0} \times X^{[0,\epsilon^{N_l}]}(T)$. Therefore, $V_{T}$ is isomorphic to a neighborhood of $K_{\ell,k,T^{max}} \times X^{[0,1]}(T)$ in $K_{\ell,k,\geq T} \times X^{[0,1]}(T)$.

{\bf 5) Attaching the $V_{T}$ cells.} 

As seen above, for every $T < T' \leq T^{max}$, let $T' \backslash T = \underset{i}{\bigcup} T^{(i)} \backslash T^{'(i)}$, then from step $4$,
\[
(\underset{i}{\prod}(X_{T^{(i)} \backslash T^{'(i)}}^\epsilon(T^{(i)}) \bigcap V_{T^{(i)}_0}) \times X^{\R_+}(T)) \bigcap X_T^{[0,\epsilon]}(T^{max})
\]
is a boundary stratum of $V_{T}$. It coincides with the boundary stratum 
\[
(\underset{i}{\prod} V_{T^{'(i)}_0} \times  X^{\R_+}(T')) \bigcap X_{T'}^{[0,\epsilon]}(T^{max}) \bigcap X_{T^{(i)} \backslash T^{'(i)}}^\epsilon(T^{(i)})
\]
of $V_{T'} = (\underset{i}{\prod} V_{T^{'(i)}_0} \times X^{\R_+}(T')) \bigcap X_{T'}^{[0,\epsilon]}(T^{max})$. More simply, those are labelings $[0,\epsilon]$-reducing on $T$ such that in the complement of $T$, they reduce to $\epsilon$-labelings exactly on $T'\backslash T$. It is not hard to see that those correspond under the $\chi$ maps to the same identifications as in the definition of $col(\klk)$.

\end{proof}

Now that we can set piecewise smooth identifications between $(\cllkr)^{PDIFF}$ and $\klk$ in disjoints neighborhoods of the maximal corner, there is no obstruction to make them fit together using the $\psi_T$ charts from a neighborhood of $K_{\ell,k, T} \times \{0\} = K_{\ell,k, T} \times X^{0}(T)$ in $K_{\ell,k, T} \times X^{\R_+}(T)$ to a neighborhood $\nu(K_{\ell,k, T})$ of $K_{\ell,k, T}$ in $\klk$.

\chapter{Smoothing $(\qcllkr)^{PDIFF}$ charts} \label{appe_B}

Using techniques very similar to those of appendix \ref{appe_A}, we want to establish the following result:

\begin{lemnn}[\ref{qsmoo_lemm}]
There exists a piecewise smooth isomorphism
\[
(\qcllkr)^{PDIFF} = col(\qklk) \rightarrow \qklk
\]
sending $\mathcal{QC}\ell_{\ell,k,T}^\otimes$ to $Q_{\ell,k,T}$ for every combinatorial type $T$.
\end{lemnn}

The main difficulty is that we must now keep track of the balancing condition on the labelings on the combinatorial trees of the quilted disks. However, we can still reduce a balanced labeling on a given tree to a balanced labeling on a smaller tree, and this will be enough to decompose neighborhoods of the boundary components of $\qklk$.

We begin by defining an identification near any maximal corner:

\begin{lem}
For every maximal combinatorial type $T^{max}$, there is a piecewise smooth isomorphism from a neighborhood of $Q_{\ell,k,T^{max}} \times X^{[0,1]}(T^{max}) \subset col(\qklk)$ to a neighborhood $V$ of $Q_{\ell,k,T^{max}} \in \qklk$, or equivalently, to a neighborhood $V$ of $X^{0}(T^{max}) \in X^{\R_+}(T^{max})$.
\end{lem}

That is, we will decompose, as in the unquilted case, a neighborhood $V$ of $X^{0}(T^{max}) \in X^{\R_+}(T^{max})$ into $V= \underset{T\leq T^{max}}{\bigcup} V_T \bigcap V$ where $V_T \bigcap V$ is isomorphic to a neighborhood of $Q_{\ell,k,T^{max}} \times X^{[0,1]}(T)$ in $Q_{\ell,k,\geq T} \times X^{[0,1]}(T)$. This will be based on the balanced labeling reduction operation.

We first set, for every edge $l \in E(T^{max})$ below the colored vertices of $T^{max}$, $b_l$ to be the number of (finite) edges below $l$, including itself. 
Say $M_l = \frac{1}{2^{b_l}}$ if $l$ is below the colored vertices without touching one of them, $M_l = \frac{1}{2^{b_l-1}}$ if it is just below a colored vertex, and $M_l = 1$ if $l$ is above the colorings. If $T_v$ is a linear subtree from the root to a colored vertex $v$, is not hard to see from this definition that $\underset{l \in T_v}{\sum} M_l =1$.

\begin{proof}
{\bf 1) Building the $V_{T^{max}}$ cell.} Set $0 < \epsilon \leq 1$ and $V \supset X^{[0,\epsilon^{M_l}]}(T^{max})$ an open neighborhood of $X^{[0,\epsilon^{M_l}]}(T^{max})$ small enough so that it is contained in the $\psi_{T^{max}}$ chart. Note that, canonically, $X^{[0,\epsilon^{M_l}]}(T^{max}) \cong X^{[0,1]}(T^{max})$ and thus $Q_{\ell,k,T^{max}} \times X^{[0,1]}(T^{max})$ gets identified with the cell $V_{T^{max}} = X^{[0,\epsilon^{M_l}]}(T^{max}) \subset V$. To simplify notation, we will refer to $X^{[0,\epsilon^{M_l}]}(T^{max})$ as $X^{[0,\epsilon]}(T^{max})$.

Remark that unlike in the preceding section, putting $\epsilon$ labels on every edge of $T^{max}$ does not generate a balanced labeling, as does assigning $\epsilon^{M_l}$ labels.

{\bf 2) Building a neighborhood of the $\geq T$ corner.} For any colored tree $T \leq T^{max}$, recall that a labeling $X$ on $T^{max}$ determines one on $T$, denoted by $X|_{T}$, by taking $X|_{T}(l) = X(l)\cdot \prod_{l'} X(l') \prod_{l''} X(l'')$, where the product is over contracted edges $l'$ below $l$ that are connected to $l$ by contracted edges only or contracted edges $l''$ connecting $l$ to a colored vertex through contracted edges only.

For $l \in E(T)$, let $N_l = M_{l} + \underset{l'}{\sum} M_{l'} + \underset{l''}{\sum} M_{l''}$  and $X^{[0,\epsilon]}(T)$ be the labelings on $T$ taking values in $[0,\epsilon^{N_l}]$ over $l$. Denote by $ X_T^{[0,\epsilon]}(T^{max}) = |_T^{-1}(X^{[0,\epsilon]}(T)) \subset X^{\R_+}(T^{max})$ the labelings on $T^{max}$ that restrict to $T \leq T^{max}$ in this range. Near $Q_{\ell,k,T^{max}}$, $X_T^{[0,\epsilon]}(T^{max})$ is as a neighborhood of the corner $Q_{\ell,k,\geq T}$.

Remark that from this definition, $T^{(1)} < T^{(2)} \leq T^{max}$ implies that $X_{T^{(2)}}^{[0,\epsilon]}(T^{max}) \subset X_{T^{(1)}}^{[0,\epsilon]}(T^{max})$ since $(X|_{T^{(2)}})|_{T^{(1)}} = X|_{T^{(1)}}$.

{\bf 3) Building the $V_{T_0}$ cell.} Now set $V_{T_0} = \overline{( \underset{|T| = 1}{\bigcup} X_T^{[0,\epsilon]}(T^{max}) )^c}$. Then $V_{T_0} \bigcap V$ could be thought of as a cell of $V$ that is complementary to the above corner neighborhoods. Note that the labeling $X^\epsilon(T^{max})$ (taking value $\epsilon^{M_l}$ on $l \in E(T^{max})$) belongs to $V_{T_0}$ since the sequence $\{X^{\epsilon+\frac{1}{n}}(T^{max})\}_{n\in \N}$ lies in $( \underset{|T| = 1}{\bigcup} X_T^{[0,\epsilon]}(T^{max}) )^c$.

More generally, $X_T^{[0,\epsilon]}(T^{max}) \bigcap V_{T_0} = X_T^{\epsilon}(T^{max}) \bigcap V_{T_0}$, the labelings (of $V_{T_0}$) reducing to $T$ with values $\epsilon^{N_l}$. This is clear when $|T|=1$, and otherwise notice that $V_{T_0} = \overline{( \underset{T}{\bigcup} X_T^{[0,\epsilon]}(T^{max}) )^c}$. $X_T^{\epsilon}(T^{max}) \bigcap V_{T_0} \neq \emptyset$ since it contains $X^\epsilon(T^{max})$, but in fact, is seen to be isomorphic to $K_{\ell,k,\geq T}$ near $K_{\ell,k, T^{max}}$. 

\begin{lem} \label{lemm_retr}
There is an isomorphism from a neighborhood of $X^{0}(T^{max})$ in $X^{\R_+}(T^{max})$ to a neighborhood of  $X^\epsilon(T^{max})$ in $V_{T_0}$.
\end{lem}

From proposition \ref{neighk}, we get that

\begin{cor} \label{coro_retr}
There is an isomorphism from a neighborhood of $Q_{\ell,k,T^{max}}$ in $\qklk$ to a neighborhood of $X^\epsilon(T^{max})$ in $V_{T_0}$.
\end{cor}

\begin{proof}
Define a smooth map
\begin{diagram}
X^{\R_+}(T^{max}) && \rTo{\chi} &&&& V_{T_0} \\
X && \rMapsto &&&& \chi(X)(l)= \begin{cases} \epsilon + X(l) &\mbox{if } l \hspace{0.15cm} \text{above color} \\ (\epsilon^{M_l} + X(l)^2) \frac{\epsilon^{M_{l'}}}{\epsilon^{M_{l'}} + X(l')^2 } \frac{\epsilon^{M_{l''}}}{\epsilon^{M_{l''}} + X(l'')^2} (1+ \delta_{l''}Y) &\mbox{if }  l \hspace{0.15cm} \text{below color} \end{cases}
\end{diagram}
where $l'$ (resp. $l''$) is the edge just below $l$ (resp. just above $l$ and that touches a colored vertex), put $X(l')=0$ (resp. $X(l'')=0$) if this quantity is not defined, $\delta_{l''}=1$ if $l''$ is just above $l$ and that touches a colored vertex, $0$ otherwise and $Y$ is the product of the labels over the linear subtree $T_v$ from a colored vertex $v$ to the root. 

First, remark that $\chi(X^0) = X^{\epsilon^{M_l}}$.

$\chi$ is well defined: It is straightforward to see that $\chi(X)$ is again balanced, with the product of the labels from a colored vertex to the root being equal to $1+Y$.

Now suppose $\chi(X) \notin V_{T_0}$, then $\chi(X)|_T$ must lie in $X_T^{[0,\epsilon^{N_l}[}(T)$ for some $T \leq T^{max}$. Since $X_{T^{(2)}}^{[0,\epsilon^{N_l}[}(T^{max}) \subset X_{T^{(1)}}^{[0,\epsilon^{N_l}[}(T^{max})$ whenever $T^{(1)} < T^{(2)}$, we can assume that $|T|=1$. This implies that on the only edge $l \in E(T)$, we must have
\[
\chi(X)|_T(l) = \begin{cases} \chi(X)(l) < \epsilon &\mbox{if } l \hspace{0.15cm} \text{above color} \\ \chi(X)(l) \cdot \prod_{l'} \chi(X)(l') \cdot \prod_{l''} \chi(X)(l'') < \epsilon^{N_l} = \epsilon &\mbox{if }  l \hspace{0.15cm} \text{below color} \end{cases}
\]
where the products are over the edges $l'$ below $l$ and $l''$ above $l$ but below the colorings. In the first case, this amounts to $\epsilon + X(l) < \epsilon$ while in the second case, it means $\epsilon (1+Y) < \epsilon$. Both are contradictions since $X(l) \in \R_+$, $Y \in \R_+$.

$\chi$ is injective: $\chi(X_1)=\chi(X_2)$ directly implies that $X_1(l)=X_2(l)$ if $l$ touches the root edge and $X_1(l')=X_2(l')$ implies $X_1(l)=X_2(l)$ when $l$ is just above $l'$. Therefore, $X_1 = X_2$.


It is not hard to see that the labelings on $T^{max}$ having $0$ values exactly on $T$ (corresponding to $K_{\ell,k,\geq T}$) correspond under $\chi$ to $X^{[0,\epsilon]}_T(T^{max}) \bigcap V_{T_0} = X^{\epsilon}_T(T^{max}) \bigcap V_{T_0}$: Let $X \in X^{\R_+}(T^{max})$ that is $0$ over $T$ and nonzero over $T^{max} \backslash T$,
\begin{itemize}
\item then an edge $l \in E(T)$ above the colorings gets label $\epsilon$ so  $X|_T(l) = \epsilon$, otherwise
\item an edge $l \in E(T)$ just above $l' \in E(T)$ gets label $\epsilon$. Therefore, $X|_T(l) = \epsilon$,
\item an edge $l \in E(T)$ not directly under an edge having a colored vertex, and just above $l_1' \in E(T^{max} \backslash T)$ gets label $\epsilon^{M_l} \frac{\epsilon^{M_{l'}}}{\epsilon^{M_{l_1'}} + X(l_1')^2}$, while $l_1'$ just above $l_2' \in E(T^{max} \backslash T)$ gets label $(\epsilon^{M_{l_1'}} + X(l_1')^2) \frac{\epsilon^{M_{l_2'}}}{\epsilon^{M_{l_2'}} + X(l_2')^2}$ and so on until we reach $l_j' \in E(T^{max} \backslash T)$ just above an edge of $T$, so it gets label $\epsilon^{M_{l_j'}}+ X(l_j')^2$. Therefore, $X|_T(l) = \epsilon^{N_l}$,
\item an edge $l \in E(T)$ directly under a contracted edge $l''$ having a colored vertex, and just above $l_1' \in E(T^{max} \backslash T)$ gets label $\epsilon^{M_l} \frac{\epsilon^{M_{l'}}}{\epsilon^{M_{l_1'}} + X(l_1')^2} \frac{\epsilon^{M_{l''}}}{\epsilon^{M_{l''}} + X(l'')^2}$, while $l_1'$ just above $l_2' \in E(T^{max} \backslash T)$ gets label $(\epsilon^{M_{l_1'}} + X(l_1')^2) \frac{\epsilon^{M_{l_2'}}}{\epsilon^{M_{l_2'}} + X(l_2')^2}$ and so on until we reach $l_j' \in E(T^{max} \backslash T)$ just above an edge of $T$, so it gets label $\epsilon^{M_{l_j'}}+ X(l_j')^2$. Furthermore, $l''$ gets label $\epsilon^{M_{l''}} +X(l'')^2$. Therefore, $X|_T(l) = \epsilon^{N_l}$.

\end{itemize}

\end{proof}


{\bf 4) Building the $V_{T}$ cell.} Given $T \leq T^{max}$, let $\underset{i}{\bigcup} T^{(i)} =T^{max} \backslash T$ be the decomposition into connected maximal subtrees. Define $V_{T^{(i)}_0} \subset X^{\R_+}(T^{(i)})$ as in step $3$, so it is made of labelings $X^{(i)}$ over every component subtree $T^{(i)}$ of $T^{max}\backslash T$ such that $X^{(i)}$ does not restrict to a $[0,\epsilon]$-labeling on any of the subtrees of $T^{(i)}$. Lemma \ref{lemm_retr} shows that near $K_{\ell^{(i)},k^{(i)}, T^{(i)}}$, the $\chi^{(i)}$ maps will retract the $K_{\ell^{(i)},k^{(i)},\geq T^{(i)}_0}$ onto $V_{T^{i}_0}$. Consequently, near $K_{\ell,k,T^{max}}$, $\underset{i}{\prod} \chi^{(i)}$ retracts the $K_{\ell,k,\geq T} = \underset{i}{\prod} K_{\ell^{(i)},k^{(i)},\geq T^{(i)}_0}$ corner onto $\underset{i}{\prod} V_{T^{(i)}_0}$. As in the proof of \ref{lemm_retr}, the labelings $X$ with value $0$ on $T$ and restricting to $\epsilon$-labelings over some subtrees of the $X^{(i)}$ constitute a boundary stratum of $\underset{i}{\prod} V_{T^{(i)}_0}$.

Combining the above labels with labels over $T$, $\underset{i}{\prod} V_{T^{(i)}_0} \times X^{\R_+}(T)$ naturally sits in $X^{\R_+}(T^{max})$ and we can therefore define $V_{T} = (\underset{i}{\prod} V_{T^{(i)}_0} \times X^{\R_+}(T)) \bigcap X_T^{[0,\epsilon]}(T^{max})$. Through the restriction map associated with $T$, $V_{T}$ is seen to be isomorphic to $\underset{i}{\prod} V_{T^{(i)}_0} \times X^{[0,\epsilon^{N_l}]}(T)$. Therefore, $V_{T}$ is isomorphic to a neighborhood of $K_{\ell,k,T^{max}} \times X^{[0,1]}(T)$ in $K_{\ell,k,\geq T} \times X^{[0,1]}(T)$.

{\bf 5) Attaching the $V_{T}$ cells.} 

As seen above, for every $T < T' \leq T^{max}$, let $T' \backslash T = \underset{i}{\bigcup} T^{(i)} \backslash T^{'(i)}$, then from step $4$,
\[
(\underset{i}{\prod}(X_{T^{(i)} \backslash T^{'(i)}}^\epsilon(T^{(i)}) \bigcap V_{T^{(i)}_0}) \times X^{\R_+}(T)) \bigcap X_T^{[0,\epsilon]}(T^{max})
\]
is a boundary stratum of $V_{T}$. It coincides with the boundary stratum 
\[
(\underset{i}{\prod} V_{T^{'(i)}_0} \times  X^{\R_+}(T')) \bigcap X_{T'}^{[0,\epsilon]}(T^{max}) \bigcap X_{T^{(i)} \backslash T^{'(i)}}^\epsilon(T^{(i)})
\]
of $V_{T'} = (\underset{i}{\prod} V_{T^{'(i)}_0} \times X^{\R_+}(T')) \bigcap X_{T'}^{[0,\epsilon]}(T^{max})$. More simply, those are labelings $[0,\epsilon]$-reducing on $T$ such that in the complement of $T$, they reduce to $\epsilon$-labelings exactly on $T'\backslash T$. It is not hard to see that those correspond under the $\chi$ maps to the same identifications as in the definition of $col(\klk)$.

\end{proof}

Now that we can set piecewise smooth identifications between $(\cllkr)^{PDIFF}$ and $\klk$ in disjoints neighborhoods of the maximal corner, we need to extend these globally. This is again possible using the $\psi_T$ charts from a neighborhood of $K_{\ell,k, T} \times \{0\} = K_{\ell,k, T} \times X^{0}(T)$ in $K_{\ell,k, T} \times X^{\R_+}(T)$ to a neighborhood $\nu(K_{\ell,k, T})$ of $K_{\ell,k, T}$ in $\klk$.

\setcounter{chapter}{3}

\vspace{0.2cm}
{\sc \small \noindent Département de mathématiques et de statistique, Université de Montréal, Montréal, QC H3C 3J7}\\
{\sc \small \noindent Department of Mathematics, Columbia University, New York, NY 10027}
{\small {\em E-mail address:} charest@math.columbia.edu}


%
%

\end{document}